\documentclass{amsart}

\usepackage{mathtools}
\usepackage{etex}
\reserveinserts{28}
\usepackage[all]{xy}
\usepackage{tikz}

\usepackage{amssymb}

\usepackage{stmaryrd}
\usepackage{multirow}
\usepackage{comment}
\usepackage{pictex}
\usepackage{amsmath,amscd}

\makeatletter
\newcommand{\xdashrightarrow}[2][]{\ext@arrow 
0359\rightarrowfill@@{#1}{#2}}
\newcommand{\xdashleftarrow}[2][]{\ext@arrow 
3095\leftarrowfill@@{#1}{#2}}
\newcommand{\xdashleftrightarrow}[2][]
{\ext@arrow 3359\leftrightarrowfill@@{#1}{#2}}
\def\rightarrowfill@@{\arrowfill@@\relax\relbar
\rightarrow}
\def\leftarrowfill@@{\arrowfill@@\leftarrow
\relbar\relax}
\def\leftrightarrowfill@@{\arrowfill@@
\leftarrow\relbar\rightarrow}
\def\arrowfill@@#1#2#3#4{%
  $\m@th\thickmuskip0mu\medmuskip\thickmuskip
\thinmuskip\thickmuskip
   \relax#4#1
   \xleaders\hbox{$#4#2$}\hfill
   #3$%
}
\makeatother

\usepackage{afterpage}
\usepackage{youngtab}
\usepackage{psfrag}
\usepackage[boxsize=.3 em]{ytableau}
\usepackage{enumerate}
\usepackage{pstricks}
\usepackage{verbatim}
\xyoption{matrix}
\xyoption{arrow}

\usepackage{graphics}

\newtheorem{thm}{Theorem}[section]
\newtheorem{cor}[thm]{Corollary}
\newtheorem{lem}[thm]{Lemma}
\newtheorem{prop}[thm]{Proposition}
\theoremstyle{definition}
\newtheorem{defn}[thm]{Definition}
\newtheorem{rem}[thm]{Remark}  

\newtheorem{conjecture}[thm]{Conjecture}

\newtheorem{example}[thm]{Example}
\numberwithin{equation}{section}

\newcommand{\m}{m}

\newcommand{\RR}{{\mathcal R}}

\newcommand{\To}{\longrightarrow}
\newcommand{\out}{\mathrm{out}}
\newcommand{\Invol}{\operatorname{Invol}}
\newcommand{\sign}{\operatorname{SignAct}}
\newcommand{\QR}{\mathcal Q_R}
\newcommand{\QL}{\mathcal Q_L}
\newcommand{\Gbar}{H_B}
\newcommand{\GbarO}{H_{B,0}}

\newcommand{\HB}{\mathcal H_{B}}
\newcommand{\HBan}{\mathcal H_{B,an}}
\newcommand{\HA}{\mathcal H_{A}}
\newcommand{\HAan}{\mathcal H_{A,an}}

\newcommand{\qa}{q}

\newcommand{\qb}{q}
\newcommand{\PAB}{\mathbf P}
\newcommand{\Gr}{{\mathbf {gr}}}
\newcommand{\Ss}{\mathcal S}
\newcommand{\Sseq}{\mathcal S^{\operatorname{eq}}}

\newcommand{\dd}{z}

\newcommand{\Xcheck}{{\check{\mathbb X}^\circ}}
\newcommand{\Xcheckbar}{\check{\mathbb X}}
\newcommand{\Xcheckfact}{{\check{\mathbb X}^{\circ,fact}}}

\newcommand{\Osc}{\mathrm{Osc}}

\newcommand{\HH}{\mathcal H}

\newcommand{\ta}{\tau}
\newcommand{\hbarr}{z}
\newcommand{\hbarrB}{z}

\newcommand{\C}{\mathbb{C}}
\newcommand{\Z}{\mathbb{Z}}
\newcommand{\inv}{^{-1}}

\newcommand{\x}{\times}

\newcommand{\Hom}{{\operatorname{Hom}}}

\newcommand{\opp}{\operatorname{opp}}
\newcommand{\xx}{\widetilde{\mathbf{x}}}
\newcommand{\A}{\mathcal{A}}
\newcommand{\T}{\mathcal{T}}
\newcommand{\Pn}{\mathcal{P}}

\newcommand{\M}{\mathcal{M}}
\newcommand{\CC}{{\widetilde{\mathcal{C}}}}
\newcommand{\CI}{\mathcal{C}}

\newcommand{\Proj}{\operatorname{Proj}}
\newcommand{\Vertical}{\operatorname{Vert}}

\newcommand{\Sym}{\operatorname{Sym}}
\newcommand{\GammaO}{{\Gamma_{\mu_{n-k}}}}

\newcommand{\p}{{p}}
\newcommand{\qq}{{\mathbf q}}

\newcommand{\Qi}{Q_i}
\newcommand{\Qil}{Q_i^{\text{\rm in}}}
\newcommand{\Qir}{Q_i^{\text{\rm out}}}
\newcommand{\Qilc}{Q_i^{\text{\rm in}}(\Sigma_i)}
\newcommand{\Qirc}{Q_i^{\text{\rm out}}(\Sigma_i)}
\newcommand{\Qic}{Q_i(\Sigma_i)}
\newcommand{\Giin}{G_i^{\text{\rm in}}}
\newcommand{\GMstar}{(G_i^*)_{M_i}}
\newcommand{\Mi}{M_i}
\newcommand{\Miin}{M_i^{\text{\rm in}}}
\newcommand{\Fi}{F_i}
\newcommand{\Fip}{F'_i}
\newcommand{\PPi}{P_i}

\newcommand{\boundary}{v}
\newcommand{\addbox}[2]{\tau_{#2}(#1)}

\newcommand{\coeff}{p_{\CC}}
\newcommand{\telt}{t}

\newcommand{\Weq}{W^{\operatorname{eq}}}
\newcommand{\Feq}{\mathcal F^{\operatorname{eq}}}
\newcommand{\WeqAlt}{{\widetilde W}^{\operatorname{eq}}}
\newcommand{\Wtilde}{\widetilde{W}^{\operatorname{eq}}}
\newcommand{\Weqq}{\Weq_{\qb}}
\newcommand{\Wtildeq}{\Wtilde_{\qb}}
\newcommand{\dWeqdq}{\frac{\partial \Weq}{\partial q}}
\newcommand{\dWdq}{\frac{\partial W}{\partial q}}
\newcommand{\fact}{\operatorname{fact}}

\newcommand{\GG}{\widetilde{G}^{W_{\qb}}}
\newcommand{\HHH}{\widetilde{H}_B}

\newcommand{\reduce}[1]{#1}
\newcommand{\cluster}{\prod_{p\in \CI}p}

\DeclareRobustCommand*{\Raiseboxdepth}{\raisebox{\depth}}

\begin{document}

\title[The $B$-model connection and
 mirror symmetry for Grassmannians]
{The $B$-model connection and
mirror symmetry for Grassmannians}
\author{R. J. Marsh}%
\address{School of Mathematics,
            University of Leeds,
            Leeds
            LS2 9JT,
            UK
}%
\email{R.J.Marsh@leeds.ac.uk}%
\author{K. Rietsch}%
\address{Department of Mathematics,
            King's College London,
            Strand, London
            WC2R 2LS
            UK
}%
\email{konstanze.rietsch@kcl.ac.uk}%

\dedicatory{This paper is dedicated to the memory of Andrei Zelevinsky.}



\thanks{ This work was supported by the Engineering and Physical Sciences Research Council
[grant numbers EP/G007497/1 and EP/S071395/1].}
\date{6 June 2017 (Introduction revised 19 February 2018). Revised again after referee's comments 31 Jan 2020.}
\begin{abstract}
We consider the Grassmannian $X=Gr_{n-k}(\C^n)$ and describe a `mirror dual'
Landau-Ginzburg model $(\Xcheck , W_\qa:\Xcheck \to \C)$, where
$\Xcheck$ is the complement of a particular anti-canonical divisor in a
Langlands dual Grassmannian $\Xcheckbar$, and we express $W$  succinctly in
terms of Pl\"ucker coordinates. First of all, we show this Landau-Ginzburg model to be isomorphic to
one proposed for homogeneous spaces in a previous work by the second author. Secondly we show it to be a partial compactification of the Landau-Ginzburg model
defined in the 1990's by Eguchi, Hori, and Xiong.
Finally we construct inside the Gauss-Manin system associated to $W_\qa$ a free submodule which recovers
the trivial vector bundle with small Dubrovin connection defined out of Gromov-Witten invariants of $X$.
We also prove a
$T$-equivariant version of this isomorphism of connections.
Our results imply in the case of Grassmannians an integral formula for a solution to the quantum cohomology $D$-module of a homogeneous space, which was conjectured by the second author.
They also imply a series expansion of the top term in Givental's $J$-function, which was conjectured in a 1998 paper by Batyrev, Ciocan-Fontanine, Kim and van Straten.

\end{abstract}
\maketitle
\tableofcontents
\bibliographystyle{amsplain}
\section{Introduction}


The genus $0$ Gromov-Witten invariants of a Grassmannian $X$ answer enumerative questions about rational curves in $X$ 
and are put together in various ways to define rich mathematical structures such quantum cohomology rings, flat pencils of connections on Frobenius manifolds \cite{CoxKatz:QCohBook}. These structures are part of the so-called `$A$-model' of $X$. Mirror symmetry in the sense we consider here seeks to describe these structures in terms a mirror dual `$B$-model' or  Landau-Ginzburg (LG) model associated to $X$. Explicitly, the data from the $A$-model should be encoded by singularity theory or oscillating integrals of a regular function $W_q$ (the superpotential) which is defined on a `mirror dual' affine Calabi-Yau variety $\Xcheck$ with a holomorphic volume form $\omega$.  In the present paper we construct such a mirror datum in canonical and concrete terms for Grassmannians and prove associated mirror conjectures. Our results in particular imply and enhance a conjecture formulated in the 1990s by Batyrev, Ciocan-Fontanine, Kim and van Straten \cite[Conjecture~5.2.3]{BCKS:MSGrass} concerning a series expansion for a coefficient of Givental's $J$-function. This conjecture was restated in a paper of Bertram, Ciocan-Fontanine and Kim \cite[Section~3]{HoriVafaConj}, where the special case of Grassmannians of $2$-planes was proved by an entirely different method. The conjecture was again restated as a problem of interest in \cite[Problem 14]{prz09} some 10 years later.  

To give a flavour of our results, consider 
\begin{equation}\label{e:introS}
\mathcal S(q)=\sum_{\lambda}\left (\oint e^{\frac 1 {\hbar } W_q}p_\lambda \omega\right ) PD(\sigma^{\lambda}),
\end{equation}
which is an $H^*(X,\C)$-valued function depending on a single variable $q$ (and a parameter $\hbar$).
 The integrand involves the superpotential $W_q$. This is a particular rational function introduced in this paper  via an explicit formula in terms of Pl\"  ucker coordinates $p_\lambda$ on an isomorphic (but Langlands~dual) Grassmannian~$\check{\mathbb X}$; see for example \eqref{e:Wexample}. The integration is over a natural compact torus in an open subvariety~$\Xcheck$ of $\check{\mathbb X}$; see Theorem~\ref{t:main2}. Note  that the function $\mathcal S(q)$ in \eqref{e:introS} is expressed as a linear combination of Schubert classes $\sigma^\lambda$ permuted by Poincar\'e duality, denoted by PD.
 By Theorem~\ref{t:main2} (one of our main results) the function
 $\mathcal{S}(q)$ satisfies the flat section equation of the Dubrovin connection, i.e.\ we have 
$q\frac{d}{dq} \mathcal S=\frac 1 {\hbar} \sigma^{\ydiagram{1}}\star_q \mathcal S$.
Here   $\sigma^{\ydiagram{1}}$ denotes the hyperplane class  of $X$, and $\star_q$ denotes the quantum cup product on the small quantum cohomology ring of $X$.   
 
 A conjecture of Batyrev, Ciocan-Fontanine, Kim and van Straten proposes an integral formula $\oint e^{\frac 1 {\hbar } L_q} \omega$ for the coefficient of the top class $PD(\sigma^{\emptyset})$ of a flat section of the Dubrovin connection, and employs a Laurent polynomial $L_q$ introduced by Eguchi, Hori and Xiong \cite{EHX:GravQCoh} in place of our $W_q$. We prove this conjecture by recovering the Laurent polynomial superpotential $L_q$ as a restriction of $W_q$ to a particular open dense torus inside $\check{\mathbb X}$ and applying our more general Theorem~\ref{t:main2}. 

Note that the Schubert basis and the Pl\"ucker coordinates above are indexed by the same set.
This has its explanation in the geometric Satake correspondence. The key point is that there is a natural identification
$H^0(\check{\mathbb X},\mathcal O(1))=H^*(X,\C)$ identifying Schubert classes of $X$ with homogeneous coordinates of $\check{\mathbb X}$. This identification comes from the fact that both left-hand and right-hand sides agree with the $k$-th fundamental representation $\bigwedge^k\C^n$ of $SL_n(\C)$, the special linear group which acts on $\check{\mathbb X}$. For the left-hand side this is by the Borel-Weil theorem, and for the right-hand side it is by a very special case of the geometric Satake correspondence~\cite{Ginzburg, Kamnitzer, Lusztig, MV} which constructs representations of $G^\vee=SL_n(\C)$ via intersection cohomology of Schubert varieties of the affine Grassmannian $\mathcal Gr_G$ for the Langlands dual group $G=PSL_n(\C)$. (The Schubert variety which arises in the construction of  $\bigwedge^k\C^n$ turns out to be homogeneous for $G$ and coincides with the Grassmannian $X$.)


The methods developed here are useful and adaptable to other co-minuscule $G/P$, which are precisely the homogeneous spaces which in the geometric Satake correspondence appear  as `minimal' Schubert varieties of the affine Grassmannian $\mathcal{G}r_G$. In particular, since this paper appeared, the methods we use have been employed to obtain results analogous to Theorem~\ref{t:main2} in the case of even and odd-dimensional quadrics; see \cite{PechRie:Quadrics, PechRieWilliams:AllQuadrics}. Moreover, in the case of Lagrangian Grassmananians, a partial result in this direction, the `canonical' description of the superpotential in terms of Pl\"ucker coordinates, has been obtained in \cite{PechRie:LGLG}.

The main work in this paper is concerned with constructing the $A$-model Dubrovin connection in terms of the Gauss-Manin system associated to a mirror LG~model, in the most natural way possible, in the case where $X$ is a Grassmannian. 
Our main theorem, Theorem~\ref{t:main1}, explicitly describes the Dubrovin connection of $X$ in terms of the Gauss-Manin system of the mirror LG model we introduce. This underlies the  formula for the global flat section obtained in Theorem~\ref{t:main2}. We note that the superpotential is additionally shown (see Theorem~\ref{t:twosuperpotentials}) to be isomorphic to the Lie-theoretic superpotential of $X$ that was defined for general $G/P$ in \cite[Section~4.2]{Rie:MSgen}. Our result therefore also  implies a version of the mirror conjecture concerning solutions to the quantum differential equations in terms of the superpotential considered there, namely \cite[Conjecture~8.1]{Rie:MSgen}, in the Grassmannian case. 

Our results also shed new light on the (small) quantum cohomology ring of $X$, which is shown to agree with the Jacobi ring of $W_q$ by an isomorphism which identifies the Schubert class $\sigma^\lambda$ with the element of the Jacobi ring represented by the Pl\"ucker coordinate $p_\lambda$.   All of our results, including this one, are also stated and proved in the $T$-equivariant setting; see Proposition~\ref{p:eqJacobi} and Theorem~\ref{t:equivmain}.

In more recent work of the second author together with L.~Williams~\cite{RW:NObody}, the superpotential $W_q$ introduced in this paper, together with the cluster structure on $\C[\Xcheck]$, is shown by tropicalisation to define a set of polytopes which can be identified as Newton-Okounkov convex bodies of $X$. This can be understood as another form of mirror symmetry relating $X$ and $(\Xcheck,W_q)$.

A version of the present paper was placed on the arXiv in December 2015
(arXiv:1307.1085v2 [math.AG]). In May 2017 a preprint of Lam and Templier \cite{LamTemplier:preprint} appeared on the minuscule $G/P$ case of the mirror conjecture \cite[Conjecture~8.1]{Rie:MSgen} for the Lie-theoretic superpotential (see Section~\ref{s:RichardsonBmodel}). Their approach, which covers the Grassmannian case, employs completely different methods and uses the minuscule property. From this approach they can in their setting deduce that the Gauss-Manin system has the expected rank. 
On the other hand, our approach, using that the Grassmannian is {\it co-minuscule}, gives rise to explicit formulas for all of the components of a flat section of the Dubrovin connection in certain canonical coordinates. As mentioned above, the proof of the mirror conjecture along the same lines as in the present paper has already been carried out for even and odd-dimensional quadrics in \cite{PechRie:Quadrics, PechRieWilliams:AllQuadrics}, the latter of which are not minuscule and hence not covered by the work in~\cite{LamTemplier:preprint}. We expect our approach to be applicable, with analogous canonical coordinates, to all co-minuscule $G/P$. 

\vskip .2cm

The outline of the paper is as follows. We begin with a concrete introduction of the $A$-model structures. Then we give definitions of the $B$-model structures, as well as a careful statement of the main results.  In Section~\ref{s:comparisonMRR} we show how our Pl\"ucker formula for the superpotential relates to the formulations in \cite{BCKS:MSGrass,EHX:GravQCoh,Rie:MSgen}.
The proof of the main theorem begins in Section~\ref{s:StartOfMainProof} and takes up the remainder of the paper. It makes use of deep properties of the cluster algebra structure of the homogeneous coordinate ring of a Grassmannian. In the final sections we prove a version of the main theorem in the torus-equivariant setting.
 \vskip .2cm

\noindent  {\it Acknowledgements~:}
The first author thanks Gregg Musiker and Jeanne Scott for useful conversations.
The second author thanks Dale Peterson for his inspiring work and lectures,
Clelia Pech for useful conversations and Lauren Williams for helpful comments. We would like to thank the referee for many helpful comments.
Declarations of Interest: None.

\section{The $A$-model introduction}\label{s:IntroAmodel}

Let us suppose $X$ is the Grassmannian $Gr_{n-k}(\C^n)$. We focus on the example $k=3$ and $n=5$ to illustrate our results in the introduction. The cohomology of $X$ has a basis called the Schubert basis, which is indexed by partitions or Young diagrams (see e.g. \cite{Fulton:YoungBook}). In the example of $X=Gr_2(\C^5)$ we denote the Schubert basis of $H^*(X)=H^*(X,\C)$ by
$$
\sigma^{\emptyset},\sigma^{\ydiagram{1}}, \sigma^{\ydiagram{2}}, \sigma^{\ydiagram{1,1}}, \sigma^{\ydiagram{3}},\sigma^{\ydiagram{2,1}}, \sigma^{\ydiagram{3,1}}, \sigma^{\ydiagram{2,2}},
 \sigma^{\ydiagram{3,2}}, \sigma^{\ydiagram{3,3}}\ .
 $$
Here $\sigma^\lambda$ is in $H^{2|\lambda|}(X)$ and $|\lambda|$ denotes the number of boxes in the Young diagram $\lambda$. Furthermore $X^\lambda$ will be a Schubert variety associated to $\lambda$, representing the Poincar\'e dual homology class. For example $X^{\ydiagram{1}}$ is a hyperplane section for the Pl\"ucker embedding. In general the set of partitions fitting in an $(n-k)\x k$-rectangle indexes the Schubert basis for $Gr_{n-k}(n)$ in an entirely analogous way and is denoted by $\Pn_{k,n}$.

In classical Schubert calculus, Monk's rule says that the cup product with $\sigma^{\ydiagram{1}}$ takes any Schubert class $\sigma^\lambda$ to the sum of all the Schubert classes $\sigma^\mu$ corresponding to the $\mu$'s in $\Pn_{k,n}$ made up of $\lambda$ and precisely one extra box. In the cohomology of $Gr_2(5)$,
\begin{equation*}
\sigma^{\ydiagram{1}}\cup \sigma^{\ydiagram{2,1}}=\sigma^{\ydiagram{2,2}}+ \sigma^{\ydiagram{3,1}},\qquad
\sigma^{\ydiagram{1}}\cup \sigma^{\ydiagram{3,2}}=\sigma^{\ydiagram{3,3}},\qquad
\text{and }\quad \sigma^{\ydiagram{1}}\cup \sigma^{\ydiagram{3,3}}=0.
\end{equation*}
The combinatorics of adding a box, in this context, just encodes what happens to a Schubert variety $X^\lambda$ homologically, when it is intersected with a hyperplane section in general position.

In the quantum cohomology ring \cite{SiebertTian} (at fixed parameter $q$), quantum Monk's rule says that the quantum cup product with $\sigma^{\ydiagram{1}}$ takes any Schubert class $\sigma^\lambda$ to $\sigma^{\ydiagram{1}}\cup \sigma^{\lambda}$ plus, if it exists, a term $\qa \sigma^\nu$, where $\nu$ is obtained by removing $n-1(=4)$ boxes which must form the rim of $\lambda$ 
(by the $a=1$ case of quantum Pieri formula from \cite{Bertram:qSchubCalc}). 
For example,
\begin{equation*}
\sigma^{\ydiagram{1}}\star_\qa \sigma^{\ydiagram{2,1}}=\sigma^{\ydiagram{2,2}}+ \sigma^{\ydiagram{3,1}},\qquad
\sigma^{\ydiagram{1}}\star_\qa \sigma^{\ydiagram{3,2}}=\sigma^{\ydiagram{3,3}} + \qa \sigma^{\ydiagram{1}}.\qquad
\text{and }\quad \sigma^{\ydiagram{1}}\star_\qa \sigma^{\ydiagram{3,3}}=\qa \sigma^{\ydiagram{2}}.
\end{equation*}
Very roughly speaking, the extra term $\qa \sigma^{\nu}$ in $\sigma^{\ydiagram{1}}\star_\qa \sigma^{\lambda}$ says that the space of degree one maps $\phi: \C P^1\hookrightarrow X$ for which $\phi(0)$ lies in $X^{\lambda}$ and $\phi(1)$ lies in a fixed general position hyperplane, is essentially parametrised by a subvariety of $X$ in the class $[X^{\nu}]$ (via sending $\phi$ to $\phi(\infty)$).

The quantum products by degree two classes were used by  Dubrovin and Givental \cite{Giv:ICM, Dub:ICM} to define a connection on the trivial bundle
$$
 H^2(X,\C)\x H^*(X,\C)\to H^2(X,\C).
$$
In our setting we have $H^2(X,\C)=\C$ spanned by the class $c_1(\mathcal O(1))$, which equals the Schubert class $\sigma^{\ydiagram{1}}$. We denote by $\ta$ the coordinate on $H^2(X,\C)$ dual to $c_1(\mathcal O(1))$.
Let us recall the usual definition of the connection in the conventions of Dubrovin \cite{CoxKatz:QCohBook}, depending on an additional parameter $\hbarr$~:
\begin{equation}\label{e:Dubrovin1}
\nabla_{\partial_\ta}  := \frac{d}{d\ta} +
 \frac{1}{\hbarr}
 \sigma^{\ydiagram{1}}\star_{e^\ta}\, \underline{\, \ }\ .
\end{equation}
 We think of this connection as being dual to the one whose flat sections are constructed by Givental~\cite{Giv:EquivGW} in terms of descendent $2$-point Gromov-Witten invariants, compare also Section~\ref{s:Osc}.

Our main result is a $B$-model construction of the above connection. However we make some small adjustments first. Instead of $\ta$ we prefer to consider $\qa=e^{\ta}$, the coordinate on the torus $H^2(X,\C)/H^2(X,2\pi i \Z)\cong \C^*_\qa$.
We write $\C_\qa$ to mean $\C$ with coordinate $\qa$, and similarly $\C^*_\qa$ for $\C\setminus\{0\}$ with (invertible) coordinate~$\qa$.
Also, following Dubrovin \cite{Dub:ICM} we will extend the connection \eqref{e:Dubrovin1} in the $\hbarr$-direction, to give a flat meromorphic connection over a larger base.  Namely, let $\HA$ denote the (sheaf of regular sections of) the trivial vector bundle with fiber $H^*(X,\C)$ over the extended base $\PAB=\C_\hbarr\x \C_\qa$,  where the $\hbarr$ and $\qa$ are coordinates. We identify a Schubert class $\sigma^\lambda\in H^*(X,\C)$ with the corresponding constant section of $\HA$. 
Using the conventions of Iritani~\cite[Definition~3.1]{Iritani2} we set~:
\begin{eqnarray}\label{e:Anablaq}
{}^A\nabla_{\qa\partial_\qa}  &:=&\qa  \frac{\partial}{\partial\qa} +
 \frac{1}{\hbarr}
 \sigma^{\ydiagram{1}}\star_{\qa}\, \underline{\,  \ }\ , \\ \label{e:Anablaz}
{}^A \nabla_{\hbarr\partial_\hbarr}&:=& \hbarr\frac{\partial  }{\partial\hbarr} + {\Gr}\, - \frac{1}{\hbarr} c_1(TX)\star_{\qa}\, \underline{\,  \ }\  ,
\end{eqnarray}
where ${\Gr}$ is a diagonal operator on $H^*(X)$ given by
${\Gr}(\sigma)=k \sigma$ whenever $\sigma\in H^{2k}(X)$.
These formulas define a flat meromorphic connection on $\HA$.


The vector bundle $\HA$ also comes equipped with a flat pairing which is non-degenerate over $\C_z^*\x\C_\qa$. Namely let  $\left< \, \underline{\,  \ }\,  ,\, \underline{\,  \ }\, \right>_{H^*(X)}$ denote the Poincar\'e duality pairing on $H^*(X)$ and define $j:\PAB\to
\PAB$ by $(\hbarr,\qa)\mapsto (-\hbarr,\qa)$. Then the pairing $S_A: j^*\HA \otimes \HA \to \mathcal O_{\PAB_A}$ is given by
\begin{eqnarray*}
S_A(\, \underline{\,  \ }\,  ,\, \underline{\,  \ }\, )&=&(2\pi i\hbarr)^N\left< \, \underline{\,  \ }\, ,\, \underline{\,  \ }\, \right>_{H^*(X)},
\end{eqnarray*}
where $N=\dim_\C(X)$, compare~\cite{Iritani2}.
For the Grassmannian $X$, the Poincar\'e duality pairing is concretely described by the following involution. For $\lambda\in \Pn_{k,n}$ we denote by $PD(\lambda)$ the Young diagram obtained by taking the complement of $\lambda$ (within its bounding $(n-k)\x k$-rectangle) and rotating by $180^\circ$. The resulting Schubert class $\sigma^{PD(\lambda)}$ pairs to $1$ with  $\sigma^\lambda$ and to $0$ with all other Schubert classes under  $\left< \, \underline{\,  \ }\,  ,\, \underline{\,  \ }\, \right>_{H^*(X)}$.  
\begin{defn}\label{d:HA}  We define
\[
H_A=\Gamma(\HA, \C^*_\hbarr\x\C^*_\qa)=H^*(X,\C[\hbarr^{\pm 1}, \qa^{\pm 1}])
\]
to be the module for $D_\PAB=\C[\hbarr^{\pm 1},\qa^{\pm 1}]\left<\partial_{\hbarr},\partial_{\qa}\right>$ where
$\partial_{\hbarr},\partial_{\qa}$ act by ${}^A\nabla_{\partial_\hbarr} $ and ${}^A\nabla_{\partial_\qa} $, respectively. Compare equations~\eqref{e:Anablaq} and \eqref{e:Anablaz}.
We also define the $\C[\hbarr,\qa]$-submodule $H_{A,0}:=H^*(X,\C[\hbarr,\qa])$, which is acted on by the subalgebra
$D_{\PAB,0}$ of $D_\PAB$ generated by $\hbarr,\qa, \hbarr(\hbarr\partial_{\hbarr})$ and $\hbarr(\qa\partial_{\qa})$.
The pairing $H_A\otimes_{\C[\hbarr^{\pm 1},\qa^{\pm 1}]} H_A\to \C[\hbarr^{\pm 1},\qa^{\pm 1}]$ defined by $S_A$ is non-degenerate and denoted again by $S_A$.
\end{defn}

The main goal of this paper is to construct the $D_\PAB$-module $H_A$ above, and with it the data, $\HH_A, H_{A,0} $ and ${}^A\nabla$,
in terms of a Gauss-Manin system defined by a
mirror LG~model.
Descriptions of the LG~model
and the Gauss-Manin system
follow in Section~\ref{s:IntroBmodel}. The main results are stated in Section~\ref{s:main}.

\section{The $B$-model introduction}\label{s:IntroBmodel}
\subsection{The mirror LG model} To give our presentation of the mirror LG~model of the Grassmannian $X=Gr_{n-k}(n)$ we need to introduce  a new Grassmannian $\Xcheckbar:=Gr_k(n)$. Both $X$ and $\Xcheckbar$ have dimension $N=k(n-k)$. To be more precise, if $X=Gr_{n-k}(\C^n)$ then we think of $\Xcheckbar $ as $Gr_{k}(({\C^n})^*)$, which is embedded by a Pl\"ucker embedding in $\mathbb P (\bigwedge^k (\C^n)^*)$. Here $(\C^n)^*$ denotes the vector space dual to $\C^n$, with an action of the Langlands dual $GL_n(\C)$ from the right.
The  Pl\"ucker coordinates $\p_\lambda$ for $\Xcheckbar$ correspond in a natural way to
the Schubert classes $\sigma^{\lambda}$ in $H^*(X)$ and are indexed by $\lambda\in \Pn_{k,n}$.
\begin{footnote}{This coincidence of Pl\"ucker coordinates and Schubert classes is due to the identification of $H^*(X)$ with $\bigwedge^k \C^n$, viewed as the $k$-th fundamental representation of the Langlands dual $GL_n(\C)$ afforded by the geometric Satake correspondence~\cite{Ginzburg, Kamnitzer, Lusztig, MV}. 
This explains also why $\Xcheckbar$ should be viewed as a homogeneous space for the Langlands dual of the group used to define the $A$-model Grassmannian $X$. Therefore even though $X$ and $\Xcheckbar$ are isomorphic we do not think of them as being the same. Compare also \cite{PechRie:LGLG,PechRie:Quadrics,PechRieWilliams:AllQuadrics}.}
\end{footnote}

We continue with the explicit example of $X=Gr_2(5)$ and $\Xcheckbar=Gr_3(5)$, where $k=3$ and $n=5$.
Define the rational function $W$ on $\Xcheckbar\x \C_\qb$ by the formula
\begin{equation}\label{e:Wexample}
W=\frac{\p_{\ydiagram{1}}}{\p_{\emptyset}}+\frac{\p_{\ydiagram{3,1}}}{\p_{\ydiagram{3}}}+
\qb\frac{\p_{\ydiagram{2}}}{\p_{\ydiagram{3,3}}}+\frac{\p_{\ydiagram{3,2}}}{\p_{\ydiagram{2,2}}}
+\frac{\p_{\ydiagram{2,1}}}{\p_{{\ydiagram{1,1}}}},
\end{equation}
in terms of Pl\"ucker coordinates $p_\lambda$.
To obtain a regular function, remove from $\Xcheckbar$ the $5$ hyperplanes defined by the Pl\"ucker coordinates which appear in the denominators. The resulting affine variety is
$$
\Xcheck := \Xcheckbar\setminus\left\{ \p_{\emptyset}= 0\right\} \cup\left\{ \p_{\ydiagram{3}}= 0\right\}
\cup\{ \p_{\ydiagram{3,3}}= 0\}\cup
\{\p_{\ydiagram{2,2}}= 0\} \cup\{\p_{\ydiagram{1,1}}= 0\}.
$$
Note that the anti-canonical class of $\Xcheckbar=Gr_3(5)$ is $5\sigma^{\ydiagram{1}}$, therefore $\Xcheck$ is the complement of an anti-canonical divisor in $\Xcheckbar$. Let $\omega=
\omega_{\Xcheck}$ be a choice of non-vanishing holomorphic
volume form on $\Xcheck$ with simple poles along $D$.  We denote the regular function again by $W$, and refer to
$W:\Xcheck\x \C_\qb\to \C$
as the superpotential.

For a general Grassmannian $X=Gr_{n-k}(n)$ we have a completely  analogously defined affine subvariety $\Xcheck$ of $\Xcheckbar=Gr_k(n)$, along with  a non-vanishing holomorphic volume form $\omega$ on $\Xcheck$, and a superpotential $W:\Xcheck\x\C_\qb\to \C$, see Sections~\ref{s:mirrorGr} and \ref{s:omega}. Note that $\omega$ is a priori only uniquely determined up to a scalar. This scalar is chosen in Theorem~\ref{t:main2}, which is our first result that depends on this choice.

It is interesting to note that the variety $\Xcheck$ is an open positroid variety in the sense of Knutson, Lam and Speyer \cite{KnutsonLamSpeyer:PositroidPublished}. It also plays a special role for a particular Poisson structure on $\Xcheckbar$, see \cite{GSV:ClusterPoisson}. We will show in Section~\ref{s:comparisonMRR} that this LG model is isomorphic to the one introduced by the second author in \cite{Rie:MSgen}, and after restriction to an open subtorus becomes isomorphic to one introduced earlier by Eguchi, Hori and Xiong~\cite{EHX:GravQCoh} and studied further in~\cite{BCKS:MSGrass}.

 \subsection{The Gauss-Manin system} Denote by $\Omega^N(\Xcheck)$ the space of algebraic $N$-forms on $\Xcheck$.
 We write  $W_q:\Xcheck\to \C$ for the superpotential where the coordinate $\qb\ne 0$ is fixed.
There is a Gauss-Manin system associated to $W_\qb$, see \cite{Douai, Douai2, Sabbah:book}, which should be thought of as describing algebraic $N$-forms $\eta$
measured by `period integrals' of the form $\int_\Gamma e^{\frac 1 z W_q}\eta$, see the definition below.
Here we let both $\hbarrB$ and $\qb$ vary to get a $2$-parameter Gauss-Manin connection.
\begin{defn}\label{d:GMconnection}
Consider the $\C [\hbarrB^{\pm 1}, \qb^{\pm 1}]$-module $G^{W_\qb}$ defined by
\[
G^{W_{\qb}}=\Omega^N(\Xcheck)[\hbarrB^{\pm 1}, \qb^{\pm 1}]/\left((d + \hbarrB\inv d W_\qb\wedge\, \underline{\ \,  }\, \ )(\Omega^{N-1}(\Xcheck)[\hbarrB^{\pm 1}, \qb^{\pm 1}])\right).
\]
Note that for fixed $(\hbarrB_0,\qb_0)\in \PAB$ the `fiber'
\[
F_{B,(\hbarrB_0,\qb_0)}:=\Omega^N(\Xcheck)/\left((d + \hbarrB_0\inv  dW_{\qb_0}\wedge\, \underline{\ \,  }\,  )(\Omega^{N-1}(\Xcheck))\right),
\]
is a twisted de Rham cohomology group in the sense going back to Witten \cite[(11)]{Witten:82}. There is a Gauss-Manin connection on $G^{W_\qb}$  defined on $\eta\in \Omega^N(\Xcheck)$ by
\begin{eqnarray}\label{e:Bnabla1def}
{}^B\nabla_{\qb\partial_{\qb}}[\eta]&:=&
\frac{1}\hbarrB\left[\qb \frac{\partial{W}}{\partial \qb}\, \eta \right],\\
\label{e:Bnabla2def}
{}^B\nabla_{\hbarrB\partial_{\hbarrB}}[\eta]&:=&
 -\frac{1}{\hbarrB} \left[ W\, \eta \right],
\end{eqnarray}
and extended using the Leibniz rule. Thanks to the flatness of ${}^B\nabla$ we get a $D_\PAB$-module structure on $G^{W_\qb}$
by letting  $\partial_\hbarr$ and $\partial_\qb$ in $D_\PAB=\C[\hbarrB^{\pm 1},\qb^{\pm 1}]\left<\partial_{\hbarr},\partial_{\qa}\right>$
act by the operators ${}^B\nabla_{\partial_{\hbarr}}$ and ${}^B\nabla_{\partial_{\qb}}$, respectively. \end{defn}

In order to state our main theorem we will define a $\C[\hbarrB,\qb]$-submodule of $G^{W_{\qb}}$ which is to play the part of regular global sections of a trivial vector bundle $\HB$ on $\C_q\x\C_z$.
 
Since the divisor $\{\p_{\emptyset}=0\}$ is removed in the definition of $\Xcheck$, we can adopt the convention of setting $\p_{\emptyset}=1$ on $\Xcheck$. Therefore, from here on we consider the remaining $\p_\lambda$ as actual functions on $\Xcheck$ as opposed to homogeneous coordinates. 

\begin{defn}\label{d:Gbar}
Recall that $\Xcheck$ has on it a non-vanishing holomorphic volume form  $\omega$.
Let $\Gbar$ be the $\C[\hbarrB^{\pm 1}
,\qb^{\pm 1}]$-submodule of $G^{W_\qb}$ spanned by the classes $[\p_\lambda\omega]$ where $\lambda$ runs through $\Pn_{k,n}$. Furthermore let
$\GbarO$ be the $\C[\hbarrB,\qb]$-submodule inside $\Gbar$ generated by the 
$[\p_\lambda\omega]$, and let $\overline{F}_{B,(\hbarrB_0,\qb_0)}$ be the corresponding $\C$-linear subspace of $F_{B,(\hbarrB_0,\qb_0)}$.
\end{defn}

We will prove the following lemma in Section~\ref{s:freebasis}.
\begin{lem}\label{l:freebasis}
$\GbarO$ is a free $\C[\hbarrB,\qb]$-module with basis $\{[\p_\lambda\omega], \lambda\in\Pn_{k,n}\}$.
\end{lem}

By this lemma  $\GbarO$ is indeed the space of regular sections of a trivial vector bundle on $\PAB=\C_\hbarr\x\C_\qb$. We let $\HB$ denote the sheaf of regular sections of this trivial bundle on $\PAB$.
The fiber of the vector bundle $\HB$ at $(\hbarrB_0,\qb_0)$ is  $\overline{F}_{B,(\hbarrB_0,\qb_0)}$ and has a basis given independently of  $(\hbarrB_0,\qb_0)$ by the classes $[\p_\lambda \omega]$.

\begin{conjecture}\footnote{Update, May 2017. This conjecture now follows from our Theorem~\ref{t:twosuperpotentials} combined with the recent work  \cite{LamTemplier:preprint} of Lam and Templier.}\label{c:RankGM}
We conjecture that $\Gbar=G^{W_\qb}$. 
\end{conjecture}

The intuitive meaning of this conjecture is that $W_q$ has no additional critical point at $\infty$ in any compactification. The conjecture would follow from a related technical condition on $W_\qb$ (cohomological tameness \cite{Sabbah:pol,Sabbah:book}).  Note that in the special case of projective space the $[\p_\lambda\omega]$ were already known to form a free basis of $G^{W_\qb}$, see \cite{Gross:MSbook}. Another related example is the the smooth quadric $Q_{2n-1}$, which is the orthogonal Grassmannian of lines in $\C^{2n+1}$. In this case the mirror LG-model of \cite{Rie:MSgen} was expressed in Pl\"ucker coordinates in \cite{PechRie:Quadrics}, and proved for $Q_3$ to agree with one given by Gorbounov and Smirnov which they showed with Sabbah and Nemethi to be cohomologically tame \cite{GorbSmir:OddQuadrics}.

\section{Main results}\label{s:main}
\subsection{Isomorphism of $D_\PAB$-modules}\label{s:main1}
In the $B$-model, we consider the Gauss-Manin system on $G^{W_\qb}$. The first main theorem shows that the $A$-model datum, consisting of $H_{A,0}=H^*(X,\C[\hbarr,\qa])$ together with its small Dubrovin connection, can be recovered inside $G^{W_\qb}$.

\begin{thm}
\label{t:main1}
The $\C[\hbarr^{\pm 1},\qb^{\pm 1}]$-module $\Gbar$ is a $D_\PAB$--submodule of $G^{W_\qb}$, and
the map
\begin{eqnarray*}
\Phi:H_A
&\to &\Gbar\\
\sigma^{\lambda}&\mapsto&[\p_{\lambda}\omega]
\end{eqnarray*}
is an isomorphism of $D_\PAB$-modules.
Under this isomorphism  $H_{A,0}=H^*(X,\C[\hbarr,\qa])$ is  identified with $\GbarO$ and $\HA$ is identified with $\HB$.
\end{thm}

The proof of this theorem hinges on verifying the following formulas for the action of ${}^{B}\nabla$,
\begin{eqnarray}\label{e:Bnabla1}
{}^B\nabla_{\qb\partial_\qb}[\p_\lambda\omega] &=&\frac{1}\hbarrB \left (\sum_{\mu}[\p_\mu \omega] + \sum_{\nu}\qb\, [\p_\nu \omega]\right ),\\ \label{e:Bnabla2}
{}^B\nabla_{\hbarrB\partial_\hbarrB}[\p_\lambda\omega] &=&|\lambda|[\p_\lambda\omega] -{\frac 1\hbarrB}\, n\,\left (\sum_{\mu}[\p_\mu \omega] + \sum_{\nu}\qb\, [\p_\nu \omega]\right ),
\end{eqnarray}
where $\mu$ and $\nu$ are exactly as in the quantum Monk's rule for $\sigma^{\ydiagram{1}}\star_\qb \sigma^\lambda$; see equations~\eqref{e:Anablaq} and~\eqref{e:Anablaz} in Section~\ref{s:IntroAmodel}. The proof of this theorem will be given in Sections~\ref{s:StartOfMainProof} to~\ref{s:completionofproof}.
\medskip

In the concrete case of $X=Gr_2(5)$, these formulas say for example
\begin{eqnarray*}
{}^{B}
\nabla_{\qb\partial_\qb}\left([\p_{\ydiagram{2,1}}\omega]\right)
&=&\frac{1}\hbarrB\left([\p_{\ydiagram{3,1}}\omega]  +[\p_{\ydiagram{2,2}}\omega] \right),\\
{}^{B}
\nabla_{\qb\partial_\qb}\left([\p_{\ydiagram{3,3}}\omega]\right)
&=&\frac{1}\hbarrB \qb \, [\p_{\ydiagram{2}}\omega].
\end{eqnarray*}
Similarly in the $\hbarr\partial_\hbarr$ direction~:
\begin{eqnarray*}
{}^{B}
\nabla_{\hbarr\partial_\hbarr}\left([\p_{\ydiagram{2,1}}\omega]\right)
&=&3[\p_{\ydiagram{2,1}}\omega]- \frac{5}{\hbarrB}\left( [\p_{\ydiagram{3,1}}\omega]  +[\p_{\ydiagram{2,2}}\omega]\right) ,\\
{}^{B}
\nabla_{\hbarr\partial_\hbarr}\left([\p_{\ydiagram{3,3}}\omega]\right)
&=&6[\p_{\ydiagram{3,3}}\omega]- \frac{5}\hbarrB \qb \, [\p_{\ydiagram{2}}\omega].
\end{eqnarray*} 
This reflects precisely the formulas on the $A$-side arising from \eqref{e:Anablaq}
 and \eqref{e:Anablaz} and quantum Schubert calculus. For example \eqref{e:Anablaz} implies
 \begin{eqnarray*}
{}^{A}
\nabla_{\hbarr\partial_\hbarr}\left(\sigma^{\ydiagram{3,3}}\right)
&=&
6\, \sigma^{\ydiagram{3,3}}- \frac 5 \hbarr\qa\,\sigma^{\ydiagram{2}}.
\end{eqnarray*}
Now we come to some consequences of Theorem~\ref{t:main1}, along with comparison results connecting $W_q$ with previously defined superpotentials.  

\subsection{Oscillating integrals}\label{s:Osc}
Recall the flat pairing $S_A$ in the $A$-model. We may think of this  pairing as identifying the bundle $H^*(X)\x\C_\hbarr^*\x\C_\qa\to \C_\hbarr^*\x\C_\qa$ with its dual. The flatness condition can then be interpreted as saying that the dual connection to ${}^A\nabla$ is given by  formulas  analogous to \eqref{e:Anablaq} and \eqref{e:Anablaz} but with $z$ replaced by $-z$.
In other words the new connection ${}^A\nabla^\vee$ defined by
\begin{eqnarray}\label{e:AnablaqDual}
{}^A\nabla^\vee_{\qa\partial_\qa}  &:=&\qa  \frac{\partial}{\partial\qa} -
 \frac{1}{\hbarr}
 \sigma^{\ydiagram{1}}\star_{\qa}\, \underline{\,  \ }\ , \\ \label{e:AnablazDual}
{}^A \nabla^\vee_{\hbarr\partial_\hbarr}&:=& \hbarr\frac{\partial  }{\partial\hbarr} + {\Gr}\, + \frac{1}{\hbarr} c_1(TX)\star_{\qa}\, \underline{\,  \ }\  ,
\end{eqnarray}
satisfies $dS_A(\sigma,\sigma')=S_A({}^A\nabla^\vee \sigma,\sigma')+S_A(\sigma,{}^A\nabla\sigma')$, and is therefore dual to ${}^A\nabla$.

In~\cite[Corollary~6.3]{Giv:EquivGW}, Givental wrote down a basis of the space of all solutions $s\in H^*(X,\C[\hbarrB\inv,\ln(\qa)][[\qa]])$ to the flat sections equation $ {}^A\nabla^\vee_{\qa\partial_{\qa}} s =0$, that is to the equation:
 \begin{equation}\label{e:qdiffeq}
\qa\frac{\partial}{\partial {\qa}} s =\frac{1}{\hbarr}\sigma^{\ydiagram{1}}\star_{\qa} s.
 \end{equation}
Equation \eqref{e:qdiffeq} is referred to as the \emph {(small) quantum differential equation}.
Givental's solution is given in terms of two-point descendent Gromov-Witten invariants. 
Explicitly in our setting, for each $\mu\in \Pn_{k,n}$ there is a solution
\begin{equation}\label{e:smu}
s_\mu=\frac{1}{{\hbarr}^N}\sum_{\lambda\in{ \Pn_{k,n}}} \sum_{d\ge 0} \qa^d\left<\frac{1}{\hbarr-\psi} {e^{\frac{\ln (q)}{\hbarr}\sigma^{\ydiagram{1}}}}\sigma^\mu ,\sigma^\lambda \right>_{2,d} \sigma^{PD(\lambda)}
\end{equation}
to \eqref{e:qdiffeq}, where to make sense of the $d=0$ term one sets
\begin{equation}\label{e:smutwo}
\left<\frac{1}{\hbarr-\psi} {e^{\frac{\ln (q)}{\hbarr}\sigma^{\ydiagram{1}}}}\sigma^\mu ,\sigma^\lambda \right>_{2,0}:=
\left< {e^{\frac{\ln (q)}{\hbarr}\sigma^{\ydiagram{1}}}}\sigma^\mu ,\sigma^\lambda,1 \right>_{3,0}.
\end{equation}
Here we use the notation $\langle \quad\rangle_{n,d}$ for genus zero $n$-point degree $d$ Gromov-Witten invariants of $X$. Moreover $\psi$ is the `psi-class' on the moduli space of stable maps $\overline{\mathcal M_{0,2}}(X,d)$, which is the first Chern class of the line bundle defined by the cotangent line at the first marked point. We refer to~\cite{CoxKatz:QCohBook} or \cite[Section~1.3]{Pandharipande:afterGivental} for this result and relevant definitions. Note that we have added the factor $\frac{1}{{\hbarr}^N}$, with $N=k(n-k)$ for degree reasons (compare equation~\eqref{e:smunk} below). Equivalently, the factor is necessary to ensure flatness of $s_\mu$ in the $\hbarr$-direction. In the case where $\mu$ is the maximal element in $\Pn_{k,n}$, which we denote $\mu_{n-k}$, the exponential disappears and the formula simplifies to

\begin{equation}\label{e:smunk}
s_{\mu_{n-k}}= \sum_\lambda \left(\delta_{\lambda,\emptyset}+ \sum_{d\ge 1} \left<{\psi}^{dn-|\lambda|-1}\sigma^{\mu_{n-k}} ,\sigma^\lambda \right>_{2,d} \left(\frac{\qa}{z^n}\right)^d\right) {{\hbarr}}^{-|PD(\lambda)|}\sigma^{PD(\lambda)}.
\end{equation}

The following theorem implies an integral formula for the solution $s_{\mu_{n-k}}\in H^*(X,\C[\hbarrB\inv][[\qa]])$ to the quantum differential equation \eqref{e:qdiffeq}.

\begin{thm}\label{t:main2} Let  $\GammaO$ be a cycle in $H_N(\Xcheck,\Z)$ represented by an oriented, compact torus (homeomorphic to $ (S_1)^N$) which is the compact real form of a cluster torus (isomorphic to $(\C^*)^N$) inside $\Xcheck$. Note that the class $\GammaO$ does not depend on the choice of the cluster torus. 
We choose $\omega$ to be dual to  $\GammaO$ in the sense that $\frac{1}{(2\pi i)^N}\int_{\GammaO}\omega=1$.  Then the formula
\[
\Ss_{\GammaO}(\hbarrB,\qb):=\frac{1}{(2\pi i\hbarr)^N}\sum_{\lambda\in\Pn_{k,n}}\left(\int_{\GammaO} e^{\frac{1}{\hbarrB}W_{\qb}} \p_\lambda\omega\right)\sigma^{PD(\lambda)}
\]
defines  a flat section for ${}^A\nabla^\vee$ inside $ H^*(X,\C[\hbarrB\inv][[\qa]])$. In particular $\Ss_{\GammaO}$ satisfies the small quantum differential equation~\eqref{e:qdiffeq}.
\end{thm}

\begin{rem} Note that $\Gamma_{\mu_{n-k}}$ and $\omega$ in the above theorem are uniquely defined up to a common sign. The function $\Ss_{\GammaO}(\hbarrB,\qb)$ is canonical and does not depend on this sign choice. 
\end{rem}

\begin{proof}
This statement follows in a standard way from  Theorem~\ref{t:main1} and the constructions.  For any $(\hbarrB,\qb)\in \PAB$ with $\hbarrB\ne 0$, consider the linear form on the fiber $F_{B,(\hbarrB,\qb)}$ of $\HB$ at the point $(\hbarrB,\qb)$  defined by the formula
\[
\Osc_{\GammaO}(\hbarr,\qb): [\eta]\mapsto \int_{\GammaO} e^{\frac{1}{\hbarrB}W_{\qb}} \eta.
\]
This formula defines a section $\Osc_{\GammaO}$
of $\HBan^\vee$  over $\C_\hbarr^*\x\C_\qb$, where $\HBan^\vee$ denotes the sheaf of analytic sections of the bundle dual to $\HB$.
The definition of the Gauss-Manin connection \eqref{e:Bnabla1},\eqref{e:Bnabla2} on $\HB$ is engineered so that $\Osc_{\GammaO}$ is a flat section of $\HBan^\vee$. 

Using Theorem~\ref{t:main1} together with the pairing $S_A$, the bundle $\HBan^\vee$ with its Gauss-Manin connection can be identified with the pair $(\HAan,{}^A\nabla^\vee)$ over $\C_\hbarr^*\x\C_\qb$, where $\HAan$ denotes the sheaf of analytic sections of the vector bundle $\HA$.  We now denote by 
$\Osc^A_{\GammaO}$ the flat section of $(\HAan,{}^A\nabla^\vee)$ over $\C_\hbarr^*\x\C_\qb$ corresponding to $\Osc_{\GammaO}$ under this identification. Then $\Osc^A_\GammaO$ is the section of $\HAan$ determined by the property that 
\begin{equation}\label{e:OscA}
S_A(\Osc^A_{\GammaO},\sigma^\lambda)= 
\int_{\GammaO} e^{\frac{1}{\hbarrB}W_{\qb}} \p_\lambda\omega.
\end{equation}
Since the basis dual to $\{\sigma^\lambda\}$ with respect to the pairing $S_A$ is $\{\frac{1}{(2\pi i\hbarr)^N}\sigma^{PD(\lambda)}\}$, equation \eqref{e:OscA} implies that  
\[
\Osc^A_{\GammaO}=
\frac{1}{(2\pi i\hbarr)^N}\sum_{\lambda\in\Pn_{k,n}}\left(\int_{\GammaO} e^{\frac{1}{\hbarrB}W_{\qb}} \p_\lambda\omega\right)\sigma^{PD(\lambda)}.
\]
So $\Osc^A_{\GammaO}=\Ss_{\GammaO}$ and we see that $\Ss_{\GammaO}$ is flat for ${}^A\nabla^\vee$. 
It remains to check that $\Ss_{\GammaO}$ lies in $H^*(X,\C[\hbarr\inv][[q]])$, in other words that the coefficients $m_\lambda$ lie in $\C[\hbarr\inv][[q]]$ as opposed to $\C[[\hbarr\inv,q]]$. 
This follows by degree considerations. Namely the flatness of $\Ss_{\GammaO}$ implies in particular for the coefficients that 
\begin{equation*}
\left(\hbarr\frac{\partial}{\partial \hbarr}+n \qa\frac{\partial}{\partial \qa}\right)\left[\frac{1}{(2\pi i\hbarr)^N}\int_{\GammaO} e^{\frac{1}{\hbarrB}W_{\qb}} \p_\lambda\omega\right]=\frac{-|PD(\lambda)|}{(2\pi i\hbarr)^N}\int_{\GammaO} e^{\frac{1}{\hbarrB}W_{\qb}} \p_\lambda\omega.
\end{equation*}
Therefore $\left(\hbarr\frac{\partial}{\partial \hbarr}+n \qa\frac{\partial}{\partial \qa}\right)$ annihilates
\begin{equation}\label{e:homogeneous}
\hbarr^{|PD(\lambda)|}\frac{1}{(2\pi i\hbarr)^N}\int_{\GammaO} e^{\frac{1}{\hbarr}W_{\qb}} \p_\lambda\omega,
\end{equation}
which implies that in the $\qa$-expansion of \eqref{e:homogeneous} the coefficient of $\qa^d$ is a scalar multiple of $\frac 1{\hbarr^{dn}}$. As a consequence the coefficient of $\sigma^{PD(\lambda)}$ in $\Ss_{\GammaO}$ has the form
\begin{equation}\label{e:FormOfResidue}
\frac{1}{(2\pi i\hbarr)^N}\int_{\GammaO} e^{\frac{1}{\hbarrB}W_{\qb}} \p_\lambda\omega=\sum_{d\ge 0} a_d\left(\frac \qa{\hbarr^{n}}\right)^d\hbarr^{-|PD(\lambda)|}.
\end{equation}
In particular it lies in $\C[\hbarr\inv][[\qa]]$.
\end{proof}
\begin{rem}\label{r:GWresidue} Note that by setting $q=1$  in Equation~\eqref{e:FormOfResidue} we have the series expansion 
\begin{equation} \label{e:1expansion}
\frac{1}{(2\pi i\hbarr)^N}\int_{\GammaO} e^{\frac{1}{\hbarrB}W_{1}} \p_\lambda\omega= 
\sum_{d\ge 0} a_d\hbarr^{-dn-|PD(\lambda)|}.
\end{equation}
Expanding the exponential $e^{\frac{1}{\hbarrB}W_{1}}$ on the left-hand side of the above equation, the coefficient $a_d$ can now be computed by the residue formula
\[
a_d=\frac{1}{(dn-|\lambda|)!}\left(\frac {1}{(2\pi i)^N}\int_{\GammaO}(W_{1})^{dn-|\lambda|} \p_\lambda\omega\right), 
\]
where we assume $dn\ge |\lambda |$. If $dn<|\lambda|$ then necessarily $a_d=0$, since in this case $-dn-|PD(\lambda)|>-N$ while the left-hand side of \eqref{e:1expansion} is contained in $\hbarrB^{-N}\C[[\hbarrB\inv]]$. We will make use of this formula for $a_d$ in Proposition~\ref{p:GWlambda}.
\end{rem}
\begin{rem}\label{r:Siss}
In this remark we check that the flat section $\Ss_{\GammaO}$ of Theorem~\ref{t:main2} agrees with Givental's flat section $s_{\mu_{n-k}}$ exactly. First we note that, up to scalar, $s_{\mu_{n-k}}$ is the unique flat section with coefficients in $\C[\hbarrB\inv][[q]]$, as follows recursively from the differential equation~\eqref{e:qdiffeq} and grading considerations. From Theorem~\ref{t:main2} it follows therefore that $\Ss_{\GammaO}$ is a scalar multiple of $s_{\mu_{n-k}}$. The relevant scalar can be determined by computing the 
coefficient $a_0$ of 
\[
\frac{1}{(2\pi i)^N}\int_{\GammaO} e^{\frac{1}{\hbarrB}W_{\qb}} \omega=\sum_{d\ge 0} a_d\qa^d\hbarr^{-dn},
\]
where the above equation is a consequence of \eqref{e:FormOfResidue} in the case $\lambda=\emptyset$, and we have multiplied out by $\hbarrB^N$. This coefficient $a_0$ is computed on the left-hand side by \[
a_0=\frac{1}{(2\pi i)^N}\int_{\Gamma_{\mu_{n-k}}}\omega=1,
\]
using the assumptions of Theorem~\ref{t:main2}. Therefore, comparing with the start of the $\lambda=\emptyset$ term of \eqref{e:smunk}, we see that  $\Ss_{\GammaO}$ and $s_{\mu_{n-k}}$ agree.
\end{rem}
\begin{rem}[Other solutions and the $J$-function]\label{r:Jfunction} Further local solutions $\Ss=\Ss_\Gamma$ to the equation  $ {}^A\nabla^\vee_{\qa\partial_{\qa}} \Ss =0$ can be obtained by replacing $\GammaO$ by some other, possibly non-compact integration cycle $\Gamma$. In this case it would be necessary to have conditions on the decay of $\Re(\frac{1}{\hbarrB}W_{\qb})  $ in unbounded directions of $\Gamma$ and let $\Gamma$ vary with $\hbarrB$ and $\qb$, to ensure convergence.  According to Givental~\cite[Section~2]{Giv:QToda} such cycles $\Gamma$ may be obtained from Morse theory for $\Re(\frac{1}{\hbarrB}W_q)$. 
 Other than the compact cycle $\GammaO$ associated to $s_{\mu_{n-k}}$, we don't know how to determine specific cycles $\Gamma_\mu$ such that the flat section $S_{\Gamma_\mu}$ recovers Givental's flat section $s_\mu$ defined via the $A$-model.  Identifying such cycles would give integral formulas for all the entries of Givental's `fundamental solution matrix', and in particular for the coefficients of Givental's `$J$-function'; see Section~4 in \cite{Giv:QToda}. 
When $X=\C P^2$, however, explicit integration cycles can be described; see~\cite{Gross:MSbook}. 

Although from our results we only obtain a formula for the constant term of Givental's $J$-function  (the `$A$-series', see Section~\ref{s:EHX1}) and not the full $J$-function, we can still consider the  $\C[\qa,\hbarr^{- 1}]\langle \partial_q\rangle$-module generated by the coefficients of the $J$-function, or equivalently generated by the $\sigma^{\mu_{n-k}}$-coefficients of the flat sectionss of ${}^A\nabla^\vee$; see \cite{Giv:toricmirror}, or for example \cite[Section~5.1]{BCKS:MSGrass}.  This is sometimes called the `quantum cohomology $D$-module', and it quantises the part of the quantum cohomology ring generated by degree $2$ elements. (It is not to be confused with the $D_\PAB$-module $H_A$.)   
With this definition, we note that the property ${}^A\nabla^\vee_{\qa\partial_{\qa} }\Ss_\Gamma =0$ implies that the integrals
\begin{equation}\label{e:qcohDmod}
\int_{\Gamma} e^{\frac{1}{\hbarrB}W_{\qb}} \omega
\end{equation}
are solutions to the quantum cohomology $D$-module. We remark that other integral expressions for solutions of the quantum cohomology $D$-module of a Grassmannian which are very different from \eqref{e:qcohDmod} were obtained by Bertram, Ciocan-Fontanine and Kim~\cite{HoriVafaConj}. Moreover, in the same paper they prove a formula for the $J$-function generalising the one for projective space due to Givental.
\end{rem}

\subsection{$A$-series conjecture and the superpotential of Eguchi, Hori and Xiong} \label{s:EHX1} In Section~\ref{s:EHX}
 we will recall the definition of the conjectural Laurent polynomial superpotential of Eguchi, Hori and Xiong  \cite{EHX:GravQCoh}. In that section we will also state in more detail the following comparison result.
\begin{thm}\label{t:EHXcomparison}
The Laurent polynomial $L_q$ associated to the Grassmannian $X$ by Eguchi, Hori and Xiong in  \cite{EHX:GravQCoh} is isomorphic to the restriction of $W_\qb$ to a certain open torus $\mathcal T$  inside $\Xcheck$. The holomorphic volume form $\omega$ restricted to this torus agrees with the standard torus-invariant volume form $\omega_{\mathcal T}$. 
\end{thm}
The proof of Theorem~\ref{t:EHXcomparison} is contained in Sections~\ref{s:EHX} and~\ref{s:omega}.

\vskip .2cm

Consider now Givental's special solution $s_{\mu_{n-k}}$~\eqref{e:smunk} to the quantum differential equation~\eqref{e:qdiffeq}.
The coefficient of $\sigma^{\mu_{n-k}}$ in $s_{\mu_{n-k}}$  with $z$ specialised to $1$ is  an element of $\C[[\qa]]$ and referred to as the $A$-series in~\cite{BCKS:MSGrass}. Namely for $X=Gr_{n-k}(\C^n)$,
\begin{equation}\label{e:Aseriesgen}
A_{X}(q)=\sum_d \left<\frac{1}{1-\psi}\sigma^{\mu_{n-k}} ,\sigma^\emptyset \right>_{2,d}\qa^d\ = \ 
1+\sum_{d\ge 1} \left<\psi^{d n-2}\sigma^{\mu_{n-k}} \right>_{1,d}\qa^d,
\end{equation}
where we have used that $\sigma^{\emptyset}$ is the fundamental class of $X$ to simplify the formula. 

In \cite{BCKS:MSGrass}  Batyrev, Ciocan-Fontanine, Kim and van Straten studied the  Laurent polynomial superpotential of \cite{EHX:GravQCoh} and conjectured an explicit combinatorial formula
for the $A$-series \eqref{e:Aseriesgen} in the case of a Grassmannian. This conjecture can now be deduced. 

We note that in the special case of $Gr_2(n)$ this conjecture was proved earlier in \cite{HoriVafaConj} using a formula (also proved in \cite{HoriVafaConj}) for the $J$-function; compare Remark~\ref{r:Jfunction}.

 \begin{cor}\label{c:BCKSconj} \cite[Conjecture~5.2.3]{BCKS:MSGrass}
The $A$-series of the Grassmannian $X=Gr_{n-k}(\C^n)$ is given by
\begin{equation}\label{e:EHXA}
A_X(q)=
\sum_{d\ge 0}\frac{1}{(d!)^{n}}\sum_{(s_{i,j})\in\mathcal S_d}
\left(\prod_{(i,j)\in \mathcal I_{k,n}}\binom{s_{i+1,j}}{s_{i,j}}\binom{s_{i,j+1}}{s_{i,j}}\right )q^d .
\end{equation}
Here the indexing sets are
\begin{eqnarray*}
\mathcal I_{k,n}&=&\{(i,j)\in \Z\x \Z\ |\ 0< i<n-k, 0<j<k\},\\
\mathcal S_d&=&\{(s_{i,j})\in (\Z_{\ge 0})^{\mathcal I_{k,n}} \ |\ s_{i+1,j}\ge s_{i,j},\ s_{i,j+1}\ge s_{i,j},\ s_{n-k,j}=s_{i,k}=d  \}.
\end{eqnarray*}
 \end{cor}
\begin{proof} The combinatorial formula \eqref{e:EHXA} is obtained in \cite{BCKS:MSPartFl} as the residue of the form
$e^{L_q}\omega_{\mathcal T}$ determined by the EHX superpotential $L_q$.  This corollary therefore follows from Theorem~\ref{t:main2} together with Theorem~\ref{t:EHXcomparison} and the residue calculation in \cite[Section~5.1]{BCKS:MSPartFl}.
\end{proof}

Note that the corollary uses a residue calculation to give combinatorial formulas for the descendent Gromov-Witten invariants $\langle\psi^{dn-2}\sigma^{\mu_{n-k}}\rangle_{1,d}$ which are certain coefficients of Givental's flat section $s_{\mu_{n-k}}$.
As a corollary to Theorem~\ref{t:main2} we also have the following residue formulas for all of the remaining descendent Gromov-Witten invariants appearing in $s_{\mu_{n-k}}$. 
\begin{prop}\label{p:GWlambda}
With notations as in Section~\ref{s:Osc} we have 
\begin{equation}\label{e:GWresidue}
\langle \psi^{dn-1}\sigma^{\mu_{n-k}}, \sigma^\lambda\rangle_{2,d}=
\frac{1}{(dn-|\lambda|)!}\left(\frac {1}{(2\pi i)^N}\int_{\GammaO}(W_{1})^{dn-|\lambda|} \p_\lambda\omega\right),
\end{equation}
assuming $dn\ge|\lambda|$.
\end{prop}
\begin{proof}
Recall that the flat section $\Ss_{\GammaO}$ of Theorem~\ref{t:main2} agrees with $s_{\mu_{n-k}}$, by Theorem~\ref{t:main2} and Remark~\ref{r:Siss}. The residue formula~\eqref{e:GWresidue} now follows from the expansion \eqref{e:smunk} and Remark~\ref{r:GWresidue}.
\end{proof}
 
\subsection{The LG model on a Richardson variety}\label{s:GPsuperpotentialIntro}
The first LG~model for the Grassmannian $X$ which accurately recovers the quantum cohomology ring is one defined in \cite{Rie:MSgen}. In this LG~model the superpotential can be interpreted as a regular function $\mathcal F_\qb$ defined on an intersection of opposite Bruhat cells $\RR=\mathcal R_{w_P,w_0}$. Here $\mathcal R$ is an affine subvariety of the full flag variety on the $B$-model side, which is associated to the Grassmannian~$X$ interpreted as a homogeneous space $GL_n/P$. 
 In Section~\ref{s:RichardsonBmodel} we recall the precise definition of the superpotential $\mathcal F_q$. Moreover we explain how our superpotential $W_q$ from Section~\ref{s:IntroBmodel} relates to it, by defining a carefully chosen embedding of  $\RR$ 
 into the Grassmannian $\Xcheckbar$. We obtain the following comparison result. 
\begin{thm}\label{t:twosuperpotentials} The LG~models $(\Xcheck, W_{\qb})$ and $(\RR,\mathcal F_{\qb})$ are isomorphic via the maps given in Proposition~\ref{p:comparisonMRR}. This isomorphism also identifies the holomorphic volume form $\omega$ on $\Xcheck$ with the holomorphic volume form on $\RR$ introduced in~\cite[Section~7]{Rie:MSgen}.
\end{thm}
The proof of this theorem is contained in Section~\ref{s:comparisonMRR} and Section~\ref{s:omega}. This proof also makes use of some special coordinates on $\Xcheck$ and
the EHX Laurent polynomial expression recalled in Definition~\ref{d:EHXBmodel}. This theorem implies that the Jacobi ring of $W_q$ recovers the quantum cohomology ring $qH^*(X)[q\inv]$, since this is what was proved for $\mathcal F_q$ in \cite[Corollary~4.2]{Rie:MSgen}. We will show in   Proposition~\ref{p:SchubertPlucker} that the classes of the Pl\"ucker coordinates in the Jacobi ring of $W_q$ recover the Schubert basis, see also Proposition~\ref{p:eqJacobi} and the paragraph preceding it.

Finally, as a consequence of Theorem~\ref{t:twosuperpotentials} and the results on oscillating integrals described in Section~\ref{s:main1} and Section~\ref{s:Osc}, we also obtain the following corollary,  compare Remark~\ref{r:Jfunction}. 

\begin{cor}[Special case of Conjecture~8.1 from \cite{Rie:MSgen}]
The quantum cohomology $D$-module of the Grassmannian
has a global holomorphic solution on $\C^*_{\hbarr}\x\C_{\qb}$ given by the residue integral $\frac{1}{(2 \pi i)^N}\oint e^{\frac{1}\hbarr\mathcal F_{\qb}}\omega$.  \end{cor}
This concludes the summary of main results which are non-equivariant. 

\section{Equivariant results}\label{s:EquivariantIntro}
The Grassmannian $X=Gr_{n-k}(\C^n)$ is a homogeneous space for $G^\vee=GL_n(\C)$. In particular, the maximal torus $T^\vee$ of $n\x n$-diagonal matrices and the Borel subgroups of upper-triangular and lower-triangular matrices,  denoted by $B^\vee_+$ and $B^\vee_-$ respectively, all act on $X$. In the $A$-model this means that we can consider the $T^\vee$-equivariant cohomology of $X$ and the
$T^\vee$-equivariant quantum cohomology of $X$, and we also have a $T^\vee$-equivariant version of the Dubrovin connection; see in particular \cite{Giv:toricmirror, Lu:EquivQCoh, Mihalcea}. Our final results are about extending Theorem~\ref{t:main1} to describe the equivariant Dubrovin connection via the $B$-model.

In Section~\ref{s:Tequivariantsuperpotential}  we define a deformation $\Weq$ of the superpotential $W$ (Definition~\ref{d:Weq}) involving the equivariant parameters $x_1,\dotsc, x_n$. These parameters are the standard generators of the equivariant cohomology ring of a point,  $H^*_{T^\vee}(\{pt\})=\C[x_1,\dotsc, x_n]$.  In our running example of $X=Gr_2(\C^5)$ the deformation $\Weq$ is given by the formula
\[
\Weq=W + (x_1+x_2)\ln(\qb) + (x_2-x_1)\ln (\p_{\ydiagram{3}}) + (x_3-x_2)\ln (\p_{\ydiagram{3,3}}) +(x_4-x_3)\ln (\p_{\ydiagram{2,2}})+(x_5-x_4)\ln (\p_{\ydiagram{1,1}}).
\]
Note that we may think of $\Weq$ either as a multivalued map
\[\Weq: \Xcheck\x \C^*_q\to \C\oplus H^2_{T^\vee}(\{pt\})\,  =\, H^{\le 2}_{T^\vee}(\{pt\})\]
or as a multivalued function 
$\Weq: \Xcheck\x \C^*_q\x \mathfrak h^\vee\to \C $, interpreting the $x_i$ as functions on $\mathfrak h^\vee$ via the identification $H^2_{T^\vee}(\{pt\})\cong (\mathfrak h^\vee)^*$.  
Our first equivariant result is the comparison result, which is proved in Section~\ref{s:Tequivariantsuperpotential}. 

\begin{prop}\label{p:eqLiecomp} The LG~model $(\Xcheck, \Weq_\qb)$ is isomorphic to the equivariant LG~model $(\RR,\mathcal F^{\operatorname{eq}}_{\qb})$ defined in \cite[Section~4.1 and 4.2]{Rie:MSgen} via the maps given in Proposition~\ref{p:comparisonMRR}. 
\end{prop}

Note that although $\Weq$ is multivalued, the derivatives of $\Weq$ along $\Xcheck$ are regular and define an ideal $(\partial_{\Xcheck}\Weq)$ in  $\C[\Xcheck][q^{\pm 1}, x_1,\dotsc x_n]$. We call the quotient by the ideal the \emph{Jacobi ring} of $(\Xcheck, \Weq_\qb)$ and let $[p_\lambda]$ denote the image in the Jacobi ring of the Pl\"ucker coordinate $p_\lambda$.  As in the non-equivariant case, combining our comparison result (Proposition~\ref{p:eqLiecomp}) with  \cite[Corollary~4.2]{Rie:MSgen} implies an isomorphism between the Jacobi ring of $(\Xcheck, \Weq_\qb)$ and the (small) equivariant quantum cohomology ring of $X$ with $\qa\inv$ adjoined. This isomorphism has the following very natural description which will be restated and proved in Proposition~\ref{p:eqLiecompEnhanced}.

\begin{prop}\label{p:eqJacobi} The Jacobi ring of $(\Xcheck, \Weq_\qb)$ is isomorphic to the equivariant quantum cohomology ring $qH^*_{T^\vee}(X,\C)[\qa\inv]$ via an isomorphism which takes the form
\[
[\p_\lambda]\mapsto [X^\lambda]_{T^\vee}.
\]
Here $X^\lambda$ is the $B^\vee_+$-invariant Schubert variety of codimension $|\lambda|$ associated to $\lambda$, and $[X^\lambda]_{T^\vee}$ is its $T^\vee$-equivariant fundamental class, viewed as an element of the equivariant quantum cohomology of $X$.
\end{prop}

Our main equivariant result is an equivariant version of Theorem~\ref{t:main1} which we now prepare to state. Equivariant quantum cohomology can in this setting be thought of as providing a $q$-deformed version  $\star_{\qa,\mathbf x}$ of the equivariant cup product on $H^*_{T^\vee}(X)$, where $\mathbf x=(x_1,\dotsc, x_n)$. Let us therefore denote by $\sigma_{T^\vee}^\lambda$ the equivariant {\it or} quantum equivariant Schubert class 
\[
\sigma^\lambda_{T^\vee}:=[X^\lambda]_{T^\vee}.
\]
We recall the equivariant quantum Monk's rule \cite[Section~1.1]{Mihalcea} which reads
\[
c_1^{T^\vee}\hskip -1mm(\mathcal O(1))\star_{\qa,\mathbf x} \sigma_{T^\vee}^\lambda=\sum_\mu\sigma_{T^\vee}^\mu +\qb\sum_\nu \sigma_{T^\vee}^\nu + x_\lambda \sigma_{T^\vee}^\lambda,
\]
where the $\mu$ and $\nu$ are as in the non-equivariant quantum Monk's rule described in  Section~\ref{s:IntroAmodel}, and  $x_\lambda$ is a particular linear combination of the $x_i$; see  \eqref{e:xlambda}.

This formula determines the equivariant Dubrovin connection on the $A$-model side. Namely, we have the following definition, which is the $T^\vee$-equivariant analogue of Definition~\ref{d:HA}.

\begin{defn}[The equivariant version of $H_A$]\label{d:HAeq}
Consider the ring of differential operators, 
\[
D_{\operatorname{eq}}=\C[x_1,\dotsc, x_n, \hbarr^{\pm1}, q^{\pm 1}]\langle q \partial_q, z\partial_z+\sum_{i=1}^n x_i\partial_{x_i}\rangle.
\]
Let $H^{\operatorname{eq}}_A$ be the $H_{T^\vee}^*(\{pt\})[\hbarr^{\pm 1},\qb^{\pm 1}]$-module defined by 
\[ 
H_A^{\operatorname{eq}}:=H_{T^\vee}^*(X)[\hbarr^{\pm 1},\qb^{\pm 1}].
\]
Then $H_A$ is a $D_{\operatorname{eq}}$-module by setting
\begin{equation}\label{e:EqGM1}
\qa\partial_{\qa} \ \sigma=\frac 1z c^{T^\vee}_1\hskip -1mm(\mathcal O(1))\star_{\qa,\mathbf x}\sigma
\end{equation}
and
\begin{equation}\label{e:EqGM2}
\left(\hbarr\partial_{\hbarr}+\sum_i x_i\partial_{x_i}\right) \ \sigma=
-\frac{1}{\hbarr}
[X_{ac}]_{T^\vee}\star_{\qa,\mathbf x}\sigma + \Gr(\sigma),
\end{equation}
for $\sigma \in H^*(X,\C)$. Here, $X_{ac}$ denotes the anticanonical divisor given by the union of $n$ different $T^\vee$-invariant hyperplanes which are permuted by the cyclic $\Z/n\Z$-action on $X$; see Section~\ref{s:eqqH}. 
\end{defn}

\begin{defn}[The equivariant version of $H_B$]\label{d:BmodeleqGMsystem} 
Let $\Omega_{\operatorname{eq}}^\bullet(\Xcheck)$ %
denote the graded algebra of algebraic differential forms on $\Xcheck$ with coefficients in 
\[H_{T^\vee}^*(pt)[\hbarr^{\pm 1},q^{\pm 1}]=\C[x_1,\dotsc, x_n, \hbarr^{\pm 1},q^{\pm 1}].
\] 
The $k$-th graded component 
\[
\Omega^k_{\operatorname{eq}}(\Xcheck)=\Omega^k(\Xcheck)\otimes H_{T^\vee}^*(\{pt\})[\hbarr^{\pm 1},\qb^{\pm 1}]
\]
consists of algebraic $k$-forms on $\Xcheck$ with coefficients in $H_{T^\vee}^*(\{pt\})[\hbarr^{\pm 1},\qb^{\pm 1}]$.   
In particular,  the $\Omega^k_{\operatorname{eq}}(\Xcheck)$ are $H_{T^\vee}^*(\{pt\})[\hbarr^{\pm 1},\qb^{\pm 1}]$-modules. Then $\frac 1 z dW^{\operatorname{eq}}$, where $d=d_{\Xcheck}$ is the exterior derivative along $\Xcheck$, can be thought of as an element of $\Omega_{\operatorname{eq}}^1(\Xcheck)$. Note that $ dW^{\operatorname{eq}}$ is algebraic despite the fact that 
$\Weq$ is not. 
An equivariant analogue of the Gauss-Manin system from Definition \ref{d:GMconnection} is defined by
\[
G^{\Weqq}:=\Omega_{\operatorname{eq}}^N(\Xcheck)/ (d+\frac{1}{\hbarr}d \Weq\wedge\, \underline{\ \ }\, )\Omega_{\operatorname{eq}}^{N-1}(\Xcheck).
\]
The elements of $G^{\Weqq}$ may be thought of as algebraic $N$-forms $\eta$ on $\Xcheck$ depending (algebraically) on parameters $q^{\pm 1},z^{\pm 1},x_1,\dotsc, x_n$, which can be measured by integrals $\int_{\Gamma}e^{\frac 1 z\Weq}\eta$.

The ring of differential operators, 
\[
D_{\operatorname{eq}}=\C[x_1,\dotsc, x_n, \hbarr^{\pm1}, q^{\pm 1}]\langle q \partial_q, z\partial_z+\sum_{i=1}^n x_i\partial_{x_i}\rangle,
\]
acts in a natural way on  $G^{ \Weqq}$. This action is given explicitly by
\[
\qb\partial_{\qb}[ \eta]  = \frac 1 {\hbarr} \left[\qb\frac{\partial \Weq}{\partial \qb}\eta \right]
\]
and
\[
(\hbarr\partial_{\hbarr}+\sum_i x_i\partial_{x_i})[ \eta]  = -\frac 1 \hbarr \left[W\eta \right],
\]
where the second formula arises from the identity $\left(\hbarr\frac{\partial}{\partial {\hbarr}}+\sum_i x_i\frac{\partial}{\partial{x_i}} \right)\left(\frac 1{\hbarr}\Weq\right)=-\frac 1{\hbarr}W$. 

Note that the differential operator $\hbarr\partial_{\hbarr}$ no longer acts on $G^{ \Weqq}$. 
Indeed, $\hbarr\frac{\partial}{\partial \hbarr} \left(\frac 1\hbarr \Weqq\right)=- \frac 1{\hbarr} \Weqq$ now involves logarithms of Pl\"ucker coordinates, so is no longer algebraic.
This is why it was necessary to replace $\hbarr\partial_{\hbarr}$ by the differental operator $\hbarr\partial_{\hbarr}+\sum_i x_i\partial_{x_i}$.


Finally, we define $H_B^{\operatorname{eq}}$ to be the  $H_{T^\vee}^*(pt)[\hbarr^{\pm 1}][\qb, \qb\inv]$-submodule of $G^{\Weqq}$ spanned by the classes $[\p_{\lambda}\omega]$ for $\lambda\in \Pn_{k,n}$.
\end{defn}

The following Theorem is the equivariant version of Theorem~\ref{t:main1}. It will be proved in Section~\ref{s:actionvectorfieldequivariant}.
\begin{thm}\label{t:equivmain}
 $H^{\operatorname{eq}}_B$ is a $D_{\operatorname{eq}}$-submodule of $G^{W_q^{\operatorname{eq}}}$. Moreover the $H_{T^\vee}^*(pt)[\hbarr^{\pm 1},\qb^{\pm 1}]$-module homomorphism defined by 
\[
\begin{array}{ccc}
H_A^{\operatorname{eq}}&\To &H_B^{\operatorname{eq}}\\
\sigma_{T^\vee}^\lambda&\mapsto &[p_{\lambda}\omega]
\end{array}
\]
is an isomorphism of $D_{\operatorname{eq}}$-submodules. 
In particular
\[
\qb\partial_{\qb} [\p_\lambda\omega]=\frac{1}z\left(\sum_\mu[\p_\mu\omega]+\qb \sum_\nu[\p_\nu\omega] + x_\lambda [\p_\lambda\omega]\right),
\]
where $\mu,\nu$ and $x_\lambda$ are as in the equivariant quantum Monk's rule \eqref{e:EquivariantQMonks}  for multiplication by $c_1^{T^\vee}\hskip -1mm(\mathcal O(1))$.
\end{thm}

\begin{rem}[Alternative superpotential] We also introduce an alternative version of the equivariant superpotential. It is denoted by $\WeqAlt$ and is the same as $\Weq$ but with the $\ln(q)$-term removed. For example in the case of $X=Gr_2(\C^5)$,   
\[
{\WeqAlt}=W +  (x_2-x_1)\ln (\p_{\ydiagram{3}}) + (x_3-x_2)\ln (\p_{\ydiagram{3,3}}) +(x_4-x_3)\ln (\p_{\ydiagram{2,2}})+(x_5-x_4)\ln (\p_{\ydiagram{1,1}}).
\]
Unlike $\Weq$, this alternative superpotential is regular in $\qb$. Note that the Jacobi ring of $(\Xcheck, \WeqAlt)$ again agrees with the equivariant quantum cohomology ring $qH^*(X)[\qa\inv]$, since $\WeqAlt$ has the same Jacobi ring as  $\Weq$. 
 However the change from $\Weq$ to $\WeqAlt$ affects the Gauss-Manin system.  With this change, an analogue to Theorem~\ref{t:equivmain} holds in which the equivariant first Chern class $c^{T^\vee}_1\hskip -1mm(\mathcal O(1))$ appearing in \eqref{e:EqGM1} is replaced by the equivariant fundamental class of the $B^\vee_-$-invariant Schubert divisor $\widetilde X^{\ydiagram 1}$. Note that the Chern class $c^{T^\vee}_1\hskip -1mm(\mathcal O(1))$ in the original version is in a sense not geometric, because $\mathcal O(1)$ has no $T^\vee$-invariant global sections; therefore $c^{T^\vee}_1(\mathcal O(1))$ is not the fundamental class of a $T^\vee$-invariant divisor. We also remark that the Schubert divisor $\tilde X^{\ydiagram 1}$ which appears here is `opposite' to the one defining $\sigma^{\ydiagram{1}}_{T^\vee}$.  
 \end{rem}

\begin{rem} [Oscillating integrals]
In analogy with Section~\ref{s:Osc}, this theorem provides solutions to the equivariant small quantum differential equations
\begin{equation}\label{e:equivariantDubrovin1}
\qa \frac{\partial}{\partial\qa}\Ss=
 \frac{1}{\hbarr}\, c_1^{T}(\mathcal O(1))\star_{\qa,\mathbf x}\, \Ss,
\end{equation}
which are of the form
\begin{equation}\label{e:equivsol}
\Sseq_{\Gamma}(\hbarrB,\qb,\mathbf x):=\frac{1}{(2\pi i\hbarr)^N}\sum_{\lambda\in\Pn_{k,n}}\left(\int_{\Gamma} e^{\frac{1}{\hbarrB}\Weq} \p_\lambda\omega\right)\sigma^{PD(\lambda)},
\end{equation}
given a suitable simply-connected choice of $\Gamma$ in $\Xcheck$ (allowed to vary continuously with $\hbarrB,\qb$ and $\mathbf x$ in some domain of $\C^{2+n}$) for which the integral converges, as in Remark~\ref{r:Jfunction}.
Moreover, $\Ss=\Sseq_{\Gamma}$ also satisfies the differential equation
\begin{equation}\label{e:equivariantDubrovin2}
\left( \hbarr \frac{\partial}{\partial\hbarr}+\sum_{i=1}^nx_i  \frac{\partial}{\partial x_i}\right)\Ss=
-\frac{1}{\hbarr}[X_{ac}]_{T^\vee}\star_{\qa,\mathbf x}\, \Ss - \Gr( \Ss).
\end{equation}

Note that for the definition of the solution \eqref{e:equivsol} we need to make a choice of a branch of $\Weq|_{\Gamma}$. This affects $\Sseq_{\Gamma} $ by a factor of the form
\begin{equation}\label{e:eqfactor}
\exp\left (2\pi i m\frac{x_1+\dotsc + x_{n-k}}{z}+2\pi i \sum_{j=2}^n m_j\frac{x_j- x_{j-1}}{z}\right),
\end{equation}
where $m, m_j\in \Z$.
We remark that the factor~\eqref{e:eqfactor} is annihilated by $\qb\partial_\qb$ and $z\partial_z+\sum x_j\partial_{x_j}$, hence the choice of branch doesn't affect the validity of \eqref{e:equivariantDubrovin1} or \eqref{e:equivariantDubrovin2}.

Finally, a solution $\Sseq_\Gamma$ of the quantum differential equations gives rise to a solution,
\[
\int_{\Gamma} e^{\frac{1}{\hbarrB}\Weq}\omega,
\]
of the equivariant quantum cohomology $D$-module; compare Remark~\ref{r:Jfunction}. 
Together with the comparison result (Proposition~\ref{p:EquivariantComparison}), Theorem~\ref{t:equivmain} implies a version of \cite[Conjecture~8.2]{Rie:MSgen} about integral solutions to the quantum cohomology $D$-module of a homogeneous space, in the special case of a Grassmannian. 
\end{rem}

This concludes the summary of results. We now begin by defining in more detail the versions of the superpotential and showing how they are related to one another.

\section{The three versions of the superpotential }\label{s:mirrorGr}
We have already mentioned the three different versions of a Landau-Ginzburg model dual to the $A$-model Grassmannian $X=Gr_{n-k}(\C^n)$. In this section their superpotentials are defined in detail and  we show how they are related to one another.
We proceed in reverse chronological order, beginning with the superpotential introduced in Section~\ref{s:IntroBmodel}. 

\subsection{The  Pl\"ucker coordinate superpotential}\label{s:GrassmannianBmodel}
\label{s:Plucker}
The Pl\"ucker coordinate formulation gives a very simple-looking expression for the superpotential, therefore it is a natural starting point. We begin by describing its domain.  Recall that the $A$-model Grassmannian, $X=Gr_{n-k}(\C^n)$, is a homogeneous space for $GL_n(\C)$ acting from the left. We think of it in the usual way as a Grassmannian of codimension $k$ subspaces in the $n$-dimensional vector space of column vectors $\C^n$. 
In this section we define a Landau-Ginzburg model taking place on a $B$-model Grassmannian.  The $B$-model Grassmannian is a Grassmannian of row vectors, $\Xcheckbar:=Gr_k((\C^{n})^*)$,  and we view it as homogeneous space for the Langlands dual $GL_n$ under the (right) action of multiplication from the right. Note that $X$ and $\Xcheckbar$ are isomorphic, but this is a type~$A$ coincidence; compare for example~\cite{PechRieWilliams:AllQuadrics,PechRie:LGLG}. In order to distinguish between the two general linear groups acting on $X$ and $\Xcheckbar$ we will refer to the group on the $A$-model side as $GL_n^\vee(\C)$ and add a check ${}^\vee$ to notations pertaining to this group.


Elements of $\Xcheckbar$ may be represented by maximal rank $(k\x n)$-matrices $M$ in the usual way, with $M$ representing its row-span.
We think of $\Xcheckbar$ as embedded in $\Proj(\bigwedge^k (\C^n)^*)$ by its Pl\"ucker embedding and denote its homogeneous coordinate ring by $\C[\Xcheckbar]$.
The Pl\"ucker coordinates are all the maximal minors of $M$, and are determined by a choice of $k$ columns. We index the
Pl\"ucker coordinates by partitions $\lambda\in \Pn_{k,n}$ as follows.
Associate to any partition $\lambda\in \Pn_{k,n}$ a $k$-subset in $[1,n]:=\{1,\dotsc, n\}$ by interpreting $\lambda$ as a path from the top right hand corner of its bounding $(n-k)\x k$ rectangle down to the bottom left hand corner. Such a path necessarily consists of $k$ horizontal and $n-k$ vertical steps.  The positions of the horizontal steps (numbered from the start of the path to the end) define a subset of $k$ elements in $[1,n]$. We denote this subset, associated to $\lambda$, by
$J_\lambda$. See Figure~\ref{fig:partitionpath} for an example.

\begin{figure}
\includegraphics[width=3cm]{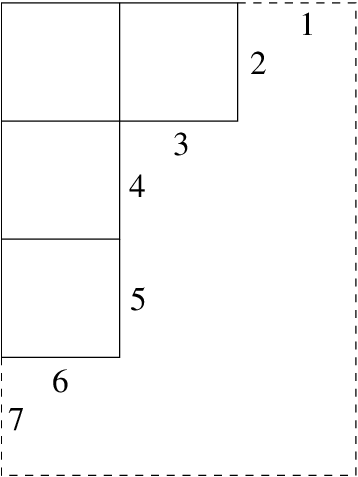}
\caption{The $k$-subset corresponding to a partition. For
$\Raiseboxdepth{\ydiagram{2,1,1}}\in\Pn_{3,7}$,
we have 
$J_{\ydiagram{2,1,1}}=\{1,3,6\}$.
}
\label{fig:partitionpath}
\end{figure}

Suppose that $J_\lambda=\{j_1,\dotsc, j_k\}$ with $1\le j_1<\dotsc< j _k\le n$. Then the Pl\"ucker coordinate associated to $\lambda$ is defined to be the determinant of a $k\x k$ submatrix of $M$,
\[
\p_\lambda(M)=\det\left((M_{i,j_l})_{i,l}\right).
\]
We will also sometimes denote the Pl\"{u}cker coordinate corresponding to the
$k$-subset $J$ by $p_J$.

A special role will be played by the $n$ Pl\"{u}cker coordinates corresponding to
the $k$-subsets which are (cyclic) intervals. These are $J_{i}:=[i+1,i+k]$, where $i\in [1,n]$. Sometimes it will be useful to index such an interval by its last element, in which case we use the notation $L_{i+k}$ for $J_i$. So 
\begin{equation}\label{JiLi}
J_{i}=L_{i+k}=[i+1,i+k].
\end{equation}
Our convention regarding indices 
is that elements in such $k$-subsets are interpreted modulo $n$, and also the subscripts
of $J_i$ and $L_i$.
The partition corresponding  to $J_i$ is denoted by $\mu_i$. 
For example since $J_1=[2,k+1]$ we have $\mu_{1}=(k)$, the maximal partition with one part. If $n-k\ge 2$, then $\mu_2$ is the maximal two row partition, $(k,k)$. For $k=3$ and $n=7$ we have for example,
\begin{equation}\label{e:muiexample}
\mu_{1}=\ydiagram{3}\, , \  \mu_2=\ydiagram{3,3}\,,\  \mu_3=\ydiagram{3,3,3}\, ,\ \   \mu_4=\ydiagram{3,3,3,3},\, \ \
  \mu_5=\ydiagram{2,2,2,2}\, ,\ \mu_6=\ydiagram{1,1,1,1}\, ,\ \mu_7=\emptyset\,.
\end{equation}
 Always $\mu_{n-k}$ is the maximal rectangle, and $\mu_n$ is the empty partition. 
 
The set $\{p_{\mu_1},\dotsc, \p_{\mu_n}\}$ of special Pl\"ucker coordinates is invariant under the $\Z/n\Z$ action on $\Xcheckbar$ defined by cyclic shift,
\[
M=\begin{pmatrix}m_{11} & m_{12}& m_{13} &\dots & m_{1,n}\\
m_{21} & m_{22}& m_{23}&\dots & m_{2,n}\\
\vdots  & \vdots &\vdots& &\vdots \\
m_{k1} & m_{k2}&  m_{k3}&\dots & m_{k,n}\\
\end{pmatrix}
\mapsto
M[1]=\begin{pmatrix}m_{12} & m_{13}&\dots & m_{1,n} & (-1)^{k-1}m_{11}\\
m_{22} & m_{23}&\dots & m_{2,n} &(-1)^{k-1}m_{21}\\
\vdots  & \vdots &    &\vdots &\vdots \\
m_{k2} & m_{k3}&\dots & m_{k,n}&(-1)^{k-1}m_{k1}\\
\end{pmatrix}.
\]
Indeed, we have $p_J(M[1])=p_{J+1}(M)$ 
for any $k$-subset $J$ of $[1,n]$ (where $J+1$ is obtained by adding $1$ to each element of $J$, modulo $n$). 
Therefore also  $\p_{\mu_i}(M[1])=\p_{\mu_{i+1}}([M])$ as $J_{\mu_i}=J_i$ and $J_{i+1}=J_i+1$. 
 
Each $\p_{\mu_i}$ is a section of $\mathcal O(1)$ for the Pl\"ucker embedding, and the union of the hyperplane sections defined by the $\p_{\mu_i}$ is an anticanonical divisor
\begin{equation}\label{e:acdivisorDef}
D=\{\p_{\mu_{1}}=0\} \cup \{\p_{\mu_{2}}=0\}\cup \dotsc \cup \{\p_{\mu_{n}}=0\}.
\end{equation}
Indeed, the Pl\"ucker embedding of the Grassmannian $\Xcheckbar$ is minimal, and $\Xcheckbar$ is Fano of index $n$, i.e. $n$ is the maximal integer for which the anti-canonical class  is divisible by $n$.

Let $\Xcheck$ be the Zariski-open
subset of $\Xcheckbar$
obtained by removing the anti-canonical divisor \eqref{e:acdivisorDef}. So
$$
\Xcheck:=\Xcheckbar\setminus D=\left\{ [M]\in \Xcheckbar\ |\ \p_{\mu_{i}}(M)\ne 0, \text{ all $i=1,\dotsc, n$ }\right\}.
$$
Note that the anticanonical divisor $D$ and its complement are invariant under the $\Z/ n\Z$-action on $\Xcheckbar$ defined by $M\mapsto M[1]$, by the discussion above.

The coordinate ring of $\Xcheck$ is denoted by $\C[\Xcheck]$. 
Recall that, to pass from the homogeneous coordinates 
$\p_\lambda$ on $\Xcheckbar$ to regular functions on $\C[\Xcheck]$ we 
make the convention of setting $\p_\emptyset=1$. Both $\C[\Xcheck]$ and 
$\C[\Xcheckbar]$ are
cluster algebras; see Section~\ref{s:CAcoordring}.

Our version of the Landau-Ginzburg model mirror dual to $X$ is a regular function $W:\Xcheck\x\C_{\qb}\to \C$, which we may call the {\it canonical superpotential}, in analogy with \cite[Section~1.1]{PechRieWilliams:AllQuadrics}. It is defined as follows.

\begin{defn}[The superpotential on $\Xcheck$]\label{d:W}
Denote by 
$\widehat{\mu}_{i}$ the partition corresponding to 
\[\widehat{J_i}=\widehat{L}_{i+k}:=[i+1,i+k-1]\cup \{i+k+1\},
\] 
where $i+k$ has been removed from $J_i$ and replaced by $i+k+1$. 
Unless $i=n-k$, the Young diagram of  $\widehat{\mu}_{i}$, is obtained from the
Young diagram of $\mu_{i}$ by adding a box. The particular shape of $\mu_{i}$
guarantees that there is only one way to do this. The partition $\widehat{\mu}_{n-k}$ is obtained by
removing the entire rim from $\mu_{n-k}$ to give an $(n-k-1)\x (k-1)$ rectangle.
 We define
\begin{equation}\label{e:W}
W:=\sum_{i=1}^n \frac {\p_{\widehat{\mu}_{i}}} {\p_{\mu_{i}}} \qb^{ \delta_{i,n-k} }=
\left( \sum_{ i\ne n-k} \frac {\p_{\widehat{\mu}_{i}}} {\p_{\mu_{i}}}\right)  + \qb \frac {\p_{\widehat{\mu}_{n-k}}} {\p_{\mu_{n-k}}}.
\end{equation}
\end{defn}
\begin{rem} Notice that in the quantum Schubert calculus of $X$ and for $i\ne n-k$,
\[
\sigma^{\ydiagram{1}}\star_{\qa}\sigma^{\mu_i}=\sigma^{\widehat{\mu_{i}}},\qquad \sigma^{\ydiagram{1}}\star_{\qa}\sigma^{\mu_{n-k}}=\qb\, \sigma^{\widehat{\mu_{n-k}}},
\]
and thus each of the individual summands of $W$ in the above formula  formally resembles the hyperplane class $\sigma^{\ydiagram{1}}$.  
\end{rem}

We will next recall the definitions of the two previously conjectured Landau-Ginzburg models for Grassmannians,
starting with the LG~model of \cite{Rie:MSgen} followed by that of Eguchi, Hori and Xiong \cite{EHX:GravQCoh, BCKS:MSGrass}, and explain how these previous definitions relate to this new one.

\subsection{The Lie-theoretic superpotential}\label{s:RichardsonBmodel}

Let the ($B$-model) group  $G=GL_n(\C)$ act (now from the left) on a full flag variety.
We fix some notation regarding this group.
 We let $B_+,B_-$ denote its upper-triangular and lower-triangular Borel subgroups, respectively, and
$T$ denote the maximal torus of diagonal matrices. The unipotent radicals of $B_+$ and $B_-$ are
denoted by $U_+$ and $U_-$. Let $\mathfrak h$ denote the Lie algebra of $T$.
We let $E_{i,j}$ denote the matrix
with entry $1$ in row $i$ and column $j$ and zeros elsewhere. Let $e_i=E_{i,i+1}$
and $f_i=E_{i+1,i}$ be the usual Chevalley elements of $\mathfrak{gl}_n$, and
let $\alpha_i\in\mathfrak h^*$ be the simple root corresponding to $e_i$.
We define the associated $1$-parameter subgroups $x_i:\C\to U_+$ and $y_i:\C\to U_-$,
\[
x_i(t)=\exp(t e_i),\quad
y_i(t)=\exp(t f_i)
\]
for $i=1,\dotsc, n-1$.
Let $W=N_{GL_n}(T)/T\cong S_n$ be the Weyl
group. Namely for every $w\in W$ we have a natural choice of representative $\bar w$ which is the permutation matrix in $GL_n$ corresponding to the permutation $w$. Alternatively we can choose the following representatives. Let
$$
\dot s_i=x_i(1)y_i(-1)x_i(1) \in N_{GL_n}(T)
$$
represent the generator $s_i=\dot s_i T$ of $W$ which is the simple transposition $(i,i+1)$. Let $\ell:W\to\Z_{\ge 0}$ be the length function.
If $\ell(w)=m$ and $s_{i_1}\dotsc s_{i_m}$ is a reduced expression for $w$, then the product $\dot w=
\dot s_{i_1}\dotsc \dot s_{i_m}$ is a well defined element of $GL_n$  and independent of
the reduced expression chosen. One advantage of the latter choice of representatives is that its definition is not specific to $GL_n$.
We will also require the root subgroups
\[
x_{\alpha_j+\alpha_{j+1}+\dotsc+\alpha_k}(a):=\bar s_{j+1}\bar s_{j+2}\cdots \bar s_{k-1} x_k(a) \bar s_{k-1}\inv\dotsc \bar s_{j+2}\inv\bar s_{j+1}\inv= \mathbf 1_n + a E_{j,k+1},
\]
where $j<k$ and $\mathbf 1_n$ is the $n\x n$ identity matrix.

Let $P\supset B_+$ be the maximal parabolic subgroup in $G$ generated by $B_+$ and the
elements $\dot s_i$ for  $i\ne n-k$. We also have $W_P=\left<s_i\ |\ i\not = n-k\right>$, the corresponding
parabolic subgroup of the Weyl group $W$. The longest element in $W_P$ is denoted by $w_P$, while the longest element of $W$ is denoted by $w_0$. .  Let $W^P$ denote the set
of minimal length coset representatives in $W/W_P$.  The longest element in $W^P$ is denoted by $w^P$. Clearly $w^Pw_P=w_0$, the longest element of $W$.

Consider the Schubert variety  inside $GL_n/B_-$ defined by,
\[
\Xcheckbar_{w_P}=\overline{B_+\dot w_P B_-/B_-}.
\]
We think of  $\Xcheckbar_{w_P}\subset GL_n/B_-$ as a Richardson variety, namely as closure of the intersection,
\[
\RR=\mathcal R_{w_P,w_0}= B_+\dot w_P B_-\cap B_-\dot w_0 B_-/B_-,
\]
 of opposite Bruhat cells. This intersection, $\RR$, is a smooth irreducible variety of dimension $\ell (w^P)$, which equals to  $k(n-k)$
 for our choice of $P$.

The construction of a mirror LG~model in \cite{Rie:MSgen}, applied in  the Grassmannian case, yields a superpotential $\mathcal F$ on $\RR\x\C^*_{\qb}$. Here the definition of $\mathcal F$ involves first identifying $\RR\x\C^*_{\qb}$ with a subset of the group $PSL_n(\C)$, Langlands dual to the group $SL_n(\C)$ acting on the $A$-model $X$, see \cite[Section~4.1]{Rie:MSgen}. Since we are considering our Grassmannians $X$ and $\Xcheckbar$ as a homogeneous spaces for (Langlands dual) general linear groups, we will replace this subset of $PSL_n(\C)$ with a natural choice of lift to the $B$-model $GL_n(\C)$. This will not change the function on $\RR\x\C^*_{\qb}$, but will make a difference (as it should) when we extend to the $T^\vee$-equivariant case in Section~\ref{s:Tequivariantsuperpotential}. 

One further practical difference between $PSL_n$ and $GL_n$ is that in \cite{Rie:MSgen} the torus in $PSL_n$ analogous to $T^{W_P}$ is identified  with $\C^*_{\qb}$, while for $GL_n$ the torus $T^{W_P}$ isn't one-dimensional, but rather is two-dimensional. In order to cut down the dimension let $\widetilde T^{W_P}$ be the one-dimensional subtorus of $T^{W_P}$ in $GL_n(\C)$ defined by 
\[
\widetilde T^{W_P}:=\left\{\left .\begin{pmatrix} d &&&&\\
&\ddots& &&&\\
&& d &&&\\
&&& 1&&\\
&&&&\ddots &\\
&&&&& 1
\end{pmatrix}\ \right | \  d\in\C^*\right\}.
\] 
Then $\widetilde T^{W_P}\cong \C^*_{\qb}$ via the identification of $\alpha_{n-k}$ and $q$. We have an isomorphism 
\begin{equation}\label{e:psiR}
\begin{array}{ccc}
\psi_R:B_-\cap U_+\widetilde T^{W_P} \dot w_P\dot w_0\inv U_+& \overset{\sim}\To &\RR\x\C^*_{\qb},\\
b = u_1 t \dot w_P\dot w_0\inv u_2   & \mapsto & (b\dot w_0B_-,\alpha_{n-k}(t)),
 \end{array}
 \end{equation}
 as in \cite[Section~4.1]{Rie:MSgen}.
To define the superpotential we need the map  $e_i^* : U_+\to\C$ which sends $u\in U_+$  to its $(i,i+1)$-entry, so $e_i^*(u):=u_{i,i+1}$. This  notation $e_i^*(u)$ stems from the fact that the matrix entry $u_{i,i+1}$ of $u$ can also be thought of as the coefficient of $e_i$ in $u$ after embedding $U_+$ into the completed universal enveloping algebra of its Lie algebra.

The following definition is an equivalent formulation of the definition from \cite[Section~4.2]{Rie:MSgen} as follows from \cite[Lemma~5.2]{Rie:MSgen}. Note that the setting in \cite{Rie:MSgen} is that of an arbitrary complex reductive algebraic group~$G$ and parabolic subgroup $P$, and we apply it here to $G=GL_n$ and $P$ a maximal parabolic. 

 \begin{defn}[The Lie-theoretic superpotential {\cite[Lemma~5.2]{Rie:MSgen}}]  \label{d:RichardsonBmodel} 
 Let $\widetilde {\mathcal F}:B_-\cap U_+\widetilde T^{W_P}\dot w_P\dot w_0\inv U_+\to \C$ be the map defined by
\[
\widetilde {\mathcal F}: b=u_1t \dot w_P\dot w_0\inv u_2\quad \mapsto\quad \sum_{i=1}^n e_i^*(u_1)+\sum_{i=1}^n e_i^*(u_2).
\] 
Note that this map $\widetilde{\mathcal F}$ is well-defined even though $u_1$ and $u_2$ are not uniquely determined by $b$, see
\cite[Equation~(4.4) and Lemma~5.2]{Rie:MSgen}.
For the $A$-model Grassmannian $X=Gr_{n-k}(\C^n)$ viewed as a homogeneous space $GL^\vee_n/P^\vee$, the Lie-theoretic version of the superpotential is the composition
\[
\mathcal F=\widetilde{\mathcal F}\circ \psi_R\inv:\RR\x\C_\qb^*\to\C.
\]
Explicitly,
\[
\mathcal F(b\dot w_0B_-,q)=\sum e_i^*(u_1)+\sum e_i^*(u_2),
\]
if $b$ is in $B_-$ and factorizes as $b=u_1 t \dot w_P\dot w_0\inv u_2$ with  $u_1,u_2\in U_+$ where $t\in \widetilde T^{W_P}$ is determined by $\alpha_{n-k}(t)=q$.
Additionally, we use the notation $\mathcal F_q$ for the map 
$
\mathcal F_q:\RR\to \C
$
defined by $\mathcal F_q(b\dot w_0 B_-)=\mathcal F(b\dot w_0B_-,q)$.
\end{defn}


We now recall the Dale Peterson presentation of the quantum cohomology ring of a homogeneous space which applies as a special case 
to $X=Gr_{n-k}(n)$,
compare \cite{Rie:QCohPFl}.
\begin{thm}[Dale Peterson \cite{Pet:QCoh}] \label{t:Peterson} Associated to the homogeneous space $X$ define a subvariety $Y^*_P$ in $\RR$, called the Peterson variety of $X$, as follows. Let $F\in \mathfrak g^*$ be the sum of the dualised positive Chevalley generators, 
\[F= e_1^*+e_2^*+\dotsc + e_{n-1}^*,
\] 
and set
\[Y^*_P:=\{ gB_-\in\RR \ |\ g\inv\cdot F\in [\mathfrak n_-,\mathfrak n_-]^\perp\},\]
using the coadjoint action of $G$. Let $\C[Y^*_P]$ denote the coordinate ring of $Y_P^*$, in the possibly non-reduced sense.
Then $\C[Y^*_P]$ is isomorphic to the quantum cohomology ring $qH^*(X,\C)[\qb\inv]$ of $X$ by an explicit isomorphism.
\end{thm}
We call the isomorphism from Theorem~\ref{t:Peterson} the {\it Peterson isomorphism}. In Proposition~\ref{p:flambda} we will recall where Peterson's isomorphism takes  a Schubert class in the Grassmannian case. For a description of the Peterson isomorphism for type $A$ partial flag varieties we refer to \cite{Rie:QCohPFl, Rie:JAMSerr}, in general type see also~\cite{Rie:MSgen}. We remark that in type $A$ the quantum cohomology rings, and with them the coordinate rings $\C[Y_P^*]$, are always reduced. 

 The superpotential $\mathcal F$ defined in \cite{Rie:MSgen} is related to the Peterson variety as follows. Denote again by $\qb$ the element of the coordinate ring $\C[Y^*_P]$ which corresponds under the Peterson isomorphism to the quantum parameter. Then $\qb$ is a finite morphism from $Y_P^*$ to $ \C^*$.
 
\begin{thm}[Mirror construction of the Peterson variety {\cite[Theorem~4.1]{Rie:MSgen}}]\label{t:MSgen}
The critical points of $\mathcal F_{\qb}$ inside $\RR$ lie in the Peterson variety and precisely recover the fibers of $\qb:Y^*_P\to \C^*$.  Moreover the subvariety of $\RR\x\C^*_q$ corresponding to the ideal  $ (\partial_{\RR}\mathcal F)$ of partial derivatives of $\mathcal F$ along $\RR$ is isomorphic to $Y_P^*$ by the restriction of the first projection $\RR\x\C^*_q\to \RR$, and we obtain
\[
\C[\RR\x\C^*_\qb]/(\partial_{\RR}\mathcal F)\cong\C[Y^*_P].
\]
\end{thm}

The Lie-theoretic superpotential $(\RR,\mathcal F_q)$ is therefore related to the quantum cohomology of $X$ by the combination of Theorems~\ref{t:Peterson} and \ref{t:MSgen}.

In order to set up the comparison of the Lie-theoretic superpotential with our new formulation in terms of Pl\"ucker coordinates we define the following maps
\[
\Xcheck\overset{\pi_L}\longleftarrow B_-\cap U_+  \dot w_P \dot w_0\inv U_+\overset {\pi_R}\To \RR.
\]
Here the the map on the left hand side is defined by setting $\pi_L(b):=Pb$, and the map on the right hand side is $\pi_R(b)=b\dot w_0 B_-$, where $b\in B_-\cap U_+  \dot w_P \dot w_0\inv U_+$. It is straightforward that $\pi_R$ is well-defined and an isomorphism. The map $\pi_L$ is a priori a map to $\Xcheckbar$,
but it is known to land in $\Xcheck$ and moreover $\pi_L$ is an isomorphism so that we have $\Xcheck\cong  \RR$. Namely the following proposition follows from {\cite[Section 5.4]{KnutsonLamSpeyer:PositroidPublished}, via a construction from  \cite[Section~2.1]{Lus:TPparabolic}}. 
\begin{prop}
\label{p:KLS}
The projection map $\pi_L$ is a well-defined isomorphism from $B_-\cap U_+  \dot w_P \dot w_0\inv U_+$ to $\Xcheck$.
 \end{prop}

In Section~\ref{s:comparisonMRR} we will prove the following proposition.
\begin{prop}\label{p:comparisonMRR} With the definitions from Section~\ref{s:GrassmannianBmodel} and \ref{s:RichardsonBmodel}, the following diagram commutes:
\[
\xymatrix{
\Xcheck\x\C^*_q \ar[d]_W &
B_-\cap U_+  \dot w_P \dot w_0\inv U_+\x \C_\qb^*  
\ar[l]_(0.65){\pi_L \x id} \ar[r]^(0.65){\pi_R \x id}
& 
\RR\x\C_{\qb}^* \ar[d]_{\mathcal{F}} \\
\C& = &  \C.
}
\]
\end{prop}

Propositions~\ref{p:KLS} and \ref{p:comparisonMRR}  imply the part of
Theorem~\ref{t:twosuperpotentials} which states
that the Landau-Ginzburg model from 
Section~\ref{s:GrassmannianBmodel} is isomorphic to the one 
from~\cite{Rie:MSgen} recalled in 
Definition~\ref{d:RichardsonBmodel}. In the case of Lagrangian 
Grassmannians and for odd-dimensional quadrics, formulas 
analogous  to Definition~\ref{d:RichardsonBmodel} and comparison 
results analogous to the above Propositions  were found by C.~Pech and 
K. Rietsch in \cite{PechRie:LGLG,PechRie:Quadrics} and by 
C.~Pech, K. Rietsch and L.~Williams in the case of even 
quadrics~\cite{PechRieWilliams:AllQuadrics}.

\subsection{The Laurent polynomial superpotential}\label{s:LaurentBmodel}\label{s:EHX}

The earliest construction of Landau-Ginzburg models for Grassmannians is due to Eguchi, Hori and Xiong \cite{EHX:GravQCoh} and associates to $X=Gr_{n-k}(\C^n)$
a Laurent polynomial $L_{\qb}$ in $k(n-k)$ variables (with parameter $\qb$).  This Laurent polynomial also appeared in \cite{Oblezin:WhittakerII} in relation to a parabolic analog of the quantum Toda lattice. Let $\mathcal T=(\C^*)^{k(n-k)}$ and define  $L_q:\mathcal T \to \C$ as follows (compare \cite{BCKS:MSGrass,BCKS:MSPartFl}).

Let $\QR=(\mathcal V,\mathcal A)$  be a quiver with vertices given by
\[
\mathcal V=\{(i,j)\in [1, n-k]\x [1, k]\}\sqcup \{(0,1),(n-k,k+1)\}
\]
and with two types of arrows $a\in \mathcal A$, namely
\[  (i,j)\To (i,j+1)\quad \text{and}\quad (i,j)\To (i+1,j),
\]
defined whenever $(i,j),(i,j+1)$, and  $(i,j),(i+1,j)$, respectively, are in $\mathcal V$. We write $h(a)$ for the head of an arrow $a$, and $t(a)$ for the tail.
To every vertex in the quiver associate a coordinate $\dd_{ij}$. We set $\dd_{0,1}=1$ and $\dd_{n-k,k+1}=\qb$, and let the remaining $(\dd_{ij})_{i=1,\dotsc, n-k}^{j=1,\dotsc, k}$ be the coordinates on
the big torus $\mathcal T=(\C^*)^{k(n-k)}$.

\begin{defn}[The EHX Laurent polynomial supotential \cite{BCKS:MSGrass,EHX:GravQCoh}]\label{d:EHXBmodel}
To every arrow $a$ in the quiver  $(\mathcal V,\mathcal A)$ one can associate a Laurent monomial  by dividing the coordinate at the head by the coordinate at the tail. The regular function $L_q:\mathcal T\to \C$ is defined to be the sum of all of the Laurent monomials obtained in this way,
\[
L_q=\sum_{a\in\mathcal A} \frac{\dd_{h(a)}}{\dd_{t(a)}},
\]
keeping in mind that  $\dd_{0,1}=1$ and $\dd_{n-k,k+1}=\qb$ also occur.
\end{defn}


\begin{example}\label{ex:EHX} Consider $k=3$ and $n=5$. So $X=Gr_2(\C^5)$ and $\Xcheckbar=Gr_3((\C^5)^*)$,
the Grassmannian of $3$-planes in the vector space of row vectors.
The big torus is $\mathcal T\cong (\C^*)^6$ with coordinates
$(\dd_{11}, \dd_{12}, \dd_{13}, \dd_{21}, \dd_{22}, \dd_{23})$. The superpotential is
\begin{equation}\label{e:EHXWGr25}
L_q= \dd_{11}+\frac{\dd_{21}}{ \dd_{11}}+\frac{\dd_{22}}{ \dd_{12}}+\frac{\dd_{23}}{ \dd_{13}}
+\frac{\dd_{12}}{ \dd_{11}}+
\frac{\dd_{13}}{ \dd_{12}}+
 \frac{\dd_{22}}{ \dd_{21}}+\frac{\dd_{23}}{ \dd_{22}}+\frac{\qb}{ \dd_{23}},
\end{equation}
and is encoded in the quiver $\QR$ shown below.
\begin{center}

\begin{tikzpicture}
\draw (0,2) [fill] circle (.4mm);
\draw[->](0,0)--(1,0);
\draw[->](1,0)--(2,0);
\draw[->](2,0)--(3,0);
\draw[->](0,2)--(0,1);
\draw[->](0,1)--(0,0);
\draw[->](1,1)--(1,0);
\draw[->](0,1)--(1,1);
\draw[->](1,1)--(2,1);
\draw[->](2,1)--(2,0);
\draw (3.05,0) [fill] circle (.4mm);
\node[right] at (3,0) {$\qb$};
\node [below] at (2,0) {$\dd_{23}$};
\node [below] at (1,0) {$\dd_{22}$};
\node [below] at (0,0) {$\dd_{21}$};
\node [above] at (1,1) {$\dd_{12}$};
\node [above] at (2,1) {$\dd_{13}$};
\node [left] at (0,1) {$\dd_{11}$};
\node [left] at (0,2) {$1$};
\end{tikzpicture}.
\end{center}
\end{example}

The relationship between this superpotential $L_q$ and the superpotential $W_\qb$ from Definition~\ref{d:W} is given in the proposition below. The Pl\"ucker coordinates indexed by rectangular Young diagrams play a special role here,
and we denote a Young diagram which is an $i\times j$ rectangle by
$i\times j$.
For example if $(k,n)=(3,7)$ the Pl\"ucker coordinates corresponding to the rectangles $\mu_i$ are $\p_{\mu_1}=\p_{1\x 3}$, $\p_{\mu_2}=\p_{2\x 3}$ and so forth; compare \eqref{e:muiexample}. We have $k(n-k)$ rectangular Pl\"ucker coordinates, not counting $\p_\emptyset$. 

\begin{prop}\label{p:comparisonMREHX}
There is a (unique) embedding $\iota:\mathcal T\to\Xcheck$ for which the Pl\"ucker coordinates corresponding to rectangular Young diagrams are related to the $\dd_{ij}$ coordinates as follows,
\begin{equation}\label{e:iota}
\dd_{ij}= \iota^*\left(\frac{\p_{i\x j}}{\p_{(i-1)\x (j-1)}}\right).
\end{equation}
Moreover, the Laurent polynomial superpotential $L_q$ agrees with the pullback of $W_\qb$ to $\mathcal T$ under $\iota$.
\end{prop}
\begin{proof}[Sketch of proof] We describe the embedding $\iota:\mathcal T\to \Xcheck $ defined in Proposition~\ref{p:comparisonMREHX} concretely. To explain the construction in our setting we continue with Example~\ref{ex:EHX}. Let us decorate the above quiver by elements $\dot s_i$ as follows and remove the arrow with head labeled $\qb$ .
\begin{center}
 \begin{tikzpicture}
\node at (-2, .4) {Row $E_4$};
\node at (-2, 1.4) {Row $E_3$};
\node at (-2, 2.4) {Row $E_2$};
\node at (-2, 3.4) {Row $E_1$};
\draw (3,0.5) circle (3mm);
\draw (0,2.5) circle (3mm);
\draw (1,2.5) circle (3mm);
\draw (0,3.5) circle (3mm);
\draw (0,2) [fill] circle (.4mm);
\node at (3,0.5) {${\dot s_4}$};
\node at (0,2.5) {${\dot s_2}$};
\node at (1,2.5) {${\dot s_2}$};
\node at (0,3.5) {${\dot s_1}$};
\draw[->](0,0)--(1,0);
\draw[->](1,0)--(2,0);
\draw[->](0,2)--(0,1);
\draw[->](0,1)--(0,0);
\draw[->](1,1)--(1,0);
\draw[->](0,1)--(1,1);
\draw[->](1,1)--(2,1);
\draw[->](2,1)--(2,0);
\node [below] at (2,0) {$\dd_{23}$};
\node [below] at (1,0) {$\dd_{22}$};
\node [below] at (0,0) {$\dd_{21}$};
\node [above] at (1,1) {$\dd_{12}$};
\node [above] at (2,1) {$\dd_{13}$};
\node [left] at (0,1) {$\dd_{11}$};
\node [left] at (0,2) {$1$};
\end{tikzpicture}
\end{center}

The new figure has $n-1$ rows, $E_1,E_2,\ldots ,E_{n-1}$ (where, in the
example, $n=5$). For $1\leq i\leq k-1$, row $E_i$ contains $i$ copies
of $\dot s_i$ (written a circle). For $k\leq i\leq n-1$, row $E_i$
contains all of the downward-pointing arrows with target $z_{i-k+1,j}$
for some $j$ (one arrow if $i=k$; $k$ arrows if $i>k$), followed by
$i-k$ copies of $\dot s_i$.

We call a path in the quiver which has precisely one vertical step a {\it $1$-path}. Notice that each $1$-path contains a downward-pointing arrow from exactly one row. For any vertex $v$ decorated with a $z_{ij}$ there is clearly a unique minimal length $1$-path which has this vertex at its lower end. We call this $1$-path the \emph{minimal} $1$-path with the given vertex $v$ at its base.

Then we read off a sequence of $\dot s_i$'s and $1$-paths, going column by column from right to left. Namely in each column we list, starting from the top and going down, any $\dot s_i$'s, followed by any minimal $1$-paths associated to
vertices from that column. In the example above the sequence is shown below.
 \begin{center} \label{e:1pathlist}
 \begin{tikzpicture}
 \node at (-1,.5) {$\dot s_4 $};
 \draw (-1,.5) circle (3mm);
  \node at (-.5,.4) {,};
\node [left] at (0,1) {$1$};
\draw (0,1) [fill] circle (.4mm);
\draw[->](0,1)--(0,0);
\draw[->](0,0)--(1,0);
\draw[->](1,0)--(2,0);
\node [above] at (2,0) {$\dd_{13} \ $};
  \node at (2.5,.4) {,};
\node [above] at (3,1) {$\dd_{13} \ $};
\draw[->](3,1)--(3,0);
\node [below] at (3,0) {$\dd_{23}\ $};
  \node at (3.5,.4) {,};
\node at (4,.5) {$\dot s_2$};
\draw (4,.5) circle (3mm);
  \node at (4.5,.4) {,};
\node [left] at (5,1) {$1$};
\draw (5,1) [fill] circle (.4mm);
\draw[->](5,1)--(5,0);
\draw[->](5,0)--(6,0);
\node [above] at (6,0) {$ \dd_{12}\ $};
  \node at (6.5,.4) {,};
\draw[->](7,1)--(7,0);
\node [above] at (7,1) {$\dd_{12}$};
\node [below] at (7,0) {$\dd_{22} \ $};
  \node at (7.5,.4) {,};
\node at (8,.5) {$\dot s_1$};
\draw (8,.5) circle (3mm);
  \node at (8.5,.4) {,};
\node at (9,.5) {$\dot s_2$};
\draw (9,.5) circle (3mm);
  \node at (9.5,.4) {,};
\draw[->](10,1)--(10,0);
\draw (10,1) [fill] circle (.4mm);
\node [above] at (10,1) {$1$};
\node [below] at (10,0) {$\dd_{11}\ $};
  \node at (10.5,.4) {,};
\draw[->](11,1)--(11,0);
\node [above] at (11,1) {$\dd_{11}$};
\node [below] at (11,0) {$\dd_{21}\ $};
\end{tikzpicture}
\end{center}
To a  $1$-path $\gamma$ in row $E_i$ we associate a factor 
$x_i\left(\frac{\dd_{h(\gamma)}}{\dd_{t(\gamma)}}\right)$, where 
$t(\gamma)$ is the initial vertex and $h(\gamma)$ is the final vertex of 
$\gamma$. The other elements of the sequence correspond in
the obvious way to factors $\dot s_i$. These factors are all 
multiplied together to give an element $g_{(\dd_{ij})}$ in the big Bruhat 
double coset $B_-\dot w_0 B_-$ 
of $GL_n$. In the above example we have the element of $GL_5$ given by
 \begin{equation*}\label{e:g2Gr25}
g_{(\dd_{ij})}:={\dot s_4}x_3\left({\dd_{13}}\right)x_4\left(\frac{\dd_{23}}{\dd_{13}}\right) {\dot s_2} x_3\left({\dd_{12}}\right)
 x_4\left(\frac{\dd_{22}}{\dd_{12}}\right){\dot s_1} {\dot s_2} x_3\left({\dd_{11}}\right) x_4\left(\frac{\dd_{21}}{\dd_{11}}\right).
\end{equation*}
Note that one can check that $\dot w_0\inv g_{(\dd_{ij})}\in \dot w_P\dot w_0\inv U_+$ by permuting all of the $\dot s_i$ factors in $g_{(\dd_{ij})}$ to the left. 
Since $g_{(\dd_{ij})}$ is also in the big Bruhat cell we have $\dot w_0\inv g_{(\dd_{ij})}\in \dot w_P\dot w_0\inv U_+\cap B_+ B_-$.
 
We now define the map $\iota:\mathcal T\to \Xcheck$ by
\[
\iota:(\dd_{ij})\mapsto P \, \dot w_0\inv\, g_{(\dd_{ij})}.
\]

To see that this map is the embedding alluded to in the proposition  it suffices to consider the $\p_{i\x j}$ Pl\"ucker coordinates of $P\, \dot w_0\inv\, g_{(\dd_{ij})}$ and check that these are related to the $\dd_{ij}$  
as follows,
\begin{equation}\label{e:dijviaminors}
\ \dd_{11}=\frac{\p_{\ydiagram{1}}}{\p_\emptyset}\ ,\   \dd_{12}=\frac{\p_{\ydiagram{2}}}{\p_\emptyset}\ ,\ \dd_{13}=\frac{\p_{\ydiagram{3}}}{\p_\emptyset}\ ,
 \dd_{21}=\frac{\p_{\ydiagram{1,1}}}{\p_\emptyset}\ ,\ \dd_{22}=\frac{\p_{\ydiagram{2,2}}}{\p_{\ydiagram{1}}}\ , \
\dd_{23}=\frac{\p_{\ydiagram{3,3}}}{\p_{\ydiagram{2}}}.
\end{equation}
%
It is easy to check that this holds in general, so that we have \eqref{e:iota}.
Finally, it is straightforward to compute the $\p_{\widehat\mu_i}$ Pl\"ucker coordinates of $P\, \dot w_0\inv\, g_{(\dd_{ij})}$ and see that substituting $P\, \dot w_0\inv \, g_{(\dd_{ij})}$ into the formula~\eqref{e:W} for $W$
recovers the Laurent polynomial  $L_q$.
\end{proof}

\begin{rem}
The construction of the map $\iota$ is inspired by the construction of factorisations  of elements in the Peterson variety introduced in~\cite{Rie:TotPosGBCKS}. In fact the two constructions are essentially related by a reflection of the quiver, see the proof of Proposition~\ref{p:involution} where both factorisations are needed.
\end{rem}

\subsection{Proof of Proposition~\ref{p:comparisonMRR}}\label{s:comparisonMRR}

In this section we prove Proposition~\ref{p:comparisonMRR}, which says that the superpotential $W$ defined in \eqref{e:W} is isomorphic to the Lie-theoretic superpotential $\mathcal F$ from \cite{Rie:MSgen}, see Definition~\ref{d:RichardsonBmodel}.

Let $\widetilde{W}$ denote the pullback of $\mathcal F$ to $\Xcheck\x \C^*_\qb$ via $(\pi_L\x id)\inv\circ(\pi_R\x id)$;
\begin{equation}
\label{e:Wtildediagram}
\xymatrix{
\Xcheck\x\C^*_q \ar[drr]_{\widetilde{W}} & \ar[l]_(0.67){\pi_L\x id} (B_-\cap U_+  \dot w_P \dot w_0\inv U_+)\x \C_\qb^* \ar[r]^(0.67){\pi_R\x id} & \RR\x\C_{\qb}^* \ar[d]^{\mathcal{F}} \\
&& \C
}
\end{equation}
compare Section~\ref{s:RichardsonBmodel}. Here we are keeping in mind Proposition~\ref{p:KLS} which says that $\pi_L$ is an isomorphism. Then to show Proposition~\ref{p:comparisonMRR} we need to prove that $\widetilde{W}$ agrees with $W$. 

Let us assume that $Pg\in \Xcheck$ is of the form
\begin{equation}\label{e:Pg}
Pg=P\dot w_0\inv g_{(\dd_{ij})}
\end{equation}
for an element $(\dd_{ij})\in \mathcal T$; compare Section~\ref{s:EHX}. 
The subset $\Xcheckfact$ of $\Xcheck$ consisting of such factorisable elements $Pg$ is an open dense subset of $\Xcheck$. Therefore  it suffices to show that $\widetilde{W}$ and $W$ agree on $\Xcheckfact$. 

\begin{defn} \label{d:ufact}
Note that any element $\dot w_0\inv g_{(\dd_{ij})}$ can also be written in the form $\dot w_P \dot w_0\inv u_2$ for 
an element $u_2\in U_+$ which can be factorized as
$u(1)\cdots u(k)$, where each $u(j)$ is a product of root subgroups,
\[
u(j)=x_{\alpha_j+\cdots +\alpha_k}(m_{jk})x_{k+1}(m_{j,k+1})\cdots x_{n-1}(m_{j,n-1}).
\]
Let $U_+^{\fact}\subset U_+$ denote the subset of $U_+$ consisting of such factorized elements $u_2=u(1)\cdots u(k)$, with nonzero entries $m_{j,l}$ in all of the root subgroup factors.
\end{defn}
We have $\Xcheckfact=P\dot w_P\dot w_0\inv U_+^{\fact}$. Therefore we may rewrite $Pg$ from \eqref{e:Pg} as
\[
Pg=P\dot w_P\dot w_0\inv u_2,
\]
where $u_2\in  U_+^{\fact}$. We can now define a map $\tilde \mu:U_+^{\fact}\to U_+$ by requiring 
\[
\tilde \mu(u_2)\inv B_-=\dot w_P\dot w_0\inv u_2B_-
 \]
for all $u_2\in U_+^{fact}$.
Notice that $\dot w_P\dot w_0\inv u_2B_-$ lies in $U_+B_-/B_-$ since it equals $\dot w_0\inv g_{(\dd_{ij})}B_-\in \dot w_0\inv B_-\dot w_0 B_-=U_+ B_-$, and therefore $\tilde \mu$ is well-defined. If the context is clear we will write
$\tilde\mu(u_2):=u_{1,0}$, noting that then $u_{1,0}$ is the  unique element in $U^+$ for which 
\[
u_{1,0}\dot w_P\dot w_0\inv u_2\in B_-.
\]

\begin{lem} \label{l:wtildeformula}
Let $u_2\in U_+^{\fact}$. We have the following formula for $\widetilde W$ on $\Xcheckfact\x \C^*_q$,
\[
\widetilde {W}(P\dot w_P\dot w_0\inv u_2,q)=\sum_i e_i^*(u_2)+\sum_{i\ne n-k}e_i^*(u_{1,0})+ q e_{n-k}^*(u_{1,0}),
\]
where $u_{1,0}=\tilde \mu(u_2)$.
\end{lem}

\begin{proof}
This lemma is straightforward. Let $t_q\in \widetilde T^{W_P}$ be the unique element with $\alpha_{n-k}(t_q)=q$. Recall the isomorphism $\psi_R$ from  \eqref{e:psiR}. We define an analogous isomorphism $\psi_L$ by the following commutative diagram
\[
\xymatrix{
\Xcheck\x\C^*_q &
\ar[l]_(0.65){\pi_L\x id} B_-\cap U_+  \dot w_P \dot w_0\inv U_+\x \C_\qb^* \ar[d]^{\psi} \ar[r]^(0.65){\pi_R\x id} &
\RR\x\C_{\qb}^*\\
& \ar[ul]^{\psi_L} \ar[ur]_{\psi_R}
B_-\cap U_+ \widetilde T^{W_P}\dot w_P\dot w_0\inv U_+ &
}
\]
in which every arrow is an isomorphism;
compare  Proposition~\ref{p:KLS} along with the paragraph preceding it. The connecting isomorphism $\psi$ in the middle is just the map
\[
\begin{array}{ccc}
\psi:  B_-\cap U_+  \dot w_P \dot w_0\inv U_+\x \C_\qb^*  &\to  &  B_-\cap U_+ \widetilde T^{W_P}\dot w_P\dot w_0\inv U_+\\
(b_0,q)&\mapsto &b:=t_q b_0.
\end{array}
\]

We can express $\widetilde W$ as a composition
\[
\begin{array}{ccccccl}
\widetilde W: 
& \Xcheck\x \C^*_{\qb} 
&\overset {\psi_L\inv}\longrightarrow  
& B_-\cap U_+ \widetilde T^{W_P}\dot w_0 U_+ 
&\overset{\widetilde{\mathcal F}} \To & \C \\
& (P\dot w_P\dot w_0\inv u_2,q) &\mapsto &
b=t_q u_{1,0}\dot w_P\dot w_0\inv u_2 &\mapsto &  \widetilde {\mathcal F} (b)=\mathcal F(b\dot w_0 B_-,q),
\end{array}
\]
where $\widetilde{\mathcal F}$ is as in Definition~\ref{d:RichardsonBmodel}. Clearly,
since $t_q u_{1,0}\dot w_P\dot w_0\inv u_2 = u_1 t_q \dot w_P\dot w_0\inv u_2$ for $u_1=t_q u_{1,0}t_q\inv$, we have  
\[\widetilde {\mathcal F} (t_q u_{1,0}\dot w_P\dot w_0\inv u_2)=
\sum_i e_i^*(u_2)+\sum_{i}e_i^*(u_1)=\sum_i e_i^*(u_2)+\sum_{i\ne n-k}e_i^*(u_{1,0})+ q e_{n-k}^*(u_{1,0}),
\]
as required.
\end{proof}

Let  $\Delta^I_J(g)$ denote the minor with row set $I$ and column set $J$, and 
recall that $J_i=[i+1,i+k]$, in interval notation, and $\widehat J_i=[i+1,i+k-1]\cup \{i+k+1\}$.
We then have the following lemma about minors. 

\begin{lem} \label{l:eistar}
Let $u_2\in U_+^{\fact}$ and $u_{1,0}=\tilde \mu(u_2)$ and $b=u_{1,0}\dot w_P\dot w_0\inv u_2\in B_-$. Then
we have
\begin{eqnarray}
e_i^*(u_{1,0})&=&
\begin{cases}
0 & 1\leq i\leq n-k-1; \\
\frac{\Delta_{\widehat J_i}^{[n-k+1,n]}(b)}{\Delta_{J_i}^{[n-k+1 ,n]}(b)} & n-k\leq i\leq n-1.
\end{cases}\label{e:eiu10}
\\
e_i^*(u_2)&=&
\begin{cases}
0 & 1\leq i\leq k-1; \\
\frac{\Delta_{\widehat J_{i-k}}^{[n-k+1,n]}(b)}{\Delta_{J_{i-k}}^{[n-k+1,n]}(b)} & k\leq i\leq n-1. \label{e:eiu2}
\end{cases}
\end{eqnarray}
\end{lem}
\begin{proof}
Since $u_2\in U_+^{\fact}$, we have a factorization $u_2=u(1)\cdots u(k)$ as in Definition~\ref{d:ufact}.
We denote by $v_1,\ldots ,v_n$ the standard basis of
the defining representation $V=\mathbb{C}^n$ for $GL_n(\mathbb{C})$.

We first consider the proof of~\eqref{e:eiu2}.
For $1\leq i\leq k-1$, it follows from the above factorization
of $u_2$ that $u_2\cdot v_{i+1}=v_{i+1}$, so $e_i^*(u_2)=0$ as required.
For $k\leq i\leq n-1$, an inductive argument using
the factorization of $u_2$ shows that
$$u_2\cdot v_{i-k+1}\wedge \cdots \wedge v_{i+1}=0.$$
It follows that
$$\Delta_{[i-k+1 ,i+1]}^{[1,k]\cup \{i\}}(u_2)=0.$$
Expanding this minor along the last row and noting that
$u_2$ is upper unitriangular, this implies that
$$\Delta_{[i-k+1 ,i-1]\cup \{i+1\}}^{[1,k]}(u_2)-\Delta_{[i-k+1,i]}^{[1,k]}(u_2)e_i^*(u_2)=0.$$
Hence, using the fact that $b=u_{1,0}\dot w_P\dot w_0^{-1}u_2$,
$$e_i^*(u_2)=
\frac{ \Delta_{[i-k+1,i-1]\cup \{i+1\}}^{[1, k]}(u_2)}
{\Delta_{[i-k+1,i]}^{[1,k]}
(u_2)}=\frac{\Delta_{\widehat J_{i}}^{[n-k+1,n]}(b)}{\Delta_{J_i}^{[n-k+1 ,n]}(b)}.$$
We now consider the proof of~\eqref{e:eiu10}.
The explicit factorization of $u_2$ implies that
the matrix $u_2^{-1}\dot w_0 \dot w_P^{-1}$ has the form
$\begin{pmatrix} A & B \\ C & D \end{pmatrix},$
where $A$ is a $k\times (n-k)$ matrix with zeros above the 
leading diagonal (i.e.\ the entries $A_{ij}$ with $j>i$ are all zero)
$B$ is a $k\times k$ identity matrix, $C$ is an upper
triangular
$(n-k)\times (n-k)$ matrix with $(-1)^k$ on the diagonal,  and $D$ is a zero $(n-k)\times k$ matrix.

Since $b^{-1}u_{1,0}=u_2^{-1}\dot w_0\dot w_P^{-1}$, we have,
for $1\leq i\leq n-1$, that
$$e_i^*(u_{1,0})=\frac{\Delta_{[1,i-1]\cup\{i+1\}}^{[1,i]}(u_2^{-1}\dot w_0\dot w_P^{-1})}
{\Delta_{[1,i]}^{[1,i]}(u_2^{-1}\dot w_0\dot w_P^{-1})}.$$
If $1\leq i\leq n-k-1$,
this is zero since the entries in the first $i$ rows
of column $i+1$ of $u_2^{-1}\dot w_0\dot w_P^{-1}$ are all zero.
If $n-k+1\leq i\leq n-1$, then,
using the above description of
$u_2^{-1}\dot w_0\dot w_P^{-1}$, we have
\begin{equation*}
e_i^*(u_{1,0})=
-\frac{\Delta_{[1,n-k]}^{\{i-n+k\}\cup [i-n+k+2, i]}(u_2^{-1}\dot w_0\dot w_P^{-1})}{\Delta_{[1,n-k]}^{[i-n+k+1,i]}(u_2^{-1}\dot w_0\dot w_P^{-1})}.
\end{equation*}
Since $b^{-1}u_{1,0}=u_2^{-1}\dot w_0\dot w_P^{-1}$ and $u_{1,0}\in U^+$, we obtain
\begin{equation*}
e_i^*(u_{1,0}) =
-\frac{\Delta_{[1,n-k]}^{\{i-n+k\}\cup [i-n+k+2,i]}(b^{-1})}{\Delta_{[1 ,n-k]}^{[i-n+k+1,i]}(b^{-1})} 
=
\frac{\Delta_{\widehat{J_i}}^{[n-k+1,n]}(b)}{\Delta_{J_i}^{[n-k+1 ,n]}(b)},
\end{equation*}
as required (using Jacobi's Theorem for the minors of an inverse matrix).
A similar argument can be made in the case $i=n-k$.
\end{proof}

Proposition~\ref{p:comparisonMRR} follows from diagram~\eqref{e:Wtildediagram}, Lemma~\ref{l:wtildeformula}
and Lemma~\ref{l:eistar}, since the $\frac{\Delta_{\widehat{J_{i}}}^{[n-k+1 ,n]}(b)}{\Delta_{J_{i}}^{[n-k+1,n]}(b)}$ are nothing other than the summands of $W(Pb)$ as defined  in \eqref{e:W}. Therefore this concludes the proof of Proposition~\ref{p:comparisonMRR}. 

In these last three sections we have proved that we have an isomorphism $\pi_R\circ\pi_L\inv:\Xcheck\to \mathcal R$ (see Proposition~\ref{p:KLS}), and that under this isomorphism the superpotentials $\mathcal F$ and $W$ are identified (see Proposition~\ref{p:comparisonMRR}). Also we have demonstrated an embedding of a $k(n-k)$-dimensional torus $\mathcal T$ into $\Xcheck$ for which $W$ restricts to the Laurent polynomial superpotential $L_q$ (see Proposition~\ref{p:comparisonMREHX}). To finish up the proof of Theorems~\ref{t:EHXcomparison} and \ref{t:twosuperpotentials} from the introduction it remains to compare the holomorphic volume forms on the domains of these three superpotentials. This will be done in Section~\ref{s:omegalemmas}, after we have introduced the cluster structure of the Grassmannian.




\section{The coordinate ring $\C[\Xcheck]$ as a cluster algebra}
\label{s:CAcoordring}

By~\cite[Thm.\ 3]{Scott:Grassmannian}, the homogeneous coordinate ring of the Grassmannian $\Xcheckbar$ has a cluster algebra
structure (see also \cite[\S3]{GSV:ClusterPoisson},\cite[Thm. 4.17]{GSV:Book}). In the latter this cluster algebra structure is shown to induce a cluster algebra structure on $\C[\Xcheck]$. We now recall these constructions.

A skew-symmetric \emph{cluster algebra} with frozen variables is defined as
follows~\cite[\S5]{FoZe:ClusterAlgebrasI}.
Let $r,m\in\mathbb{N}$ and consider the field $\mathbb{F}=\mathbb{C}(u_1,\ldots ,u_{r+m})$
of rational functions in $r+m$ indeterminates.
A \emph{seed} in $\mathbb{F}$ is a pair $(\xx,\widetilde{Q})$
where $\xx=\{x_1,x_2,\ldots ,x_{r+m}\}$ is a set
freely generating $\mathbb{F}$ as a field over $\mathbb{C}$
and $\widetilde{Q}$ is a quiver with vertices
$1,2,\ldots ,r+m$ which has no loops ($1$-cycles) or $2$-cycles.
The vertices $r+1,\ldots r+m$ are said to be \emph{frozen}, and there are no
arrows between them. The corresponding
variables are called \emph{frozen variables}. The subset $\mathbf{x}=\{x_1,x_2,\ldots ,x_r\}$ of $\xx$
is known as a \emph{cluster} while $\xx$ is known as an \emph{extended cluster}.

Given $1\leq k\leq r$, the seed $(\xx,\widetilde{Q})$ can be mutated at $k$ to produce a new
seed $\mu_k(\xx,\widetilde{Q})=(\xx',\mu_k\widetilde{Q})$ where
$\xx'=(\xx\setminus \{x_k\})\cup \{x'_k\}$, and
$$x_kx'_k=\prod_{i\rightarrow k}x_i+\prod_{k\rightarrow i}x_i.$$
The new quiver, $\mu_k\widetilde{Q}$, is obtained from $Q$ as follows:
\begin{enumerate}
\item For every path $i\to k\to j$ in $Q$, add an arrow $i\to j$
(with multiplicity).
\item Reverse all arrows incident with $k$.
\item Remove a maximal collection of $2$-cycles in the resulting quiver.
\end{enumerate}
The \emph{cluster algebra} associated to $(\xx,\widetilde{B})$ is the
$\mathbb{C}$-subalgebra of $\mathbb{F}$ generated by the elements of the
extended clusters which can be obtained from $(\xx,\widetilde{Q})$ by arbitrary
finite sequences of mutations; these
elements are called \emph{cluster variables}. Note that the cluster algebra can be defined over
$\mathbb Z$ or $\mathbb Q$.

Recall that in the $B$-model we are working with $\Xcheckbar=P\backslash GL_n^{\vee}$, a Grassmannian of $k$-planes in $(\C^n)^*$, in its Pl\"ucker embedding.  We have the following:
\begin{thm} \cite[Thm. 3]{Scott:Grassmannian}
(see also~\cite[\S3]{GSV:ClusterPoisson},\cite[Thm. 4.17]{GSV:Book}).
The homogeneous coordinate ring $\C[\Xcheckbar]$ is a
cluster algebra.
\end{thm}

We follow~\cite{Scott:Grassmannian}, which describes a cluster structure on
$\C[\Xcheckbar]$ in terms of \emph{Postnikov diagrams}, i.e.
\emph{alternating strand diagrams} from~\cite[Defn. 14.1]{Postnikov:Totalpositivity}.
We restrict here to the Postnikov diagrams arising in the cluster structure
of the Grassmannian.

\begin{defn} \label{defn:Postnikov}
A \emph{Postnikov diagram} of type $(k,n)$
consists of a disk $\mathbb{D}$ with $2n$ marked points
\sloppy $b_1,b'_1,b_2,b'_2,\ldots ,b_n,b'_n$ marked clockwise on its boundary,
together with $n$ smooth oriented curves in the disk, known as \emph{strands}.
Strand $i$ starts at $b_i$ or $b'_i$ and ends at $b_{i+k}$ or $b'_{i+k}$.
Here we regard strands (and thus the subscripts
of the $b_i$) as elements of $[1,n]$ interpreted
modulo $n$.
The arrangement must satisfy the following additional conditions:

\begin{enumerate}[(a)]
\item Only two strands can intersect at any given point and all such
crossings must be transversal.
\item There are finitely many crossing points.
\item If strand $i$ starts at $b_i$
(respectively, $b'_i$), the first strand crossing it (if such a strand exists) comes from the right (respectively, left). Similarly, if strand $i$ ends at $b_{\reduce{i+k}}$ 
(respectively, $b'_{\reduce{i+k}}$), the last string crossing 
it (if such a strand
exists) comes from the right (respectively, left). Following a strand from
its starting point to its ending point, the crossings alternate between left and
right.
\item A strand has no self-crossings.
\item Suppose two strands meet at more than one point. For any two distinct intersection points $p$ and $q$ one strand must be oriented from $p$ to $q$ and the other from $q$ to $p$.  
\end{enumerate}

Postnikov diagrams are considered up to \emph{isotopy} (noting that such an isotopy can neither create nor delete crossings).
One may also consider  \emph{twisting/untwisting moves} and \emph{boundary twists}; see Figures~\ref{fig:untwistingmove} and~\ref{fig:boundarytwist} (note that these are not isotopies).
Two Postnikov diagrams are said to be equivalent if one can be obtained from the other using a sequence of such moves.
These moves are local in the sense that no other strands must
cross the strands involved in the area where the rule is applied.

\begin{figure}
\includegraphics[width=6cm]{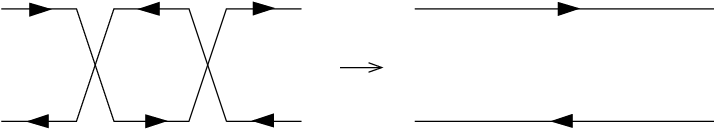}
\caption{The twisting/untwisting move.}
\label{fig:untwistingmove}
\end{figure}

\begin{figure}
\psfragscanon
\psfrag{bi}{$b_i$}
\psfrag{bip}{$b'_i$}
\includegraphics[width=6cm]{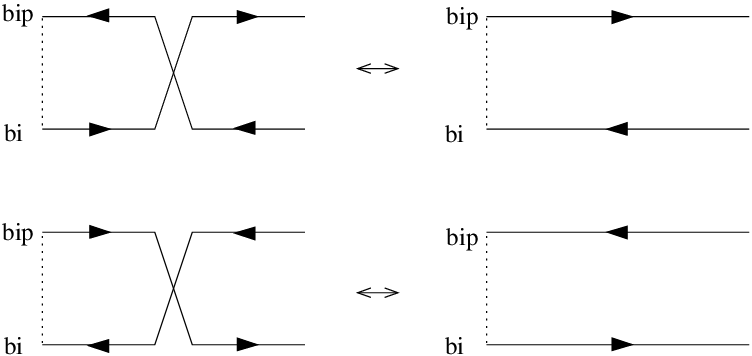}
\caption{The boundary twist.}
\label{fig:boundarytwist}
\end{figure}
\end{defn}

When actually drawing Postnikov diagrams, we usually drop the labels of the
vertices $b_i$ and $b'_i$ and instead indicate the start of strand $i$ by writing
an $i$ in a circle (i.e.\ at $b_i$ or $b'_i$) and the end of strand $i$ (i.e.\ at
$b_{\reduce{i+k}}$ or $b'_{\reduce{i+k}}$) by writing $i$ in a rectangle.
We draw a dotted line between $b_i$ and $b'_i$ to make it clearer where they are.
Thus each $i$ in a rectangle should be linked by a dotted line to $\reduce{i+k}$ in a circle.
For an example of a Postnikov diagram, of type $(3,6)$,
see Figure~\ref{fig:post36}.

\begin{figure}
\psfragscanon
\psfrag{1s}{\pscirclebox{$1$}}
\psfrag{1e}{\psframebox{$1$}}
\psfrag{2s}{\pscirclebox{$2$}}
\psfrag{2e}{\psframebox{$2$}}
\psfrag{3s}{\pscirclebox{$3$}}
\psfrag{3e}{\psframebox{$3$}}
\psfrag{4s}{\pscirclebox{$4$}}
\psfrag{4e}{\psframebox{$4$}}
\psfrag{5s}{\pscirclebox{$5$}}
\psfrag{5e}{\psframebox{$5$}}
\psfrag{6s}{\pscirclebox{$6$}}
\psfrag{6e}{\psframebox{$6$}}
\psfrag{p}{$\emptyset$}
\includegraphics[width=8cm]{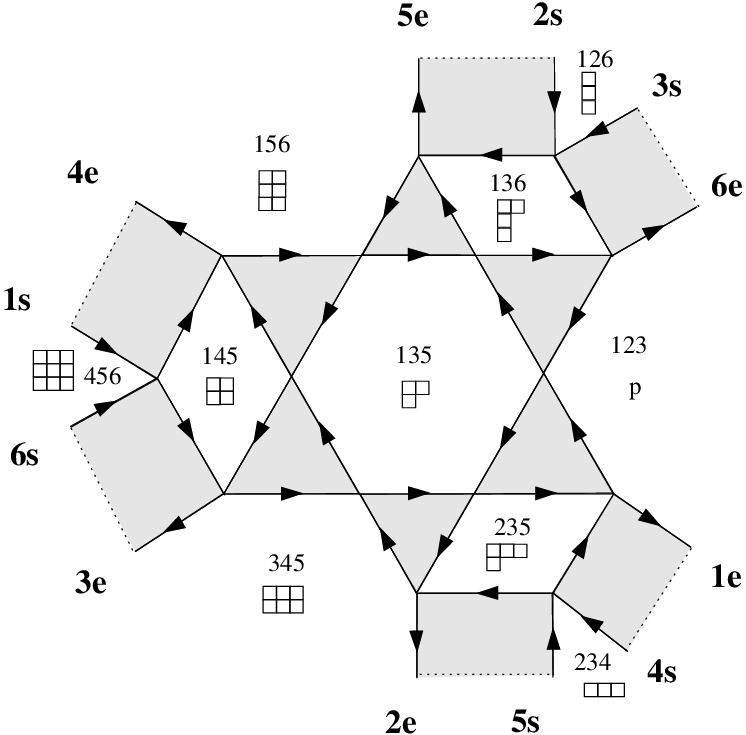}
\caption{A Postnikov diagram for $Gr_3(6)$.}
\label{fig:post36}
\end{figure}

The complement of a Postnikov diagram in the disk is a disjoint union of disks, called \emph{faces} (note that faces often appear as polygons in the Figures, e.g.\ Figure~\ref{fig:post36}).
A face whose boundary includes part of the boundary of the disk $\mathbb D$ is
called a \emph{boundary face}. A face whose boundary (excluding the boundary of $\mathbb D$) is oriented (respectively, alternating) is said to be an \emph{oriented}
(respectively, \emph{alternating}) \emph{face}; it is easy to check that all faces
are of one of these types.

We label each alternating face $F$ with the subset $L(F)$ of
$[1,n]$ which contains $i$ if and only if $F$ lies to the left of
strand $i$. The corresponding Pl\"ucker coordinate is denoted by $p_F=p_{L(F)}$.

The \emph{geometric exchange} on a Postnikov diagram is the local move shown in Figure~\ref{fig:flipmove}.

\begin{figure}
\centering
\includegraphics[width=6cm]{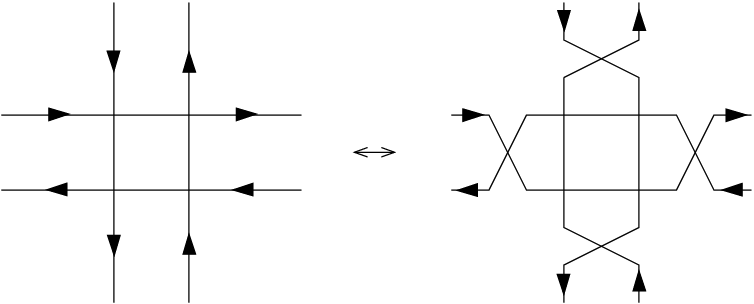}
\caption{Geometric Exchange.}
\label{fig:flipmove}
\end{figure}

We recall the following (see~\cite[Props. 5 and 6]{Scott:Grassmannian}; see also~\cite[\S1]{ops} for more recent developments).

\begin{thm}[Postnikov] \label{theorem:Postnikov}
\begin{enumerate}
\item[(a)]
Each Postnikov diagram of type $(k,n)$ has exactly $k(n-k)+1$ alternating faces.
\item[(b)]
Each alternating face is labelled by a $k$-subset of
$[1,n]$.
\item[(c)]
Every $k$-subset of $[1,n]$ appears as the label of an alternating face in some Postnikov diagram of type $(k,n)$.
\item[(d)]
Any two Postnikov diagrams of type $(k,n)$ (up to equivalence) are connected by a sequence of geometric exchanges.
\item[(e)] The labels of the faces on the boundary of any Postnikov diagram
are the $L_i=[\reduce{i-k+1},i]$ for $i=1,2,\ldots ,n$.
Indeed, $L_{i}$ labels the boundary face between $b'_i$ and $b_{\reduce{i+1}}$.
\end{enumerate}
\end{thm}

A Postnikov diagram of type $(k,n)$ encodes a seed for the cluster algebra structure of the homogeneous coordinate ring of the Grassmannian $\Xcheckbar$ as follows.
Scott~\cite[Sect.\ 5]{Scott:Grassmannian} defines a quiver $Q=Q(D)$ for any
Postnikov diagram $D$. The vertices of $Q$ are the alternating faces of $D$.
The arrows between vertices correspond to
points of incidence of the corresponding faces, such
that whenever two faces $X,Y$ of $D$ are related as in Figure~\ref{fig:neighbourfaces}
there is an arrow in $Q$ from $X$ to $Y$.

We regard $Q$ as being embedded in $\mathbb{D}$, with each vertex mapping
to a point in the middle of the corresponding alternating face and each
arrow drawn as a line between its endpoints passing through the corresponding point of incidence of the corresponding faces.
We will consider the label of an alternating
region to also label the corresponding vertex in $Q$. We refer to the vertices
$L_i$ of $Q$ as its \emph{boundary vertices}.

\begin{figure}
\psfragscanon
\psfrag{I}{$I$}
\psfrag{J}{$J$}
\psfrag{D}{$D$}
\psfrag{Q(D)}{$Q(D)$}
\includegraphics[width=3cm]{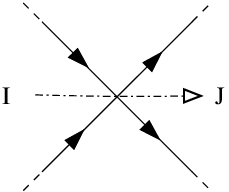}
\caption{Neighbouring faces in a Postnikov diagram.}
\label{fig:neighbourfaces}
\end{figure}

We consider the field $\mathbb{F}$ obtained by adjoining to $\mathbb{C}$ indeterminates
$u_I$ for $I$ the label of an alternating face in $D$.
Let $\xx(D)$ be the free generating set for $\mathbb{F}$ containing these
indeterminates $u_I$ for $I$ coming from $D$. We regard the indeterminates corresponding to boundary
faces (the $L_i$) as frozen variables. Note that there are $k(n-k)-n+1$
non-frozen variables and $n$ frozen variables, making a total of $k(n-k)+1$
variables. Each variable is naturally associated to an alternating face of
$D$ and thus to a vertex of $Q(D)$.

\begin{defn}
Fix a Postnikov diagram $D_0$ of type $(k,n)$. We set $\A$ to be the cluster algebra corresponding to the seed
$(\xx(D_0),Q(D_0))$.

Recall for a $k$-subset $I$ of $[1, n]$ there is an associated Pl\"ucker coordinate denoted by $p_I$; see Section~\ref{s:GrassmannianBmodel}. For a general Postnikov diagram $D$ we denote by
$\CI(D)$ the set of Pl\"ucker coordinates labelling non-boundary alternating faces of $D$, and by $\CC(D)$ the set of Pl\"{u}cker coordinates labelling arbitrary alternating faces of $D$.
\end{defn}

\begin{thm} \cite[Thm.\ 2]{Scott:Grassmannian} \label{thm:Grassmanniancluster}
\begin{enumerate}[(a)]
\item
There is an isomorphism $\varphi$ from $\C[\Xcheckbar]$ to $\A$ taking
$p_I$ to $u_I$ for each $p_I$ in $\CC(D_0)$.
\item
Let $D$ be an arbitrary Postnikov diagram of type $(k,n)$.
Then $S(D):=(\varphi(\CC(D)),Q(D))$ is a seed of $\A$.
\item
If $D,D'$ are two Postnikov diagrams of type $(k,n)$ related by a geometric
exchange corresponding to a quadrilateral face $X$ of $D'$ then $S(D')$ is
the mutation at $p_X$ of $S(D)$.
\end{enumerate}
\end{thm}

We see that the elements $u_I=\varphi(p_I)$ of $\A$, for $I$ a $k$-subset of $[1,n]$
are all cluster variables.

Let $\A'$ be the cluster algebra defined the same way as for $\A$ except that the elements $u_{J_i}^{-1}$, for $i=1,2,\ldots ,n$, of $\mathbb{F}$
are added to the generating set. Thus $\A'$ is the localisation of $\A$
obtained by adjoining inverses to the elements $u_{J_i}$
(see~\cite[Sect.\ 3.4]{GSV:Book}). Recall that $\Xcheck$ is defined to be the
subset of $\Xcheckbar$ where the $p_{J_i}$ do not vanish.
We have the following:

\begin{prop} \label{prop:Invertedcoefficients}
\begin{enumerate}
\item[(a)]
There is an isomorphism $\varphi'$ from
$\mathbb{C}[\Xcheck]$ to $\A'$, taking
$p_I$ to $u_I$ for each $p_I$ in $\CC(D_0)$.
\item[(b)]
Let $D$ be an arbitrary Postnikov diagram of type $(k,n)$.
Then $S(D):=(\varphi(\CC(D)),Q(D))$ is a seed of $\A'$.
\item[(c)]
If $D,D'$ are two Postnikov diagrams of type $(k,n)$
related by a geometric exchange corresponding to a quadrilateral face $X$
of $D$ then $S(D')$ is the mutation at $p_X$ of $S(D)$.
\end{enumerate}
\end{prop}

The proof of this result involves applying~\cite[Prop.\ 
3.37]{GSV:Book} to get (a) (see also~\cite[Prop.\ 11.1]
{FoZe:ClusterAlgebrasII}). Parts (b) and (c) can be
shown by following the proof in~\cite{Scott:Grassmannian}
of Theorem~\ref{thm:Grassmanniancluster}.

\begin{defn} We identify $\C[\Xcheck]$ with $\mathcal A'$ via the isomorphism $\varphi'$. We refer to any seed $(\CC(D), Q(D))$ associated to a Postnikov diagram $D$ as a Postnikov seed and call $\CC(D)$ a Postnikov extended cluster. The set of cluster variables contains the Pl\"{u}cker coordinates $p_I$.
The frozen variables are the $p_I$, where $I$ is a cyclic interval.
\end{defn}

\begin{rem} \label{r:post1n}
If $k=1$ or $n-1$, there is a unique Postnikov diagram of  
type $(k,n)$ up to equivalence. It can be chosen to have no 
crossings at all. The case $k=1$, $n=6$ is shown in Figure~
\ref{fig:postnikovk1}.
\end{rem}

\begin{figure}
\psfragscanon
\psfrag{1s}{\pscirclebox{$1$}}
\psfrag{1e}{\psframebox{$1$}}
\psfrag{2s}{\pscirclebox{$2$}}
\psfrag{2e}{\psframebox{$2$}}
\psfrag{3s}{\pscirclebox{$3$}}
\psfrag{3e}{\psframebox{$3$}}
\psfrag{4s}{\pscirclebox{$4$}}
\psfrag{4e}{\psframebox{$4$}}
\psfrag{5s}{\pscirclebox{$5$}}
\psfrag{5e}{\psframebox{$5$}}
\psfrag{6s}{\pscirclebox{$6$}}
\psfrag{6e}{\psframebox{$6$}}
\psfrag{1}{$1$}
\psfrag{2}{$2$}
\psfrag{3}{$3$}
\psfrag{4}{$4$}
\psfrag{5}{$5$}
\psfrag{6}{$6$}
\includegraphics[width=5cm]{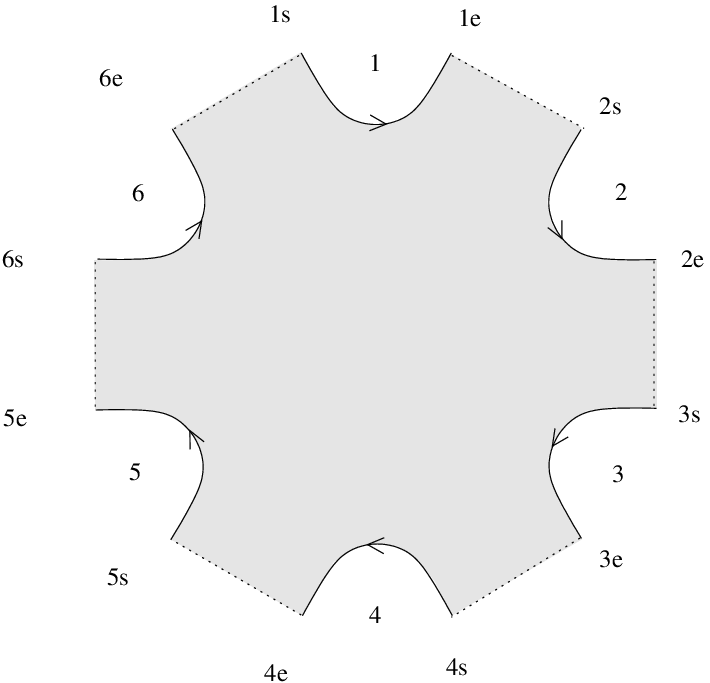}
\caption{A Postnikov diagram for $Gr(1,6)$.}
\label{fig:postnikovk1}
\end{figure}

 \section{The three versions of the holomorphic volume form}\label{s:omega}

In this section we conclude the proof of the comparison theorems~\ref{t:EHXcomparison} and \ref{t:twosuperpotentials}\label{s:omegalemmas}
which was begun in Section~\ref{s:mirrorGr}.
To do this it remains to compare the holomorphic volume forms on the domains, $\mathcal T,\Xcheck$ and  $\RR$, of the three superpotentials $L_{\qb}, W_{\qb}$ and $\mathcal F_{\qb}$. Recall that these domains are related by maps
\begin{equation}\label{e:compatiblemaps}
\mathcal T \overset\iota{\hookrightarrow} 
\Xcheck \overset{\pi_L\inv} {\longrightarrow}  B_-\cap U_+\dot w_P\dot w_0\inv U_+  \overset{\pi_R}\longrightarrow \RR,
\end{equation}
see Section~\ref{s:RichardsonBmodel}. We begin by recalling how each of the holomorphic volume forms is constructed.

\begin{defn}[the holomorphic volume form on $\mathcal T$ \cite{Giv:QToda}]\label{d:omegaT}
The domain $\mathcal T$ of the Laurent polynomial superpotential $L_q$ is naturally a torus. The holomorphic volume form $\omega_{\mathcal T}$ is the standard invariant volume form on the torus,
\begin{equation}\label{e:omegaT}
\omega_{\mathcal T}= \bigwedge_{i,j} \frac{d \dd_{ij}}{\dd_{ij}},
\end{equation}
where the wedge product is over pairs $i,j$ with $1,\dotsc, n-k$ and $j=1,\dotsc, k$. The sign is determined by the ordering of the indexing set, which we may choose to be lexicographic.
\end{defn}

\begin{defn}[the holomorphic volume form on $\Xcheck$]\label{d:omegaX}
We define $\omega_\Xcheck$ to be a choice of holomorphic volume form on $\Xcheck$ which extends to a meromorphic form on $\Xcheckbar$ with degree one poles along the $\Z/n\Z$-invariant anticanonical divisor $D$ from \eqref{e:acdivisorDef}. The normalisation will be chosen after Lemma~\ref{l:omegaXlemma}. 
\end{defn}

\begin{defn}[the holomorphic volume form on $\RR$ \cite{Rie:MSgen}]\label{d:omegaR}
The holomorphic volume form on $\RR$ is defined in outline as follows; for details we refer to {\cite[Section~7]{Rie:MSgen}}.  Choose a reduced expression $s_{i_1}\dotsc s_{i_N}$ of $w_0$, recorded by the sequence $\mathbf i=(i_1,\dotsc, i_N)$. Associated to $\mathbf i$ there is an open torus $\RR_{\mathbf i}\cong (\C^*)^{k(n-k)}$ consisting of `factorised' elements $u_1 \dot w_P B_- $ in $\RR$; see~\cite{MarRie:Parametrizations}. 
We let $z^{(\mathbf i)}_j$ denote the coordinates on $\RR_{\mathbf i}$, and consider the standard invariant holomorphic volume form \[
\omega_{\mathbf i}=\bigwedge_j \frac{dz_j^{(\mathbf i)}}{z_j}.
\] 
Note that the reduced expression $\mathbf i$ implies a natural ordering on these coordinates. 
It is shown in \cite[Proposition~7.2]{Rie:MSgen} that this form $\omega_{\mathbf i}$ extends to $\mathcal R$ and (up to sign) is independent of the  initial choice of reduced expression $\mathbf i$. 

For example in the case where $\Xcheck=Gr_3(\C^5)$ and $\mathbf i=(1,2,3,4,1,2,3,1,2,1)$ we have 
\begin{equation}\label{e:Riexample}
\RR_{\mathbf i}=
\{ x_1(m_1)x_2(m_2)\dot s_3\dot s_4 x_1(m_5)x_2(m_6)\dot s_3 x_1(m_8) x_2(m_9)\dot s_1 B_- \mid m_j\in \C^*
\},
\end{equation}
and the form $\omega_\RR$ is the extension of the form $\bigwedge_j \frac{dm_j}{m_j}$ from $\RR_{\mathbf{i}}$ to $\RR$.
\end{defn}
Using the maps in \eqref{e:compatiblemaps} we can pull back each one of the three forms to $\mathcal T$. For any two of the forms we say they are {\it {compatible}} if their pullbacks to $\mathcal T$ agree up to a nonzero scalar. 

\begin{prop}\label{p:compatibleforms} The forms $\omega_{\mathcal T},\omega_\Xcheck$ and $\omega_{\RR}$ are compatible.
\end{prop}

Once we have proved this proposition the normalisations of $\omega_{\Xcheck}$ and $\omega_\RR$ can be chosen so that the pullbacks actually agree with $\omega_{\mathcal T}$, and the proof of Theorems~\ref{t:EHXcomparison} and \ref{t:twosuperpotentials} will be complete. Our strategy for proving that the holomorphic volume forms $\omega_{\mathcal T},\omega_\Xcheck$ and $\omega_{\RR}$ are compatible is to make use as much as possible of the symmetries of $\Xcheck$, and of the rectangles cluster.  

We begin with the form $\omega_\Xcheck$. Recall the open subset $\Xcheckfact$ of $\Xcheck$, and the maps introduced in Sections~\ref{s:comparisonMRR} and \ref{s:LaurentBmodel},
\begin{equation*}\begin{array}{ccccc}
\mathcal T &\overset \sim \To &\dot w_P \dot w_0\inv U^{fact}_+ &  \overset \sim\longrightarrow &\Xcheckfact\\
(z_{ij}) &\mapsto & g=\dot w_0\inv g_{(z_{ij})} &\mapsto &Pg.  
\end{array}
\end{equation*}
Let $\omega_{\Xcheckfact}$ denote the holomorphic volume form on $\Xcheckfact$ which agrees with $\omega_{\mathcal T}$ under the above isomorphism $\mathcal T\overset\sim\To\Xcheckfact$. Proposition~\ref{p:comparisonMREHX} implies that $\omega_{\Xcheckfact}$ is given by 
\begin{equation}\label{e:omegaXfact}
\omega_{\Xcheckfact}=\bigwedge_{i,j}\frac{d\,\p_{i\x j}}{\p_{i\x j}},
\end{equation}
where the indexing set is as for $\omega_\mathcal T$ in \eqref{e:omegaT}. 

\begin{lem} \label{l:omegaXlemma}
The holomorphic volume form $\omega_{\Xcheckfact}$ extends to a meromorphic form $\omega_{\Xcheckbar}$ on $\Xcheckbar$ which is regular on $\Xcheck$.
Moreover $\omega_{\Xcheckbar}$  has poles of order one along the divisor $D$. 
\end{lem}

\begin{proof}
The form $\omega_{\Xcheckfact}$ extends to a meromorphic form on $\Xcheckbar$ as a consequence of \eqref{e:omegaXfact}, and we need to analyse its poles. We denote this meromorphic form by $\omega_{\Xcheckbar}$.
Note that the rectangular Pl\"{u}cker coordinates form an extended cluster in $\mathbb{C}[\Xcheck]$
by~\cite[Lemma 8.1]{MarSco}, since, for $1\leq i\leq n-k$,
$1\leq j\leq k$,
the $k$-subset $J_{i\times j}$ associated to $p_{i\times j}$ is
given by
\begin{equation}\label{e:Jixj}
J_{i\times j}=[1,k-j]\cup [i+k-j+1,i+k],
\end{equation}
and hence coincides with the $k$-subset $M_{k,n}(k-j,i)$ defined in~\cite[Lemma 8.3]{MarSco}.

It follows that $\omega_{\Xcheckbar}$ is regular and nonvanishing on $\Xcheck$ by an argument entirely analogous to the one used in the construction of $\omega_{\RR}$ in \cite[Section~7]{Rie:MSgen}, which is now standard in the context of cluster algebras; see for example Section~13 of \cite{Lam:TNGlectures} and references therein. Namely, by applying mutations to $\omega_{\Xcheckbar}$ one can show that restriction of $\omega_{\Xcheckbar}$ to  any cluster torus is given by  the same formula, up to sign, in terms of the corresponding cluster variables. So $\omega_{\Xcheckbar}$ is regular and nonvanishing on the union of the cluster tori. To extend  one uses that the union of cluster tori has complement of codimension $\ge 2$ inside $\Xcheck$. This shows that $\omega_{\Xcheckbar}$ is regular and nonvanishing on $\Xcheck$.

We now use this compatibility of $\omega_{\Xcheckbar}$ with the cluster structure to show that $\pm\omega_{\Xcheckbar}$ is invariant under the action of $\Z/n\Z$ on $\Xcheckbar$
defined by the cyclic
shift (see Section~\ref{s:GrassmannianBmodel}).
Note that if $1$ is added to the label of each strand in a Postnikov diagram $D$, we obtain another Postnikov diagram $D'$
with the property that the corresponding extended cluster
$\CC(D')$ is the pull-back of $\CC(D)$ under the cyclic shift.
Since $\omega_{\Xcheckbar}$ has the same form (up to sign) in terms of the shifted cluster $\CC(D')$ as it does in terms of $\CC(D)$, it follows that
the $\Z/n\Z$ action on $\Xcheck$ preserves $\omega_{\Xcheckbar}$ up to sign.



We can now determine the poles of $\omega_{\Xcheckbar}$. Since $\omega_{\Xcheckbar}$ is regular and nonvanishing on $\Xcheck$, it must have poles contained in $D$. We pick an irreducible component $D_0$ of $D$ and let $m_0$ be the order of the pole of $\omega_{\Xcheckbar}$ along $D_0$. The other irreducible components of $D$ are obtained from $D_0$ by the $\Z/n\Z$-action. Since $\omega_{\Xcheckbar}$ is $\Z/n \Z$-invariant up to sign, it must have the same order $m_0$ of pole along each component of $D$. Finally, since the index of $\Xcheckbar$ is $n$ it follows that $m_0=1$. Therefore $\omega_{\Xcheckbar}$ has poles of order one along each component of $D$. 
\end{proof}
 
We may set $\omega_{\Xcheck}$ to be the restriction of $\omega_{\Xcheckbar}$ to $\Xcheck$, thus fixing the normalisation which was missing from Definition~\ref{d:omegaX}. The above lemma implies that $\omega_\Xcheck$ is compatible with $\omega_{\mathcal T}$, thus proving the first part of Proposition~\ref{p:compatibleforms}.   
 
Next we turn out attention to $\omega_\RR$.   In order to be able to make the comparison of forms we introduce another symmetry of $\Xcheck$, which is in a sense an obstruction to our isomorphism from $\RR$ to $\Xcheck$  extending to a map from the closure of $\RR$ to $\Xcheckbar$.
Namely, the main ingredient to proving that $\omega_\Xcheck$ and $\omega_\RR$ are compatible will be the involution described in the following proposition. 

For $b\in B_-\dot w_0\cap U_+ \widetilde T^{W_P}\dot w_P U_-$
and $\lambda\in \mathcal P_{k,n}$, let us denote by $P_{\lambda}(b)$
the minor of $B$ with row set $[n-k+1,n]$ and column set $J_{\lambda}$.
We also use the shorthand $P_m(b)=P_{\mu_m}(b)$.

\begin{prop}\label{p:involution}
Transposition of matrices restricts to define an involution $(\quad)^t$ on the subvariety $B_-\dot w_0\cap U_+ \widetilde T^{W_P}\dot w_P U_-$ of $GL_n(\C)$. We conjugate this involution by right multiplication by ${\dot w_0}$ to obtain an involution 
\begin{equation}\label{e:geninvolution}
\begin{array} {lccc}
\Invol:& B_-\cap U_+ \widetilde T^{W_P}\dot w_P \dot w_0\inv U_+ &\To & B_-\cap U_+ \widetilde T^{W_P}\dot w_P\dot w_0\inv U_+,\\
&b&\mapsto & (b \dot w_0)^t \dot w_0\inv.
\end{array}
\end{equation}
The map $\Invol$ satisfies (and is uniquely determined by) the following equality of minors,
\begin{equation}\label{e:deltaidentity}
{P_{i\x j}(\Invol(b))} = (-1)^{ij}\frac {P_{(n-k-i)\x (k-j)}(b)} {P_{n-k+i-j}(b)} (\qq(b))^{\min(i,j)},
\end{equation}
where $1\le i\le n-k$ and $1\le j\le k$. Here 
\[
\qq(b)
:=\alpha_{n-k}(t),
\]
where 
\begin{equation}\label{e:factorsofb}
b=u_1t\dot w_P \dot w_0\inv u_2\quad \text{ for }\quad u_1, u_2\in U_+ \quad\text{and}\quad t\in \widetilde T^{W_P}.
\end{equation} 
\end{prop}
\begin{rem}\label{r:involutionq}
Note that in the setting of Proposition~\ref{p:involution}, transposition of $b=z\dot w_0\inv$ affects the torus factor $t$ in \eqref{e:factorsofb} by $t\mapsto \dot w_P^{-2} t$. This changes the sign of $\alpha_{n-k}$ precisely if $k$ and $n-k$ have different parity.  Hence  $\qq(\Invol(b))=(-1)^n\qq(b)$ 
\end{rem}
\begin{rem}\label{r:involutiotildetau} 
The proposition can be interpreted as saying that there exists an involution 
\begin{equation}\label{e:tautildeinvolution}
\widetilde\tau:\Xcheck\x\C^*_q\to
\Xcheck\x\C^*_q
\end{equation}
which (using rectangles cluster coordinates) satisfies
\begin{equation}\label{e:involutionwithq}
(p_{i\x j},q)\mapsto \left(\frac{p_{(n-k-i)\x (k-j)}}{p_{J_{n-k+j-i}}}q^{\min(i,j)},q\right).
\end{equation}
Indeed via the identification of $\Xcheck\x\C^*_q$ with $B_-\cap U_+\widetilde T^{W_P}\dot  w_P\dot w_0\inv U_+$, the involution $\Invol$ becomes defined on $\Xcheck\x\C^*_q$.
We now note that the involution $\Invol$ commutes with the $\C_u^*$-action on $\Xcheck\x \C^*_q$ given in coordinates by: 
\begin{equation}\label{e:gradingastorusaction}
u\cdot ((p_{\lambda})_{\lambda\in\Pn_{k,n}},q)=((u^{|\lambda|}p_\lambda)_{\lambda\in\Pn_{k,n}},u^n q).
\end{equation}
This is straightforward to check using the formula \eqref{e:involutionwithq}. Let $\sign:\Xcheck\x \C^*_q\to \Xcheck\x \C^*_q$ denote the map given by the action of $(-1)\in \C^*_u$. Then 
\[
\widetilde\tau:=\sign\circ\Invol=\Invol\circ\sign.
\]
The involution $\widetilde\tau$ preserves the superpotential $W:\Xcheck\x\C^*_q\to \C$. For example  this follows from a direct calculation using the formula for the superpotential in the rectangles cluster. Alternatively the invariance of $W$ relates to the symmetry in the formula for~$\mathcal F$ with regard to $u_1$ and $u_2$, see Definition~\ref{d:RichardsonBmodel},  via the comparison of superpotentials result (Proposition~\ref{p:comparisonMRR}).

Finally, we observe that $\widetilde\tau$ fixes the critical points of $W_q$. This is because the critical points are represented by Toeplitz matrices $b$ in $B_-\cap U_+\widetilde T^{W_P}\dot w_P\dot w_0\inv U_+$; see \cite[Equation~(5.14)]{Rie:MSgen}, with $h=0$ for the non-equivariant case. (Recall that a Toeplitz matrix is a matrix for which the entries along each diagonal are the same, and for $b\in B_-$ this is equivalent to the condition of stabilizing the standard principal nilpotent $F$ from Theorem~\ref{t:Peterson}). It follows that the matrices
$z=b\dot w_0$ are almost Hankel matrices (i.e.\ the entries along each anti-diagonal are the same up to sign, with the signs alternating along the anti-diagonal).
Transposition of $z$ swaps all of the signs on even length anti-diagonals. From this we can observe that 
\[
\sign(Pz^t\dot w_0\inv, q)=(Pz\dot w_0\inv, q),\quad \text{ if $Pz\dot w_0\inv$ is a critical point of $W_q$.}
\]
In other words, $\widetilde \tau$ fixes the critical points $(Pz\dot w_0\inv, q)$ of $W_q$. 
\end{rem}

\begin{rem}\label{r:involutiontau}
If we set $q=1$ the restricted domain of \eqref{e:tautildeinvolution} can be identified with $\Xcheck$. In this case the proposition can be interpreted as saying that there is an involution  $\tau$ on $\Xcheck$ which extends the involution 
\[
p_{J_j}\mapsto \frac{1}{p_{J_{n-j}}}
\]
on frozen variables. Moreover this involution preserves the rectangles cluster torus and is given there by
\begin{equation}\label{e:translatedinvolution}
p_{i\x j}\mapsto \frac{p_{(n-k-i)\x (k-j)}}{p_{J_{n-k+j-i}}}.
\end{equation}
\end{rem}

\begin{rem}\label{r:involutionqcoh}
The involution $\tilde \tau$ from Remark~\ref{r:involutiotildetau} has an interpretation which involves the quantum cohomology ring $qH^*(X,\C)[\qa\inv]$. Namely by the combination of  Theorem~\ref{t:MSgen} with Theorem~\ref{t:Peterson}, the quantum cohomology ring is the ring of functions on the critical point locus $\{\partial_{\Xcheck}W=0\}$ inside $\Xcheck\x\C^*_q$. 
Since the involution acts trivially on the critical point locus, it also acts trivially on the quantum cohomology ring $qH^*(X,\C)[\qa\inv]$.
In Section~\ref{s:Schubert} we will show that the Pl\"ucker coordinates $p_\lambda$ restricted to the critical locus of $W_q$ are identified with the quantum Schubert classes $\sigma^\lambda$ via Peterson's isomorphism (see Proposition~\ref{p:eqJacobi}). 
Therefore in the Grassmannian case, since we have the formula \eqref{e:involutionwithq}, this means that we obtain the relations
\begin{equation}\label{e:freeqcohrelations}
\sigma^{i\x j}\star\sigma^{\mu_{n-k+j-i}} =q^{\min(i,j)}\sigma^{(n-k-i)\x (k-j)}
\end{equation}
in the quantum cohomology ring $qH^*(X,\C)$, as a result of the symmetry $\widetilde\tau$ of the mirror. These relations also follow from repeated application of the quantum Pieri rule of \cite[p. 293]{Bertram:qSchubCalc}.

We note that the involution~\eqref{e:geninvolution} makes sense for arbitrary reductive algebraic groups. By analogous arguments to above this construction gives rise to an involution on the Lie-theoretic mirror  $\RR_P\x T^{W_P}$ of a general $G^\vee/P^\vee$ such that the superpotential $\mathcal F$ is invariant and the critical points are fixed. In particular this involution should again induce relations in the quantum cohomology ring. 
\end{rem}

We will outline the proof of Proposition~\ref{p:involution} after proving the following corollary, which is our first application of the proposition. 

\begin{cor}\label{c:compatibility2}
The forms $\omega_{\Xcheck}$ and $\omega_\RR$ are compatible. 
\end{cor}

\begin{proof}
To prove the statement of the corollary we need to  compare $(\Xcheck,\omega_{\Xcheck})$ and $(\RR, \omega_{\RR})$ under the isomorphisms
\begin{equation}\label{e:piLpiR}
\Xcheck \overset{\pi_L}\longleftarrow 
 B_-\cap U_+\dot w_P\dot w_0\inv U_+ 
\overset{\pi_R}\To 
\RR.
\end{equation}
Let $\Invol_0$ be the restriction of the involution $\Invol$ to $B_-\cap U_+\dot w_P\dot w_0\inv U_+$,
\[\Invol_{\, 0} :B_-\cap U_+\dot w_P\dot w_0\inv U_+ \to B_-\cap U_+\dot w_P\inv\dot w_0\inv U_+
\] 
given by  $b\mapsto \dot w_0\inv b^t \dot w_0\inv$. Also let
\begin{eqnarray*} \tau_0 :\Xcheck &\to &\Xcheck\\
Pb &\mapsto & P\Invol_{\, 0}(b),
\end{eqnarray*} 
for $b\in B_-\cap U_+\dot w_P\dot w_0\inv U_+ $.
We consider the commutative diagram
\begin{equation} \label{e:piLpiRenhanced}
\xymatrix{
\Xcheck \ar[d]_{\tau_0} & \ar_(0.7){\pi_L}[l] B_-\cap U_+\dot w_P\dot w_0\inv U_+ \ar[dl]^{\widetilde{\pi}_{L}} 
\ar[r]^(0.75){\pi_R} &  \RR. \\
\Xcheck & &
}
\end{equation}

Recall that we have a torus $\Xcheckfact$ inside $\Xcheck$, see Section~\ref{s:comparisonMRR}, and on the right hand a torus $\RR^{fact}:=\RR_{\mathbf i}$ inside $\RR$, where $\mathbf i=(1,2,\dotsc, n;1,2,\dotsc, n-1;\dotsc ; 1,2; 1)$ (compare Definition~\ref{d:omegaR}). 

The holomorphic volume form $\omega_{\Xcheck}$ is characterised up to a scalar by the fact that its restriction to $\Xcheckfact$ is a torus invariant volume form, and similarly for $\omega_{\RR}$ and $\RR^{fact}$.  Therefore it suffices to prove that the pull-backs of the tori agree, namely that
\[\pi_L\inv\left(\Xcheckfact\right)=\pi_R\inv\left(\RR^{fact}\right).
\]
We have that the preimage $\pi_{R}\inv\left(\RR^{fact}\right)$ is described, carrying on with the example from Definition~\ref{d:omegaR},  by
\[
\pi_{R}\inv\left(\RR^{fact}\right)=\{
b\in B_-\mid b =x_1(m_1)x_2(m_2)\dot s_3\dot s_4 x_1(m_5)x_2(m_6)\dot s_3 x_1(m_8) x_2(m_9)\dot s_1 \dot w_0\inv u_2,\  u_2\in U_+, m_i\in \C^* 
\}.
\]
We compare this with the preimage $\widetilde\pi_L\inv\left(\Xcheckfact\right)$. Namely, in the same example, $\widetilde\pi_L(b)=P\dot w_0\inv b^t\dot w_0\inv$ lies in \[
\Xcheckfact=\{P\dot w_0\inv \dot s_4 x_3(n_1) x_4(n_2)\dot s_2x_3(n_4)x_4(n_5)\dot s_1\dot s_2 x_3(n_8)x_4(n_9) \mid n_i\in \C^*\}
\]
if and only if $\dot w_0\inv b^t\in  B_- \dot w_0\cap U_+\dot w_P\inv U_-$ is of the form
\begin{equation*}
\dot w_0\inv b^t =b_+\dot w_0\inv \dot s_4\inv x_3(n_1) x_4(n_2)\dot s_2\inv x_3(n_4)x_4(n_5)\dot s_1\inv\dot s_2\inv x_3(n_8)x_4(n_9)\dot w_0
\end{equation*}
for some $b_+\in B_+$, or equivalently if $b$ is of the form 
\begin{equation*}
b =
x_1(n_9)
x_2(n_8)
\dot s_3
\dot s_4
x_1(n_5)
x_2(n_4)
\dot s_3
 x_1(n_2)
x_2(n_1)
\dot s_1
\dot w_0\inv
b_+  .
\end{equation*}
Therefore $\widetilde\pi_L\inv\left(\Xcheckfact\right)=\pi_R\inv\left(\RR^{fact}\right)$. This argument clearly works in general, although we only wrote it out in an example.

Now note that the torus $\Xcheckfact$ is also characterised by the condition that $\p_{i\x j}\ne 0$ for all of the Pl\"ucker coordinates in the rectangles cluster, thanks to Proposition~\ref{p:comparisonMREHX}.  From Proposition~\ref{p:involution} it follows that $\tau_0$ is an isomorphism which takes $\Xcheckfact$ to $\Xcheckfact$. Therefore 
\[
\pi_L\inv\left(\Xcheckfact\right)=\widetilde\pi_L\inv\left(\tau_0(\Xcheckfact)\right)=
\widetilde\pi_L\inv\left(\Xcheckfact\right).
\]
and $\widetilde\pi_L\inv\left(\Xcheckfact\right)=\pi_R\inv\left(\RR^{fact}\right)$ as we saw above. 
This concludes the proof that $\pi_L\inv\left(\Xcheckfact\right)=\pi_R\inv\left(\RR^{fact}\right)$, and hence the proof of the Corollary.
\end{proof}

\begin{proof}[Outline of the proof of Proposition~\ref{p:involution}]
Our initial proof of Proposition~\ref{p:involution} involved a sequence of equalities of minors. However we describe our second proof which involves an extension of the construction from Proposition~\ref{p:comparisonMREHX}. 

The idea for this proof is to factorise a generic element of $B_- \cap U_+\widetilde T^{W_P}\dot w_P\dot w_0\inv U_-$ in a way that makes it easy to take the transpose. Recall that we have an embedding $\iota:\mathcal T\to\Xcheck$ defined in Section~\ref{s:EHX} which was constructed via multiplying together factors from simple root subgroups. One can `extend' $\iota$ to an embedding 
\[
\zeta:\mathcal T\x \C^*_q\ \hookrightarrow \ B_-\cap U_+\widetilde T^{W_P}\dot w_P \dot w_0\inv U_+,
\]
for which $b=\zeta(z_{ij},q)$ satisfies $\mathbf q(b)=q$ and $Pb\in \Xcheck$ is $\iota(\dd_{ij})$. 
The map $\zeta$ has the following explicit description, which relates to the construction of $\iota$ from Section~\ref{s:EHX}. As in the  proof of Proposition~\ref{p:comparisonMREHX}, we demonstrate our construction of the map $\zeta$ in the special case of $X=Gr_2(5)$ and $\Xcheckbar=Gr_3(5)$, that is, $k=3$ and $n=5$. 

We now consider both the quiver $\QR$ associated to $X=Gr_2(5)$ in Section~\ref{s:EHX} and its reflection through the $xy$-axis which we denote by $\QL$. We decorate the two quivers as shown below. 
Note that we have purposefully left out some of the $s_i$'s. In both cases we have also dropped the bottom right hand corner arrow. 
\begin{center}
 \begin{tikzpicture}
\node at (-2, .4) {Row $E_4$ };  
\node at (-2, 1.4) {Row $E_3$ };  
\node at (-2, 2.4) {Row $E_2$ };  
\node at (-2, 3.4) {Row $E_1$ };  
\draw (0,3.5) circle (3mm);
\draw (0,3) [fill] circle (.4mm);
\node at (0,3.5) {${\dot s_1}$};
\draw[->](1,0)--(0,0);
\draw[->](0,1)--(0,2);
\draw[->](0,2)--(0,3);
\draw[->](1,0)--(1,1);
\draw[->](0,0)--(0,1);
\draw[->](1,1)--(0,1);
\draw[->](1,2)--(0,2);
\draw[->](1,1)--(1,2);
\node [left] at (0,2) {$z_{23}$};
\node [below] at (1,0) {$z_{11}$};
\node [below] at (0,0) {$z_{21}$};
\node [right] at (1,1) {$z_{12}$};
\node [above] at (1,2) {$z_{13}$};
\node [left] at (0,1) {$z_{22}$};
\node [left] at (0,3) {$q$};
\end{tikzpicture}
\qquad \qquad
 \begin{tikzpicture}
\node at (-2, .4) {Row $E_4$ };
\node at (-2, 1.4) {Row $E_3$ };
\node at (-2, 2.4) {Row $E_2$ } ;
\node at (-2, 3.4) {Row $E_1$ };
\draw (0,2.5) circle (3mm);
\draw (1,2.5) circle (3mm);
\draw (0,3.5) circle (3mm);
\draw (0,2) [fill] circle (.4mm);
\node at (0,2.5) {${\dot s_2}$};
\node at (1,2.5) {${\dot s_2}$};
\node at (0,3.5) {${\dot s_1}$};
\draw[->](0,0)--(1,0);
\draw[->](1,0)--(2,0);
\draw[->](0,2)--(0,1);
\draw[->](0,1)--(0,0);
\draw[->](1,1)--(1,0);
\draw[->](0,1)--(1,1);
\draw[->](1,1)--(2,1);
\draw[->](2,1)--(2,0);
\node [below] at (2,0) {$\dd_{23}$};
\node [below] at (1,0) {$\dd_{22}$};
\node [below] at (0,0) {$\dd_{21}$};
\node [above] at (1,1) {$\dd_{12}$};
\node [above] at (2,1) {$\dd_{13}$};
\node [left] at (0,1) {$\dd_{11}$};
\node [left] at (0,2) {$1$};
\end{tikzpicture}
\end{center}
We remark that the analogue of the quiver from the proof of Proposition~\ref{p:comparisonMREHX} is the one on the right hand side, above. The left hand side quiver is related to the construction of an open part of the Peterson variety from  \cite[Theorem~7.2]{Rie:TotPosGBCKS}.
 
To the quiver on the right hand side we simply associate an element $g_R(z_{ij})\in \dot s_2\dot s_1\dot s_2 U_+$ using the construction from the proof of Proposition~\ref{p:comparisonMREHX}. Namely in this example the element is
\begin{equation}\label{e:gR}
g_R(z_{ij})= 
   x_3\left({z_{13}}\right)
   x_4\left(\frac {z_{23}}{z_{13}}\right)
 \dot s_2  
x_3\left(z_{12}\right) 
x_4\left(\frac {z_{22}}{z_{12}}\right)
\dot s_1\dot s_2
 x_3\left(z_{11}\right)
x_4\left(\frac {z_{21}}{z_{11}}\right).
\end{equation}
To the reflected quiver on the left hand side we associate an element $g_L(z_{ij},q)\in  U_+\dot s_1$ using the construction from~\cite{Rie:TotPosGBCKS} which goes as follows. This time we start from the left-most column at the bottom working upwards and then carry on column by column to the right. Namely to every vertex labelled by a $z_{ij}$ we associate the minimal $1$-path starting at that vertex. Then we list these $1$-paths one for each $z_{ij}$ vertex in our column, starting from the bottom and working upwards. These are followed by any $\dot s_i$'s at the top of the column. Then we repeat for the next column to the right, until the list is complete. 
As before to any $1$-path which crosses row $E_i$  we associate the element $x_i(z_{t}/z_s)\in U_+$, where $z_s$ is the coordinate at the beginning vertex (`source') of the $1$-path, and $z_t$ the vertex at the end (`target'). 
We again obtain a matrix given as the product of the listed factors. 
 In the running example we have,  
\begin{equation}\label{e:gL}
g_L(z_{ij},q)=  
x_4\left(\frac {z_{22}}{z_{21}}\right)
x_3\left(\frac {z_{23}}{z_{22}}\right)
x_2\left(\frac {q}{z_{23}}\right)
 \dot s_1  
x_4\left(\frac {z_{12}}{z_{11}}\right)
x_3\left(\frac {z_{13}}{z_{12}}\right)
x_2\left(\frac {q}{z_{13}}\right).
\end{equation}

 Moreover we set $t_q=\operatorname{diag}(q,q,1,1,1)$ to be the diagonal matrix in $\widetilde T^{W_P}$ with $\alpha_{2}(t_q)=q$. The map $\zeta$ is defined by  multiplying the matrices as follows, 
 \[
 \zeta(\dd_{ij},q)=g_L(z_{ij},q)\, t_q \dot w_0\inv \ g_R(z_{ij}).
 \]
 From the factorisation of $g_R$ and $g_L$ we see that $\zeta(\dd_{ij},q) $ is of the form $u_1 t_q \dot s_1 \dot s_3\dot s_4\dot s_3\dot w_0\inv u_2\in U_+\widetilde T^{W_P}\dot w_P\dot w_0\inv U_+$. One then has to check that the $ab$-entries of the $n\x n$-matrix $\zeta(\dd_{ji},q)$ vanish whenever $a<b$. This can be shown by careful study of the
representation of the product $\zeta(\dd_{ij},q)$ in terms of a concatenation
of `chips' (corresponding to terms $x_i$; see~\cite[Figure 4]{FoZe:TestsParametrizations}) and wiring diagrams (corresponding to products of terms $\dot s_i$). We leave out the details. It follows that $\zeta(\dd_{ij},q)\in B_-\cap U_+\widetilde T^{W_P}\dot w_P\dot w_0\inv U_+$, as was intended. 
 
We can now work out the minors we are interested in, in the case where $b=\zeta(\dd_{ij},q)$.
As in Proposition~\ref{p:comparisonMREHX} we have
\begin{equation}\label{e:minorgR}
\frac{P_{i\x j}(\zeta(z_{ij},q))}{P_{(i-1)\x(j-1)}(\zeta(z_{ij},q))}=z_{ij}
\end{equation}
and similarly we obtain
\begin{equation}\label{e:minorgL}
\frac{P_{i\x j}(\Invol(\zeta(z_{ij},q)))}{P_{(i-1)\x(j-1)}(\Invol(\zeta(z_{ij},q)))}=\frac{(-1)^{i+j-1}q}{z_{n-k-i+1,k-j+1}}.
\end{equation}
Note that the $P_{i\x j}$ are basically the rectangle Pl\"ucker coordinates $p_{i\x j}$ appearing in Proposition~\ref{p:comparisonMREHX}.
In particular, note that $P_{i\times 0}=P_{0\times j}=P_{\emptyset}$, and
it is easy to check that $P_{\emptyset}(\zeta(\dd_{ij},q))=
P_{\emptyset}(\Invol(\zeta(\dd_{ij},q))=1$.

We now finish the proof of the proposition. Clearly, $\zeta$ has open dense image in $B_-\dot w_0\cap U_+\widetilde T^{W_P}\dot w_P U_-$ (as follows for example from the above formulas).  It therefore suffices to prove the identity \eqref{e:deltaidentity} in the case $b=\zeta(\dd_{ij},q)$.

From \eqref{e:minorgR} and \eqref{e:minorgL} it follows that
\[
\frac{P_{i\x j}(\Invol(b))}{P_{(i-1)\x(j-1)}(\Invol(b))}=(-1)^{i+j-1}\frac{q P_{(n-k-i)\x(k-j)}(b)}{P_{(n-k-i+1)\x (k-j+1)}(b)}.
\]
A telescopic product identity then implies the desired formula,
\[
P_{i\x j}(\Invol(b))=(-1)^{ij}q^{\min(i,j)}\frac{P_{(n-k-i)\x(k-j)}(b)}{P_{n-k+j-i}(b)}.
\]
We note that $P_{n-k+j-i}(b)$ is the minor of $b$ associated to the rectangle $(n-k-i+\min{(i,j)})\x (k-j+\min{(i,j)})$, which is indeed a maximal rectangle.
\end{proof}

\section{Schubert classes and proof of the free basis lemma}\label{s:consequences}
The first aim of this section is to show that the isomorphism between $\C[\Xcheck\x \C^*_q]/(\partial_{\Xcheck} W_q)$ and $qH^*(X)[q\inv]$ identifies the classes of the Pl\"ucker coordinate $\p_\lambda$ with the Schubert class $\sigma^\lambda$. After that we will prove Lemma~\ref{l:freebasis}.

\subsection{The Schubert basis}\label{s:Schubert} In this Section we prove the non-equivariant version of Proposition~\ref{p:eqJacobi}. 
As in Section~\ref{s:comparisonMRR}, let $\Delta^I_J(g)$ denote the minor with row set $I$ and column set $J$, and recall that to each $\lambda\in \Pn_{k,n}$ we have associated a $k$-subset $J_\lambda$ of $[1,n]$, see Section~\ref{s:GrassmannianBmodel}. Recall that by Peterson's theorem, Theorem~\ref{t:Peterson}, there is an explicit isomorphism between $qH^*(X,\C)[q\inv]$ and the Peterson variety $Y_P^*$ inside $G/B_-$.  We now state in the context of Grassmannians a result from Dale Peterson's theory which makes this isomorphism precise. It can be found with proof as Proposition~11.3 in~\cite{Rie:QCohPFl}.

\begin{prop}[Dale Peterson~\cite{Pet:QCoh}]\label{p:flambda}
Let $\lambda\in\Pn_{k,n}$. If $J_\lambda=\{\lambda_1,\dotsc,\lambda_k\}$ set $J_\lambda^{\opp}=\{n+1-\lambda_1,n+1-\lambda_2,\dotsc, n+1-\lambda_k\}$. Note that $J_\emptyset=[1,k]$ and $J_\emptyset^{\opp}=[n+1-k,n]$. We write $\Delta^{J}$  for $\Delta_{J_\emptyset^{\opp}}^{J}$. Define a rational function on $G/B_-$ by
\[
f_\lambda(gB_-)=\frac{\Delta^{J_\lambda^{\opp}}(g)}{\Delta^{J_\emptyset^{\opp}}(g)}.
\]
The restriction of $f_\lambda$ to the Peterson variety $Y_P^*$ is identified with the Schubert class $\sigma^\lambda$ under the isomorphism of Theorem~\ref{t:Peterson}. 
\end{prop}

Next we recall that by Theorem~\ref{t:MSgen}, the Peterson variety $Y_P^*$ is isomorphic  to the subvariety of $\RR\x \C^*_q$ defined by $\partial_\RR \mathcal F=0$, for the Lie theoretic superpotential $\mathcal F:\RR\x \C^*_q\to \C$ from Definition~\ref{d:RichardsonBmodel}. Moreover by the comparison of superpotentials result, Theorem~\ref{t:twosuperpotentials}, this variety must be isomorphic to the subvariety of $\Xcheck\x \C^*_{q}$ defined by $\partial_{\Xcheck}W=0$. 

We now apply Proposition~\ref{p:flambda} and isomorphisms above to prove the following proposition. 

\begin{prop}\label{p:SchubertPlucker}
The Jacobi ring of $(\Xcheck, W_q)$ is isomorphic to the quantum cohomology ring $qH^*(X,\C)[\qa\inv]$ via an isomorphism which sends the coordinate $q$ to the quantum parameter $q$ and takes the form
\[
[\p_\lambda]\mapsto \sigma^\lambda
\]
on Pl\"ucker coordinates. This isomorphism agrees with the one obtained by combining the isomorphism between $(\Xcheck, W_q)$ and $(\RR, \mathcal F_q)$, the isomorphism between the critical locus $\{\partial_\RR \mathcal F=0\}$ and the Peterson variety $Y_P^*$, and Peterson's isomorphism between the coordinate ring of the Peterson variety and the quantum cohomology ring $qH^*(X,\C)[q\inv]$.
\end{prop}

\begin{proof}
For $\lambda\in\mathcal P_{k,n}$ and $b=z\dot w_0\inv\in B_-\cap U_+ \widetilde T^{W_P}\dot w_P\dot w_0\inv U_+$ recall that $P_\lambda(b)$ denotes the minor of $b$ involving the last $k$ rows and with column set given by $J_\lambda$. These
minors are related to the Pl\"ucker coordinates on $\Xcheck$ as follows:
\[
\frac{P_\lambda(b)}{P_\emptyset (b)}=p_\lambda(Pb),
\]
keeping in mind our normalisation $p_\emptyset(Pb)=1$.

The isomorphism between $\Xcheck$ and $\RR$ sends $Pb$ to $b\dot w_0 B_-$. We want to relate 
$p_\lambda(Pb)$ to $f_\lambda(zB_-)$,
whenever $b\dot w_0 B_-=zB_-$ is in the Peterson variety $Y_P^*$. 

Recall the involution $b\mapsto \Invol(b)$ from Proposition~\ref{p:involution} and the related involution 
\[\widetilde\tau: (Pz\dot w_0\inv,q)\to  \sign(Pz^t\dot w_0\inv,(-1)^n q)
\] 
from Remark~\ref{r:involutiontau}, where $z=b\dot w_0$ and $q=\qq(b)$. We have that 
\begin{equation*}
p_\lambda\circ \widetilde\tau\ (Pb,q)=(-1)^{|\lambda|}\frac{P_\lambda(\Invol(b))}{P_\emptyset(\Invol(b))}=
(-1)^{|\lambda|}\frac{P_\lambda(z^t\dot w_0\inv)}{P_\emptyset(z^t\dot w_0\inv)}=
\frac{\Delta^{J_\emptyset^{\opp}}_{J_\lambda^{\opp}}(z^t)}{\Delta^{J_\emptyset^{\opp}}_{J_\emptyset^{\opp}}(z^t)} 
=\frac{\Delta^{J_\lambda^{\opp}}_{J_{\emptyset}^{\opp}}(z)}{\Delta^{J_{\emptyset}^{\opp}}_{J_\emptyset^{\opp}}(z)}
=f_\lambda(zB_-).
\end{equation*}
Assume now that $Pb$ is a critical point of $W_q$. In this case $\widetilde\tau(Pb,q)=(Pb,q)$, as follows from Remark~\ref{r:involutiotildetau}. Therefore 
\[
p_\lambda(Pb)={f_\lambda(zB_-)} \qquad\text{if $Pb$ is a critical point of $W_q$ for $q=\qq(b)$.} 
 \]  
 Recall that if $Pb$ is a critical point of $W_q$ then $b\dot w_0B_-$ is in the Peterson variety $Y_P^*$. Therefore by Proposition~\ref{p:flambda}, Peterson's isomorphism maps the function $f_\lambda$ on the right hand side to the quantum Schubert class $\sigma^\lambda$. All in all we have shown that $\p_{\lambda}$ modulo $(\partial_{\Xcheck}W)$ maps to $f_\lambda$ which maps to $\sigma^\lambda$ under the relevant isomorphisms. Note that all the isomorphisms also preserve $q$, therefore the proof is complete. 
\end{proof}

\subsection{The free basis lemma}\label{s:freebasis} Recall Definitions~\ref{d:GMconnection} and \ref{d:Gbar} of the Gauss-Manin system $G^{W_\qb}$ and its submodule $\Gbar$. Now that we are finished comparing the holomorphic volume forms we simply write $\omega$ for $\omega_\Xcheck$. In this section we prove the following result and, as a consequence, we show Lemma~\ref{l:freebasis}.
\begin{lem}\label{l:freebasis2}

$\GbarO$ is a free $\C[\hbarrB,\qb]$-module with basis $\{[\p_\lambda\omega], \lambda\in\Pn_{k,n}\}$ and
\[
\Gbar=\GbarO\otimes_{\C[\hbarrB,\qb]}{\C[\hbarrB^{\pm 1},\qb^{\pm 1}]}.
\]
In particular $\Gbar$ is a free $\C[\hbarrB^{\pm 1},\qb^{\pm 1}]$-module with basis $\{[\p_\lambda\omega], \lambda\in\Pn_{k,n}\}$.
\end{lem}
\begin{proof}
It suffices to show that the elements $[\p_\lambda\omega]$ in $G^{W_{\qb}}$ are linearly independent over $\C[\hbarr, q^{\pm 1}]$. To prove this we show first the following claim. 

\vskip .1cm
\noindent{{\it Claim:} } Suppose we have a relation $ \sum_\lambda c_\lambda(\hbarr,\qb,\qb\inv)\, [\p_\lambda\omega]=0 $
 in $G^{W_\qb}$. Then for any point $x\in\Xcheck$ there is a form $\nu\in \Omega^{N-1}({\mathcal U})[q,q\inv]$ defined locally around $x$ such that 
 \[
 \sum_\lambda c_\lambda(0,\qb,\qb\inv)\, \p_\lambda\omega= dW_{\qb}\wedge\nu.
 \]
 \vskip .1cm

\noindent{{\it Proof of the Claim:} }
The assumption says that in $\Omega^N(\Xcheck)[\hbarr,q^{\pm 1}]$ we have
\begin{equation}\label{e:plamdarelation}
 \sum_\lambda c_\lambda(\hbarr,\qb,\qb\inv)\, \p_\lambda\omega=d\eta + \frac 1 \hbarr dW_q\wedge \eta
\end{equation}
where $\eta\in \Omega^{N-1}(\Xcheck)[\hbarr^{\pm 1},q^{\pm 1}]$. Let us write 
\[
\eta=\eta^{(<-1)} + \frac 1\hbarr\eta^{[-1]}+\eta^{[0]} + \hbarr \eta^{[1]} + \eta^{(>1)}
\]
where $\eta^{(<-1)}\in z^{-2} \Omega^{N-1}(\Xcheck)[\hbarr\inv,q^{\pm 1}]$ and $\eta^{(>1)}\in \hbarr^2\Omega^{N-1}(\Xcheck)[\hbarr,q^{\pm 1}]$, while $\eta^{[-1]},\eta^{[0]}, \eta^{[1]}\in \Omega^{N-1}(\Xcheck)[q^{\pm 1}]$. By Equation~\eqref{e:plamdarelation}
we have that $d\eta + \frac 1 \hbarr dW_q\wedge \eta\in \Omega^{N}(\Xcheck)[\hbarr,q^{\pm 1}]$. This implies firstly that 
\[
d\eta^{(<-1)} + \frac 1 \hbarr dW_q\wedge \left(\frac 1\hbarr \eta^{[-1]}+\eta^{(<-1)}\right)=0,
\]
and secondly that
\[
 d\eta^{[-1]}+dW_q\wedge \eta^{[0]}
=0,
\]
from the $\hbarr\inv$-component of \eqref{e:plamdarelation}. 
We record that taking the exterior derivative of the second equality we obtain
\begin{equation}\label{e:dWqwedgeeta}
dW_q\wedge d\eta^{[0]}=0.
\end{equation}
Now setting $\hbarr$ to zero in the non-zero terms of $d\eta + \frac 1 \hbarr dW_q\wedge \eta$ and rewriting equation~\eqref{e:plamdarelation} we obtain,
\[
 \sum_\lambda c_\lambda(0,\qb,\qb\inv)\, \p_\lambda\omega=d\eta^{[0]} + dW_q\wedge \eta^{[1]}.
\]
Following \cite[Example 2.32]{Gross:MSbook}, a local argument using the fact that $W_q$ has isolated critical points shows that the sequence of sheaves
\[
\Omega_\Xcheck^{N-2}\overset{dW_q\wedge }{\longrightarrow}\ker\left(\Omega_\Xcheck^{N-1}\overset{dW_q\wedge }{\longrightarrow}\Omega_\Xcheck^N\right)\overset{dW_q\wedge }{\longrightarrow} 0
\]
is exact. Since $d\eta^{[0]}$ is in the kernel above, by \eqref{e:dWqwedgeeta}, we see that there exists locally an $(N-2)$-form $\varepsilon$ such that $d\eta^{[0]}=dW_q\wedge\varepsilon$. As a result we obtain the local equation
\[
 \sum_\lambda c_\lambda(0,\qb,\qb\inv)\, \p_\lambda\omega=dW_q\wedge \varepsilon + dW_q\wedge \eta^{[1]}= dW_q\wedge(\eta^{[1]}+\varepsilon).
\]     
Setting $\nu =\eta^{[1]}+\varepsilon$ this proves the Claim.

\vskip .1cm

We can now finish the proof of the lemma. Suppose that $\sum_\lambda c_\lambda(\hbarr, \qb,\qb\inv)[\p_\lambda\omega]=0$. Then by the Claim we have $\sum_\lambda c_\lambda(0, \qb,\qb\inv)\p_\lambda\omega=dW_q\wedge\nu$ locally. Since $\omega$ is non-vanishing on $\Xcheck$ this implies that $\sum_\lambda c_\lambda(0, \qb,\qb\inv)\p_\lambda$  vanishes on the critical points of $W_q$. Therefore it lies in the ideal $(\partial_\Xcheck W_q)$ in $\C[\Xcheck][q^{\pm 1}]$ and we have proved that $\sum_\lambda c_\lambda(0,\qb,\qb\inv)\p_\lambda=0$ in the Jacobi ring of $W_q$. By Proposition~\ref{p:SchubertPlucker} this implies the relation $\sum_\lambda c_\lambda(\qb)\,\sigma^\lambda=0$ in the quantum cohomology ring $qH^*(X,\C)[q\inv]$ of $X$. Since the Schubert classes are linearly independent over $\C[q^{\pm 1}]$ it follows that $c_\lambda(0,\qb,\qb\inv)=0$ for all $\lambda$. 

We have therefore proved that the first term of $c_\lambda(z,\qb,\qb\inv)=c_\lambda(0,\qb,\qb\inv)+\sum_{i>0} \hbarr^i c_\lambda^{(i)}(\qb,\qb\inv)$ vanishes. Now  assume we know that $c_\lambda^{(i)}(\qb,\qb\inv)=0$ for all $i<i_0$ and all $\lambda$. Then since $z$ is a non-zero-divisor in $H_B$ we can apply the same arguments to
\[
\hbarr^{-i_0}\sum_{\lambda} c_\lambda(z,\qb,\qb\inv)[p_\lambda\omega]=0,
\]
which is another a relation in $\Omega^N(\Xcheck)[\hbarr,q^{\pm 1}]$. Again it follows that the left hand side vanishes after setting $\hbarr=0$, giving $c_\lambda^{(i_0)}(\qb,\qb\inv)=0$. Therefore by induction all terms vanish and $c_\lambda(z,\qb,\qb\inv)=0$ for every $\lambda$. It follows that the $[\p_{\lambda}\omega]$ are linearly independent over $\C[\hbarr,\qb\inv,\qb]$.

\end{proof}

\section{Outline of the proof of Theorem~\ref{t:main1}}
\label{s:StartOfMainProof}
The remainder of the paper will be devoted to proving
Theorem~\ref{t:main1} and its equivariant counterpart. Here we summarize the proof. We keep all of the notation of the previous sections.
We need to show that the following hold for all $\lambda\in \Pn_{k,n}$:
\begin{align} \label{e:dwdtaction1}
\left[ q\dWdq \p_{\lambda}\omega \right] &=
\sum_{\mu}[\p_{\mu}\omega]+q\sum_{\nu}[\p_{\nu}\omega]; \\
\label{e:dwdtaction2}
\frac{1}{z}\left[ W\p_{\lambda}\omega\right] &=
\frac{n}{z}\left(\sum_{\mu}[\p_{\mu}\omega]+q\sum_{\nu}[\p_{\nu}\omega] \right) -|\lambda|[\p_{\lambda}\omega],
\end{align}
where $\mu$, $\nu$ are exactly as in the quantum Monk's rule for
$\sigma^{\Box}*_q\sigma^{\lambda}$. This will be shown
in Theorem~\ref{t:dwdtaction}.

In Section~\ref{s:CAcoordring}, we have recalled the cluster structure on the Grassmannian, following Scott~\cite{Scott:Grassmannian},
in terms of Postnikov diagrams.
Next, in Section~\ref{s:specialdiagram}, for
any partition $\lambda$, we
construct an extended cluster containing
$p_{\lambda}$ and also $p_{\mu}$ and $p_{\nu}$ for all $\mu$ and $\nu$ appearing in the
quantum Monk's rule for $\sigma^{\Box}*_q\sigma^{\lambda}$. This extended
cluster corresponds to a Postnikov diagram $D_{\lambda}$ with good properties;
in particular, strands $i$ and $i+1$ cross at a single point
in the diagram.

In order to work with $W$ in this extended cluster we use Theorem~\ref{thm:MarSco}, a special case of~\cite[Thm. 1.1]{MarSco}. In Section~\ref{s:anycluster}, we recall this result in the special case we will need. This gives us an expansion of the numerators $p_{{\widehat{\mu}_i}}$ occurring in
the definition of $W_q$, in terms of an extended cluster
arising from an arbitrary Postnikov diagram, $D$. Namely, the expansion of
$p_{\widehat{\mu}_i}$ is given as a sum of Laurent monomials,
with the set of terms in bijection with the perfect matchings on
a bipartite graph $G_i$ related to the Postnikov diagram $D$.

In Sections~\ref{s:matchingconstruction} and \ref{s:matchingproperties}
we analyse the perfect matchings on the graphs $G_i$, by
constructing a natural initial perfect matching, $\Mi$,
and showing that all other perfect matchings can be obtained from
$\Mi$ by face flips, starting with a face adjacent to the
crossing point of $i$ and $i+1$. This involves calculating the
elementary components of $G_i$.
In Section~\ref{s:matchingevaluation}
we compute the matching monomial
corresponding to $\Mi$.

If $\xi$ is a regular vector field on $\Xcheck$ then it can be inserted
into $\omega$ to give an $(n-1)$-form $i_{\xi}\omega$. By the definition
of $G^W$, we have the relation:
\begin{equation}
[di_{\xi}\omega]+\frac{1}{z}[(\xi W_q)\omega]=0
\label{e:insertionrelation}
\end{equation}
(see equation~\ref{e:insertion}).
The proof then proceeds by constructing
explicit vector fields for which the
relation~\eqref{e:insertionrelation}
implies equations~\eqref{e:dwdtaction1} and~\eqref{e:dwdtaction2}.

So, in Section~\ref{s:vectorfield} we use the cluster structure to define a regular vector field $X_{\lambda}$ on $\Xcheck$, together with twisted versions $X^{(\m)}_{\lambda}$, $m\in [1,n]$ (satisfying $X^{(n)}_{\lambda}=X_{\lambda}$).
To obtain the desired relations,
we need to compute $X_{\lambda}^{(m)}W_q$.
We compute the action of $X_{\lambda}^{(\m)}$ on each of the monomials in the expansion of $W_q$ 
in terms of the extended cluster associated to
the Postnikov diagram $D_{\lambda}$.
This is done by first computing the action on the monomial associated to
$\Mi$. Then in Section~\ref{s:actionvectorfield2} we use face
flips starting from $\Mi$ to compute the action of $X_{\lambda}^{(m)}$
on an arbitrary monomial in the expansion of $W_q$. We obtain the expression:

\begin{equation} \label{e:vfequation}
X^{(\m)}_{\lambda}W=\left(
\sum_{\mu}\p_{\mu}+q\sum_{\nu}\p_{\nu}\right) - q^{\delta_{mn}}\frac{\p_{\widehat{L}_m}}{\p_{L_m}}\p_{\lambda},
\end{equation}
where $\mu$, $\nu$ are exactly as in the quantum Monk's rule for $\sigma^{\Box}*_q \sigma^{\lambda}$ (Theorem~\ref{t:nonequivariantaction}).
Recall that $L_i=[i-k+1,i]$ (see~\eqref{JiLi} in Section~\ref{s:Plucker}).

The case $\m=n$ gives:
\begin{equation} \label{e:vfequation1}
X_{\lambda}W =
\left(
 \sum_{\mu} \p_{\mu}+q\sum_{\nu}\p_{\nu}
\right)-q\dWdq p_{\lambda},
\end{equation}
and summing~\eqref{e:vfequation} over $\m=1,2,\ldots ,n$, we obtain:
\begin{equation} \label{e:vfequation2}
\sum_{m=1}^n X^{(m)}_{\lambda}W =
n\left( \sum_{\mu} \p_{\mu}+q\sum_{\nu}\p_{\nu}
\right)- Wp_{\lambda}, 
\end{equation}
where, in each case, $\mu$, $\nu$ are exactly as in the quantum Monk's rule
for $\sigma^{\Box}*_q\sigma^{\lambda}$; see Corollary~\ref{c:nonequivariantaction}.
In Section~\ref{s:completionofproof}, we use these identities and equation~\eqref{e:insertionrelation} to complete
the proof of Theorem~\ref{t:main1}.
See Theorem~\ref{t:dwdtaction}.


\section{A special Postnikov diagram associated to a partition $\lambda$}
\label{s:specialdiagram}

Given $\lambda\in \Pn_{k,n}$, we denote by $\lambda^{\Box}$ any partition obtained from $\lambda$ by adding a single box. 
Our aim in this section is, given $\lambda\in \Pn_{k,n}$, to define a Postnikov
diagram containing $J_{\lambda}$ and all $J_{\lambda^{\Box}}$ as labels. Moreover, the faces labelled $J_{\lambda^{\Box}}$ should be adjacent to the face labelled $J_{\lambda}$. This diagram will be used later in
explicitly computing the action of $X_{\lambda}$ on $W_q$. We start by constructing special, symmetric Postnikov diagrams in the case $n=2k$ and $J_{\lambda}=\{1,3,\ldots ,2k-1\}$; the
diagrams for arbitrary $(k,n)$ can then be obtained by adding strands
to these in an appropriate way. We assume that $k\not=1,n-1$.

If $J$ is any $k$-subset of $[1,n]$ and $1\leq i\leq n$ with
$i\in J$, $\reduce{i+1}\not\in J$, then we can form a $k$-subset
$J^i$ in which $i$ is replaced by $\reduce{i+1}$.
In this case we say that $i$ is
(clockwise) \emph{moveable} in $J$.
Note that $J=L_i$ for some $i$ if and only if exactly one element of $J$ is
moveable. Therefore any $J$ for which $\p_J$ is a cluster variable
(and not a frozen variable) has at least two moveable elements.

Our aim in this section is to prove the following theorem.

\begin{thm} \label{thm:diagramconstruction}
Let $J$ be a $k$-subset of $[1,n]$.
Then there is a Postnikov
diagram $D(J)$ containing an alternating
face labelled $J$ such that if $i$ is
moveable in $J$, there is an adjacent
alternating face labelled $J^i$.
\end{thm}

For examples in the cases $J=\{1,3\}$ in $Gr_2(4)$, $J=\{1,3,5\}$ in $Gr_3(6)$
and $J=\{1,3,5,7\}$ in $Gr_4(8)$, see Figures~\ref{fig:post24},~\ref{fig:post36}
and~\ref{fig:post48} respectively.
We assume for now that $J\not=L_j$ for
any $j$; we will deal with the case $J=L_j$ at the end of the Section.
\begin{figure}
\psfragscanon
\psfrag{1s}{\pscirclebox{$1$}}
\psfrag{1e}{\psframebox{$1$}}
\psfrag{2s}{\pscirclebox{$2$}}
\psfrag{2e}{\psframebox{$2$}}
\psfrag{3s}{\pscirclebox{$3$}}
\psfrag{3e}{\psframebox{$3$}}
\psfrag{4s}{\pscirclebox{$4$}}
\psfrag{4e}{\psframebox{$4$}}
\psfrag{p}{$\emptyset$}
\includegraphics[width=6cm]{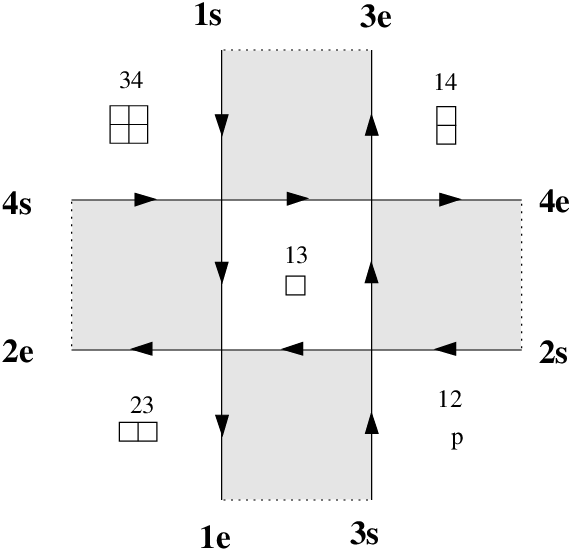}
\caption{The Postnikov diagram $D(\{1,3\})$ for $Gr_2(4)$.}
\label{fig:post24}
\end{figure}


\begin{figure}
\psfragscanon
\psfrag{1s}{\pscirclebox{$1$}}
\psfrag{1e}{\psframebox{$1$}}
\psfrag{2s}{\pscirclebox{$2$}}
\psfrag{2e}{\psframebox{$2$}}
\psfrag{3s}{\pscirclebox{$3$}}
\psfrag{3e}{\psframebox{$3$}}
\psfrag{4s}{\pscirclebox{$4$}}
\psfrag{4e}{\psframebox{$4$}}
\psfrag{5s}{\pscirclebox{$5$}}
\psfrag{5e}{\psframebox{$5$}}
\psfrag{6s}{\pscirclebox{$6$}}
\psfrag{6e}{\psframebox{$6$}}
\psfrag{7s}{\pscirclebox{$7$}}
\psfrag{7e}{\psframebox{$7$}}
\psfrag{8s}{\pscirclebox{$8$}}
\psfrag{8e}{\psframebox{$8$}}
\psfrag{p}{$\emptyset$}
\includegraphics[width=10cm]{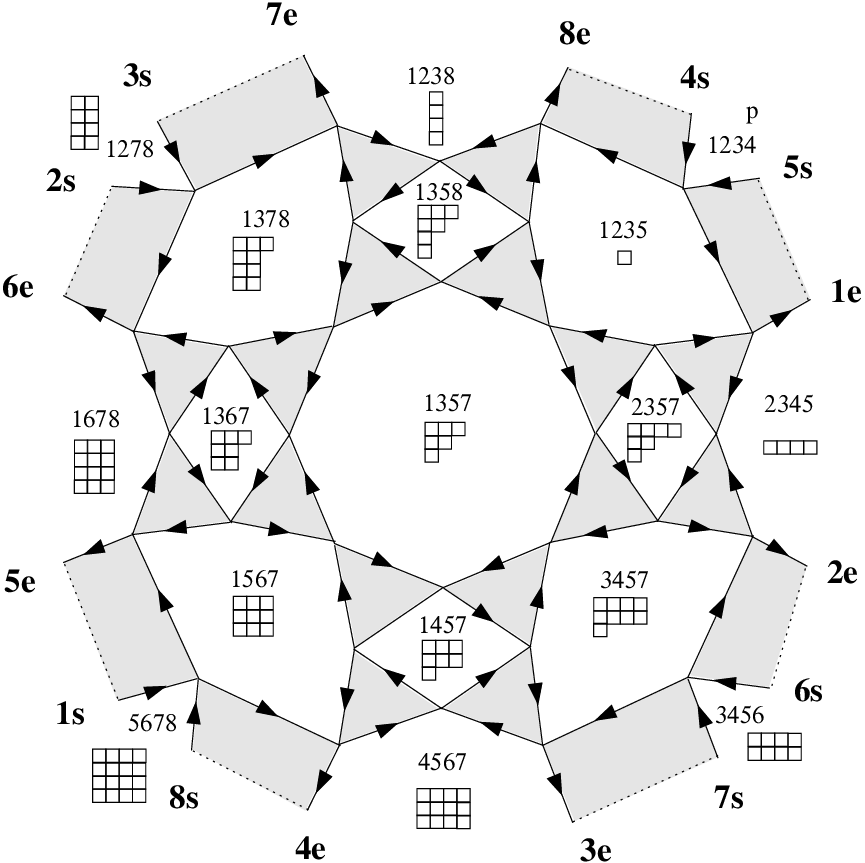}
\caption{The Postnikov diagram $D(\{1,3,5,7\})$ for $Gr_4(8)$.}
\label{fig:post48}
\end{figure}

Our strategy for proving Theorem~\ref{thm:diagramconstruction} is as follows.
Let $K=\{i\in J\,:\,\reduce{i+1}\not\in J\}$ be the set of moveable elements in $J$;
then $|K|\geq 2$ by the assumption above on $J$.
For arbitrary $k\geq 2$ and $n=2k$, we construct an explicit diagram (Proposition~\ref{prop:symmetricdiagram}) demonstrating
Theorem~\ref{thm:diagramconstruction} for the case $J=\{1,3,\ldots ,2k-1\}$.
We then show that, given a diagram satisfying Theorem~\ref{thm:diagramconstruction},
it is possible to add a new strand to it around the boundary (clockwise or anticlockwise)
to obtain another such diagram (Propositions~\ref{prop:addclockwise} and~\ref{prop:addanticlockwise}). Adding a clockwise strand corresponds to
increasing $n$ by $1$, while adding an anticlockwise strand corresponds to increasing
$k$ by $1$.

\begin{prop} \label{prop:symmetricdiagram}
Suppose $k\geq 2$ and $n=2k$. Set
$J=\{1,3,\ldots ,2k-1\}$.
Then Theorem~\ref{thm:diagramconstruction} holds for $J$.
\end{prop}

\begin{proof}
For $k=2,3,4$ it is easy to verify that the diagrams in
Figures~\ref{fig:post24},~\ref{fig:post36} and~\ref{fig:post48} are Postnikov
diagrams and that they show that the theorem holds in this case.
We now show how to construct such diagrams for arbitrary $k\geq 4$.

Consider a tiling $\T$ of the plane by regular hexagons and equilateral triangles in which there are two triangles and two hexagons incident with each vertex
(in the order hexagon, triangle, hexagon, triangle) and for which each of
the hexagons has a vertical pair of parallel sides. We convert this into a
tiling of the upper half-plane by cutting along a horizontal line through the
lowest points of a row of hexagons.
The faces in the new tiling incident with this line are hexagons, triangles and
quadrilaterals, repeating in this order
left to right.

We number the edges on the boundary
in the order
$\ldots -2,-3,0,-1,2,1,4,3,6,5,\ldots$
from left to right (i.e.\ switching each pair of integers in which the first is
odd in the natural ordering on $\mathbb{Z}$).
We orient an edge left to right if it has an even label,
and right to left if it has an odd label.
This determines orientations for the
adjacent triangles.
We label all the vertical edges in the tiling
upwards and all the remaining edges downwards.
This gives an orientation of every edge in 
the tiling. See
Figure~\ref{fig:halfplanetiling} for an 
example.
The tiling can be seen as a collection of 
infinite strands, each of which has 
precisely one edge incident with the 
boundary and thus inherits a label from
this edge.

\begin{figure}
\includegraphics[width=8cm]{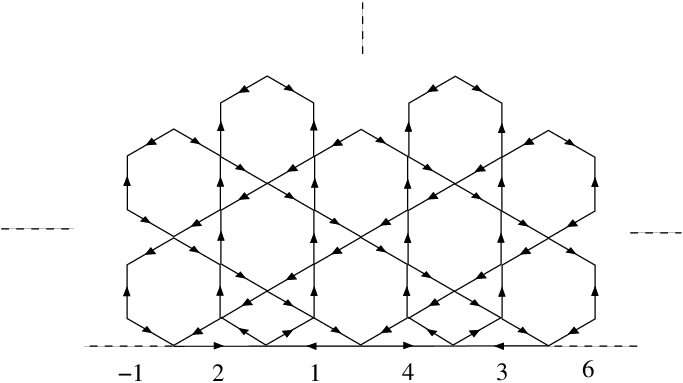}
\caption{A tiling of the upper half-plane.}
\label{fig:halfplanetiling}
\end{figure}

Consider the horizontal line in the plane
which passes through the vertices at the
tops of the triangles above the $(k-3)$rd row of hexagons.
Let $S$ be the subset of the plane
bounded by the boundary of the
half-plane and this horizontal line.
Consider also the vertical lines in the strip
$S$ bisecting edges $-1$ and $2k-1$.
We identify these to obtain a cylinder. The 
tiling of the half-plane induces a tiling of this cylinder, with the strands 
in the half-plane tiling corresponding to 
$2k$ strands in the cylinder labelled by
$1,2,\ldots ,2k$.

Since the cylinder is homeomorphic to an 
annulus, we obtain a corresponding tiling of 
an annulus, in which the lower boundary of 
the cylinder corresponds to the inner 
boundary of the annulus. We glue a disk onto
the inner boundary, obtaining a diagram $\Gamma_k$ on a disk with $2k$ vertices on
its boundary. We label the vertex where
strand $j$ starts by $j$.
It is clear that parts (a) and (b) in the definition of a Postnikov diagram
hold for $\Gamma_k$. 

Furthermore, it is easy to see inductively
that the vertices $1,2,\ldots ,2k$ are 
arranged clockwise around the boundary and 
that if $i$ is even (respectively, odd),
strand $i$ goes clockwise (respectively, 
anticlockwise) around the disk from vertex 
$i$ to vertex $i+k$. Hence property (d)
in the definition of a Postnikov diagram
holds.

For each $i$, we replace vertex $i$
with a pair of vertices $b_i$ and $b'_i$,
with $b'_i$ clockwise of $b_i$. We
move the start of strand $i$ to one of $b_i$ and $b'_i$ and the end of strand $i-k$ to the other in such a way that we don't introduce any additional crossings.
Then the odd strands start at $b_i$ and
the even strands start at $b'_i$.
We call the resulting
diagram $D_{k,2k}$. It is easy to see that it
satisfies part (c) in the definition of a Postnikov diagram (and also parts (a), (b)
and (d) by the above).

It is clear from the construction that
strands $j,j+k$ do not meet, while all
other pairs of strands which are both even or both odd meet exactly once. A pair
of strands of mixed parity may meet several
times, but since one goes clockwise
around the disk and the other goes anticlockwise, they satisfy the requirements
of part (e) in the definition of a Postnikov diagram. Hence we see that this condition
holds for any pair of strands.
We have shown that $D_k$ is a Postnikov diagram as required.
%
%


The label of the inserted disk is easily seen to
be $J=\{1,3,\ldots ,2k-1\}$ and the labels of the $k$ quadrilaterals sharing a
vertex with it are exactly the $J^i$ for $i$ moveable in $J$.
The proof of the proposition is complete.
\end{proof}

Note that, for $k=4$, the diagram in Figure~\ref{fig:post48} is equivalent to
the diagram we have constructed here using the tiling. Furthermore, the
diagrams in~\ref{fig:post24} and~\ref{fig:post36} can also be constructed from
the tiling by taking appropriate subsets of the tiling near the half-plane boundary.
We also remark that none of the faces constructed in Proposition~\ref{prop:symmetricdiagram}
(i.e.\ giving the required labels) is a boundary face.

We next consider the addition of a clockwise strand.

\begin{prop} \label{prop:addclockwise}
Let $\Gamma$ be a Postnikov diagram for $Gr_k(n)$.
Then a strand $s'$ can be added to $\Gamma$ to produce a Postnikov diagram
on strands $[1,n]\cup \{s'\}$ (with ordering
$1,2,\ldots ,s,s',\ldots ,n$) for $Gr_k(n+1)$ in such a way that all labels of
non-boundary faces remain after adding the strand.
\end{prop}

\begin{proof}
By applying the boundary twist if necessary, we
ensure that the oriented regions adjacent to
the vertices $b_{\reduce{s+1}},b'_{\reduce{s+1}},b_{\reduce{s+2}},b'_{\reduce{s+2}},\ldots ,b'_{\reduce{s+k}}$
are all oriented clockwise, as shown in Figure~\ref{fig:clockwiserulea}.
We add an extra strand labelled $s'$ with starting point between $b_{\reduce{s+1}}$ and
$b'_{\reduce{s+1}}$. The ending point of $s'$ is placed between $b_{\reduce{s+k}}$ and $b'_{\reduce{s+k}}$.
The strand $s'$ crosses all the strands between these points
close to the boundary
(i.e.\ with no additional crossings between it and the boundary).
See Figure~\ref{fig:clockwiseruleb}.

\begin{figure}
\psfragscanon
\psfrag{s+1-kE}{\psframebox{$\scriptstyle \reduce{s+1-k}$}}
\psfrag{s+2-kE}{\psframebox{$\scriptstyle \reduce{s+2-k}$}}
\psfrag{sE}{\psframebox{$\scriptstyle s$}}
\psfrag{s+1S}{\pscirclebox{$\scriptstyle \reduce{s+1}$}}
\psfrag{s+2S}{\pscirclebox{$\scriptstyle \reduce{s+2}$}}
\psfrag{s+kS}{\pscirclebox{$\scriptstyle \reduce{s+k}$}}
\psfrag{bs+1}{$\scriptstyle b_{s+1}$}
\psfrag{b's+1}{$\scriptstyle b'_{s+1}$}
\psfrag{bs+2}{$\scriptstyle b_{s+2}$}
\psfrag{b's+2}{$\scriptstyle b'_{s+2}$}
\psfrag{bs+k}{$\scriptstyle b_{s+k}$}
\psfrag{b's+k}{$\scriptstyle b'_{s+k}$}

\includegraphics[width=12cm]{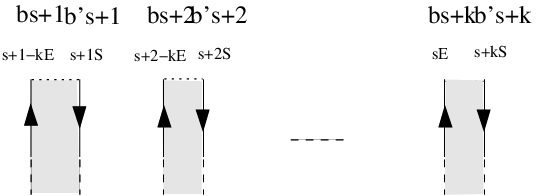}
\caption{Adding a clockwise strand: before.}
\label{fig:clockwiserulea}
\end{figure}

\begin{figure}
\psfragscanon
\psfrag{s+1-kE}{\psframebox{$\scriptstyle \reduce{s+1-k}$}}
\psfrag{s+2-kE}{\psframebox{$\scriptstyle \reduce{s+2-k}$}}
\psfrag{sE}{\psframebox{$\scriptstyle s$}}
\psfrag{ssE}{\psframebox{$\scriptstyle s'$}}
\psfrag{ssS}{\pscirclebox{$\scriptstyle s'$}}
\psfrag{s+1S}{\pscirclebox{$\scriptstyle \reduce{s+1}$}}
\psfrag{s+2S}{\pscirclebox{$\scriptstyle \reduce{s+2}$}}
\psfrag{s+kS}{\pscirclebox{$\scriptstyle \reduce{s+k}$}}
\includegraphics[width=12cm]{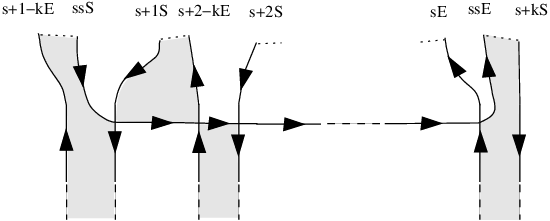}
\caption{Adding a clockwise strand: after.}
\label{fig:clockwiseruleb}
\end{figure}

The new diagram is again a Postnikov diagram.
In particular, for (c), note that the crossing nearest to the
starting point of strand $\reduce{s+1}$ must cross it from left (before adding the
new strand), since $\reduce{s+1}$ starts at $b'_{\reduce{s+1}}$. After adding the new
strand, the starting point of $s+1$ is $b_{\reduce{s+1}}$ and the
crossing nearest the starting point comes from the right, as required.
The next crossing point on strand $s+1$ is the original first crossing, which is from the
left, so the alternation rule, property (c) in Definition~\ref{defn:Postnikov}, does not fail. Similar arguments apply to the other strands.

We also note that, relabelling $1,2,\ldots ,s,s',\ldots ,n$ with
$1,2,\ldots ,n+1$, strand $i$ ends at $b_{\reduce{i+k}}$ or $b'_{\reduce{i+k}}$ for all $i$,
in the new diagram.

Since only boundary alternating faces are on the left of the added strand,
$s'$, all of the labels of the non-boundary alternating faces of the
original diagram will be labels of non-boundary alternating faces of the
new diagram, as required.
\end{proof}

For an example, see Figure~\ref{fig:clockwiseexample}.
Here we add a strand $3'$ to the diagram in Figure~\ref{fig:post36}.
Note that we have to first apply the boundary twist to the two strings
incident with $b_5$ and $b'_5$. We obtain a Postnikov diagram for $Gr_{3,7}$
(with the numbering $1,2,3,3',4,5,6$).

\begin{figure}
\includegraphics[width=8cm]{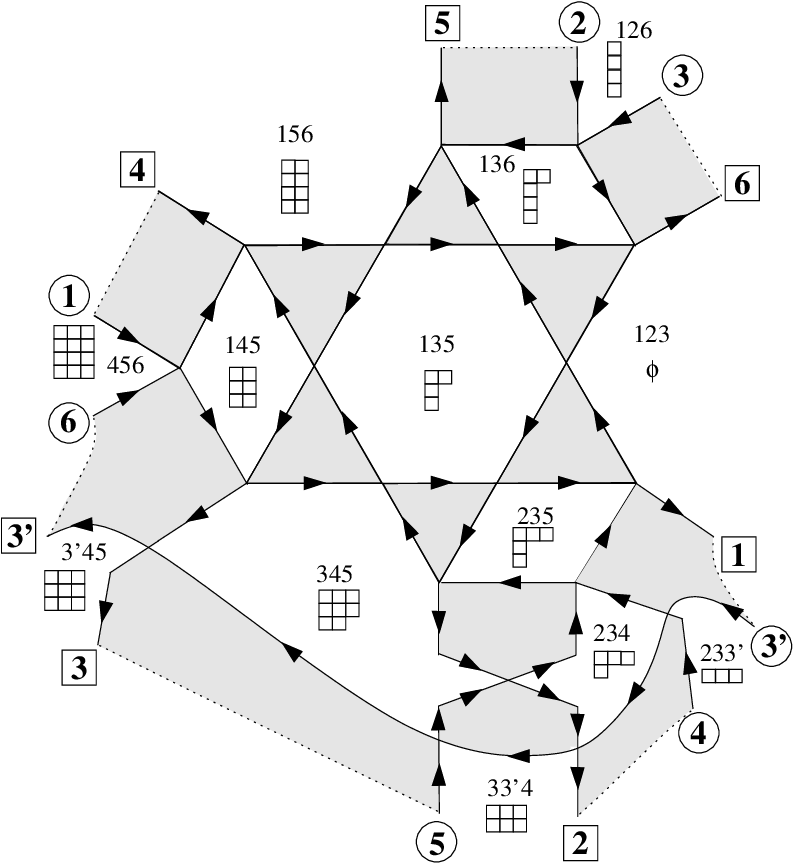}
\caption{Example: adding a clockwise strand $3'$.}
\label{fig:clockwiseexample}
\end{figure}

Secondly, we consider adding an anticlockwise strand.

\begin{prop} \label{prop:addanticlockwise}
Let $\Gamma$ be a Postnikov diagram on strands $[1,n]$
for $Gr_k(n)$. Then a strand $a'$ can be added to $\Gamma$ to produce a
Postnikov diagram on strands $[1,n]\cup \{s'\}$
(with ordering $1,\ldots ,s,s',\ldots ,n$) for $Gr_{k+1}(n+1)$ in such a way
that, for all labels $I$ of non-boundary alternating faces in the original
diagram, $I\cup\{s'\}$ labels a non-boundary alternating
face in the augmented diagram.
\end{prop}

\begin{proof}
By applying the boundary twist if necessary, we ensure that the oriented regions adjacent to the
$b_s,b'_s,b_{\reduce{s-1}},b'_{\reduce{s-1}} \ldots ,b'_{\reduce{s-(n-k-1)}}=b'_{s+k+1}$ are all oriented anticlockwise,
as shown in Figure~\ref{fig:anticlockwiserulea}. We add an 
extra strand labelled $s'$ with starting point between $b_s
$ and $b'_s$.
The ending point of $s'$ is placed between $b_{\reduce{s-(n-k-1)}}$ and $b'_{\reduce{s-(n-k-1)}}$.
It crosses all the strands between these points
close to the boundary (i.e.\ with no additional crossings between it and the boundary).
See Figure~\ref{fig:anticlockwiseruleb}.

\begin{figure}
\vskip 0.4cm
\psfragscanon
\psfrag{s-(n-k-1)S}{\pscirclebox{$\scriptstyle \reduce{s-(n-k-1)}$}}
\psfrag{s+1E}{\psframebox{$\scriptstyle \reduce{s+1}$}}
\psfrag{s-1S}{\pscirclebox{$\scriptstyle \reduce{s-1}$}}
\psfrag{sS}{\pscirclebox{$\scriptstyle s$}}
\psfrag{s-kE}{\psframebox{$\scriptstyle \reduce{s-k}$}}
\psfrag{s-k-1E}{\psframebox{$\scriptstyle \reduce{s-k-1}$}}
\psfrag{s-1S}{\pscirclebox{$\scriptstyle \reduce{s-1}$}}
\includegraphics[width=12cm]{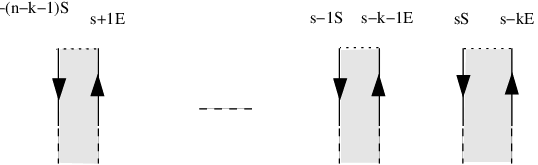}
\caption{Adding an anticlockwise strand: before.}
\label{fig:anticlockwiserulea}
\end{figure}

\begin{figure}
\vskip 0.4cm
\psfragscanon
\psfrag{s-(n-k-1)S}{\pscirclebox{$\scriptstyle \reduce{s-(n-k-1)}$}}
\psfrag{ssS}{\pscirclebox{$\scriptstyle s'$}}
\psfrag{ssE}{\psframebox{$\scriptstyle s'$}}
\psfrag{s+1E}{\psframebox{$\scriptstyle \reduce{s+1}$}}
\psfrag{s-k-1E}{\psframebox{$\scriptstyle \reduce{s-k-1}$}}
\psfrag{sS}{\pscirclebox{$\scriptstyle s$}}
\psfrag{s-kE}{\psframebox{$\scriptstyle \reduce{s-k}$}}
\psfrag{s-1S}{\pscirclebox{$\scriptstyle \reduce{s-1}$}}
\includegraphics[width=12cm]{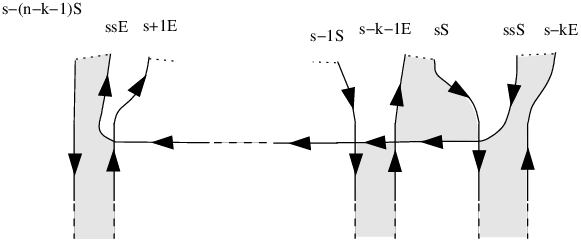}
\caption{Adding an anticlockwise strand: after.}
\label{fig:anticlockwiseruleb}
\end{figure}

It is easy to check that the new diagram is again a 
Postnikov diagram.
The proof is as for Proposition~\ref{prop:addclockwise}.
Since all non-boundary alternating faces in the original 
diagram are on the left of the added strand, $s'$,
the non-boundary labels of the new diagram are indeed
precisely the sets $I\cup \{s'\}$ where $I$ is a
non-boundary label of the old diagram.

Finally, it is also easy to check that, relabelling
$1,2,\ldots ,s,s',\ldots ,n$ with $1,2,\ldots,n+1$, strand $i$ ends at
$b_{\reduce{i+k}}$ or $b'_{\reduce{i+k}}$ for all $i$, in the new diagram.
\end{proof}

\begin{proof}[Proof of Theorem~\ref{thm:diagramconstruction}]
Suppose we are given a $k$-subset $J\subseteq [1,n]$ with
$J\not=L_j$ for any $j$.
Let $K=\{i\in J\,:\,\reduce{i+1}\not\in J\}$ be the
set of moveable elements in $J$; then $|K|\geq 2$.
We start with the diagram $D_{|K|}$ as in Proposition~\ref{prop:symmetricdiagram}.
Let $K+1=\{\reduce{i+1}\,:\,i\in K\}$; note that
$K$ and $\reduce{K+1}$ are disjoint.
Renumber the strands $\{1,2,3,\ldots ,2k\}$ with the elements of $K\cup (\reduce{K+1})$ taken
in cyclic order, starting with the smallest element of $K$.
Then, for each $i\in [1,n]\setminus (J\cup (K+1))$ (in numerical order), use
Proposition~\ref{prop:addclockwise} to add a clockwise strand $i$ in
such a way that all labels of non-boundary faces remain the 
same after adding the strand.
Finally, for each $i\in J\setminus K$ (again in numerical
order), use Proposition~\ref{prop:addanticlockwise} to
add an anticlockwise strand $i$ in such a way that for all 
labels $I$ of non-boundary faces, the resulting diagram has 
a label of form $I\cup \{i\}$. The diagram
constructed in this way satisfies the requirements of Theorem~
\ref{thm:diagramconstruction}.

If $J=L_j$ for some $j$, then $K=\{j\}$. Suppose that $j$ is even. Recall from the proof of Proposition~\ref{prop:symmetricdiagram} that in $D_k$, strand $j$ goes clockwise around the disk while strand $j+1$ goes anticlockwise around the disk. Hence these two strands must cross on the boundary of $L_j$ and, in fact, this must be the only crossing on the boundary. If $j$ is odd, we may rotate $D_k$ to obtain a diagram with the same property; this shows the result for $Gr(k,2k)$. A
diagram for arbitrary $n$ can then be
obtained by adding strands using Propositions~\ref{prop:addclockwise} and~\ref{prop:addanticlockwise}.
\end{proof}

\begin{defn} \label{d:Dlambda}
If $\lambda\in\Pn_{k,n}$ corresponds to a $k$-subset $J=J_{\lambda}$ of 
$[1,n]$, we write $D_{\lambda}$ for $D(J)$. We denote the face 
of $D_{\lambda}$ labelled $J_{\lambda}$ by $F(\lambda)$.
\end{defn}

\begin{defn} \label{d:cyclicgroup}
For $\m\in [1,n]$, let $\mu^{(\m)}$ denote the partition
defined by 
\[J_{\mu^{(\m)}}:=J_{\mu}-\m\mod n,
\]
where $J_\mu-m$ denotes the set obtained from $J_\mu$ by subtracting $m$ from every element. For example for the empty partition $\emptyset$, we have $J_{\emptyset}=\{1,2,\dotsc, k\}$ and $\emptyset^{(m)}=\mu_{n-m}$ with associated subset $J_{n-m}$, as defined in Section~\ref{s:GrassmannianBmodel}.
Note that $\mu\mapsto \mu^{(\m)}$ gives
an action of the cyclic group of degree $n$ on the set $\Pn_{k,n}$.
\end{defn}

\begin{rem} \label{r:rotateddiagrams}
We can arrange that, for each partition $\lambda\in \Pn_{k,n}$,
the Postnikov diagram $D_{\lambda^{(m)}}$ can be obtained from $D_{\lambda}$ by relabelling strand $i$ as strand $i-\m$ for all $i$.
We follow the proof of Theorem~\ref{thm:diagramconstruction} as given for a single representative $\lambda$ in
each orbit. Then we define $D_{\lambda^{(\m)}}$ to be $D_{\lambda}$ with strand $i$ relabelled as $\reduce{i-\m}$ for all $i$. It is easy to check that $D_{\lambda^{(\m)}}$ can then still be constructed as in the proof of Theorem~\ref{thm:diagramconstruction}:
this holds by construction for $Gr(k,2k)$. In general, suppose that in the
construction, $D_{\lambda}$ is obtained from a diagram $D_0$ for $Gr(k,2k)$ by adding
strands $i_1,i_2,\ldots ,i_r$ in order. Then we construct $D_{\lambda^{(\m)}}$ from
$D_0$ (with strand $i$ relabelled as $\reduce{i-\m}$) by adding
strands $\reduce{i_1-\m},\reduce{i_2-\m},\ldots ,\reduce{i_r-\m}$ in order. We shall assume that the construction has been done in this way throughout the rest of the paper.
\end{rem}

\section{The superpotential written in terms of an arbitrary
Pl\"{u}cker extended cluster}
\label{s:anycluster}
Restricting the regular function $W_q$ on $\Xcheck$ to a cluster torus $\Xcheck_{\CC}$ gives a Laurent polynomial in the associated cluster variables. We will need an explicit formula for this in the case where $\CC$ is a Postnikov extended cluster associated to a Postnikov diagram $D$.
To obtain such a formula, we shall use \cite[Thm.\ 1.1]{MarSco} which 
expresses a twisted version of an arbitrary Pl\"{u}cker coordinate $\p_\mu$
in terms of an arbitrary Postnikov extended cluster.
The Laurent polynomial expansion of $\p_{\mu}|_{\Xcheck_{\CC}}$ is given in terms of perfect matchings
on a bipartite graph associated to the Pl\"ucker coordinate  $\p_{\mu}$
and the Postnikov diagram $D$. 
The Pl\"{u}cker coordinates appearing in the numerators in the
definition of $W_q$ (see~\eqref{e:W}) are themselves twists of
Pl\"{u}cker coordinates, by~\cite[Prop.\ 3.5]{MarSco} (up to frozen variables), so it follows that~\cite[Thm.\ 1.1]{MarSco} gives an expression for these Pl\"{u}cker coordinates themselves. We now go into more 
detail.

Let us label the frozen variables by $L_i=\{i-k+1, \dots, i\}=J_{i-k}$ and recall from \eqref{e:W} the formula for $W_q$, which with this notation reads
\begin{equation}
\label{e:W2}
W_q=\sum_{i=1}^n \frac {\p_{\widehat{L}_{i}}} {\p_{L_{i}}}\, \qb^{ \delta_{i,n} }.
\end{equation}
We will now explain how to express the $\p_{\widehat{L}_i}$ in terms of an arbitrary Pl\"ucker extended cluster, by an application of~\cite[Thm.\ 1.1]{MarSco}.

\begin{defn} \label{d:Di}
Let $\CC$ be an arbitrary Postnikov extended cluster and $D$ its Postnikov diagram.
Fix $1\leq i\leq n$ and set
$$N_i=\{i\}\cup [\reduce{i+2},\reduce{i+k}].$$
Let $D_i$ be the Postnikov diagram obtained from $D$ by applying
boundary twists (see Figure~\ref{fig:boundarytwist}) where necessary to ensure that the oriented boundary region adjacent to $b_j,b'_j$ is anticlockwise for all $j\in N_i$ and
clockwise otherwise. If $k=1$ or $k=n-1$, we take $D$ to be the
Postnikov diagram described in Remark~\ref{r:post1n} (see Figure~\ref{fig:postnikovk1}). We have $N_i=\{i\}$, so, as in the other cases, to get $D_i$ we add a boundary twist between the end of strand $i-1$ and the start of strand
$i$ on the boundary. We also add a double boundary
twist between strands $\reduce{j-1}$ and $j$ for all
$j\not=i$ to avoid degeneracies.
\end{defn}

Figure~\ref{fig:quiverexample} shows the Postnikov
diagram $D_3$ constructed from the diagram $D$ in Figure~\ref{fig:post36}. Note that $N_3=\{3,5,6\}$, so in $D_3$,
the oriented regions adjacent to the starting points of
strands $1,2$ and $4$ are black and $3,5$ and $6$ are white.

\begin{defn} \label{def:QiGi}
We denote the quiver of $D_i$ by $\Qi$. We also associate a bipartite graph $G_i$ to $D_i$ as follows~\cite[\S 14]{Postnikov:Totalpositivity}.
Note that any intersection point of two
strands is surrounded by four distinct faces, two of which are oriented faces.
Let $G_i$ be the graph with vertices corresponding to the oriented faces of $D_i$ and edges determined by the intersection points of strands. The vertex in $G_i$ corresponding
to the oriented boundary face adjacent to $b_j,b'_j$ is
labelled $\boundary_j$, for $j\in [1,n]$.   Marking a vertex black if corresponds to a clockwise face and white if corresponds to an anticlockwise face makes
$G_i$ a bipartite graph. Thus the vertices $\boundary_j$ for $j\in N_i$ are white and the other vertices are black.
\end{defn}

The graph $G_i$ can be regarded as being naturally embedded in the disk, with the vertex corresponding to an oriented face of $D_i$ plotted at an interior point in the middle of the face. If there is an edge between two vertices of $G_i$, it is drawn to pass through the corresponding strand intersection point mentioned above. The faces of $G_i$ are in bijection with the alternating faces of $D_i$, and hence with the  Pl\"ucker coordinates in the Postnikov extended cluster $\CC$.
We refer to $G_i$ as the \emph{dual bipartite graph} of $D_i$.
See Figure~\ref{fig:quiverexample}, where the dual bipartite graph $G_3$ of $D_3$ and the quiver $Q_3$ have been superimposed on the
same diagram. 

\begin{figure}
\psfragscanon
\psfrag{v1}{$\boundary_1$}
\psfrag{v2}{$\boundary_2$}
\psfrag{v3}{$\boundary_3$}
\psfrag{v4}{$\boundary_4$}
\psfrag{v5}{$\boundary_5$}
\psfrag{v6}{$\boundary_6$}
\psfrag{1s}{\pscirclebox{$1$}}
\psfrag{1e}{\psframebox{$1$}}
\psfrag{2s}{\pscirclebox{$2$}}
\psfrag{2e}{\psframebox{$2$}}
\psfrag{3s}{\pscirclebox{$3$}}
\psfrag{3e}{\psframebox{$3$}}
\psfrag{4s}{\pscirclebox{$4$}}
\psfrag{4e}{\psframebox{$4$}}
\psfrag{5s}{\pscirclebox{$5$}}
\psfrag{5e}{\psframebox{$5$}}
\psfrag{6s}{\pscirclebox{$6$}}
\psfrag{6e}{\psframebox{$6$}}
\psfrag{p}{$\emptyset$}
\psfrag{P3}{$P_3$}
\includegraphics[width=10cm]{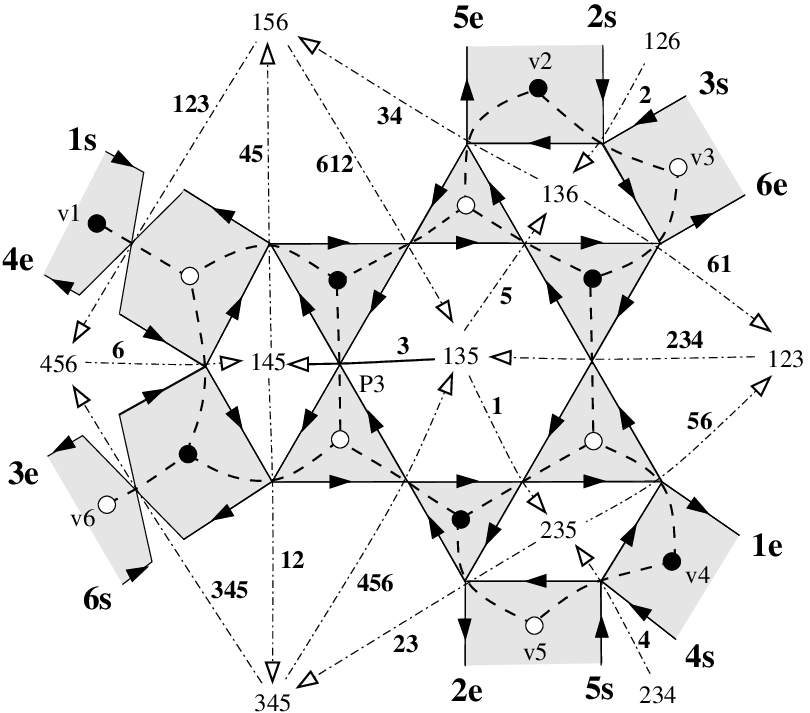}
\caption{The Postnikov diagram $D_3$ constructed from the
diagram in Figure~\ref{fig:post36}, together with its
bipartite dual, $G_3$ (equal dashed lines between coloured vertices) and the quiver $Q_3$ (uneven dashed arrows). Note that $N_3=\{3,5,6\}$. For an explanation of the notation $P_3$, see the comment after Lemma~\ref{l:uniquecrossing}.}
\label{fig:quiverexample}
\end{figure}

We assign monomial weights $w_e\in \C[\Xcheck_{\CC}]$ to the edges of $G_i$ as follows. Let $v$ be the unique black vertex incident with an edge $e$. The weight $w_e$ of $e$  is defined to be the product of the Pl\"ucker coordinates labelling the faces of $G_i$ which are incident with $v$ but not with the rest of $e$ (i.e.\ excluding the two faces on each side of $e$). See Figure~\ref{fig:weighting} for an illustration
of the rule. Figure~\ref{fig:bipartiteweighted} shows the weighting on the dual bipartite graph of Figure~\ref{fig:quiverexample}.

\begin{figure}
\psfragscanon
\psfrag{M1}{$\p(1)$}
\psfrag{M2}{$\p(2)$}
\psfrag{Md}{$\p(d)$}
\psfrag{e}{$e$}
\includegraphics[width=4cm]{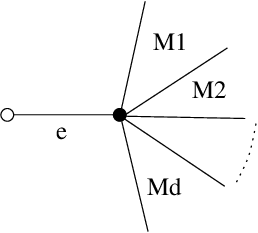}
\caption{Weighting of an edge: $w_e=\p(1)\p(2)\cdots \p(d)$.}
\label{fig:weighting}
\end{figure}

\begin{figure}
\psfragscanon
\psfrag{v1}{$\boundary_1$}
\psfrag{v2}{$\boundary_2$}
\psfrag{v3}{$\boundary_3$}
\psfrag{v4}{$\boundary_4$}
\psfrag{v5}{$\boundary_5$}
\psfrag{v6}{$\boundary_6$}
\psfrag{p123}{$p_{123}$}
\psfrag{p234}{$p_{234}$}
\psfrag{p345}{$p_{345}$}
\psfrag{p456}{$p_{456}$}
\psfrag{p156}{$p_{156}$}
\psfrag{p126}{$p_{126}$}
\psfrag{p136}{$p_{136}$}
\psfrag{p145}{$p_{145}$}
\psfrag{p235}{$p_{235}$}
\psfrag{p135}{$p_{135}$}
\psfrag{1}{$1$}
\includegraphics[width=8cm]{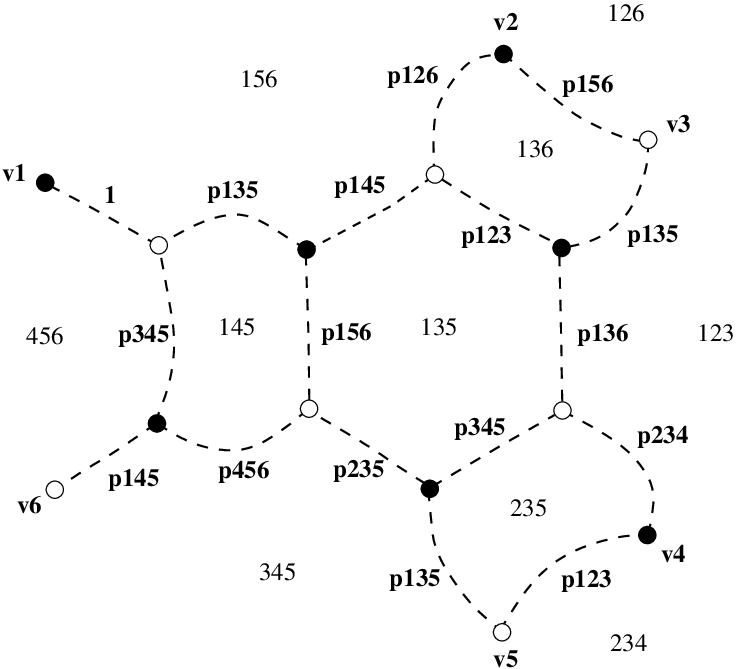}
\caption{The weighting on the dual bipartite graph $G_3$ in Figure~\ref{fig:quiverexample}.}
\label{fig:bipartiteweighted}
\end{figure}

Recall that a \emph{perfect matching} of a graph $\Gamma$ with edge-set $E$ is a subset $M$ of $E$ such that each vertex of $\Gamma$ is incident with precisely one edge in $M$. If $\Gamma$ is weighted with weighting $w_e$ for each edge $e$, its \emph{matching polynomial} is given by: $$w_{\Gamma}=\sum_{M} \prod_{e\in M} w_e,$$
where the sum is over all perfect matchings of $\Gamma$. We will sometimes write $w_M=\prod_{e\in M} w_e$.

We then have the following theorem, which is the special case of a more general result,~\cite[Thm.\ 1.1]{MarSco}.

\begin{thm} \cite{MarSco} \label{thm:MarSco} \\
Let $D$ be a Postnikov diagram, with corresponding Postnikov cluster
$\CI=\CI(D)$ and extended cluster $\CC=\CC(D)$.
Fix $1\leq i\leq n$. Let $G_i$ be the bipartite graph defined above.
Then the following holds in $\C[\Xcheck_\CC]$~:
$$\p_{\widehat{L}_i}=
\frac{w_{G_i}}{\cluster} \p_{L_{\reduce{i-1}}} \p_{L_{i}}
\cdots \p_{L_{\reduce{i+k}}}
=\sum_{M}
\frac{w_M}{\cluster} \p_{L_{\reduce{i-1}}} \p_{L_{i}}
\cdots \p_{L_{\reduce{i+k}}},
$$
where the sum is over all perfect matchings
$M$ of $G_i$.
\end{thm}


Recall that the superpotential $W_q$ is given by the formula:
\begin{equation*}
W_q=\sum_{i=1}^n \frac {\p_{\widehat{L}_{i}}} {\p_{L_{i}}}\, \qb^{ \delta_{i,n} }.
\end{equation*}
Theorem~\ref{thm:MarSco} states that 
\begin{equation}\label{e:Lihatterm}
\frac {\p_{\widehat{L}_{i}}} {\p_{L_{i}}}=\sum_{M}
\frac{w_M}{\cluster} \p_{L_{\reduce{i-1}}} \p_{L_{i+1}} \p_{L_{i+2}}
\cdots \p_{L_{\reduce{i+k}}}.
\end{equation}
  We thus have the following:

\begin{cor} \label{cor:Wanycluster}
Let $D$ be a Postnikov diagram, with corresponding Postnikov cluster
$\CI=\CI(D)$ and extended cluster $\CC=\CC(D)$.
Let $G_1,G_2,\ldots ,G_n$ be the bipartite graphs associated above to $D$.
Then we have the following expression for
the superpotential $W_q$ (see equation~\eqref{e:W2}) in $\C[\Xcheck_\CC]$~:
\begin{equation}
W_q=\frac{1}{\cluster} \sum_{i=1}^n
w_{G_i} \p_{L_{\reduce{i-1}}} \p_{L_{\reduce{i+1}}}
\cdots \p_{L_{\reduce{i+k}}} q^{\delta_{in}}.
\end{equation}
\end{cor}

\section{Construction of a perfect matching}
\label{s:matchingconstruction}
Suppose $\lambda\in \Pn_{k,n}$ and let $D=D_{\lambda}$ be the Postnikov diagram constructed by Theorem~\ref{thm:diagramconstruction}.
For example, Figure~\ref{fig:post36} shows the Postnikov diagram $D_{\ydiagram{2,1}}$ for $G(3,6)$, noting that $J_{\ydiagram{2,1}}=\{1,3,5\}$.
Let $D_1,D_2,\ldots ,D_n$ be the associated boundary-adjusted Postnikov diagrams and $G_1,G_2,\ldots ,G_n$ be the corresponding dual bipartite graphs, associated to $D$ in Section~\ref{s:anycluster}. We also have corresponding quivers $Q_1,Q_2,\ldots ,Q_n$ (see Definitions~\ref{d:Di} and~\ref{def:QiGi}).
For example, Figure~\ref{fig:quiverexample} shows $D_3$, $G_3$ and $Q_3$ for the case $D_{\ydiagram{2,1}}$.

Our aim in this section is to construct an explicit perfect matching $\Mi$ on each of the $G_i$. This perfect matching will correspond to
a distinguished monomial summand in $p_{\widehat{L}_i}/p_{L_i}$. In Section~\ref{s:matchingproperties} we will show
that every other perfect matching can be
obtained from $\Mi$ by face flips
and in Section~\ref{s:matchingevaluation} we
will compute the distinguished monomial
summand explicitly.
We assume that $k\not=1,n-1$.
We first make the following observation.

\begin{lem} \label{l:uniquecrossing}
Let $\lambda\in \Pn_{k,n}$. Then strands $i,\reduce{i+1}$ cross at exactly one point in
$D_{\lambda}$.
\end{lem}

\begin{proof}
Suppose first that $J_{\lambda}\not=L_j$
for any $j$.
For the case $J_{\lambda}=\{1,3,5,\ldots ,2k-1\}$ in $Gr(k,2k)$, the result can be observed
from the construction of $D_{\lambda}$ in Proposition~\ref{prop:symmetricdiagram}.
Note that, for $i$ even, the intersection point of strands $i$ and $\reduce{i+1}$ is the first
intersection point for each of these strands, while for $i$ odd, the intersection
point is on the boundary of the central $n$-sided alternating face in $D_{\lambda}$.

Adding strands does not change this property, since every
new strand crosses the existing strands at most once. Because strand $i$ starts at
$b_i$ or $b'_i$ and ends at $b_{\reduce{i+k}}$ or $b'_{\reduce{i+k}}$, every pair of strands $i,\reduce{i+1}$ must cross at least once. The argument for $J_{\lambda}=L_j$ is similar.
\end{proof}

We define $\PPi$ to be the unique crossing point of strands
$i$ and $i+1$ in $D_{\lambda}$ (and also the corresponding point in $D_i$). The point $P_3$ is shown in Figure~\ref{fig:quiverexample}.
We note that in the case $k=1$ or $n-1$, with our convention (see Remark~\ref{r:post1n}), Lemma~\ref{l:uniquecrossing} does not hold.


\begin{rem} \label{rem:cyclematchings}Perfect matchings for $G_i$ also have an
interpretation in terms of the quiver $Q_i$.
This interpretation will be useful for our construction.
The quiver $\Qi$ is embedded into a disk in such a way that its complement forms a disjoint union of disks. The boundary of such a disk is either a cycle in $\Qi$, which we call a \emph{minimal cycle},
or an oriented path in $\Qi$,
together with part of the boundary of the
whole disk $\mathbb{D}$.
In the latter case, the oriented path goes
from a boundary vertex to an adjacent one
and is called a \emph{boundary path}.

If $E$ is a set of edges in $G_i$, we denote
by $\Sigma(E)$ the set of arrows $\Sigma(E)$ in $Q_i$ crossing the edges in $E$. Then $E$ is a perfect matching in $G_i$ if and only if $\Sigma(E)$ contains exactly one arrow in
each minimal cycle and one arrow in each boundary path
in $Q_i$. We call such a collection of
arrows a \emph{perfect cut} (note that
a set of arrows satisfying the first property is referred to as an \emph{admissible cut} in~\cite{Her:thesis} (see also~\cite[\S1]{BFPPT}).
\end{rem}

In order to construct a perfect cut on
$Q_i$, we use a weighting
on the arrows of $Q_i$ from~\cite[Defn.\ 4.1]{bkm}.
Recall that for $p,q\in [1,n]$, we denote by 
$[p,q]$ the cyclic interval
$\{p,\reduce{p+1},\ldots ,q\}$.

\begin{defn} \label{d:weighting}
Fix $i\in [1,n]$. Then we give each arrow
$\alpha$ in the quiver $\Qi$ a
\emph{weight} given by the set
$$d(\alpha)=[p,q-1]=
\{p,\reduce{p+1},\ldots ,\reduce{q-1}\},$$
in the case where
$\alpha:I\rightarrow J$ and
$J=I-\{p\}+\{q\}$.
Note that the strands $p$ and $q$ cross on 
the arrow $\alpha$:
a typical arrow in $\Qi$ is shown in Figure
\ref{f:grading}.
This weighting is extended to paths
by defining the weight of a path to be the
multiset union of the labels of its arrows,
regarded as a multisubset of $[1,n]$.
\end{defn}

\begin{figure}
\psfragscanon
\psfrag{I}{$I$}
\psfrag{J}{$J$}
\psfrag{p}{$p$}
\psfrag{q}{$q$}
\psfrag{[p,q-1]}{$[p,q-1]$}
\includegraphics[width=4.5cm]{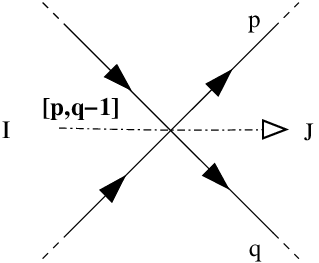}
\caption{The labelling of arrows in $Q(D)$. Strands $p$ and $q$ cross the
arrow from $I$ to $J$, which is given the label $[p,q-1]$.}
\label{f:grading}
\end{figure}

In order to study this weighting,
we need to consider a slightly modified
version of a Postnikov diagram, $D$.
We define the \emph{closure}
$\overline{D}$ of $D$ as follows. Let
$\overline{\mathbb{D}}$ be a disk slightly
larger than $\mathbb{D}$, and extend the 
strands in $D$ incident
with $b_i$ and $b'_i$ on the boundary of
$\mathbb{D}$ to a
common point $\overline{b}_i$ on the boundary of
$\overline{\mathbb{D}}$. This produces a 
Postnikov diagram
$\overline{D}$ according to the original 
definition~\cite[Defn.\ 14.1]{Postnikov:Totalpositivity}.

Let $Q(\overline{D})$ be the quiver 
associated to $\overline{D}$ in~\cite[Defn.\ 2.4]{bkm}: 
the vertices are the same as
the vertices of $Q(D)$, and we take all of 
the arrows in $Q(D)$,
together with additional \emph{boundary arrows} between the boundary 
vertices $L_j$ of the quiver. There
is a boundary arrow for each point
$\overline{b_i}$, oriented and weighted according to a one-sided version of the rule for $Q(D)$; see Figure~\ref{f:boundaryneighbourfaces}.
Thus this boundary arrow through $\overline{b_i}$ is oriented clockwise
if $\boundary_i$ is black and anticlockwise if $\boundary_i$ is white.

\begin{figure}
\psfragscanon
\psfrag{j}{$j$}
\psfrag{j-k}{$j-k$}
\psfrag{Lj}{$L_j$}
\psfrag{Lj-1}{$L_{j-1}$}
\psfrag{bj}{$\overline{b_j}$}
\psfrag{D}{$D$}
\psfrag{Q(D)}{$Q(D)$}
\includegraphics[width=7cm]{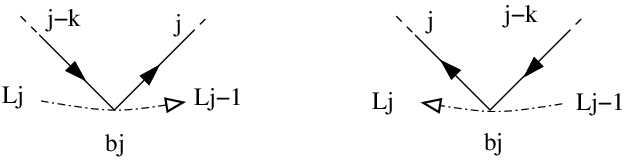}
\caption{Neighbouring faces on the boundary of a closed Postnikov diagram. The weight of the arrow in the left hand figure is $[j,j-k-1]$ and the weight of the arrow in the right hand figure is $[j-k,j-1]$.}
\label{f:boundaryneighbourfaces}
\end{figure}

Note that in $Q(\overline{D})$, there is always a boundary arrow between $L_i$
and $L_{\reduce{i+1}}$, in one direction or 
the other, corresponding
to the intersection point $\overline{b}_i$.
We define $\overline{\Qi}$ to be the quiver
of $\overline{D_i}$.
The closure $\overline{D_3}$ of the Postnikov
diagram $D_3$ in Figure~\ref{fig:quiverexample} 
is shown in Figure~\ref{fig:bipartiteexampleclosed}, and the
corresponding weighted quiver $\overline{Q_3}$ is shown in
Figure~\ref{fig:quiverclosedexample}.

\begin{figure}
\psfragscanon
\psfragscanon
\psfrag{v1}{$\boundary_1$}
\psfrag{v2}{$\boundary_2$}
\psfrag{v3}{$\boundary_3$}
\psfrag{v4}{$\boundary_4$}
\psfrag{v5}{$\boundary_5$}
\psfrag{v6}{$\boundary_6$}
\psfrag{1s}{$\overline{b_1}$}
\psfrag{2s}{$\overline{b_2}$}
\psfrag{3s}{$\overline{b_3}$}
\psfrag{4s}{$\overline{b_4}$}
\psfrag{5s}{$\overline{b_5}$}
\psfrag{6s}{$\overline{b_6}$}
\psfrag{p}{$\emptyset$}
\includegraphics[width=10cm]{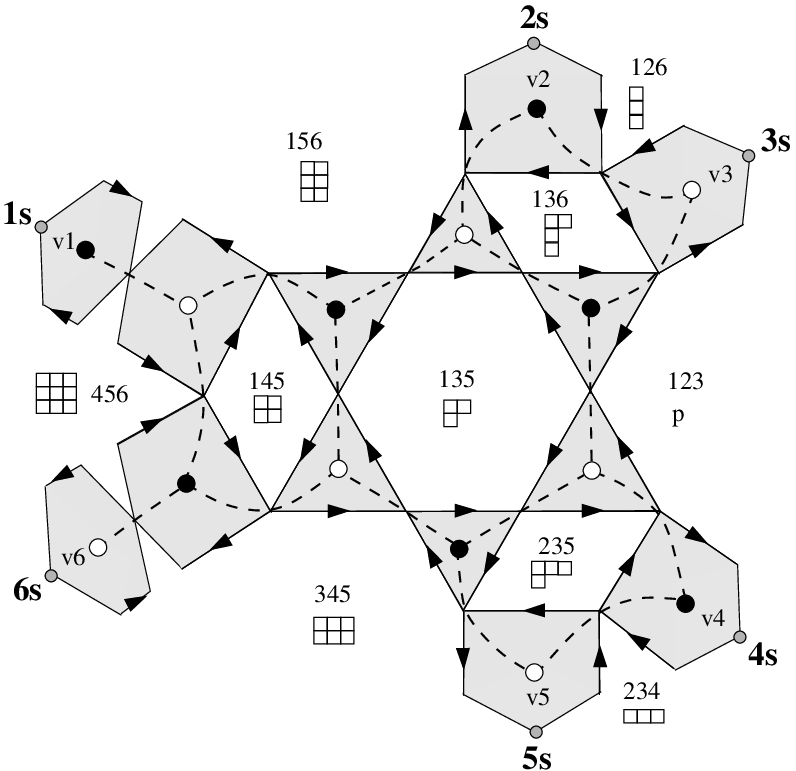}
\caption{The closure $\overline{D_3}$ of the Postnikov diagram $D_3$ in Figure~\ref{fig:quiverexample} and the corresponding dual bipartite
graph.}
\label{fig:bipartiteexampleclosed}
\end{figure}

\begin{figure}
\psfragscanon
\psfrag{v1}{$\boundary_1$}
\psfrag{v2}{$\boundary_2$}
\psfrag{v3}{$\boundary_3$}
\psfrag{v4}{$\boundary_4$}
\psfrag{v5}{$\boundary_5$}
\psfrag{v6}{$\boundary_6$}
\psfrag{1s}{$\overline{b_1}$}
\psfrag{2s}{$\overline{b_2}$}
\psfrag{3s}{$\overline{b_3}$}
\psfrag{4s}{$\overline{b_4}$}
\psfrag{5s}{$\overline{b_5}$}
\psfrag{6s}{$\overline{b_6}$}
\includegraphics[width=10cm]{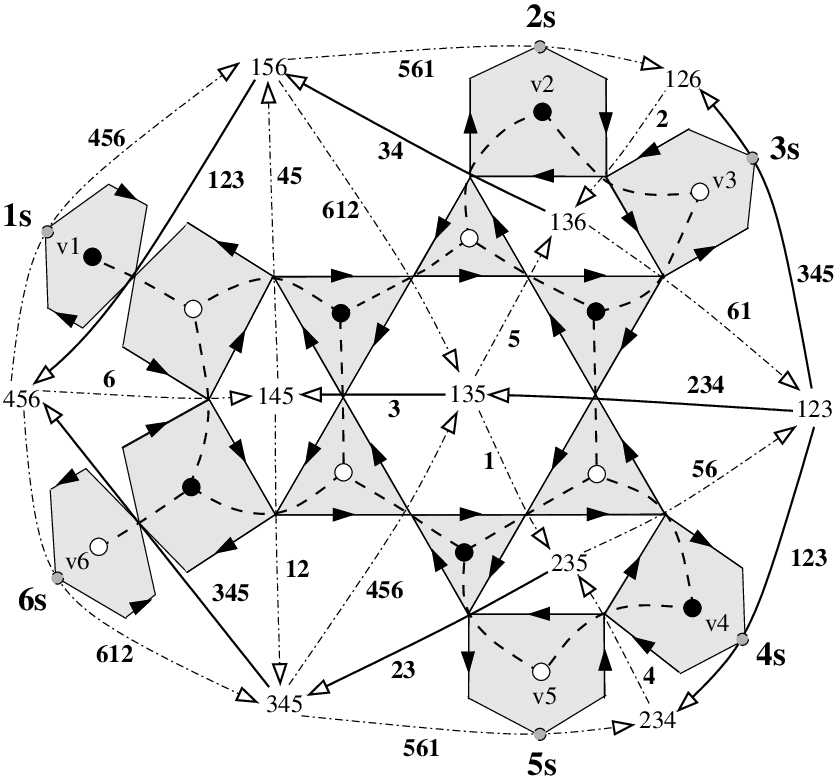}
\caption{The weighted quiver $\overline{Q_3}$ of the Postnikov diagram $\overline{D_3}$ in Figure~\ref{fig:bipartiteexampleclosed}. The arrows in $\overline{S_3}$ (see Definition~\ref{def:Siarrows}) are shown as unbroken arrows.}
\label{fig:quiverclosedexample}
\end{figure}

By~\cite[\S9]{ops} or~\cite[Cor.\ 4.4]{bkm}, we have:
\begin{lem} \label{l:orientedcycle}
Let $D$ be a Postnikov diagram. Then the weights of the arrows in a 
minimal oriented cycle in 
$Q(\overline{D})$ are of the form
$[p_1,p_2-1]$,
$[p_2,p_3-1],\ldots\ ,[p_r ,p_1-1]$,
where $p_1,p_2,\ldots ,p_r$ are in cyclic 
order around $[1,n]$. In particular, the 
weight of such a cycle is $[1,n]$.
\end{lem}

We note, for future reference, the
following corollary:

\begin{cor} \label{c:largercycle}
Let $D$ be a Postnikov diagram and let $c$ be
a cycle in $\overline{D}$. Let $r_0$ be the number of minimal
cycles made up of arrows from $c$ and its
interior, whose orientation is the same as that of $c$. Let $r_1$
be the number of such minimal cycles whose orientation is opposite to that of $c$.
Then $r_0>r_1$ and the weight of $c$ is equal to the multiset union of
$r_0-r_1$ copies of $[1,n]$.
\end{cor}
\begin{proof}
The interiors of the minimal cycles
in the statement of the corollary tile the interior of $c$ completely.
Let $\mathcal{M}_0$ (respectively $\mathcal{M}_1$) denote the set of
such minimal cycles which are oriented in the same way as $c$
(respectively, opposite to $c$).
Each arrow in $c$ lies in exactly one
minimal cycle in $\mathcal{M}_0$, and
each arrow in the interior of $c$ lies on exactly one minimal cycle
in $\mathcal{M}_0$ and exactly one minimal cycle in $\mathcal{M}_1$. Hence, considering the weights, we have:
$$\bigcup_{c_0\in \mathcal{M}_0} d(c_0) = d(c)\cup \bigcup_{c_1\in \mathcal{M}_1}d(c_1),$$
where the unions are multiset unions.
By Lemma~\ref{l:orientedcycle},
$d(c_0)=d(c_1)=[1,n]$ for all minimal
cycles $c_0\in \mathcal{M}_0$ and $c_1\in \mathcal{M}_1$.
Since $r_i=|\mathcal{M}_i|$ for $i=0,1$, the result follows.
\end{proof}

\begin{defn} \label{def:Siarrows}
We set $S_i$ to be the set of arrows in $\Qi$ whose weight contains $i$ (see Figure~\ref{fig:Siexample} for an example), and $\overline{S_i}$ be the set of arrows in $\overline{Q_i}$ whose weight contains $i$ (see Figure~\ref{fig:quiverclosedexample} for an example).
\end{defn}

Applying Lemma~\ref{l:orientedcycle} 
to $\overline{D_i}$, it follows that each minimal cycle in $\overline{\Qi}$ contains exactly one arrow in $\overline{S_i}$. Since $\overline{\Qi}$ contains no boundary paths, it follows that $\overline{S_i}$
can be considered as a perfect cut on $\overline{\Qi}$ in the sense of Remark~\ref{rem:cyclematchings}.

Each minimal cycle in $\Qi$ is a minimal cycle in $\overline{\Qi}$. Therefore,
it contains exactly one arrow in $S_i$.
However, $S_i$ is not a perfect cut
in $\Qi$, because there are boundary
paths which do not contain any arrows in $S_i$, as we shall now see.
For an example of this, see Figure~\ref{fig:Siexample}, which shows the quiver $Q_3$ in our running example and the set of arrows $S_3$.

\begin{figure}
\psfragscanon
\psfrag{v1}{$\boundary_1$}
\psfrag{v2}{$\boundary_2$}
\psfrag{v3}{$\boundary_3$}
\psfrag{v4}{$\boundary_4$}
\psfrag{v5}{$\boundary_5$}
\psfrag{v6}{$\boundary_6$}
\psfrag{1s}{\pscirclebox{$1$}}
\psfrag{1e}{\psframebox{$1$}}
\psfrag{2s}{\pscirclebox{$2$}}
\psfrag{2e}{\psframebox{$2$}}
\psfrag{3s}{\pscirclebox{$3$}}
\psfrag{3e}{\psframebox{$3$}}
\psfrag{4s}{\pscirclebox{$4$}}
\psfrag{4e}{\psframebox{$4$}}
\psfrag{5s}{\pscirclebox{$5$}}
\psfrag{5e}{\psframebox{$5$}}
\psfrag{6s}{\pscirclebox{$6$}}
\psfrag{6e}{\psframebox{$6$}}
\psfrag{p}{$\emptyset$}
\psfrag{P3}{$P_3$}
\includegraphics[width=10cm]{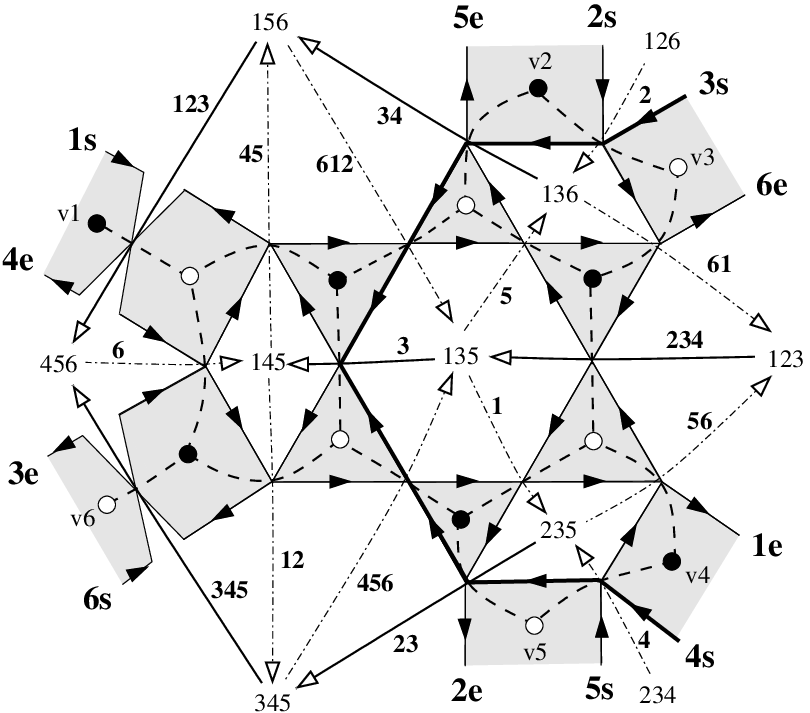}
\caption{The quiver $Q_3$ from Figure~\ref{fig:quiverexample}. The arrows in $S_3$ (see Definition~\ref{def:Siarrows}) are shown as unbroken arrows. The path $\gamma_3$ is shown as a thickened line, with end points given by the encircled $3$ and encircled $4$; see the paragraph after Lemma~\ref{l:strandalternating}.}
\label{fig:Siexample}
\end{figure}

\begin{lem} \label{l:boundaryweight}
Let $D$ be a Postnikov diagram. Then the 
boundary path between
$L_{\reduce{j-1}}$ and $L_j$ has weight
$[j,\reduce{j-k-1}]$ (respectively,
$[\reduce{j-k},\reduce{j-1}]$) if it is 
oriented towards $L_{\reduce{j-1}}$ 
(respectively, $L_j$).
\end{lem}
\begin{proof}
In $Q(\overline{D})$, there is an arrow between $L_j$ and
$L_{\reduce{j-1}}$ crossed by strands $\reduce{j-k}$ and $j$. This arrow
has weight $[j,\reduce{j-k-1}]$ (respectively, $[\reduce{j-k},\reduce{j-1}]$)
if it points towards $L_{\reduce{j-1}}$ (respectively, $L_j$).
Hence the boundary path in $Q(D)$, which
completes this arrow to a minimal cycle
in $Q(\overline{D})$, has weight as claimed by
Lemma~\ref{l:orientedcycle}.
\end{proof}

Recall that the vertex $\overline{b_j}$ in
$\overline{D_i}$ lies on the boundary of the
oriented region of $\overline{D_i}$
corresponding to the vertex $\boundary_j$ in $G_i$.

\begin{lem} \label{l:boundarypaths}
The boundary path around $\overline{b_j}$ in $\Qi$ contains exactly one arrow from $S_i$ if $j\not\in\{i,i+1\}$ and contains no arrow from $S_i$ otherwise.
\end{lem}

\begin{proof}
We use Lemma~\ref{l:boundaryweight}.
If $j\in [\reduce{i+2},\reduce{i+k}]$ then, since $\boundary_j$ is white, the weight of the boundary path around $\boundary_j$ is $[\reduce{j-k},\reduce{j-1}]$ and contains $i$.
Hence exactly one of the arrows on this path lies in $S_i$.
Since $\boundary_{\reduce{i+1}}$ is black, the boundary path around this vertex has weight $[\reduce{i+1},\reduce{i-k}]$ and hence does not contain $i$.

If $j\in [\reduce{i+k+1},\reduce{i-1}]$, then $\boundary_j$ is black, and the weight of the boundary path around this vertex is $[j,\reduce{j-k-1}]$,
so includes $i$, and exactly one of the arrows on this path lies in $S_i$. The vertex $\boundary_i$ is white, so the boundary path around around this vertex has weight $[\reduce{i-k},\reduce{i-1}]$ and has no arrow in $S_i$.
\end{proof}

\begin{rem} \label{r:boundaryarrows}
Since $\overline{S_i}$ is a perfect cut
(and $S_i$ is the restriction of $\overline{S_i}$ from $\overline{\Qi}$ to
$\Qi$), it follows from Lemma~\ref{l:boundarypaths} that the boundary
arrow in $\overline{\Qi}$ through $\overline{b_j}$ lies in $\overline{S_i}$
if and only if $j\in \{i,i+1\}$.
\end{rem}

%
The proof of Lemma~\ref{l:crossingstrandi} arose from discussions of R.~Marsh with K.~Baur and A.~King in an alternative approach towards the results in~\cite{bkm}.

We need to alter $S_i$ in order to obtain a perfect cut on $\Qi$. This
will involve reversing membership of
$S_i$ for certain arrows in $\Qi$.
In order to make this construction,
we need more information about $S_i$.

\begin{lem} \label{l:crossingstrandi}
Let $D$ be a Postnikov diagram and let $\alpha$ be an
arrow in $Q(D)$ or $Q(\overline{D})$. Then $\alpha$ crosses strand $i$ if and only
if the weight of $\alpha$ contains $i$ and not $\reduce{i-1}$, or $\reduce{i-1}$ and not $i$.
\end{lem}

\begin{proof}
If an arrow $\alpha$ crosses strand $i$ from left to right, its weight is $[p,\reduce{i-1}]$ where $p\not=i$ since strands cannot self-intersect. Hence it contains $\reduce{i-1}$ but not $i$.
If $\alpha$ crosses strand $i$ from right to left, its weight is $[i,\reduce{q-1}]$ for some $q\not=i$, and hence contains $i$ but not $\reduce{i-1}$. If $\alpha$ does not cross strand $i$ at all, its weight is $[p,\reduce{q-1}]$
where neither $p$ nor $q$ is equal to $i$. If $i\in [p,\reduce{q-1}]$, then,
since $p\not=i$, we have $i-1\in [p,\reduce{q-1}]$ also. If $i\not\in [p,\reduce{q-1}]$, then,
since $q\not=i$, we have $\reduce{i-1}\not\in [p,\reduce{q-1}]$ also. Thus in this case, either
$\reduce{i-1}$ and $i$ both lie in the weight of $\alpha$ or neither $i$ nor $\reduce{i-1}$
lies in the weight of $\alpha$. The result follows.
\end{proof}

\begin{lem} \label{l:strandalternating}
The arrows in $\overline{\Qi}$ crossing strand $i$ (respectively, strand $\reduce{i+1}$), in order from the start of the strand to its end, alternate between lying in $\overline{S_i}$ and not lying in $\overline{S_i}$.

The first arrow in $\overline{\Qi}$ crossing strand $i$ lies in $\overline{S_i}$. Similarly the first arrow in $\overline{\Qi}$ crossing strand $i+1$ lies in $\overline{S_i}$.  The first arrows in $\Qi$ crossing strand $i$ and strand $i+1$, respectively,  do not lie in $S_i$.

The arrow crossing $\PPi$ lies in $S_i$.
\end{lem}

\begin{proof}
Since any two consecutive arrows
crossing strand $i$ lie in the same
cycle, it follows from Lemma~\ref{l:crossingstrandi} and
Lemma~\ref{l:orientedcycle} that the arrows of $\Qi$ (or $\overline{\Qi}$) crossing strand $i$ alternate between lying in
$S_i$ and not in $S_i$. By Lemma~\ref{l:boundarypaths} the boundary path around the white vertex $\boundary_i$ does not contain an arrow in $S_i$. It follows that the first arrow crossing strand $i$ in $\Qi$ does not lie in $S_i$. The alternating property implies that the first arrow crossing
strand $i$ in $\overline{\Qi}$ does lie in $\overline{S_i}$. A similar argument applies to the case of strand $\reduce{i+1}$.

Since $\PPi$ is the unique crossing point
of strands $i$ and $i+1$ (Lemma~\ref{l:uniquecrossing}), it must be the 
case that strand $i$ crosses the arrow 
crossing $\PPi$ from right to left, 
while strand $\reduce{i+1}$ crosses
it from left to right (looking along the 
arrow). Hence the weight of this arrow is
$\{i\}$ and the second part follows.
\end{proof}

Let $\gamma_i$ be the path in $D_i$ which proceeds along strand $i$ from the beginning up until the crossing point $\PPi$ with strand $\reduce{i+1}$, and then carries on along the reverse of strand $\reduce{i+1}$, to the start of that strand. 
See Figure~\ref{fig:Siexample} for an example.
Let $\overline{\gamma_i}$ be the analogously defined path in $\overline{D_i}$.

\begin{cor} \label{c:gammaalternating}
The arrows in $\overline{\Qi}$ crossing $\overline{\gamma_i}$ alternate between lying in $\overline{S_i}$ and not in $\overline{S_i}$, starting and ending with the former case. The arrows in $Q_i$ crossing $\gamma_i$ alternate between lying in $S_i$ and not in $S_i$, starting and ending with the latter case.
\end{cor}

We define a new set of arrows $\Sigma_i$ in $\Qi$ as follows. We will see that $\Sigma_i$ is a perfect cut on $\overline{\Qi}$ and on $\Qi$.

\begin{defn} \label{d:sigmai}
Let $\Sigma_i$ be the set of arrows $\alpha$ in $\Qi$ which satisfy one of the following:
\begin{enumerate}[(a)]
\item $\alpha$ does not cross $\gamma_i$ and $\alpha$ lies in $S_i$, or
\item $\alpha$ crosses $\gamma_i$ and $\alpha$ does not lie in $S_i$.
\end{enumerate}

Thus $\Sigma_i$ is obtained from $S_i$ by toggling membership
for arrows crossing $\gamma_i$. We refer to this operation as the
swap $\rho$.
\end{defn}

See Figure~\ref{fig:sigmai} for an example of the set $\Sigma_i$.

\begin{figure}
\psfragscanon
\psfrag{v1}{$\boundary_1$}
\psfrag{v2}{$\boundary_2$}
\psfrag{v3}{$\boundary_3$}
\psfrag{v4}{$\boundary_4$}
\psfrag{v5}{$\boundary_5$}
\psfrag{v6}{$\boundary_6$}
\psfrag{1s}{\pscirclebox{$1$}}
\psfrag{1e}{\psframebox{$1$}}
\psfrag{2s}{\pscirclebox{$2$}}
\psfrag{2e}{\psframebox{$2$}}
\psfrag{3s}{\pscirclebox{$3$}}
\psfrag{3e}{\psframebox{$3$}}
\psfrag{4s}{\pscirclebox{$4$}}
\psfrag{4e}{\psframebox{$4$}}
\psfrag{5s}{\pscirclebox{$5$}}
\psfrag{5e}{\psframebox{$5$}}
\psfrag{6s}{\pscirclebox{$6$}}
\psfrag{6e}{\psframebox{$6$}}
\psfrag{p}{$\emptyset$}
\psfrag{P3}{$P_3$}
\includegraphics[width=10cm]{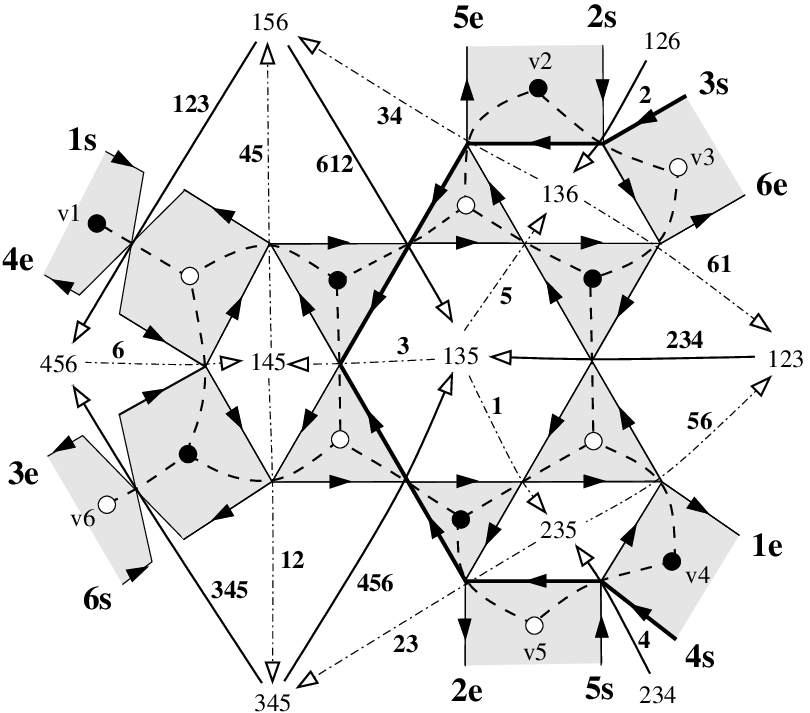}
\caption{The arrows in $\Sigma_3$ (see Definition~\ref{d:sigmai}), in the quiver $Q_3$ from Figure~\ref{fig:quiverexample}, as unbroken arrows.}
\label{fig:sigmai}
\end{figure}

\begin{prop} \label{p:Ximatching}
The set $\Sigma_i$ of arrows is a perfect cut of $Q_i$.
\end{prop}

\begin{proof}
Note first that $\gamma_i$ does not self-intersect, since the strands themselves do not
self-intersect, and strands $i$ and $\reduce{i+1}$ have a unique crossing point by
Lemma~\ref{l:uniquecrossing}.
Consider a minimal oriented cycle in $\Qi$ and the part of the Postnikov diagram lying in the interior of the cycle.
If $\gamma_i$ crosses this cycle, then the part of $\gamma_i$ in the interior of the
cycle is a union of arcs joining mid-points of arrows on the cycle. 
One end-point of each such arc lies on an arrow in $S_i$ by
Corollary~\ref{c:gammaalternating}. However, since each minimal oriented cycle contains exactly one arrow in
$S_i$, there must be only one such arc. Then the swap $\rho$ has
the effect of swapping which of the two arrows at the ends of this arc lies in the matching. This does not change the property
that the oriented cycle contains exactly one  arrow chosen by the matching. Therefore $\Sigma_i$ also contains precisely one arrow from each minimal oriented cycle in $\Qi$.

A similar argument applies to a path around a boundary vertex $\boundary_j$ for
$j\not=i,\reduce{i+1}$, since in this case $\gamma_i$ does not start or end on the
boundary side of such a path. We see that such boundary paths contain
exactly one arrow in $\Sigma_i$, possibly swapped from the one that was contained in $S_i$.

However, the two boundary paths around the vertices where strands $i$ (respectively, $\reduce{i+1}$) start
have the property that $\gamma_i$ starts (respectively, finishes) inside the boundary path around
the corresponding white (respectively, black) vertex. Furthermore, the first arrow that
$\gamma_i$ crosses does not lie in $S_i$, by Corollary~\ref{c:gammaalternating}, so it
does lie in $\Sigma_i$.
Similarly, the last arrow that $\gamma_i$ crosses does not lie in
$S_i$, by Corollary~\ref{c:gammaalternating}, so it does lie in $\Sigma_i$.

We need to check that
$\gamma_i$ crosses the
boundary path around $\boundary_{\reduce{i+1}}$ only once, at its last
arrow and that it crosses the boundary path around $\boundary_i$ only once, at its first arrow.

If $\gamma_i$ were to cross the boundary path around
the vertex $\boundary_{\reduce{i+1}}$ a second time, it would be entering the boundary region at the crossing point, describing an arc, and crossing back out at the next arrow (since it does not
end at $b_{\reduce{i+1}}$). By Lemma~\ref{l:strandalternating}, one of these arrows would be contained in $S_i$, giving a contradiction to Lemma~\ref{l:boundaryweight}.
An analogous argument shows that $\gamma_i$  crosses the boundary path around the vertex $\boundary_{\reduce{i}}$ only once. 



It follows that each of these boundary paths contains exactly 
one element of $\Sigma_i$. We have shown that $\Sigma_i$ is a 
perfect cut of $\Qi$ as required.
\end{proof}

\begin{cor}\label{c:sigmaicut}
The set $\Sigma_i$ of arrows is a perfect
cut of $\overline{\Qi}$.
\end{cor}

\begin{proof}
The quiver $\overline{\Qi}$ can be obtained
from $\Qi$ by adding the boundary arrows
between $L_j$ and $L_{j+1}$ for all $j$
(as in Figure~\ref{f:boundaryneighbourfaces}). This completes each boundary path in $\Qi$ to a minimal cycle in $\overline{\Qi}$. Each such minimal cycle
must contain a single arrow of $\Sigma_i$
by Proposition~\ref{p:Ximatching}, and
the result follows.
\end{proof}

\begin{rem} \label{r:Ximatchingk1}
If $k=1$, there is a unique perfect
matching (which we also denote by $M_i$) on $G_i$ (where $G_i$ is as defined at the start of Section~\ref{s:anycluster}). It contains the unique
edge incident with $\boundary_j$ for each $j$. A
similar description holds for $k=n-1$.
\end{rem}


\begin{defn} \label{d:perfectmatching}
Let $\Mi$ be the set of edges in $G_i$ such
that $\Sigma(\Mi)=\Sigma_i$. By Proposition~\ref{p:Ximatching} and
Remark~\ref{rem:cyclematchings}, $\Mi$
is a perfect matching on $G_i$.
\end{defn}

For an example of the perfect matching
$\Mi$, in the case of the diagram $D_3$ in
Figure~\ref{fig:quiverexample}, see Figure~\ref{fig:miexample}.

\begin{figure}
\psfragscanon
\psfrag{v1}{$\boundary_1$}
\psfrag{v2}{$\boundary_2$}
\psfrag{v3}{$\boundary_3$}
\psfrag{v4}{$\boundary_4$}
\psfrag{v5}{$\boundary_5$}
\psfrag{v6}{$\boundary_6$}
\psfrag{1s}{\pscirclebox{$\scriptstyle 1$}}
\psfrag{1e}{\psframebox{$\scriptstyle 1$}}
\psfrag{2s}{\pscirclebox{$\scriptstyle 2$}}
\psfrag{2e}{\psframebox{$\scriptstyle 2$}}
\psfrag{3s}{\pscirclebox{$\scriptstyle 3$}}
\psfrag{3e}{\psframebox{$\scriptstyle 3$}}
\psfrag{4s}{\pscirclebox{$\scriptstyle 4$}}
\psfrag{4e}{\psframebox{$\scriptstyle 4$}}
\psfrag{5s}{\pscirclebox{$\scriptstyle 5$}}
\psfrag{5e}{\psframebox{$\scriptstyle 5$}}
\psfrag{6s}{\pscirclebox{$\scriptstyle 6$}}
\psfrag{6e}{\psframebox{$\scriptstyle 6$}}
\psfrag{p}{$\emptyset$}
\includegraphics[width=16cm]{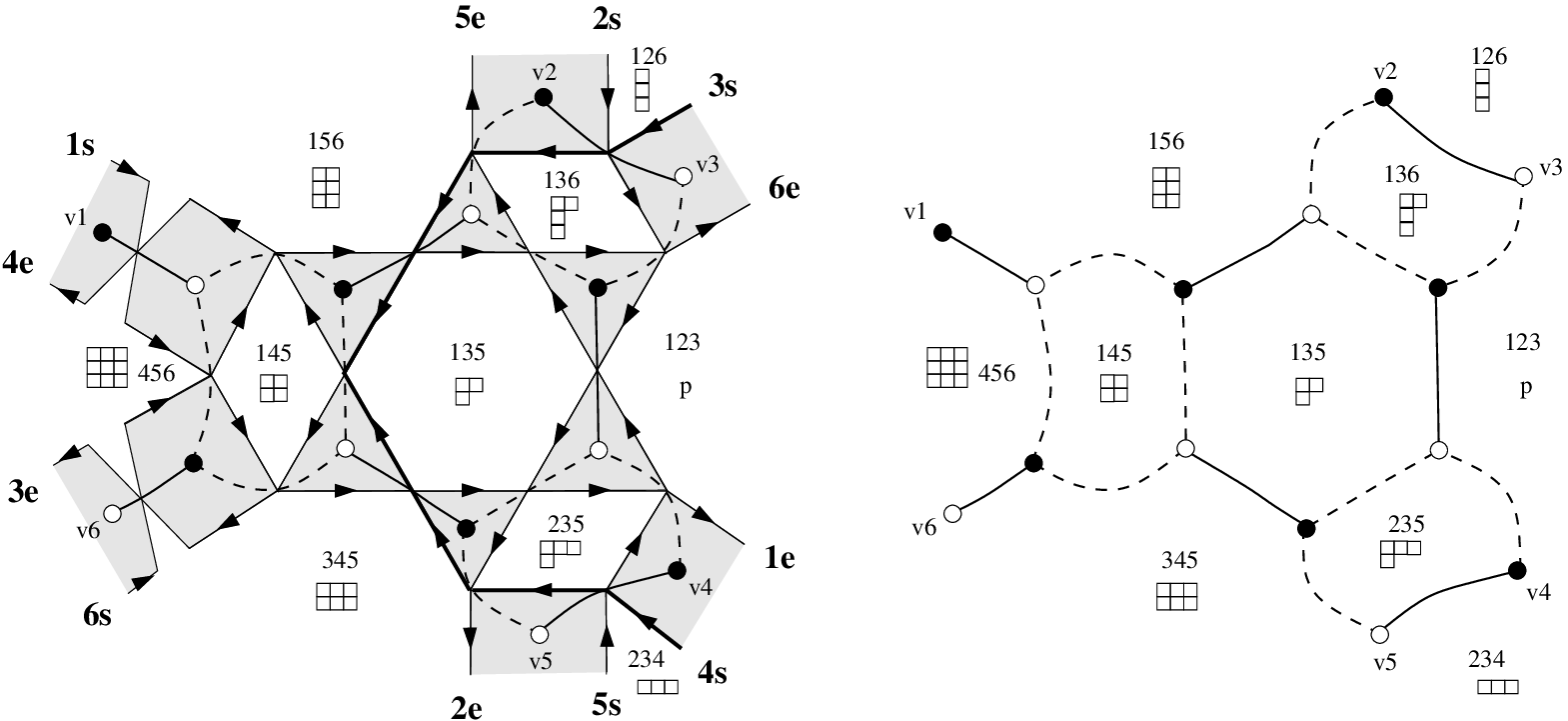}
\caption{The perfect matching $M_3$
on the graph $G_3$ in the case of
the example in Figure~\ref{fig:quiverexample}. Edges in
$G_3$ are drawn as full edges if they lie
in $M_3$ and as dashed edges otherwise.
The path $\gamma_3$ is drawn with thick
lines.}
\label{fig:miexample}
\end{figure}

\begin{figure}
\psfragscanon
\psfrag{p}{$\emptyset$}
\includegraphics[width=6.7cm]{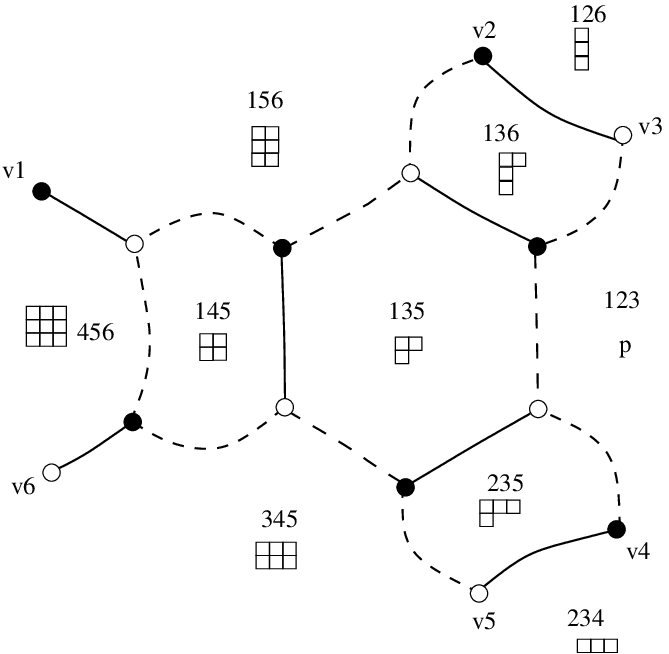}
\caption{The perfect matching on $G_3$ obtained from the perfect matching
shown in Figure~\ref{fig:miexample} by performing a face flip on the hexagonal
face. Edges in $G_3$ are drawn as full edges if they lie
in the perfect matching and as dashed edges otherwise.}
\label{fig:flippedmatching}
\end{figure}

\section{How to obtain all perfect matchings from $\Mi$}
\label{s:matchingproperties}
If $k=1$ or $n-1$, $\Mi$ is the unique perfect matching on $G_i$,
so we assume in this section that $k\not=1,n-1$. Our main aim is to show
that every other perfect matching on $G_i$
can be obtained from $M_i$ by performing
a sequence of face flips, in the sense of Definition~\ref{d:faceflip} below, on interior faces.
We will also show that the sequences
of face flips can be constructed in such a way that a distinguished face $\Fi$
of $G_i$ always appears at the beginning of the sequence and then never appears again.

\begin{defn} \label{d:faceflip}
Let $M$ be a perfect matching on a plane
graph $G$. An \emph{$M$-flippable face}
of $G$ is a (possibly unbounded) face for which the edges on
the boundary alternate between lying in $M$ and not lying in $M$.
If $F$ is an $M$-flippable face, then we
can construct a new perfect matching out of
$M$ by reversing membership of $M$ for
those edges along the boundary of $F$.
We call this operation a \emph{face flip}.
\end{defn}

Recall that the interior faces of $G_i$ correspond bijectively to the internal
alternating regions of the Postnikov diagram $D_i$, so are labelled by the non-boundary vertices of $\Qi$.

A plane bipartite graph is said to be \emph{factorizable} if it has at least one perfect matching. An edge is said to be \emph{allowed} if it appears in some perfect matching.

For example, the graph $G_3$ in Figure~\ref{fig:miexample} has a perfect
matching, $M_3$, as shown. The face labelled
\raisebox{\depth}{$\ydiagram{2,1}$} is $M_3$-flippable in $G_3$. 
Performing a face flip on this face produces a new perfect matching $M$, 
shown in Figure~\ref{fig:flippedmatching}. In $M$, the faces labelled
\raisebox{\depth}{$\ydiagram{3,1}$} and
\raisebox{\depth}{$\ydiagram{2,1,1}$} are both
$M$-flippable (as well as \raisebox{\depth}{$\ydiagram{2,1}$}). Flipping at either of these two faces in $M$ or both gives
three new perfect matchings on $G_3$. In fact, these five perfect matchings are all of the perfect matchings on $G_3$. We shall see later that, in general, all perfect matchings on $G_i$ can be obtained by performing sequences of face flips on $M_i$ (see Theorem~\ref{t:startwithQi}).
We also see that the allowed edges in the graph $G_3$ shown in Figure~\ref{fig:miexample} are as shown in Figure~\ref{fig:elementary}:
these are the edges appearing in the five perfect matchings listed above.

By definition, a connected factorizable plane bipartite graph is \emph{elementary} if and only if every edge is allowed~\cite[\S 4]{Lovasz}. We recall the following.

\begin{thm} \cite[Thm. 2]{propp},~\cite[Thm. 3.3]{zz}
\label{t:ztransformationconnected}
Let $G$ be a connected elementary plane bipartite graph.
Then, given any two perfect matchings
$M,M'$ of $G$, there is a sequence of face flips taking $M$ to $M'$.
\end{thm}

Note that to apply the result in~\cite[Thm. 2]{propp} to obtain Theorem~\ref{t:ztransformationconnected}, we need the fact that every edge in $G$ appears in some perfect matchings but not others. But this holds for any plane bipartite elementary graph,
e.g. by~\cite[Thm. 2.4]{zz}, which states that for any face of a plane bipartite
elementary graph $G$, there is a perfect matching in which that face is flippable.

In general, $G_i$ is not elementary, so in order to apply 
Theorem~\ref{t:ztransformationconnected} to $G_i$ we need to study the elementary components of $G_i$. The
\emph{elementary components} of a plane
bipartite graph $G$ are the connected components of the graph obtained by removing all disallowed edges from $G$
(see the definition before Lemma 2 in~\cite{Fournier:PerfectMatchings}).
The elementary components of the graph $G_3$ shown in Figure~\ref{fig:miexample} are shown in Figure~\ref{fig:elementary}.
We show that the elementary components
of $G_i$ are all \emph{elementary blocks}
(see~\cite[\S1]{lsz}), i.e.\ each interior
face of the component is also a face
of $G_i$. In fact, either every
elementary component of $G_i$ is a single
edge, or $G_i$ has a unique elementary
component which is not a single edge
(but is an elementary block).

To compute the elementary components of $G_i$ we need to find the allowed edges of $G_i$. This can be done using a result which we now recall.
Suppose that $G$ is a connected
factorizable plane bipartite graph with vertex bipartition $W\sqcup B$.
Let $G^*$ denote the dual graph of $G$,
oriented in such a way that the boundary
of a face in $G^*$ corresponding to a
black (respectively, white) vertex of
$G$ is oriented clockwise (respectively,
anticlockwise).
Let $M$ denote a perfect matching of $G$.
We denote by $G^*_M$ the quiver obtained by removing the edges in $G^*$
dual to the edges of $M$. Then we have
the following.

\begin{prop} \cite[Lemma 5]{Fournier:PerfectMatchings} \label{p:circuitcondition}
Suppose that G is a connected plane factorizable bipartite graph.
Then an edge of $G$ not in $M$ is allowed if and only if the dual edge in $G^*_M$ does not belong to a cycle in $G^*_M$.
\end{prop}

Note that it is clear that all edges in $M$ are allowed. Thus Proposition~\ref{p:circuitcondition}
states that the disallowed edges in
$G$ are those dual to cycles in $G_M^*$.
Thus in order to apply Proposition~\ref{p:circuitcondition} we need to find out which edges of $G_M^*$ lie in cycles.
We do this by first studying $\Qi$,
noting that $G_i^*$ can be obtained from
$\Qi$ by identifying all of the boundary
vertices $L_j$ (recall that $L_j=\{j-k+1,\ldots ,j\}$).

\begin{defn} \label{d:Giin}
Let $\Qil$ denote the full subquiver of
$\Qi$ whose vertices are those which are to the left of strand $i$ and to the right of
strand $i+1$, excluding $L_i$.
Let $\Qir$ denote the full subquiver of $\Qi$ on the remaining vertices of $\Qi$. In particular, $\Qir$ contains $L_i$.

Let $\Giin$ be the subgraph of $G_i$ whose edges are those which either cross an arrow between two vertices in $\Qil$ or cross an arrow between a vertex in $\Qil$ and a vertex in $\Qir$.

Figure~\ref{fig:elementary} shows $\Qil$ and $\Giin$ (for $i=3$) in our running example.
\end{defn}

\begin{figure}
\psfragscanon
\psfrag{v1}{$\boundary_1$}
\psfrag{v2}{$\boundary_2$}
\psfrag{v3}{$\boundary_3$}
\psfrag{v4}{$\boundary_4$}
\psfrag{v5}{$\boundary_5$}
\psfrag{v6}{$\boundary_6$}
\psfrag{1s}{\pscirclebox{$1$}}
\psfrag{1e}{\psframebox{$1$}}
\psfrag{2s}{\pscirclebox{$2$}}
\psfrag{2e}{\psframebox{$2$}}
\psfrag{3s}{\pscirclebox{$3$}}
\psfrag{3e}{\psframebox{$3$}}
\psfrag{4s}{\pscirclebox{$4$}}
\psfrag{4e}{\psframebox{$4$}}
\psfrag{5s}{\pscirclebox{$5$}}
\psfrag{5e}{\psframebox{$5$}}
\psfrag{6s}{\pscirclebox{$6$}}
\psfrag{6e}{\psframebox{$6$}}
\psfrag{p}{$\emptyset$}
\psfrag{P3}{$P_3$}
\includegraphics[width=10cm]{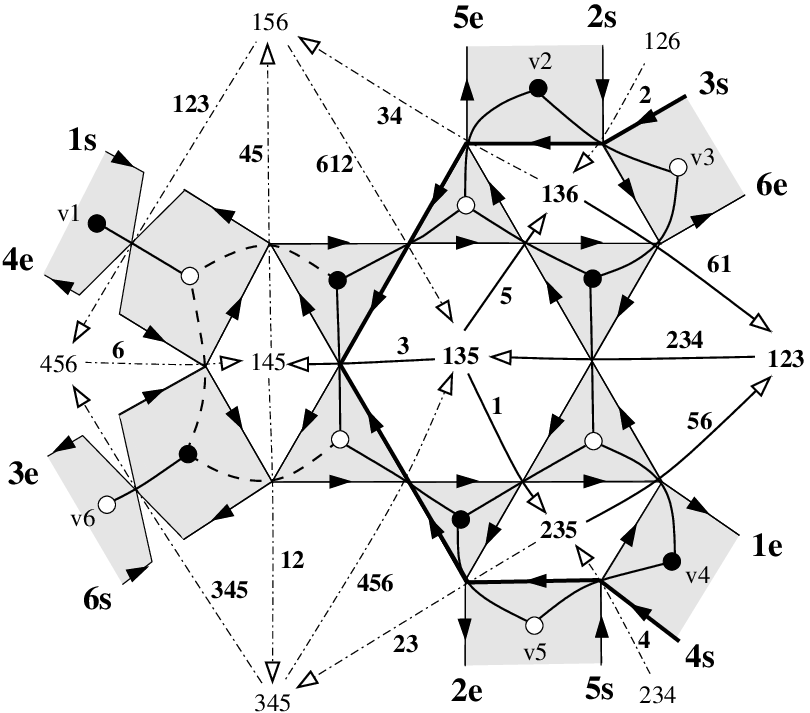}
\caption{The quiver $Q_3$ from Figure~\ref{fig:quiverexample}. The arrows in $Q_3^{\text{in}}$ (see Definition~\ref{d:Giin}) are shown as unbroken arrows (its vertices are 123,135,136, 235). As usual, the path $\gamma_3$ is shown as a thickened line. The allowed edges in $G_i$
are shown as unbroken edges. We see three elementary components: $G_i^{\text{in}}$ and two singleton edges on the left hand side incident with $v_1$ and $v_6$.}
\label{fig:elementary}
\end{figure}

We shall show that $\Giin$ forms an elementary component of $G_i$, and 
that all other elementary components of $G_i$
consist of single edges crossing arrows between vertices of $\Qir$ which 
lie in $\Sigma_i$.
We start with the following.

\begin{lem} \label{l:boundaryvertex}
Consider the arrows in $\Qi$ incident with the vertex $L_i$, in order anticlockwise from the boundary. Note that these arrows alternate between starting at $L_i$ and
ending at $L_i$, with the first and last
arrows ending at $L_i$. Arrows starting
at $L_i$ lie in $S_i$, while arrows ending
at $L_i$ do not lie in $S_i$.
\end{lem}

\begin{proof}
Since $\boundary_i$ is white and $\boundary_{i+1}$ is
black, the first and last arrows in the above ordering must be oriented
towards $L_i$.
The fact that these arrows do not lie in $S_i$ follows from Lemma~\ref{l:boundarypaths}.
An arrow oriented away from $L_i$ must have 
target labelled $(L_i\setminus \{j\})\cup \{l\}$
for some $j\in L_i=[\reduce{i-k+1},i]$ and $l\not\in L_i$, and thus has 
label $[j,\reduce{l-1}]$ containing $i$. Thus all arrows starting at 
$L_i$ lie in $S_i$. Similarly all arrows ending at $L_i$ do 
not lie in $S_i$, as required.
\end{proof}

\begin{lem} \label{l:boundaryarrow}
An arrow in $\Qi$ which has one end-point in $\Qil$ and one end-point in $\Qir$
is either
\begin{enumerate}
\item[(a)] in the perfect cut $\Sigma_i$, and oriented towards the end-point in
$\Qil$, or
\item[(b)] not in the perfect cut $\Sigma_i$, and oriented towards the end-point in $\Qir$.
\end{enumerate}
\end{lem}

\begin{proof}
Let $\alpha$ be an arrow as in the statement of the lemma.
Suppose first that $\alpha$ crosses $\gamma_i$. By the assumption on
$\alpha$, neither end-point of $\alpha$ is $L_i$ (otherwise both of its
end-points would lie in $\Qir$). In this case, the result follows from the
definition of $\Sigma_i$ and Corollary~\ref{c:gammaalternating}, noting that
the first and last arrows crossing $\gamma_i$ are oriented towards $\Qil$.

Next, suppose that one of the end-points of $\alpha$ is $L_i$.
Then the other end-point lies in $\Qil$, so $\alpha$ does not cross $\gamma_i$.
The result in this case follows from Lemma~\ref{l:boundaryvertex} and the
definition of $\Sigma_i$,
noting that $\alpha$ lies in $S_i$ if
and only if it lies in $\Sigma_i$, as it does not cross $\gamma_i$.
\end{proof}

Note that in the case where $i,\reduce{i+1}$ cross on the boundary of the face
labelled by $L_i$, the subquiver $\Qil$ is empty; see Figure~\ref{f:Lispecial}.
In this case, Lemma~\ref{l:boundaryarrow} does not say anything.

We recall the following result from~\cite[Prop.\ 4.9]{bkm}, which gives
information concerning the weights
of arrows incident with an internal vertex
of $\Qi$ which we shall use several times.

\begin{lem} \cite[Prop.\ 4.9]{bkm} \label{l:winding}
Let $I$ be an internal vertex of $\Qi$
of valency $2r$.
Then the weights of the arrows incident with
it, taken in order anticlockwise around $I$, 
follow on from each other and wrap around 
$[1,n]$ exactly $r-1$ times.
\end{lem}

For example, the weights of the arrows incident with the vertex $135$ in
Figure~\ref{fig:Siexample}, which has valency $3$, are, taken in order anticlockwise around the vertex, $1$, $234$, $5$, $612$, $3$, $456$,
which wrap around $[1,6]$ twice.

\begin{defn} \label{d:Fi}
Let $\Fi$ be the alternating face adjacent to the crossing point $\PPi$ of
strands $i$ and $\reduce{i+1}$ which is to the left of strand $i$ and to the right
of strand $i+1$. Let $I_i$
be the $k$-subset labelling this face.
Let $\Fip$ be the alternating face 
adjacent to $\PPi$ on the other side of
$\gamma_i$, i.e.\ to the right of strand
$i$ and to the left of strand $i+1$.
Let $I'_i$ be the $k$-subset labelling this face. Note that if $i,i+1$ cross
on the boundary of $D_i$ then $I_i=L_i$
and $I'_i=\widehat{L}_i$.
See Figure~\ref{f:boundary} for a schematic
illustration. In Figure~\ref{fig:quiverexample}, we have $I_3=\{1,3,5\}$ and $I'_3=\{1,4,5\}$. The faces labelled by these subsets are $F_3$ and
$F'_3$ respectively.
\end{defn}

\begin{figure}
\psfragscanon
\psfrag{Li}{$L_i$}
\psfrag{Li+1}{$L_{i+1}$}
\psfrag{Li-1}{$L_{i-1}$}
\psfrag{Li+k}{$L_{i+k}$}
\psfrag{iS}{\pscirclebox{$\scriptstyle i$}}
\psfrag{i+1S}{\pscirclebox{$\scriptstyle i+1$}}
\psfrag{Giin}{$\Giin$}
\psfrag{Fi}{$\Fi$}
\psfrag{Fip}{$\Fip$}
\includegraphics[width=6cm]{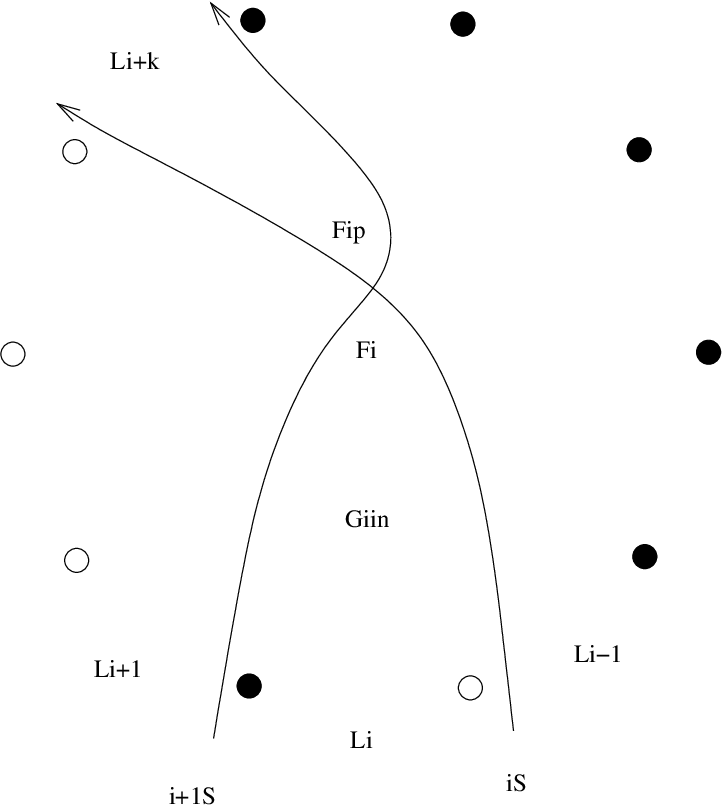}
\caption{The faces $\Fi$ and $\Fip$.}
\label{f:boundary}
\end{figure}

\begin{defn} \label{def:quiverinout}
We denote by $\Qilc$ (respectively,
$\Qirc$), the quiver $\Qil$ (respectively,
$\Qir$) with all arrows in $\Sigma_i$
removed. For the example in Figure~\ref{fig:elementary}, we show the
subquivers $Q_3^{\text{in}}(\Sigma_3)$ and $Q_3^{\text{out}}(\Sigma_3)$ in Figure~\ref{fig:quiverinout} (recall that the arrows in $\Sigma_i$ are shown in Figure~\ref{fig:sigmai}).
\end{defn}

\begin{figure}
\psfragscanon
\psfrag{v1}{$\boundary_1$}
\psfrag{v2}{$\boundary_2$}
\psfrag{v3}{$\boundary_3$}
\psfrag{v4}{$\boundary_4$}
\psfrag{v5}{$\boundary_5$}
\psfrag{v6}{$\boundary_6$}
\psfrag{1s}{\pscirclebox{$1$}}
\psfrag{1e}{\psframebox{$1$}}
\psfrag{2s}{\pscirclebox{$2$}}
\psfrag{2e}{\psframebox{$2$}}
\psfrag{3s}{\pscirclebox{$3$}}
\psfrag{3e}{\psframebox{$3$}}
\psfrag{4s}{\pscirclebox{$4$}}
\psfrag{4e}{\psframebox{$4$}}
\psfrag{5s}{\pscirclebox{$5$}}
\psfrag{5e}{\psframebox{$5$}}
\psfrag{6s}{\pscirclebox{$6$}}
\psfrag{6e}{\psframebox{$6$}}
\psfrag{p}{$\emptyset$}
\psfrag{P3}{$P_3$}
\includegraphics[width=10cm]{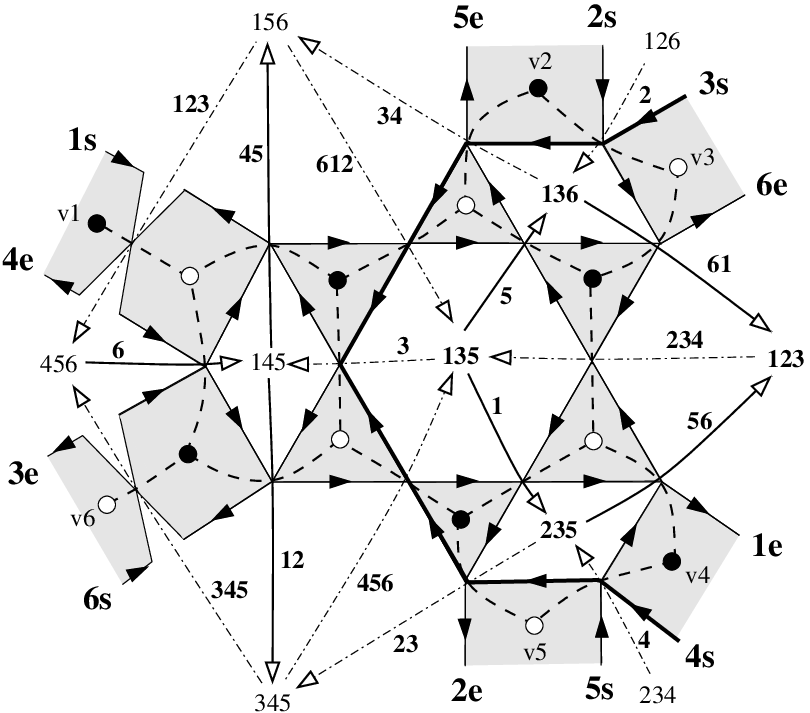}
\caption{The quiver $Q_3$ from Figure~\ref{fig:quiverexample}. The arrows in $Q_3^{\text{in}}(\Sigma_3)$ and $Q_3^{\text{out}}(\Sigma_3)$ (see Definition~\ref{def:quiverinout}) are shown as unbroken arrows.}
\label{fig:quiverinout}
\end{figure}

Our next step is to obtain more information about
paths and cycles in $\Qilc$ and $\Qirc$.
The proofs of Lemma~\ref{l:connected}(b) and Lemma~\ref{l:cycles} build
on discussions of R.~Marsh with K.~Baur and A.~King in an alternative
approach to the results in~\cite{bkm}.

\begin{lem} \label{l:connected}
\begin{enumerate}
\item[(a)] The quivers $\Qilc$ and $\Qirc$ are acyclic.
\item[(b)] Given any vertex $I$ of $\Qilc$, there is a path from
$I_i$ to $I$ in $\Qilc$.
\end{enumerate}
\end{lem}

\begin{proof}
(a) By Corollary~\ref{c:largercycle}, the weight of any cycle in $\overline{\Qi}$ is a multiset union of $[1,n]$. It follows that the weight of any cycle in $\Qi$ also has this property and, in particular, contains $i$. Therefore, any cycle in $\Qi$ contains
an arrow from $S_i$.
Since the arrows in $\Qil$ do not cross $\gamma_i$, the arrows
in $\Qil$ which lie in $\Sigma_i$ are exactly those which lie in $S_i$.
Removing these arrows from $\Qil$ breaks
up every cycle in $\Qil$.
It follows that $\Qilc$ is acylic. A similar argument applies to $\Qirc$.

(b) We will show that if $I$ is any vertex of $\Qilc$ not equal to $I_i$, then there
is an arrow in $\Qilc$ with target $I$.
Since $\Qilc$ is finite and acyclic,
repeated application of this argument
must produce a path from $I$ to $I_i$ in
$\Qilc$ as required. We find the required
arrow by showing that there is an arrow
in $\Qi$ with target $I$ whose weight does
not include $i-1,i$ or $i+1$. Such an arrow
does not lie in $S_i$, but also cannot
cross $\gamma_i$ by Lemma~\ref{l:crossingstrandi}. Hence it
also does not lie in $\Sigma_i$. Again
using the fact that the arrow does not cross $\gamma_i$, the source of the arrow must lie  in $\Qil$. Hence the arrow lies in
$\Qilc$ as required.

We consider an arbitrary vertex
$I$ of $\Qil$, of valency $2r$ in $\Qi$. We assume that $I\ne I_i$.
This implies that $I$ is not adjacent to the crossing point $\PPi$ of strands $i$ and $i+1$, and has no arrow  incident with it whose weight is $\{i\}$ (since the arrow
through $\PPi$ is the unique arrow with
weight $\{i\}$).

By Lemma~\ref{l:winding}, there are $r-1$ arrows, $\alpha_1,\ldots ,\alpha_{r-1}$ incident with $I$ whose weight contains $i-1$. For every arrow $\alpha_j$ we let $\beta_j$ refer to the arrow incident with $I$ which is adjacent to $\alpha_j$ in an anti-clockwise direction around $I$. We obtain a set
$\mathcal{A}=\{\alpha_j\}\cup\{\beta_j\}$
consisting of at most $2r-2$ arrows.
Clearly (again using Lemma~\ref{l:winding}), any arrow incident with $I$ whose weight contains $i$ must lie in $\mathcal A$. Moreover any arrow whose weight contains $i+1$ must also lie in $\mathcal A$, since otherwise there would be a $\beta_j$ with weight $\{i\}$, which contradicts our assumption that $I\ne I_i$. 

By its definition, $\mathcal{A}$ contains
at most $r-1$ arrows whose target is $I$.
Hence there is an arrow with target $I$ in $\Qi$ whose weight does not contain $i-1,i$ or $i+1$ and we are done.
\end{proof}

\begin{lem} \label{l:cycles}
Let $I$ be a non-boundary vertex in $\Qirc$. Then there is a path in $\Qirc$
from a boundary vertex to $I$ and a path from $I$ to a boundary vertex.
\end{lem}

\begin{proof}
Assume first that $I$ lies to the right of strand $i$. Let $2r$ be the valency of 
$I$ in $\Qi$. Consider the arrows incident with $I$ whose weight contains 
$i-1$, $i$ or both. By Lemma~\ref{l:winding}, cyclically ordering the arrows 
incident with $I$ anticlockwise around $I$, we see that such arrows occur either 
as singletons or in adjacent pairs, a total of $r-1$ singletons and
adjacent pairs. Since $I$ has valency $2r$ in $\Qi$, we see that there is 
always an arrow with source $I$ whose weight does not contain $i$ or $\reduce{i-1}$.
Similarly, there is always an arrow with target $I$ having this property.
These arrows do not intersect strand $i$
by Lemma~\ref{l:crossingstrandi}.

Since $I$ lies to the right of strand
$i$ and $\gamma_i$ is on or the left
of strand $i$, these arrows do not intersect
$\gamma_i$. Hence the other end-points of these arrows lie in $\Qir$.

A similar argument applies in the case where $I$ lies to the left of strand $\reduce{i+1}$.
Since any vertex in $\Qirc$ lies to the right of strand $i$ or to the left of
strand $\reduce{i+1}$, we see that repeating this argument gives the statement in the 
Lemma, noting that $\Qirc$ is finite and acyclic (by Lemma~\ref{l:connected}).
\end{proof}

We can now describe the allowed edges in $G_i$ and thus the elementary components.

\begin{lem} \label{l:allowededges}
An edge in $G_i$ which crosses
an arrow $\alpha$ in $\Qi$ is allowed
if and only if either
at least one endpoint of
$\alpha$ lies in $\Qilc$ or $\alpha$
lies in $\Sigma_i$.
\end{lem}

\begin{proof}
Since edges crossing arrows in $\Sigma_i$
are allowed, we are reduced to the
case of edges crossing arrows which do not
lie in $\Sigma_i$.
Recall that $\GMstar$ can be obtained
from $\Qic$ by identifying the boundary vertices. We denote the image of an arrow $\alpha$ in $\Qic$ under this procedure by $\overline{\alpha}$.

Note that $G_i$ is a connected plane bipartite graph, which is factorizable
by Proposition~\ref{p:Ximatching}.
By Proposition~\ref{p:circuitcondition} we 
need to show that, for any arrow $\alpha$
in $\Qic$, $\overline{\alpha}$ lies in a cycle in $\GMstar$ if and only if both of the endpoints of $\alpha$ lie in $\Qir$.

If both endpoints of $\alpha$ lie in
$\Qir$, then $\overline{\alpha}$
lies in a cycle in $\GMstar$ by Lemma~\ref{l:cycles}.
For the converse, note that
by Lemma~\ref{l:connected}(a),
$\Qilc$ is acyclic.
Also, by Lemma~\ref{l:boundaryarrow}, 
all arrows in $\Qic$ which have
one endpoint in $\Qilc$ and one in $\Qirc$
are oriented towards $\Qirc$.
It follows that if $\alpha$ has at least
one endpoint in $\Qil$ then $\overline{\alpha}$ does not lie in a
cycle in $\GMstar$.
%
\end{proof}

\begin{cor} \label{c:elementarycomponents}
The elementary components of $G_i$
are as follows:
\begin{itemize}
\item[(a)]
The full subgraph $\Giin$ of $G_i$ 
(see Definition~\ref{d:Giin}) is an
elementary component of $G_i$.
\item[(b)] Any single edge in $\Mi$
which has no end point in $\Giin$
is an elementary component of $G_i$.
\end{itemize}
\end{cor}

\begin{proof}
This follows from the description above of the allowed edges in $G_i$ and the definition of elementary components.
\end{proof}

Lemma~\ref{l:allowededges} and Corollary~\ref{c:elementarycomponents}
can both be verified in our running example in Figure~\ref{fig:elementary}.

In the case where $I_i$ is on the boundary, we see that the elementary components 
of $G_i$ are exactly the allowed edges
(i.e.\ the edges in $M_i$),
considered as subgraphs, and that
$M_i$ is the unique perfect matching on
$G_i$.

An elementary component of a plane graph which has the property that each interior
face of the component is also a face of the whole graph is called an
\emph{elementary block} (see~\cite[\S1]{lsz}). We need the following
important property of the elementary
components of $G_i$.

\begin{cor} \label{c:elementaryblock}
Each elementary component of $G_i$ is an elementary block of $G_i$.
\end{cor}

\begin{proof}
This is trivial for the elementary components of $G_i$ which consist of a single edge,
since they have no interior faces. So we consider the elementary component
$\Giin$.
By the construction of the Postnikov diagram
$D$, the region to the left of strand $i$
and to the right of strand $i+1$, together with the adjacent vertices in $G_i$, is a union of interior  
faces of $G_i$; these are exactly the interior faces of $\Giin$.
It follows that $\Giin$ is also an elementary block of $G_i$.
\end{proof}

This gives us the first key result.

\begin{prop} \label{p:biztransformationconnected}
The faces of $\Giin$ are faces of $G_i$,
and the set of perfect matchings of $G_i$ is connected under flips of faces of $\Giin$.
\end{prop}

\begin{proof}
The first statement follows from Corollary~\ref{c:elementaryblock}.
Let $M,M'$ be arbitrary perfect matchings on $G_i$. By Lemma~\ref{l:allowededges}, $M$ and $M'$ must coincide with
$\Mi$ on all edges of $G_i$ crossing arrows between vertices of $\Qir$.
By Corollary~\ref{c:elementarycomponents} and
Theorem~\ref{t:ztransformationconnected},
there is a sequence of face flips taking $M$ to
$M'$, noting that every face of the elementary graph $\Giin$
is also a face of $G_i$, by Corollary~\ref{c:elementaryblock}.
\end{proof}

Note that, by~\cite[Thm. 2.4]{zzy},
this implies that $G_i$ is \emph{weakly
elementary} (see e.g.~\cite[\S2]{zzy} for 
the definition).

%
%
%

In the remainder of this section, we will
show that, as in the example, every perfect matching on
$G_i$ can be obtained from $\Mi$ by
a sequence of face-flips in which the
first (and only the first) face is $\Fi$.
We first show that $\Fi$ is the unique
$M_i$-flippable face of $G_i$.

Given a vertex $I$ of $\Qi$, we say that
two arrows $\alpha,\beta$ incident with $I$
are \emph{adjacent} provided one follows
the other in the cyclic ordering around
$I$.

\begin{rem} \label{r:flippable}
Let $I$ be the label of an internal alternating face $F$ of $D_i$ and suppose that $I$ has valency $2r$ in $\Qi$.
Since it is not possible for two adjacent
arrows incident with $I$ to lie in
$\Sigma_i$ (as it is a perfect cut),
we have that $F$ is an $\Mi$-flippable face in $G_i$ if and only if the number of
arrows incident with $I$ lying in
$\Sigma_i$ is $r$.
\end{rem}

%

\begin{lem} \label{l:arrownumber}
Let $F$ be an internal face of $D_i$,
labelled with the $k$-subset $I$.
Suppose that $I$ has valency $2r$
in $\Qi$. Then the number of arrows in
$\Sigma_i$ incident with $I$ is
$r$ if $F=\Fi$, $r-2$ if $F=\Fip$, and is
$r-1$ otherwise.
\end{lem}

\begin{proof}
By Lemma~\ref{l:winding}, exactly $r-1$ of the arrows incident with the vertex $I$ lie in $S_i$.
Recall that $\Sigma_i$ is obtained from
$S_i$ by applying the swap $\rho$
(see Definition~\ref{d:sigmai}).
Since $I$ is an internal vertex of $\Qi$, the set of arrows incident with $I$
that are crossed by a fixed strand
consists of a number of pairs of adjacent arrows (see Figure~\ref{f:neighbourhood}).
This applies, in particular, to the strands $i$ and $i+1$ appearing in the definition of $\gamma_i$. The unique crossing point
$\PPi$ of these strands lies on the boundary
of the faces $\Fi$ of $\Fip$, and not on
the boundary of any other face.

So, if $I\not=I_i,I'_i$, then $\PPi$ is not on the boundary of $F$. It follows
that the set of arrows incident with $I$
which cross $\gamma_i$ consists of a
collection (possibly empty) of pairs of adjacent arrows, some pairs arising from the part of $\gamma_i$ along strand $i$ and some pairs arising from the part of $\gamma_i$ along (the reverse of) strand $i+1$.
The effect of the swap $\rho$ on the set
of arrows incident with $I$ is to
reverse membership in each such pair. Hence
the number of arrows incident with $I$ which
lie in $\Sigma_i$ is the same as the number
of such arrows lying in $S_i$, i.e.\ $r-1$.

The set of arrows incident with $I_i$ which cross $\gamma_i$ consists of the arrow $I_i\rightarrow I'_i$ and the two adjacent arrows incident with
$I_i$ together with a collection (possibly empty) of pairs of adjacent arrows incident with $I_i$.
The flip $\rho$ reverses membership in each of the pairs.
It also replaces the arrow $I_i\rightarrow I'_i$ with the two adjacent arrows
incident with $I_i$.
Hence exactly $r$ of the
arrows incident with $I_i$ lie in
$\Sigma_i$.

Similarly, the set of arrows incident with $I'_i$ which cross $\gamma_i$ consists of the single arrow $I_i\rightarrow I'_i$, together with a collection (possibly empty) of pairs of adjacent arrows incident with $I'_i$. The flip $\rho$ reverses membership in each of these pairs and
deletes the arrow $I_i\rightarrow I'_i$.
Hence exactly $r-2$ of the
arrows incident with $I'_i$ lie in
$\Sigma_i$.
\end{proof}

\begin{cor} \label{c:Miflippable}
Let $I$ be a vertex in $\Qil$
labelling an alternating face $F$. Then
$F$ is $\Mi$-flippable if and only if
$I=I_i$.
\end{cor}
\begin{proof}
Since $F$ must be an internal alternating face of $D_i$, this follows from
Remark~\ref{r:flippable} and Lemma~\ref{l:arrownumber}.
\end{proof}

Let $M$ be a perfect matching on a
plane bipartite graph $G$. Then a \emph{positive $M$-flippable face} is an
$M$-flippable face with the property that the
matched edges, when oriented from black
vertices towards white vertices, are
oriented in an anticlockwise direction
around the face.
Otherwise, an $M$-flippable face is
said to be \emph{negative}.
For example, in the perfect matching $M_3$ in Figure~\ref{fig:miexample},
the face $F_3$, labelled $\raisebox{\depth}{\ydiagram{2,1}}$, is an $M_3$-flippable face and is negative.

\emph{Twisting down} 
a face is the operation of flipping a positive $M$-flippable face. See~\cite[\S1]{propp}.
We recall the following, which follows
from~\cite[Prop.\ 1.11, Thm.\ 2]{propp}.

\begin{thm} \cite{propp} \label{t:distributive}
Let $G$ be a connected plane elementary factorizable bipartite graph and fix
a face $F$ of $G$.
Let $\M$ be the set of perfect matchings of $G$. Then the covering relation given by twisting down at a face other than $F$ makes $\M$ into a distributive lattice $\M_F$.
The unique minimum element of $\M_F$
is the unique perfect matching $M$ on $G$ which has no positive $M$-flippable face except for $F$.
\end{thm}

Note that, since $\Giin$ is an elementary component of $G_i$, the restriction $\Miin$
of $\Mi$ to $\Giin$ is a perfect matching
on $\Giin$ by~\cite[Lemma 2]{Fournier:PerfectMatchings}.
We denote by $\mathcal{M}_i^{\text{\rm in}}$ the set of perfect matchings of $\Giin$.

\begin{thm} \label{t:startwithQi}
Each perfect matching on $G_i$ can be obtained from $\Mi$ by a sequence of
flips of faces of $G_i$ which are faces
of $\Giin$, starting with the face $\Fi$ and never involving that face again.
\end{thm}

\begin{proof}
By Proposition~\ref{p:biztransformationconnected}, it is enough to prove
the result for the perfect matching
$\Miin$ on $\Giin$. 
By Corollary~\ref{c:Miflippable},
$\Fi$ is the unique internal $\Miin$-flippable face in $\Giin$. Since $\Fi$ is negative,
$\Giin$ has no positive internal $\Miin$-flippable faces.
But, by~\cite[Prop.\ 1.11]{propp}, there
must be at least one positive $\Miin$-flippable face in $\Giin$. Hence the boundary face of $\Giin$ is a positive $\Miin$-flippable face.

Taking $F$ to be the boundary face of
$\Giin$ in Theorem~\ref{t:distributive},
we see that $\Miin$ is the minimum
element of the lattice $(\mathcal{M}_i^{\text{\rm in}})_F$.
Hence, if $M$ is any perfect matching
on $\Giin$ not equal to $\Miin$,
there is a twisting down sequence from $M$ to $\Miin$ involving only internal faces.
Since $\Fi$ is the unique negative $\Miin$-flippable face in $\Giin$,
this twisting down sequence must involve $\Fi$ as its final flip. 
By~\cite[Cor.\ 4]{pretzel},
the number of times that any given face can be flipped in such a twisting down sequence
is at most $1$, so $\Fi$ cannot occur
in the twisting down
sequence at any other place than the end.
\end{proof}

Note that we can see that Theorem~\ref{t:startwithQi} holds in the case of
example $G_3$ in Figure~\ref{fig:miexample} from the description of the
perfect matchings given in the paragraph after Definition~\ref{d:faceflip}.

\begin{rem}
Theorem~\ref{t:startwithQi} holds trivially for $k=1,n-1$, since there is a unique perfect matching in this case (see Remark~\ref{r:Ximatchingk1}).
\end{rem}

\section{The matching monomial associated to $\Mi$}
\label{s:matchingevaluation}
Our main aim in this section is to compute the matching monomial associated
to $\Mi$, i.e.\ its contribution towards the
matching polynomial of $G_i$.
This, combined with Theorem~\ref{t:startwithQi}, will be used in Section~\ref{s:actionvectorfield2} in
order to compute the action of $X_{\lambda}$ on $W_q$.
We assume in this section that $k\not=1,n-1$. Recall the definition (Definition~\ref{d:weighting}) of the
weighting on $G_i$. We make the following
useful definition.

\begin{defn}
Let $D$ be a Postnikov diagram with closure
$\overline{D}$.
Let $I$ be a vertex of $Q(\overline{D})$.
For any minimal cycle $$I=I_0\rightarrow I_1\rightarrow I_2\rightarrow 
\cdots \rightarrow I_{s-1}\cdots \rightarrow I_s=I$$
in $Q(\overline{D})$ containing $I$, we call the path
$$I_1\rightarrow I_2\rightarrow \cdots \rightarrow I_{s-1}$$
a \emph{peripheral path} of $I$.
Note that the underlying unoriented graph of the union of the peripheral
paths of $I$ forms a circle around $I$ if $I$ is internal, or an arc of a circle if $I$ is external. Each peripheral path is either oriented clockwise or anticlockwise
in this circle (or arc). We call any arrow in a peripheral path of
$I$ a \emph{peripheral arrow} of $I$. The \emph{neighbourhood} of $I$
is the union of the cycles (and their interiors) containing $I$ (i.e.\ the region bounded by the peripheral paths
of $I$ and the boundary arrows incident with $I$ if any).
See Figure~\ref{f:neighbourhood} for an example of the neighbourhood
of an internal vertex.

We also say that an arrow in $Q(D)$ is peripheral, etc. as above, if it is peripheral when regarded as an arrow in $Q(\overline{D})$.
\end{defn}

The following lemma is a straightforward reformulation of the definitions.
\begin{lem} \label{l:monomialexponent}
The exponent of $p_I$ in the matching monomial $w_{\Mi}$ is equal to the number of clockwise peripheral arrows of $I$ in $\Qi$ which lie in $\Sigma_i$.
\end{lem}

Note that, by Proposition~\ref{p:Ximatching}, each clockwise peripheral path of $I$ contains at most one arrow in $\Sigma_i$.

%
%


\begin{lem}\label{l:sigma'}
Let $\Sigma'_i$ be the set of arrows
in $\overline{\Qi}$ obtained from
$\overline{S_i}$ by reversing membership for all
arrows crossing $\overline{\gamma_i}$.
Then $\Sigma'_i=\Sigma_i$.
\end{lem}
\begin{proof}
The intersection of $\Sigma'_i$ with $Q_i$ agrees with $\Sigma_i$, because both subsets of $Q_i$ are constructed in the same way. It only remains to check that $\Sigma'_i$ contains no arrows outside of $Q_i$. The arrows in $\overline{Q_i}\setminus Q_i$ crossing $\overline{\gamma_i}$ are in $\overline{ S_i}$ by Corollary~\ref{c:gammaalternating} and therefore are not in $\Sigma_i'$ by construction. The remaining arrows in $\overline{Q_i}\setminus Q_i$ are also not in $\Sigma_i'$.
Indeed they are not in $\overline{S_i}$
by Remark~\ref{r:boundaryarrows}.
And $\Sigma_i'$ agrees with $\overline{S_i}$ away from $\overline{\gamma_i}$.
\end{proof}

We denote the flip (defined
on subsets of the set of arrows of $\overline{\Qi}$) which reverses membership 
for arrows crossing $\overline{\gamma_i}$
by $\overline{\rho}$.

\begin{rem} \label{r:peripheral}
Fix a vertex $I$ of $\Qi$. Then, since
every arrow in $\Sigma_i$ is contained in $\Qi,$ the number of peripheral paths of $I$ in $\Qi$ containing an arrow in $\Sigma_i$ is the same as
the number of peripheral paths of $I$ in
$\overline{\Qi}$ containing an arrow in $\Sigma_i$.
\end{rem}


\begin{lem} \label{l:lightshining}
Let $I$ be the label of an internal alternating face $F$ of $D_i$. The
exponent $e_I$ of $p_I$ in the matching monomial $w_{\Mi}$ is given by:
\begin{equation}e_I=\begin{cases} 0 & \text{ if $I=I_i$ }; \\
2 & \text{ if $I=I'_i$ }; \\
1 & \text{ otherwise.}
\end{cases}
\end{equation}
\end{lem}

\begin{proof}
By Lemma~\ref{l:arrownumber}, the number of
arrows incident with $I$ in $\Qi$ (and
hence also in $\overline{\Qi}$, by Remark~\ref{r:peripheral}) which lie in
$\Sigma_i$ is equal to $r$ if $F=\Fi$,
$r-2$ if $F=\Fip$, and $r-1$ otherwise.
By Corollary~\ref{c:sigmaicut},
$\Sigma_i$ is a perfect cut
of $\overline{\Qi}$. In the case where $F=\Fi$ the arrows incident with $I_i$ alternate between lying in $\Sigma_i$ and not. Every minimal cycle passing through $I_i$ therefore has an arrow incident with $I_i$ as its unique arrow in $\Sigma_i$. It  follows that
none of the clockwise peripheral paths
of $I_i$ contain an arrow in
$\Sigma_i$. This shows that  $e_I=0$ if $I=I_i$. The other cases follow similarly. 
\end{proof}

\begin{figure}
\psfragscanon
\psfrag{I}{$I$}
\includegraphics[width=8cm]{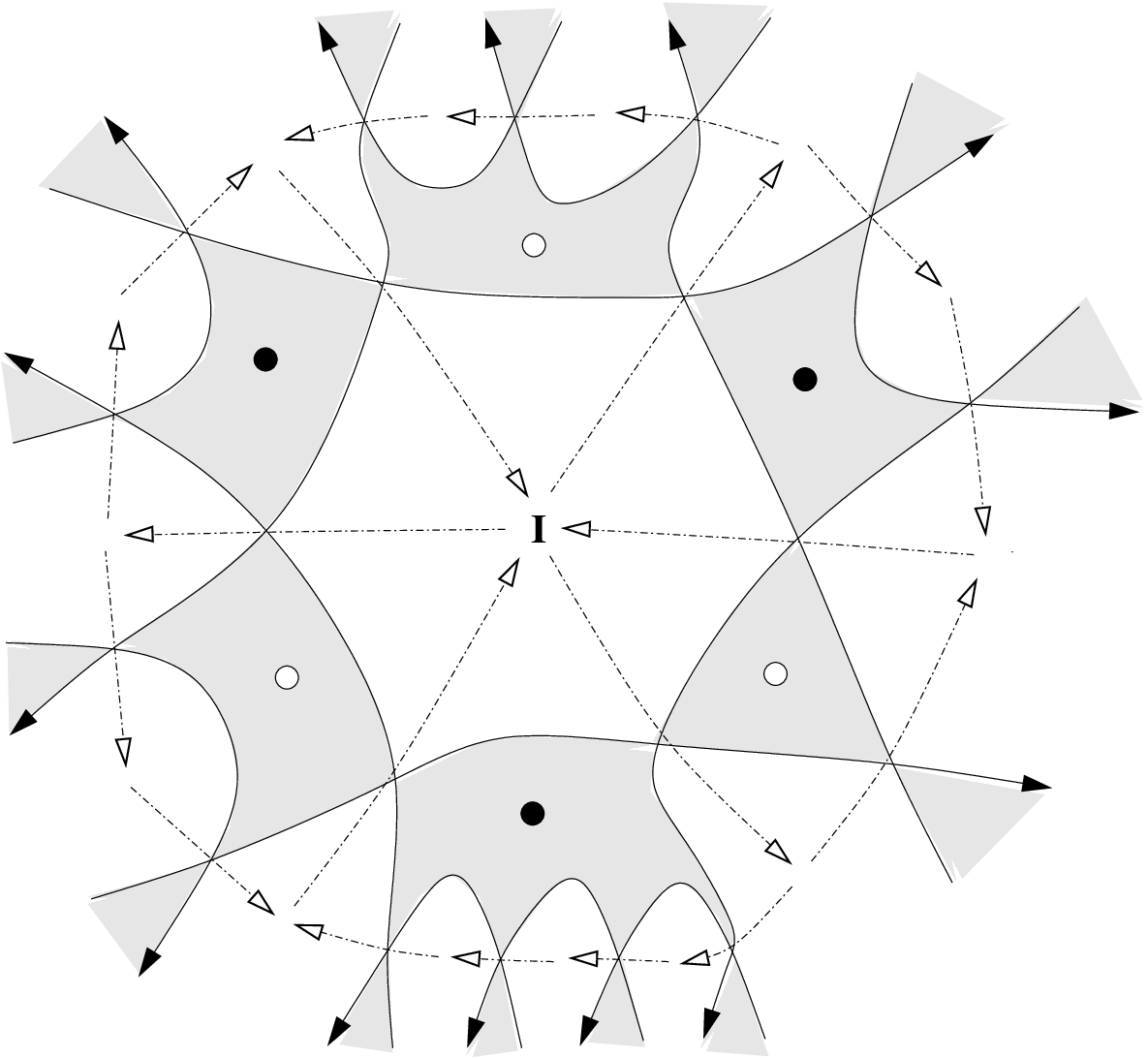}
\caption{The neighbourhood of a internal vertex $I$.}
\label{f:neighbourhood}
\end{figure}

We must next consider the boundary vertices $L_j$ of $\Qi$. The same approach essentially works, but with some additional complications arising from the behaviour at the boundary. For a subset $S$ of $[1,n]$
and $r\in \mathbb{N}$ we denote by $rS$ the
multiset union of $r$ copies of $S$. We recall:

\begin{prop} \label{p:wedge}
\cite[Prop.\ 4.11]{bkm}
Let $I=L_j$ be a boundary vertex in the quiver of a closed Postnikov diagram
$\overline{D}$.
Let $\alpha_+$ be the arrow with source
$I$ which is most anticlockwise (inside the boundary of the disk) and let
$\alpha^-$ be the arrow with source $I$ which is most clockwise. Let $W^{\out}(L_j)$
be the set of arrows incident with $I$ between $\alpha_-$ and $\alpha_+$
(including these two arrows). Let $r_{\out}$ be the number of arrows in $W^{\out}(L_j)$ with source $I$.
Then the multiset union of weights of the arrows in $W^{\out}(L_j)$ is computed by the formula:
\begin{equation}
\label{e:firstunion}
\bigcup_{\alpha\in W^{\out}(L_j)} d(\alpha)=(r_{\out}-1)
[1,n] \cup \{j\}.
\end{equation}
\end{prop}

Now we can prove an analogue of Lemma~\ref{l:lightshining} for the boundary case:

\begin{lem} \label{l:boundarylightshining}
The exponent $e_{L_j}$ of $p_{L_j}$ in the matching monomial $w_{\Mi}$ is determined as follows. Suppose first that $L_j\notin\{I_i,I_i'\}$. Then
\begin{equation}
\label{e:boundarydelta}
e_{L_j}=\begin{cases}
1, &  j\in [i+k+1,i-2]\cup \{i\};\\
0, & \text{otherwise.}\\
\end{cases}
\end{equation}
If $L_j=I_i$ then $j=i$ and $e_{L_j}=0$. 
If $L_j=I_i'$ then $j=i+k$ and $e_{L_j}=1$.
Written more compactly, we have 
\begin{equation}
e_{L_j}=\begin{cases}
1-\delta_{L_jI_i}+\delta_{L_jI'_i} & j\in [i+k+1,i-2]\cup \{i\}; \\
-\delta_{L_jI_i}+\delta_{L_jI'_i} & \text{otherwise.}
\end{cases}
\end{equation}
\end{lem}

\begin{proof}
By Lemma~\ref{l:monomialexponent} and Remark~\ref{r:peripheral}, it is
sufficient to show that the number of clockwise peripheral paths of $L_j$
in $\overline{\Qi}$ containing an arrow in $\Sigma_i$ is given by the
formula~\eqref{e:boundarydelta}.

Let $F$ be the alternating boundary face of $\overline{D_i}$ labelled $L_j$.
If $\PPi$, the unique crossing point of strands $i,i+1$, lies on the boundary
of $F$, then either $i\in L_j$, $i+1\not \in L_j$, or $i\not\in L_j$,
$i+1\in L_j$. It follows that $L_j=L_i=I_i$ (in the first case) or
$L_j=L_{i+k}=I'_i$ (in the second case).
We assume first that neither of these cases occur. This implies that, since the strand crossing points on the arrows incident with $L_j$ are precisely those on the boundary of $F$, strands $i,i+1$ cannot cross on an arrow incident with $L_j$. The separate
cases with this assumption are illustrated in Figure~\ref{f:boundarycases}.

\begin{figure}
\psfragscanon
\psfrag{ci}{$\overline{\gamma_i}$}
\psfrag{Lj}{$L_j$}
\psfrag{Lj-1}{$L_{j-1}$}
\psfrag{Lj+1}{$L_{j+1}$}
\psfrag{Li}{$L_i$}
\psfrag{Li-1}{$L_{i-1}$}
\psfrag{Li+1}{$L_{i+1}$}
\psfrag{Li+2}{$L_{i+2}$}
\psfrag{Li-2}{$L_{i-2}$}
\psfrag{Li+k}{$L_{i+k}$}
\psfrag{Li+k-1}{$L_{i+k-1}$}
\psfrag{Li+k+1}{$L_{i+k+1}$}
\psfrag{The case I}{The case $j\in [i+2,i+k-1]$}
\psfrag{The case II}{The case $j\in [i+k+1,i-2]$}
\psfrag{The case III}{The case $j=i+k$}
\psfrag{The case IV}{The case $j=i-1$}
\psfrag{The case V}{The case $j=i+1$}
\psfrag{The case VI}{The case $j=i$}
\includegraphics[width=12cm]{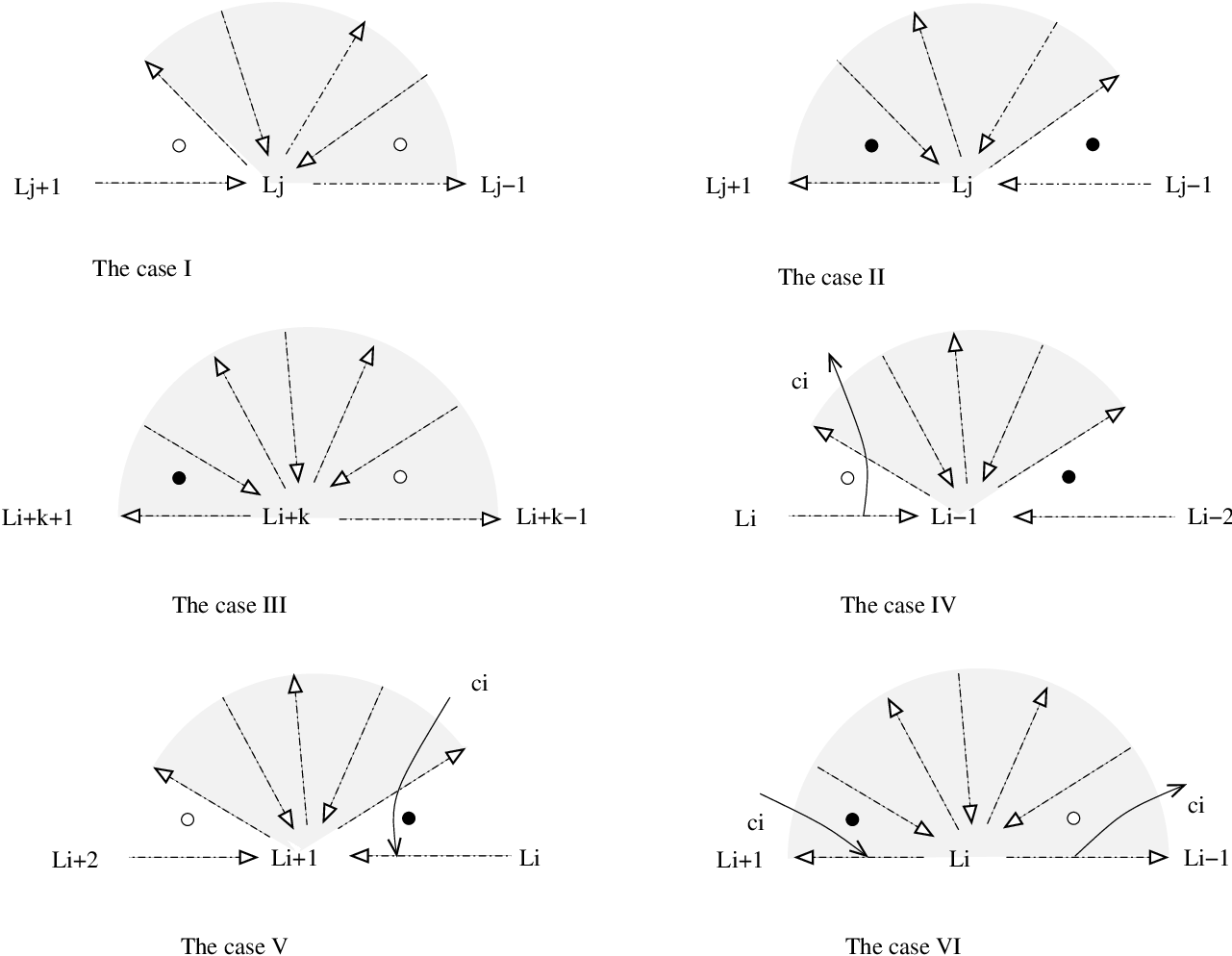}
\caption{Cases in the proof of Lemma~\ref{l:boundarylightshining}. The shaded
region indicates the arrows in $W^{\out}(L_j)$.}
\label{f:boundarycases}
\end{figure}

By Corollary~\ref{c:sigmaicut}, the number of clockwise peripheral paths of $L_j$ containing an arrow in $\Sigma_i$ is equal to the number of clockwise minimal cycles incident with $L_j$ minus the number
of arrows in $\Sigma_i$ incident with $L_j$ which lie in such cycles. We denote the
number of clockwise minimal cycles incident with $L_j$ by $n_c$.

Recall that the boundary arrows do not lie in
$\Sigma_i$, and note that any arrow incident with $L_j$ which is not a boundary arrow lies in some clockwise minimal cycle. Therefore the set of arrows in $\Sigma_i$ incident with $L_j$ coincides with the set of arrows in $\Sigma_i$ incident with $L_j$ which lie in a clockwise minimal cycle. Similarly, this set is also equal to the set of
arrows in $W^{\out}(L_j)$ which lie
in $\Sigma_i$.

Suppose first that $j\not=i-1,i,i+1$. Then
$\overline{\gamma_i}$ does not start or end on a boundary arrow of $\overline{\Qi}$ incident with $L_j$. Also $\PPi$ does not lie on an arrow incident with $L_j$. Hence the set of arrows incident with $L_j$ that cross $\overline{\gamma_i}$ consists of a (possibly empty) collection of pairs of adjacent arrows. This collection of arrows coincides
with the set of arrows in $W^{\out}(L_j)$
crossing $\overline{\gamma_i}$.
By an application of Corollary~\ref{c:gammaalternating}, we then see that   the number of arrows in $W^{\out}(L_j)$ which lie in $\Sigma_i$ is the same as the number of arrows in $W^{\out}(L_j)$ which lie in $\overline{S_i}$. This number is $r_{\out}-1$ by Proposition~\ref{p:wedge}. Hence,
by the above, $e_{L_j}=n_c-(r_{\out}-1)$.

We now compute $n_c$ for $j\ne\{i-1,i,i+1\}$ to complete the proof in this case. There are two cases. If the boundary arrow between $L_j$ and $L_{j-1}$ is oriented towards $L_{j-1}$, then it is an arrow with source $L_j$ which is not part of a clockwise minimal cycle. There is a bijection between the remaining arrows with source $L_j$ and the clockwise minimal cycles incident with $L_j$. Therefore $n_c=r_{out}-1$  and $e_{L_j}=0$ in this case.
If on the other hand the boundary arrow  between $L_j$ and $L_{j-1}$ is oriented towards $L_{j}$, then every arrow with source $L_j$ is part of a clockwise minimal cycle and we have $n_c=r_{out}$. In this case $e_{L_j}=1$.  By definition of the graph $G_i$ we are in the first case if $j\in[i+2,i+k]$ and in the second case if $j\in [i+k+1,i-2]$, as illustrated in Figure~\ref{f:boundarycases}.
  
We now consider the cases where
$\overline{\gamma_i}$ starts or ends on a boundary
arrow of $\overline{\Qi}$ incident with $L_j$, assuming still that $L_j\notin\{I_i,I'_i\}$. Note that $\PPi$ does not lie on an arrow incident with $L_j$ by this assumption. Then we are in the last three cases of Figure~\ref{f:boundarycases}. 

Recall that the exponent $e_{L_j}$ is given by the number of clockwise minimal cycles incident with $L_j$ minus  the number of arrows in $W^{\out}(L_j)$ which lie in $\Sigma_i$. By Proposition~\ref{p:wedge} if $j=i$ there are  $r_{\out}$ arrows in 
$W^{\out}(L_j)$ which lie in $\overline{S_i}$; otherwise there are $r_{\out}-1$. We again use this and the fact that $\Sigma_i$ is constructed out of $\overline{S_i}$ by the swap $\overline{\rho}$ to compute $e_{L_j}$ in each of the remaining cases.

For the cases $j=i-1$ and $j=i+1$ the path $\overline{\gamma_i}$ crosses one boundary arrow and the adjacent arrow incident with $L_j$. The perfect cut $\Sigma_i$ contains the adjacent arrow but not the boundary arrow, while $\overline{S_i}$ contains the boundary arrow but not the adjacent arrow. Therefore 
the number of arrows in $W^{\out}(L_j)$ which lie in $\Sigma_i$ is equal to $r_{\out}$. The exponent is then computed by $e_{L_j}=n_c-r_{\out}$, which in both cases equals $0$.

For the case $j=i$ the set of arrows incident with $L_i$ that cross $\overline{\gamma_i}$ consists of
the two boundary arrows incident with $L_i$,
together with a (possibly empty) collection of pairs of adjacent arrows. The two unpaired boundary arrows
incident with $L_i$ lie in $\overline{S_i}$ and do not lie in $\Sigma_i$.  Both of these boundary arrows lie in $W^{\out}(L_j)$. It follows that 
the number of arrows in $W^{\out}(L_j)$ which lie in $\Sigma_i$ is equal to $r_{\out}-2$. The exponent is then computed by $e_{L_i}=n_c-r_{\out}+2$, which equals $1$.

\begin{figure}
\psfragscanon
\psfrag{Li}{$L_i$}
\psfrag{Li+1}{$L_{i+1}$}
\psfrag{Li-1}{$L_{i-1}$}
\psfrag{Pi}{$\PPi$}
\psfrag{is}{$\overline{b_i}$}
\psfrag{i+1s}{$\overline{b_{i+1}}$}
\includegraphics[width=8cm]{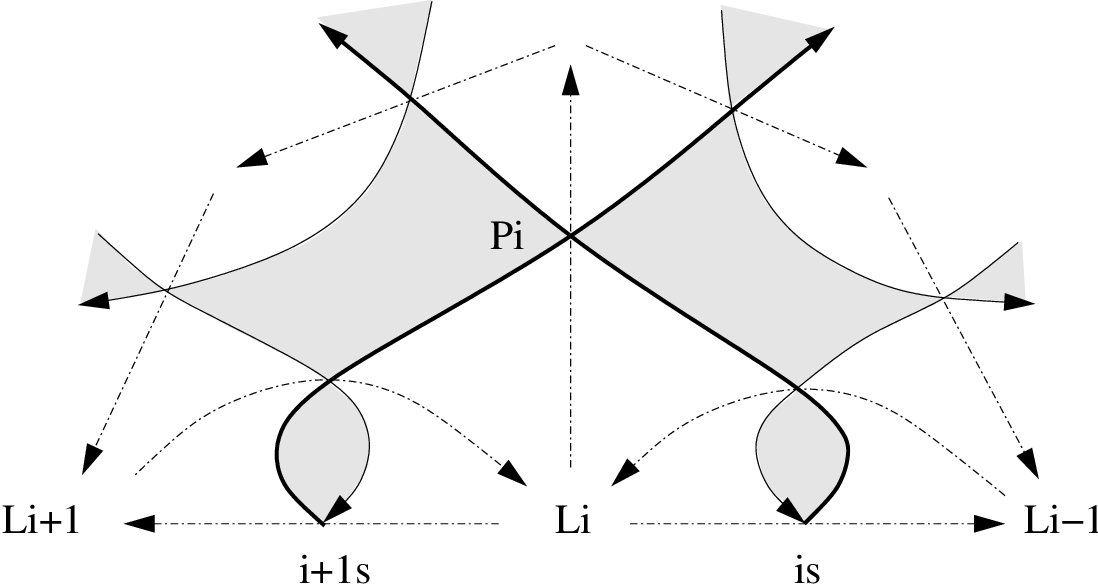}
\caption{Case where strands $i,i+1$ cross on the boundary
of the face labelled $L_i$.}
\label{f:Lispecial}
\end{figure}

\begin{figure}
\psfragscanon
\psfrag{Li+k}{$L_{i+k}$}
\psfrag{Li+k+1}{$L_{i+k+1}$}
\psfrag{Li+k-1}{$L_{i+k-1}$}
\psfrag{Pi}{$\PPi$}
\psfrag{i+ks}{$\overline{b_{i+k}}$}
\psfrag{i+k+1s}{$\overline{b_{i+k+1}}$}
\includegraphics[width=8cm]{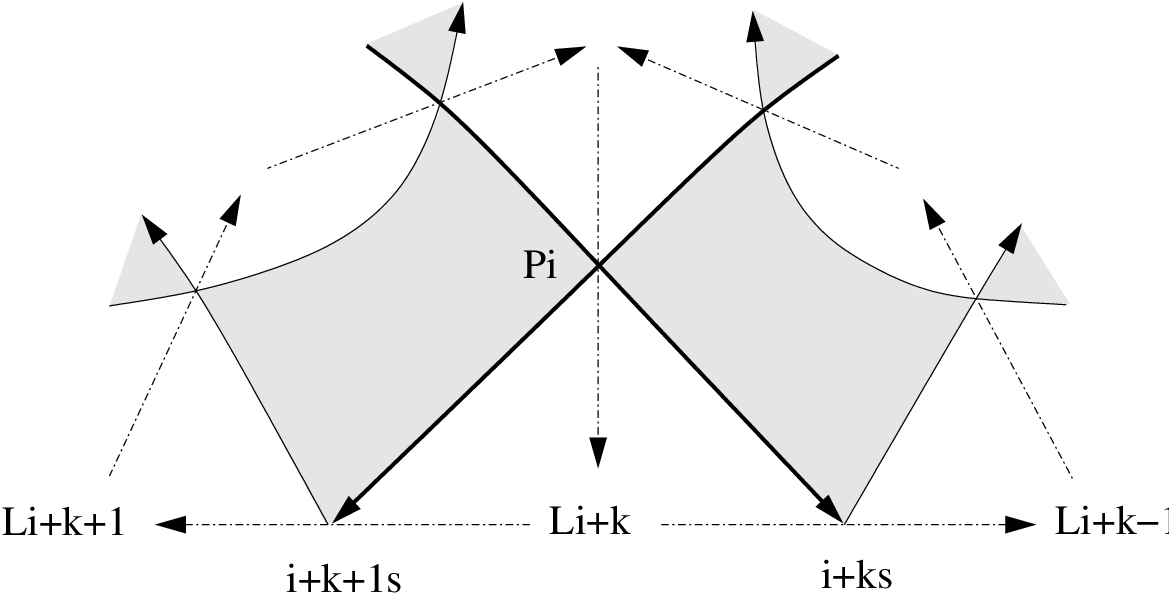}
\caption{Case where strands $i,i+1$ cross on the boundary
of the face labelled $L_{i+k}$.}
\label{f:Lipluskspecial}
\end{figure}

We are left with the cases $L_j=L_i=I_i$ and
$L_j=L_{i+k}=I'_i$.
We first consider the case $L_j=L_i=I_i$. In this case, strands
$i,i+1$ cross on the boundary of $F$.
For this to happen, the boundary arrows
incident with $L_i$ must lie in $2$-cycles and the crossing point of strands $i,i+1$ must lie on a unique internal
arrow with source $L_i$, lying in $\overline{S_i}$.
See Figure~\ref{f:Lispecial}.
The boundary arrows incident with $L_i$ lie in $\overline{S_i}$, and we see that the arrows completing these
arrows to $2$-cycles both lie in $\Sigma_i$.
Therefore, there are two clockwise minimal cycles incident with $L_i$, and each contains an arrow in $\Sigma_i$ incident with $L_i$. It follows that $e_{L_i}=0$.

Finally, we consider the case $L_j=L_{i+k}=I'_i$.
Note that the boundary arrows incident with $L_{i+k}$ point away from $L_{i+k}$.
For the crossing point $\PPi$ of strands $i,i+1$
to lie on an arrow incident with $L_{i+k}$,
there must be a unique internal arrow
incident with $L_{i+k}$, necessarily
oriented towards $L_{i+k}$, and $\PPi$
must lie on this arrow.
See Figure~\ref{f:Lipluskspecial}.
The boundary arrows incident with
$L_{i+k}$ do not lie in $\Sigma_i$.
Also, the arrow through $\PPi$ does not lie in $\Sigma_i$ (by Lemma~\ref{l:strandalternating}).
Therefore, there is one clockwise minimal
cycle incident with $L_{i+k}$, but no
arrows incident with $L_{i+k}$ lie in
$\Sigma_i$. It follows that $e_{L_{i+k}}=1$.

We have now considered all possible cases, so the proof is complete.
\end{proof}

We have proved the following theorem, which also holds in the cases $k=1,n-1$.

\begin{thm} \label{t:pmproperties}
Fix $\lambda\in \Pn_{k,n}$ and let $D_{\lambda}$
be the Postnikov diagram constructed by
Theorem~\ref{thm:diagramconstruction}.
Let $\CI$ be the corresponding Postnikov
cluster. Fix
$i\in [1,n]$ and let
$D_i$ be the Postnikov diagram
associated to $D_{\lambda}$ in Definition~\ref{d:Di}.
Let $G_i$ be the corresponding
dual bipartite graph.
Let $\Mi$ be the perfect
matching on $G_i$ defined in Definition~\ref{d:perfectmatching}. Let $\Fi$ be
the alternating face in $D_i$ defined
in Definition~\ref{d:Fi}. Then the following holds.
\begin{enumerate}[(a)]
\item
If $\Fi$ is a boundary face, then $\Mi$ is the unique perfect matching on $G_i$.
If $\Fi$ is an internal face, then it
is the face of the elementary component
$\Giin$ of $G_i$ (see Definition~\ref{d:Giin}).
In this case, every other perfect matching on $G_i$ can be obtained from $\Mi$ by flipping $\Fi$ and then applying a further sequence of flips involving faces of $\Giin$ distinct from $\Fi$.
\item
The following identity holds:
$$\frac{w_{\Mi}}{\cluster}\p_{L_{i-1}}\p_{L_{i+1}}\cdots \p_{L_{i+k}}=\frac{\p_{I'_i}}{\p_{I_i}},$$
and computes the summand corresponding to $M_i$ in the formula \eqref{e:Lihatterm} for $\frac{p_{\widehat{L}_i}}{p_{L_i}}$. 
associated to $\Mi$.
\end{enumerate}
\end{thm}

\begin{proof}
Suppose first that $k\not=1,n-1$. Then part (a)
is Theorem~\ref{t:startwithQi}.
Part (b) follows from Lemmas~\ref{l:lightshining}
and~\ref{l:boundarylightshining}.
For the case $k=1,n-1$, part (a) holds by Remark~\ref{r:Ximatchingk1}.
For part (b), note first that for our choice of Postnikov diagram (see
Section~\ref{s:anycluster}) there is not a unique choice of crossing point for
strands $i,i+1$. However, the adjacent faces $\Fi$ and $\Fip$ do not depend on this choice,
and we have $p_{I_i}=p_i=p_{L_i}$ and
 $p_{I'_i}=p_{i+1}=p_{\widehat{L}_i}$. Since $\Mi$ is the unique perfect matching on $G_i$ in this 
case, part (b) follows from Theorem~\ref{thm:MarSco}.
\end{proof}

\section{The vector fields $X_\lambda^{(\m)}$}
\label{s:vectorfield}
In this section we define a family of vector fields $X^{(\m)}_{\lambda}$,
$\lambda\in \Pn_{k,n}$, $\m\in [1,n]$, on $\Xcheck
$, denoting $X^{(n)}_{\lambda}$ also by $X_{\lambda}$. We allow the cases $k=1,n-1$.
We start by defining a vector field $X^{(\m)}_{\lambda,\CC}$
in a fixed Postnikov extended cluster $\CC$ containing $p_\lambda$ via an  explicit formula involving combinatorially defined coefficients $c^{(\m)}_{\lambda}(\mu)$. We then prove that the
coefficients satisfy an additivity property which in particular will imply that the vector field $X^{(\m)}_{\lambda,\CC}$ extends to a regular vector field $X^{(\m)}_{\lambda}$ on the whole of $\Xcheck$.

Fix a Young diagram $\lambda\in \Pn_{k,n}$. For each $\mu\in \Pn_{k,n}$
we define a nonnegative integer $c_{\lambda}(\mu)$ as follows.
Write $J_{\mu}\setminus J_{\lambda}=\{m_1<m_2<\cdots <m_r\}$ and
$J_{\lambda}\setminus J_{\mu}=\{l_1<l_2<\cdots <l_r\}$ in increasing numerical
order. Then set
$$c_{\lambda}(\mu)=|\{1\leq j\leq r\,:\,m_j>l_j\}|.$$
Note that $c_{\lambda}(\lambda)=c_{\lambda}(\emptyset)=0$.

Given an extended cluster $\CC$ (for the cluster structure on $\mathbb{C}[\Xcheck]$
discussed in Section~\ref{s:CAcoordring}), we define:
$$\Xcheck_{\CC}=\{x\in \Xcheck\,:\,f(x)\ne 0\text{ for all $f\in \CC$}\}.$$
Note that $\Xcheck_{\CC}$ is isomorphic to $(\mathbb{C}^*)^{k(n-k)}$
by~\cite[Theorem 4]{Scott:Grassmannian}. We call this a \emph{cluster torus} in $\Xcheck$.
If $\CC$ is an extended cluster of $\C[\Xcheck]$ coming from a Postnikov diagram
containing a region labelled $J_{\lambda}$ (so that $\p_{\lambda}\in \CC$),
we consider the regular vector field
\begin{equation} \label{e:clusterfield}
X_{\lambda,\CC}:=\p_{\lambda} \sum_{\p_{\mu}\in \CC}
c_{\lambda}(\mu)\p_{\mu}\frac{\partial}{\partial \p_{\mu}}
\end{equation}
on $\Xcheck_{\CC}$.
We shall see later that $X_{\lambda,\CC}$ can be extended to a regular vector field on the whole of
$\Xcheck$.

\begin{example}
Let $\CC(D)$ be the extended cluster associated to the Postnikov diagram $D$ in Figure~\ref{fig:post36}, and take $\lambda=\raisebox{\depth}{\ydiagram{1,1}}$. Note that $J_{\lambda}=\{1,2,5\}$ We calculate the coefficients $c_{\lambda}(\mu)$ for $p_{\mu}\in \CC$ in the following table:

\begin{center}
\begin{tabular}{|c|c|c|c|c|}
\hline
$\mu$ & $J_{\mu}$ & $J_{\mu} \setminus J_{\lambda}$ &
$J_{\lambda} \setminus J_{\mu}$ & $c_{\lambda}(\mu)$ \\
\hline
$\emptyset$ & $\{1,2,3\}$ & $\{3\}$ & $\{5\}$ & $0$ \\
\hline
$\raisebox{\depth}{\ydiagram{3}}$ & $\{2,3,4\}$ & $\{3,4\}$ & $\{1,5\}$ & $1$ \\
\hline
$\raisebox{\depth}{\ydiagram{3,3}}$ & $\{3,4,5\}$ & $\{3,4\}$ & $\{1,2\}$ & $2$ \\
\hline
$\raisebox{\depth}{\ydiagram{3,3,3}}$ & $\{4,5,6\}$ & $\{4,6\}$ & $\{1,2\}$ & $2$ \\
\hline
$\raisebox{\depth}{\ydiagram{2,2,2}}$ & $\{1,5,6\}$ & $\{6\}$ & $\{2\}$ & $1$ \\
\hline
$\raisebox{\depth}{\ydiagram{1,1,1}}$ & $\{1,2,6\}$ & $\{6\}$ & $\{5\}$ & $1$ \\
\hline
$\raisebox{\depth}{\ydiagram{2,1}}$ & $\{1,3,5\}$ & $\{3\}$ & $\{2\}$ & $1$ \\
\hline
$\raisebox{\depth}{\ydiagram{2,1,1}}$ & $\{1,3,6\}$ & $\{3,6\}$ & $\{2,5\}$ & $2$ \\
\hline
$\raisebox{\depth}{\ydiagram{2,2}}$ & $\{1,4,5\}$ & $\{4\}$ & $\{2\}$ & $1$ \\
\hline
$\raisebox{\depth}{\ydiagram{3,1}}$ & $\{2,3,5\}$ & $\{3\}$ & $\{1\}$ & $1$ \\
\hline
\end{tabular}
\end{center}
We see that
\begin{equation*}
\begin{split}
X_{\ydiagram{1,1},\CC(D)}=
p_{\ydiagram{1,1}} \left(
p_{\ydiagram{3}}\frac{\partial}{\partial \p_{\ydiagram{3}}}+
2p_{\ydiagram{3,3}}\frac{\partial}{\partial \p_{\ydiagram{3,3}}}+
2p_{\ydiagram{3,3,3}}\frac{\partial}{\partial \p_{\ydiagram{3,3,3}}}+
p_{\ydiagram{2,2,2}}\frac{\partial}{\partial \p_{\ydiagram{2,2,2}}}+
p_{\ydiagram{1,1,1}}\frac{\partial}{\partial \p_{\ydiagram{1,1,1}}}+\right.
\\
\left. +p_{\ydiagram{2,1}}\frac{\partial}{\partial \p_{\ydiagram{2,1}}}+
2p_{\ydiagram{2,1,1}}\frac{\partial}{\partial \p_{\ydiagram{2,1,1}}}+
p_{\ydiagram{2,2}}\frac{\partial}{\partial \p_{\ydiagram{2,2}}}+
p_{\ydiagram{3,1}}\frac{\partial}{\partial \p_{\ydiagram{3,1}}}\right).
\end{split}
\end{equation*}
\end{example}

We also consider a family of `twisted' versions of
$X_{\lambda,\CC}$, defined as follows. For $\m\in [1,n]
$, recall that $\mu^{(\m)}$ denotes the partition
defined by $J_{\mu^{(\m)}}:=J_{\mu}-\m\mod n$ (see Definition~\ref{d:cyclicgroup}).

For $m\in [1,n]$ and Young diagrams $\lambda$, $\mu$ in $\Pn_{k,n}$, we set
\begin{equation} \label{e:clambdam}
c^{(\m)}_{\lambda}(\mu)=c_{\lambda^{(\m)}}(\mu^{(\m)})-c_{\lambda^{(\m)}}(\emptyset^{(\m)}).
\end{equation}
Note also that, by definition,
$c^{(\m)}_{\lambda}(\emptyset)=0$ for any $m,\lambda$. We also have
that $c^{(n)}_{\lambda}(\mu)=c_{\lambda}(\mu)$ for any $\lambda,\mu$.

For a Postnikov extended cluster $\CC$ as above, we have the regular vector field
\begin{equation} \label{e:twistedclusterfield}
X^{(\m)}_{\lambda,\CC}:=\p_{\lambda} \sum_{\p_{\mu}\in \CC}
c^{(\m)}_{\lambda}(\mu)\p_{\mu}\frac{\partial}{\partial \p_{\mu}}
\end{equation}
on $\Xcheck_{\CC}$. Note that $X^{(n)}_{\lambda,\CC}=X_{\lambda,\CC}$.


Let $P_n$ be a regular polygon with vertices $1,2,\ldots ,n$ numbered
clockwise. Then, in~\cite[\S1]{LeZe:Quasicommuting}, two $k$-subsets $I,J$ of
$[1,n]$ are said to be \emph{weakly separated}
if none of the chords between the vertices of $P_n$ corresponding
to $I\setminus J$ crosses any of the chords between the vertices corresponding
to $J\setminus I$. Scott~\cite[Defn.\ 3]{Scott:Grassmannian} uses the term
\emph{non-crossing} and we shall use this terminology.

We recall the following result of Scott:

\begin{thm} \label{thm:Noncrossing}
Let $D$ be a Postnikov diagram of type $(k,n)$.
Then the collection of $k$-subsets
labelling the alternating faces of $D$ is an inclusion-maximal collection of
pairwise non-crossing $k$-subsets of $[1,n]$.
\end{thm}

We will now show that the coefficients $c^{(\m)}_{\lambda}(\mu)$ satisfy a
certain additivity property on Postnikov diagrams, which will be very useful for understanding $X_{\lambda,\mathcal C}^{(m)}$. We first need some notation for the
labels of the alternating faces around a given internal alternating face.

\begin{defn} \label{d:adjacentfaces}
Let $D$ be a Postnikov diagram.
Let $F$ be a non-boundary alternating face
of $D$, labelled by the $k$-subset
$J_{\mu}$, with $\mu\in \Pn_{k,n}$.
Consider the labels of the strands passing along the boundary of $F$, which are alternatingly anticlockwise and clockwise around $F$. Let $a_1$ be the  minimal label amongst all labels of anticlockwise strands around $F$. Let $c_1$ be the label of the clockwise strand meeting $a_1$ where it enters the boundary of $F$. We can extend this labelling to all of the strands passing along the boundary of $F$ to get a list  
\[
a_1,c_1,a_2,c_2,\dotsc, a_r,c_r
\] 
of elements of $[1,n]$. These are the labels of all of the strands in the boundary of $F$ starting from the minimal anticlockwise strand $a_1$ and going around $F$ in clockwise order.

Consider next the faces adjacent to $F$. We denote the face meeting $F$ in the crossing point of $c_i$ and $a_i$ by $F(i)$, and the face meeting $F$ in the crossing point of $a_i$ and $c_{i-1}$ by $F'(i)$.
Note that the strands
on the boundary of $F$ are oriented towards the intersection point where $F'(i)$ and
$F$ meet, and away from the intersection point where $F(i)$ and $F$ meet. The labels of $F(i)$ and of $F'(i)$ are determined from the label $J_\mu$ of $F$ as follows. The label of $F(i)$ is $(J_{\mu}\setminus \{a_{i}\})\cup \{c_{i}\}$.
The label of $F'(i)$ is $(J_{\mu}\setminus \{a_{i}\})\cup \{c_{i-1}\}$.

Thus, the faces adjacent to $F$ are, in order clockwise around $F$:
$$F'(1),F(1),\ldots ,F'(r),F(r).$$
We denote the partitions corresponding to the $k$-subsets labelling these
faces by:
$$\mu'(1),\mu(1),\mu'(2),\mu(2),\ldots ,\mu'(r),\mu(r).$$
Thus $J_{\mu'(i)}$ labels $F'(i)$ and $J_{\mu(i)}$ labels $F(i)$ for all $i$.
See Figure~\ref{fig:alternatingface}.
\end{defn}

\begin{figure}
\psfragscanon
\psfrag{I1}{$J_{\mu(1)}$}
\psfrag{I2}{$J_{\mu(2)}$}
\psfrag{J1}{$J_{\mu'(1)}$}
\psfrag{J2}{$J_{\mu'(2)}$}
\psfrag{I3}{$J_{\mu(3)}$}
\psfrag{I4}{$J_{\mu(4)}$}
\psfrag{J3}{$J_{\mu'(3)}$}
\psfrag{J4}{$J_{\mu'(4)}$}
\psfrag{Jmu}{$J_{\mu}$}
\psfrag{a1}{$a_1$}
\psfrag{a2}{$a_2$}
\psfrag{a3}{$a_3$}
\psfrag{a4}{$a_4$}
\psfrag{c1}{$c_1$}
\psfrag{c2}{$c_2$}
\psfrag{c3}{$c_3$}
\psfrag{c4}{$c_4$}
\includegraphics[width=7cm]{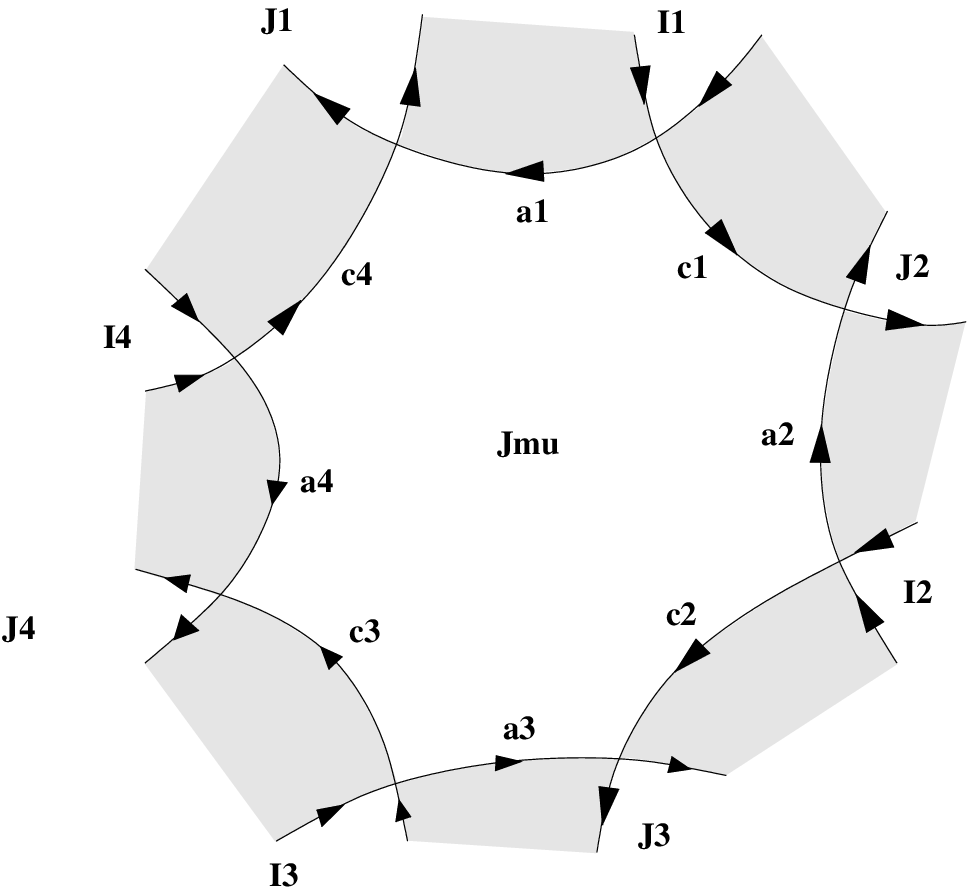}
\caption{The part of a Postnikov diagram around an alternating face.}
\label{fig:alternatingface}
\end{figure}

\begin{prop} \label{Prop:VectorFieldAdditivity}
Suppose $\lambda\in \Pn_{k,n}$ is a partition such that $J_{\lambda}$ is not a frozen variable. Fix a Postnikov diagram $D$ such that $D$ has an internal alternating
face labelled $J_{\lambda}$.
Let $F$ be a different, non-boundary alternating face of $D$, labelled by the $k$-subset $J_{\mu}$.
Then, for $\m\in [1,n]$, we have
$$\sum_{i=1}^r c^{(\m)}_{\lambda}(\mu(i))=\sum_{i=1}^r c^{(\m)}_{\lambda}(\mu'(i)).$$
In particular, taking $m=n$, we have:
$$\sum_{i=1}^r c_{\lambda}(\mu(i))=\sum_{i=1}^r c_{\lambda}(\mu'(i)).$$
\end{prop}

\begin{proof} We use the notation of Definition~\ref{d:adjacentfaces}.
We consider the strands along the boundary of the face $F$ labelled $J_\mu$.
Note that:
\begin{align*}
J_{\mu(i)}&=(J_{\mu}\setminus \{a_i\})\cup \{c_i\};\\
J_{\mu'(i)}&=(J_{\mu}\setminus \{a_i\})\cup \{c_{i-1}\}.
\end{align*}
Here we adopt the convention that the subscripts of the $a_i$ and $c_i$ are
interpreted modulo $r$. 
Note that since $c_i\not\in J_{\mu}$ and $a_j\in J_{\mu}$, we have
$c_i\not=a_j$ for all $i,j$.

We consider first the untwisted case, $\m=n$. We express $c_\lambda(\mu(i))$ in terms of $c_\lambda(\mu)$ and summands $C_i$ and $A_i$ associated to the strands $c_i$ and $a_i$, respectively, which separate $\mu$ from $\mu(i)$. The $C_i$ and $A_i$ are defined in such a way that simultaneously $c_\lambda(\mu'(i))=c_\lambda(\mu)+C_{i-1}+A_i$ (see the claim below).

For subsets $I,I'$ of $[1,n]$ we write
$I<I'$ to indicate that every element of $I$ is less than every element of $I'$
and for an element $a$ we write $a<I$ (respectively, $I<a$) to denote $\{a\}<I$
(respectively, $I<\{a\}$).
By 
Theorem~\ref{thm:Noncrossing} $J_{\lambda}$ and $J_{\mu}$ are non-crossing.  Since $J_{\lambda}\not =J_{\mu}$ there are four possible configurations for the subsets $J_{\lambda}\setminus J_{\mu}$ and $J_\mu\setminus J_{\lambda}$ in
$[1,n]$:
\begin{enumerate}
\item[(I)] $J_{\mu}\setminus J_{\lambda}<J_{\lambda}\setminus J_{\mu}$;
\item[(II)] $J_{\lambda}\setminus J_{\mu}=K_1\sqcup K_2$, $K_1,K_2\not=\emptyset$ and $K_1<J_{\mu}\setminus J_{\lambda}<K_2$;
\item[(III)] $J_{\lambda}\setminus J_{\mu}<J_{\mu}\setminus J_{\lambda}$;
\item[(IV)] $J_{\mu}\setminus J_{\lambda}=K_1\sqcup K_2$,
$K_1,K_2\not=\emptyset$ and $K_1<J_{\lambda}\setminus J_{\mu}<K_2$.
\end{enumerate}
We display these in Figure~\ref{f:cases}.
\begin{figure}
\psfragscanon
\psfrag{1}{$\scriptstyle 1$}
\psfrag{n}{$\scriptstyle n$}
\psfrag{Case I}{Case I}
\psfrag{Case II}{Case II}
\psfrag{Case III}{Case III}
\psfrag{Case IV}{Case IV}
\psfrag{lm}{$J_{\lambda}\setminus J_{\mu}$}
\psfrag{ml}{$J_{\mu}\setminus J_{\lambda}$}
\includegraphics[width=10cm]{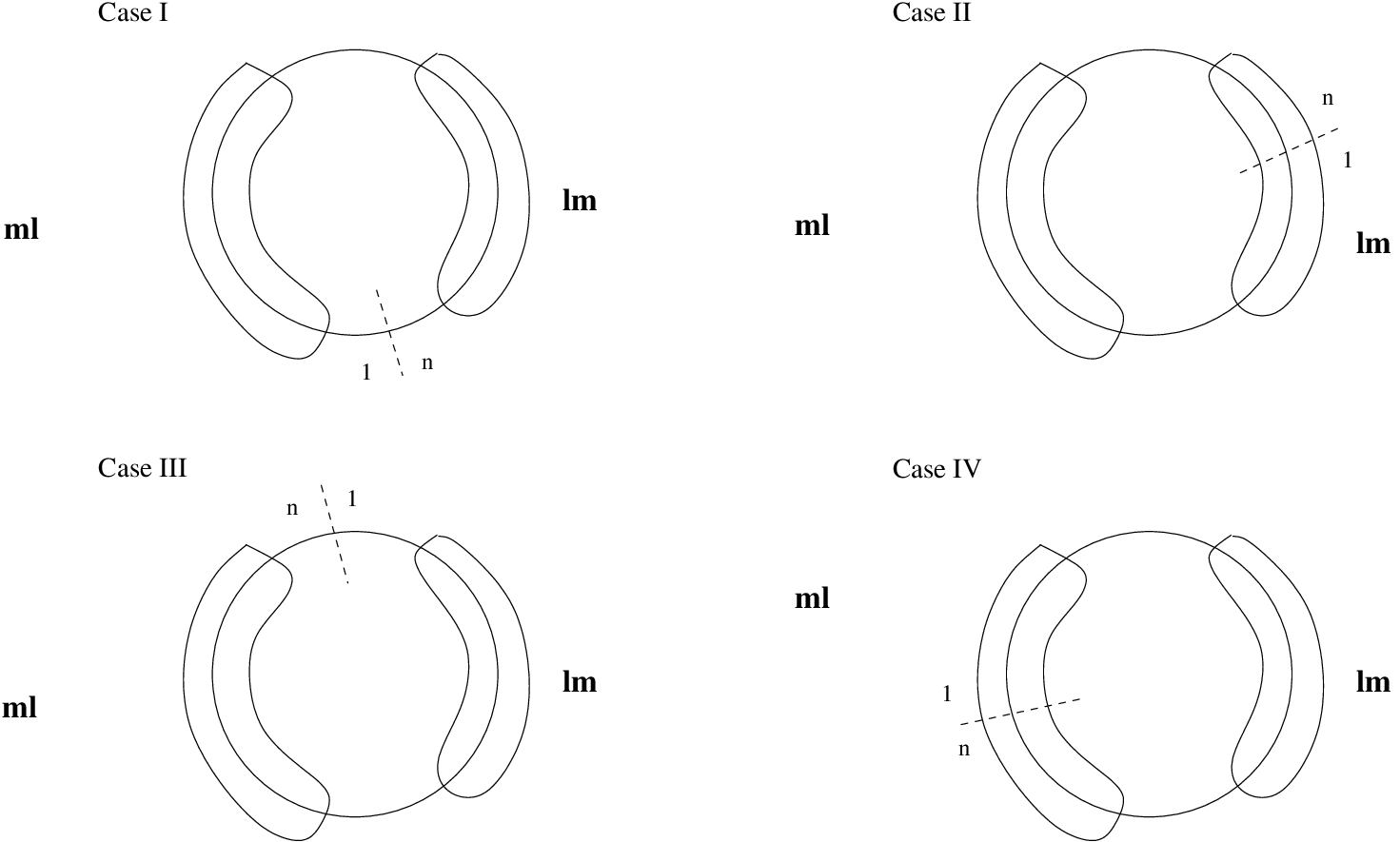}
\caption{The four configurations of $J_{\lambda}\setminus J_{\mu}$ and $J_{\mu}\setminus J_{\lambda}$.}
\label{f:cases}
\end{figure}

In each case, we define integers $C_i$ and $A_i$ for $i=1,2,\ldots ,r$ (again treating subscripts modulo $r$), as follows. Note that $C_i$ depends only on $\lambda,\mu$ and $c_i$ and similarly $A_i$ depends on $\lambda, \mu$ and $a_i$. 

\begin{description}
\item[Case I] If $J_{\mu}\setminus J_{\lambda}<J_{\lambda}\setminus J_{\mu}$, then set
\begin{alignat}{2}
C_i &=
\begin{cases}
0 & c_i\in J_{\lambda}; \\
0 & c_i<J_{\lambda}\setminus J_{\mu},\,c_i\not\in J_{\lambda}; \\
1 & c_i>J_{\lambda}\setminus J_{\mu},\,c_i\not\in J_{\lambda};
\end{cases}
\text{\ \ \  and\ } &\quad A_i &=
\begin{cases}
0 & a_i\not\in J_{\lambda}; \\
1 & a_i<J_{\mu}\setminus J_{\lambda},\,a_i\in J_{\lambda}; \\
0 & a_i>J_{\mu}\setminus J_{\lambda},\,a_i\in J_{\lambda}.
\end{cases}
\end{alignat}
\item[Case II] If $J_{\lambda}\setminus J_{\mu}=K_1\sqcup K_2$ where
$K_1$ and $K_2$ are nonempty and $K_1<J_{\mu}\setminus J_{\lambda}<K_2$, then set
\begin{alignat}{2}
C_i &=
\begin{cases}
0 & c_i\not\in J_{\lambda}; \\
-1 & c_i<J_{\mu}\setminus J_{\lambda},\,c_i\in J_{\lambda}; \\
0 & c_i>J_{\mu}\setminus J_{\lambda},\,c_i\in J_{\lambda};
\end{cases}
\text{\ \ \  and\ } &\quad A_i &=
\begin{cases}
0 & a_i\not\in J_{\lambda}; \\
1 & a_i<J_{\mu}\setminus J_{\lambda},\,a_i\in J_{\lambda}; \\
0 & a_i>J_{\mu}\setminus J_{\lambda},\,a_i\in J_{\lambda}.
\end{cases}
\end{alignat}
\item[Case III] If $J_{\lambda}\setminus J_{\mu}<J_{\mu}\setminus J_{\lambda}$, then set
\begin{alignat}{2}
C_i &=
\begin{cases}
0 & c_i\in J_{\lambda}; \\
0 & c_i<J_{\lambda}\setminus J_{\mu},\,c_i\not\in J_{\lambda}; \\
1 & c_i>J_{\lambda}\setminus J_{\mu},\,c_i\not\in J_{\lambda};
\end{cases}
\text{\ \ \  and\ } &\quad A_i &=
\begin{cases}
-1 & a_i\not\in J_{\lambda}; \\
0 & a_i<J_{\mu}\setminus J_{\lambda},\,a_i\in J_{\lambda}; \\
-1 & a_i>J_{\mu}\setminus J_{\lambda},\,a_i\in J_{\lambda}.
\end{cases}
\end{alignat}
\item[Case IV] If $J_{\mu}\setminus J_{\lambda}=K_1\sqcup K_2$ where
$K_1$ and $K_2$ are non-empty and $K_1<J_{\lambda}\setminus J_{\mu}<K_2$, then set
\begin{alignat}{2}
C_i &=
\begin{cases}
0 & c_i\in J_{\lambda}; \\
0 & c_i<J_{\lambda}\setminus J_{\mu},\,c_i\not\in J_{\lambda}; \\
1 & c_i>J_{\lambda}\setminus J_{\mu},\,c_i\not\in J_{\lambda};
\end{cases}
\text{\ \ \  and\ } &\quad A_i &=
\begin{cases}
0 & a_i\in J_{\lambda}; \\
0 & a_i<J_{\lambda}\setminus J_{\mu},\,a_i\not\in J_{\lambda}; \\
-1 & a_i>J_{\lambda}\setminus J_{\mu},\,a_i\not\in J_{\lambda}.
\end{cases}
\end{alignat}
\end{description}

\textbf{Claim:}
\begin{enumerate}[(a)]
\item For $1\leq i\leq r$, we have $c_{\lambda}(\mu(i))-c_{\lambda}(\mu)=C_i+A_i$.
\item For $1\leq i\leq r$, we have $c_{\lambda}(\mu'(i))-c_{\lambda}(\mu)=C_{i-1}+A_i$.
\end{enumerate}

\textbf{Proof of claim:}
The proof is case by case. Let $1\leq i\leq r$ and set $j=i$ or $i-1$ reduced modulo $r$.
Set $\kappa=\mu(i)$ in the former case and $\kappa=\mu'(i)$ in the latter case.
Then $J_{\kappa}=(J_{\mu}\setminus \{a_i\})\cup \{c_j\}$.

Note that, since $c_j\not\in J_{\mu}$ and
$a_i\in J_{\mu}$, we can describe $J_{\lambda}\setminus J_{\kappa}$ in terms of $J_{\lambda}\setminus J_{\mu}$
and $J_{\kappa}\setminus J_{\lambda}$ in terms of $J_{\mu}\setminus J_{\lambda}$, as follows:
\begin{itemize}
\item If $c_j,a_i\in J_{\lambda}$, we have $J_{\lambda}\setminus J_{\kappa}=((J_{\lambda}\setminus J_{\mu})\setminus \{c_j\})\cup \{a_i\}$ and $J_{\kappa}\setminus J_{\lambda}=J_{\mu}\setminus J_{\lambda}$.
\item If $c_j\in J_{\lambda}, a_i\not\in J_{\lambda}$, we have
$J_{\lambda}\setminus J_{\kappa}=(J_{\lambda}\setminus J_{\mu})\setminus \{c_j\}$ and $J_{\kappa}\setminus J_{\lambda}=(J_{\mu}\setminus J_{\lambda})\setminus \{a_i\}$.
\item If $c_j\not\in J_{\lambda}$ and $a_i\in J_{\lambda}$, we have
$J_{\lambda}\setminus J_{\kappa}=(J_{\lambda}\setminus J_{\mu})\cup \{a_i\}$ and $J_{\kappa}\setminus J_{\lambda}=(J_{\mu}\setminus J_{\lambda})\cup \{c_j\}$.
\item If $c_j,a_i\not\in J_{\lambda}$, we have
$J_{\lambda}\setminus J_{\kappa}=J_{\lambda}\setminus J_{\mu}$ and $J_{\kappa}\setminus J_{\lambda}=((J_{\mu}\setminus J_{\lambda})\setminus \{a_i\})\cup \{c_j\}$.
\end{itemize}

We consider each of the Cases (I)-(IV), from before, in turn.
We divide each case into separate subcases depending on whether $c_j,a_i$
lie in $J_{\lambda}$ or not and also their position in relation to the subsets
$J_{\mu}\setminus J_{\lambda}$ and $J_{\lambda}\setminus J_{\mu}$. We display
the calculations as a table for each of the cases (I)--(IV). The claim then follows
from the observation that the sum of the entries in the columns headed $C_j$ and $A_i$ 
coincides with the entry in the final column. Below the table for each case is a diagram 
illustrating the arrangement of the subsets $J_{\lambda}\setminus J_{\mu}$ and
$J_{\mu}\setminus J_{\lambda}$ (note that these are not intervals, just contained
in intervals in the arrangement shown, by the noncrossing property).
The subset $[1,n]$ is drawn as a circle (with the numbering 
$1,2,\ldots ,n$ clockwise around the boundary). A dotted line cutting across the circle 
indicates the gap between $1$ and $n$.

For example, in Case III, if $c_j,a_i\in J_{\lambda}$ and $a_i>J_{\mu}\setminus J_{\lambda}$,
then we have $J_{\lambda}\setminus J_{\mu}<J_{\mu}\setminus J_{\lambda}$. We also
have
$J_{\lambda}\setminus J_{\kappa} = ((J_{\lambda}\setminus J_{\mu})\setminus \{c_j\})\cup \{a_i\}$
and $J_{\kappa}\setminus J_{\lambda} = J_{\mu}\setminus J_{\lambda}$.
i.e.\ replacing $J_{\mu}$ with $J_{\kappa}$ has the effect of removing $c_j$
from $J_{\lambda}\setminus J_{\mu}$ and adding $a_i$. Since
$a_i>J_{\mu}\setminus J_{\lambda}$, it can be seen from the definition that
$c_{\lambda}(\kappa)=c_{\lambda}(\mu)-1$, giving the entry $-1$ in the second row,
final column of the table for Case III. From the definitions, $C_j=0$ and $A_i=-1$
in this case, and we see that $C_j+A_i=c_{\lambda}(\kappa)-c_{\lambda}(\mu)$ as required.
The arguments in the other cases are similar.

Occasionally we need to use the fact that the $J_{\lambda}$ and $J_{\kappa}$ are also noncrossing; this 
is why two of the subcases, as indicated in the tables, cannot occur.
This is also used, for example,  in Case III for the subcase $c_j\not\in J_{\lambda}$, $a_i\in 
J_{\lambda}$, $c_j<J_{\lambda}\setminus J_{\mu}$, $a_i<J_{\mu}\setminus J_{\lambda}$. The 
noncrossing property implies that we must have $c_j<a_i$.

\begin{center}
\begin{minipage}{\textwidth}
\begin{center}
\textbf{Case I}: $J_{\mu}\setminus J_{\lambda}<J_{\lambda}\setminus J_{\mu}$
\vskip 0.3cm
\begin{tabular}{|c|c|c|c|c|}
\hline
Membership in $J_{\lambda}$ & Extra condition(s) & $C_j$ & $A_i$ & $c_{\lambda}(\kappa)-c_{\lambda}(\mu)$ \\
\hline
\multirow{2}*{$c_j,a_i\in J_{\lambda}$}
& $a_i<J_{\mu}\setminus J_{\lambda}$ & $0$ & $1$ & $1$ \\ \cline{2-5}
& $a_i>J_{\mu}\setminus J_{\lambda}$ & $0$ & $0$ & $0$ \\ \hline
$c_j\in J_{\lambda},a_i\not\in J_{\lambda}$
& none                               & $0$ & $0$ & $0$ \\ \hline
\multirow{4}*{$c_j\not\in J_{\lambda},a_i\in J_{\lambda}$}
& $c_j<J_{\lambda}\setminus J_{\mu}, a_i<J_{\mu}\setminus J_{\lambda}$ & $0$ & $1$ & $1$ \\ \cline{2-5}
& $c_j<J_{\lambda}\setminus J_{\mu}, a_i>J_{\mu}\setminus J_{\lambda}$ & $0$ & $0$ & $0$  \\ \cline{2-5}
& $c_j>J_{\lambda}\setminus J_{\mu}, a_i<J_{\mu}\setminus J_{\lambda}$ & \multicolumn{3}{c|}{cannot occur}     \\ \cline{2-5}
& $c_j>J_{\lambda}\setminus J_{\mu}, a_i>J_{\mu}\setminus J_{\lambda}$ & $1$ & $0$ & $1$  \\ \hline
\multirow{2}*{$c_j,a_i\not\in J_{\lambda}$}
& $c_j<J_{\lambda}\setminus J_{\mu}$ & $0$ & $0$ & $0$ \\ \cline{2-5}
& $c_j>J_{\lambda}\setminus J_{\mu}$ & $1$ & $0$ & $1$ \\ \hline
\end{tabular}
\end{center}
\end{minipage}
\vskip 0.2cm
\begin{minipage}{\textwidth}
\begin{center}
\textbf{Case II}: $J_{\lambda}\setminus J_{\mu}=K_1\sqcup K_2$, $K_1,K_2\not=\emptyset$ and $K_1<J_{\mu}\setminus J_{\lambda}<K_2$
\vskip 0.3cm
\begin{tabular}{|c|c|c|c|c|}
\hline
Membership in $J_{\lambda}$ & Extra condition(s) & $C_j$ & $A_i$ & $c_{\lambda}(\kappa)-c_{\lambda}(\mu)$ \\
\hline
\multirow{4}*{$c_j,a_i\in J_{\lambda}$}
& $c_j<J_{\mu}\setminus J_{\lambda}, a_i<J_{\mu}\setminus J_{\lambda}$ & $-1$ & $1$ & $0$ \\ \cline{2-5}
& $c_j<J_{\mu}\setminus J_{\lambda}, a_i>J_{\mu}\setminus J_{\lambda}$ & $-1$ & $0$ & $-1$  \\ \cline{2-5}
& $c_j>J_{\mu}\setminus J_{\lambda}, a_i<J_{\mu}\setminus J_{\lambda}$ & $0$  & $1$ & $1$    \\ \cline{2-5}
& $c_j>J_{\mu}\setminus J_{\lambda}, a_i>J_{\mu}\setminus J_{\lambda}$ & $0$ & $0$ & $0$  \\ \hline
\multirow{2}*{$c_j\in J_{\lambda},a_i\not\in J_{\lambda}$}
& $c_j<J_{\mu}\setminus J_{\lambda}$ & $-1$ & $0$ & $-1$ \\ \cline{2-5}
& $c_j>J_{\mu}\setminus J_{\lambda}$ & $0$ & $0$ & $0$ \\ \hline
\multirow{2}*{$c_j\not\in J_{\lambda},a_i\in J_{\lambda}$}
& $a_i<J_{\mu}\setminus J_{\lambda}$ & $0$ & $1$ & $1$ \\ \cline{2-5}
& $a_i>J_{\mu}\setminus J_{\lambda}$ & $0$ & $0$ & $0$ \\ \hline
$c_j,a_i\not\in J_{\lambda}$
&  none                             & $0$ & $0$ & $0$ \\ \hline
\end{tabular}
\end{center}
\end{minipage}
\vskip 0.2cm
\begin{minipage}{\textwidth}
\begin{center}
\textbf{Case III}: $J_{\lambda}\setminus J_{\mu}<J_{\mu}\setminus J_{\lambda}$
\vskip 0.3cm
\begin{tabular}{|c|c|c|c|c|}
\hline
Membership in $J_{\lambda}$ & Extra condition(s) & $C_j$ & $A_i$ & $c_{\lambda}(\kappa)-c_{\lambda}(\mu)$ \\
\hline
\multirow{2}*{$c_j,a_i\in J_{\lambda}$}
& $a_i<J_{\mu}\setminus J_{\lambda}$ & $0$ & $0$ & $0$ \\ \cline{2-5}
& $a_i>J_{\mu}\setminus J_{\lambda}$ & $0$ & $-1$ & $-1$ \\ \hline
$c_j\in J_{\lambda},a_i\not\in J_{\lambda}$
&  none                              & $0$ & $-1$ & $-1$ \\ \hline
\multirow{4}*{$c_j\not\in J_{\lambda},a_i\in J_{\lambda}$}
& $c_j<J_{\lambda}\setminus J_{\mu}, a_i<J_{\mu}\setminus J_{\lambda}$ & $0$ & $0$ & $0$ \\ \cline{2-5}
& $c_j<J_{\lambda}\setminus J_{\mu}, a_i>J_{\mu}\setminus J_{\lambda}$ & \multicolumn{3}{c|}{cannot occur}  \\ \cline{2-5}
& $c_j>J_{\lambda}\setminus J_{\mu}, a_i<J_{\mu}\setminus J_{\lambda}$ & $1$ & $0$ & $1$  \\ \cline{2-5}
& $c_j>J_{\lambda}\setminus J_{\mu}, a_i>J_{\mu}\setminus J_{\lambda}$ & $1$ & $-1$ & $0$  \\ \hline
\multirow{2}*{$c_j,a_i\not\in J_{\lambda}$}
& $c_j<J_{\lambda}\setminus J_{\mu}$ & $0$ & $-1$ & $-1$ \\ \cline{2-5}
& $c_j>J_{\lambda}\setminus J_{\mu}$ & $1$ & $-1$ & $0$ \\ \hline
\end{tabular}
\end{center}
\end{minipage}
\vskip 0.2cm
\begin{minipage}{\textwidth}
\begin{center}
\textbf{Case IV}: $J_{\mu}\setminus J_{\lambda}=K_1\sqcup K_2$,
$K_1,K_2\not=\emptyset$ and $K_1<J_{\lambda}\setminus J_{\mu}<K_2$
\vskip 0.3cm
\begin{tabular}{|c|c|c|c|c|}
\hline
Membership in $J_{\lambda}$ & Extra condition(s) & $C_j$ & $A_i$ & $c_{\lambda}(\kappa)-c_{\lambda}(\mu)$ \\
\hline
$c_j,a_i\in J_{\lambda}$
& none                               & $0$ & $0$ & $0$ \\ \hline
\multirow{2}*{$c_j\in J_{\lambda},a_i\not\in J_{\lambda}$}
& $a_i<J_{\lambda}\setminus J_{\mu}$ & $0$ & $0$ & $0$ \\ \cline{2-5}
& $a_i>J_{\lambda}\setminus J_{\mu}$ & $0$ & $-1$ & $-1$ \\ \hline
\multirow{2}*{$c_j\not\in J_{\lambda},a_i\in J_{\lambda}$}
& $c_j<J_{\lambda}\setminus J_{\mu}$ & $0$ & $0$ & $0$ \\ \cline{2-5}
& $c_j>J_{\lambda}\setminus J_{\mu}$ & $1$ & $0$ & $1$ \\ \hline
\multirow{4}*{$c_j,a_i\not\in J_{\lambda}$}
& $c_j<J_{\lambda}\setminus J_{\mu}, a_i<J_{\lambda}\setminus J_{\mu}$ & $0$ & $0$ & $0$ \\ \cline{2-5}
& $c_j<J_{\lambda}\setminus J_{\mu}, a_i>J_{\lambda}\setminus J_{\mu}$ & $0$ & $-1$ & $-1$  \\ \cline{2-5}
& $c_j>J_{\lambda}\setminus J_{\mu}, a_i<J_{\lambda}\setminus J_{\mu}$ & $1$  & $0$ & $1$    \\ \cline{2-5}
& $c_j>J_{\lambda}\setminus J_{\mu}, a_i>J_{\lambda}\setminus J_{\mu}$ & $1$ & $-1$ & $0$  \\ \hline
\end{tabular}
\end{center}
\end{minipage}
\vskip 0.2cm
\end{center}
By the claim we have
\begin{equation}
\sum_{i=1}^r c_{\lambda}(\mu(i))=rc_{\lambda}(\mu)+\sum_{i=1}^r (C_i+A_i)
=\sum_{i=1}^r c_{\lambda}(\mu'(i)).
\label{e:additivity}
\end{equation}
This proves the Proposition in the case $n=m$.

For the general case, note that relabelling strand $i$ as $\reduce{i-m}$ in $D$ for each $i$, we
obtain a new Postnikov diagram $D^{(\m)}$ with a region labelled $J_{\lambda^{(\m)}}$. Furthermore, the partitions corresponding to the $k$-subsets labelling the regions surrounding this region are
$\mu'(1)^{(\m)},\mu(1)^{(\m)},\ldots ,\mu'(r)^{(\m)},\mu(r)^{(\m)}$.
Hence, applying equation~\eqref{e:additivity} with $\lambda$ replaced by $\lambda^{(\m)}$,
we have, for $\m\in [1,n]$:
\begin{equation*}
\begin{split}
\sum_{i=1}^r c_{\lambda}^{(\m)}(\mu(i))
&=\sum_{i=1}^r \left(c_{\lambda^{(\m)}}(\mu(i)^{(m)})-c_{\lambda^{(\m)}}(\emptyset^{(\m)})\right) \\
&=\sum_{i=1}^r \left(c_{\lambda^{(\m)}}(\mu'(i)^{(m)})-c_{\lambda^{(\m)}}(\emptyset^{(\m)})\right)
=\sum_{i=1}^r c_{\lambda}^{(\m)}(\mu'(i)),
\end{split}
\end{equation*}
and we are done.
\end{proof}

\begin{lem} \label{l:vectorfield}
Fix a Postnikov extended cluster $\CC$.
Let $X_{\lambda,\CC}$ and $X_{\lambda,\CC}^{(\m)}$ be the vector fields on $\Xcheck_{\CC}$ defined in equations~\eqref{e:clusterfield} and~\eqref{e:twistedclusterfield}.
Then there is a regular vector field 
$X_{\lambda}$ on $\Xcheck$ such that $X_{\lambda}$ and 
$X_{\lambda,\CC}$ coincide on $\Xcheck_{\CC}$.
Similarly, for each $m\in [1,n]$, there is a regular vector 
field $X^{(\m)}_{\lambda}$ on $\Xcheck$ such that $X^{(\m)}
_{\lambda}$ and $X^{(\m)}_{\lambda,\CC}$ coincide on $\Xcheck_{\CC}$.
\end{lem}

\begin{proof}
By~\cite[Lemma 2.3]{GSV:BaecklundDarboux},
the complement of $\Xcheck_{\CC}\cup \bigcup_{p_\mu\in \CC} \Xcheck_{\CC(\mu)}$ in $\Xcheck$ has codimension at least two.
Hartog's Theorem says that a function that is regular on the complement of a codimension two subvariety of an algebraic variety $X$ extends to a regular function on all of $X$. Hence, it suffices to prove that there is a regular extension of $X_{\lambda,\CC}^{(\m)}$ to
$\Xcheck_{\CC}\cup \bigcup_{p_\mu\in \CC} \Xcheck_{\CC(\mu)}$,
where $\CC(\mu)$ denotes the extended cluster obtained from $\CC$
by mutating at $\p_{\mu}$. We consider first the case
$\m=n$.


Fix $\mu\not=\lambda$ such that $p_{\mu}\in \CC$.
Let $\mu(i),\mu'(i)$, for $i=1,2,\ldots ,r$ be as in
Proposition~\ref{Prop:VectorFieldAdditivity}. 
Then the mutation of $\CC$ at $p_{\mu}$ corresponds to
the change of variables $\tilde{\p}_{\kappa}=\p_{\kappa}$ for
$\p_\kappa\in \CC$, $\p_\kappa\not=\p_\mu$, and
$$\tilde{\p}_{\mu}=\frac{\prod_{i=1}^r \p_{\mu(i)}+\prod_{i=1}^r \p_{\mu'(i)}}{\p_\mu}.$$

Consider the natural vector fields $p_{\kappa} \frac{\partial}{\partial p_{\kappa}}$ on the cluster torus $\Xcheck_{\CC}$ and the vector fields $\widetilde{p}_{\kappa} \frac{\partial}{\partial \widetilde{p}_{\kappa}}$ on the cluster torus $\Xcheck_{\CC(\mu)}$.
On the intersection of these cluster tori we obtain:
\begin{equation*}
\begin{split}
\p_{\mu(i)}\frac{\partial}{\partial \p_{\mu(i)}}&=
\tilde{\p}_{\mu(i)}\frac{\partial{\tilde{\p}_{\mu(i)}}}{\partial \p_{\mu(i)}}
\frac{\partial}{\partial \tilde{\p}_{\mu(i)}}+
\tilde{\p}_{\mu(i)}\frac{\partial \tilde{\p}_{\mu}}{\partial \p_{\mu(i)}}
\frac{\partial}{\partial \tilde{\p}_{\mu}} \\
&=\tilde{\p}_{\mu(i)}\frac{\partial}{\partial \tilde{\p}_{\mu(i)}}+
\tilde{\p}_{\mu(i)}
\frac{\prod_{j=1,j\not=i}^r \p_{\mu(j)}}{\p_{\mu}}
\frac{\partial}{\partial \tilde{\p}_{\mu}} \\
&=\tilde{\p}_{\mu(i)}\frac{\partial}{\partial \tilde{\p}_{\mu(i)}}+
\frac{\prod_{j=1}^r \tilde{\p}_{\mu(j)}}{\prod_{j=1}^r \tilde{\p}_{\mu(j)}+
\prod_{j=1}^r \tilde{\p}_{\mu'(j)}} \tilde{p}_{\mu}\frac{\partial}{\partial \tilde{\p}_{\mu}}.
\end{split}
\end{equation*}
Similarly,
\begin{equation*}
\p_{\mu'(i)}\frac{\partial}{\partial \p_{\mu'(i)}}=\tilde{\p}_{\mu'(i)}\frac{\partial}{\partial \tilde{\p}_{\mu'(i)}}+
\frac{\prod_{j=1}^r \tilde{\p}_{\mu'(j)}}{\prod_{j=1}^r \tilde{\p}_{\mu(j)}+
\prod_{j=1}^r \tilde{\p}_{\mu'(j)}} \tilde{p}_{\mu}\frac{\partial}{\partial \tilde{\p}_{\mu}}.
\end{equation*}
We also have:
\begin{equation*}
\p_{\mu}\frac{\partial}{\partial \p_{\mu}}=
\p_{\mu}\frac{\partial \tilde{\p}_{\mu}}{\partial \p_{\mu}}\frac{\partial}{\partial \tilde{\p}_{\mu}} \\
= -\tilde{p}_{\mu}\frac{\partial}{\partial \tilde{\p}_{\mu}}.
\end{equation*}
Hence, using Proposition~\ref{Prop:VectorFieldAdditivity}, we have:
\begin{equation*}
\begin{split}
\p_{\lambda}\sum_{\p_{\kappa}\in \CC} c_{\lambda}(\kappa)p_{\kappa}\frac{\partial}{\partial \p_{\kappa}}&=
\tilde{\p}_{\lambda}\sum_{\p_{\kappa}\in\CC,p_{\kappa}\not=p_{\mu}}
c_{\lambda}(\kappa)
\tilde{\p}_{\kappa}\frac{\partial}{\partial \tilde{\p}_{\kappa}}+ \\
&\quad\quad\quad\quad\left(\frac{\sum_{i=1}^r c_{\lambda}(\mu(i))
\prod_{j=1}^r \tilde{\p}_{\mu(j)}+
\sum_{i=1}^r c_{\lambda}(\mu'(i))
\prod_{j=1}^r \tilde{\p}_{\mu'(j)}}{\prod_{j=1}^r \tilde{\p}_{\mu(j)}+
\prod_{j=1}^r \tilde{\p}_{\mu'(j)}}-c_{\lambda}(\mu)\right)
\tilde{\p}_{\lambda}\tilde{\p}_{\mu}\frac{\partial}{\partial \tilde{\p}_{\mu}} \\
&=\tilde{\p}_{\lambda}\sum_{\p_{\kappa}\in\CC,\p_{\kappa}\not=\p_{\mu}}
c_{\lambda}(\kappa)\tilde{\p}_{\kappa}\frac{\partial}{\partial \tilde{\p}_{\kappa}}+
\left(\left(\sum_{i=1}^r c_{\lambda}(\mu(i))\right)-c_{\lambda}(\mu)\right)
\tilde{\p}_{\lambda}\tilde{\p}_{\mu}\frac{\partial}{\partial \tilde{\p}_{\mu}}.
\end{split}
\end{equation*}
This is regular on the cluster torus
$\Xcheck_{\CC(\mu)}$.

For the case $\mu=\lambda$, we have
\begin{equation*}
\p_{\lambda}\p_{\lambda(i)}\frac{\partial}{\partial p_{\lambda(i)}}
=\p_{\lambda}\tilde{\p}_{\lambda(i)}
\frac{\partial}{\partial \tilde{p}
_{\lambda(i)}}+
\left( \prod_{j=1}^r \tilde{\p}_{\lambda(j)} 
\right)
\frac{\partial}{\partial \tilde{\p}
_{\lambda}},
\end{equation*}
and, similarly:
\begin{equation*}
\p_{\lambda}\p_{\lambda'(i)}\frac{\partial}{\partial p_{\lambda'(i)}}
=\p_{\lambda}\tilde{\p}_{\lambda'(i)}\frac{\partial}{\partial \tilde{p}_{\lambda'(i)}}+
\left( \prod_{j=1}^r \tilde{\p}_{\lambda'(j)} \right)
\frac{\partial}{\partial \tilde{\p}_{\lambda}}.\\
\end{equation*}
As before, we have
\begin{equation*}
\p_{\lambda}\frac{\partial}{\partial \p_{\lambda}}
= -\tilde{\p}_{\lambda}\frac{\partial}{\partial \tilde{\p}_{\lambda}}.
\end{equation*}
Hence, we have
\begin{equation*}
\begin{split}
\p_{\lambda} \sum_{\p_{\kappa}\in \CC}
c_{\lambda}(\kappa)\p_{\kappa}\frac{\partial}{\partial \p_{\kappa}}&=
\p_{\lambda} \sum_{\tilde{p}_{\kappa}\in \CC(\lambda),\tilde{\p}_{\kappa}\not=\tilde{p}_{\lambda}}c_{\lambda}(\kappa)\tilde{\p}_{\kappa}\frac{\partial}{\partial \tilde{\p}_{\kappa}}
+ \\
& \quad \quad \quad \quad
\left(\sum_{i=1}^r c_{\lambda}(\lambda(i))\prod_{j=1}^r \tilde{\p}_{\lambda(j)}+
\sum_{i=1}^r c_{\lambda}(\lambda'(i))\prod_{j=1}^r \tilde{\p}_{\lambda'(j)}-
\tilde{\p}_{\lambda}
\right) \frac{\partial}{\partial \tilde{\p}_{\lambda}} \\
&=
\frac{
 \left(
  \prod_{j=1}^r \tilde{\p}_{\lambda(j)}+
  \prod_{j=1}^r \tilde{\p}_{\lambda'(j)}
 \right)
 }
{\tilde{\p}_{\lambda}}
\sum_{\tilde{\p}_{\kappa}\in\CC(\lambda),\tilde{\p}_{\kappa}\not=\tilde{\p_{\lambda}}}
c_{\lambda}(\kappa)\tilde{\p}_{\kappa}
\frac{\partial}{\partial\tilde{\p}_{\kappa}}+ \\
& \quad \quad \quad \quad
\left(\sum_{i=1}^r c_{\lambda}(\lambda(i))\prod_{j=1}^r \tilde{\p}_{\lambda(j)}+
\sum_{i=1}^r c_{\lambda}(\lambda'(i))\prod_{j=1}^r \tilde{\p}_{\lambda'(j)}-
\tilde{\p}_{\lambda}
\right) \frac{\partial}{\partial \tilde{\p}_{\lambda}}.
\end{split}
\end{equation*}

This is regular on the cluster torus
$\Xcheck_{\CC(\lambda)}$.
Hence there is a regular extension of $X_{\lambda,\CC}$ to
$\Xcheck_{\CC}\cup \bigcup_{\mu\in \CC} \Xcheck_{\CC(\mu)}$
as required.


A similar argument can be used for $X_{\lambda,\CC}^{(\m)}$, using the additivity given by Proposition~\ref{Prop:VectorFieldAdditivity}.
\end{proof}

Note that we have not shown that $X_{\lambda}^{(\m)}$ is independent of the choice of initial extended cluster $\CC$. We expect this to hold, but we do not need it here.
For $X_{\lambda}^{(\m)}$, we shall always choose $\CC$ to be the
extended cluster corresponding to the Postnikov
diagram $D_{\lambda}$ associated to
$\lambda$ (see Theorem~\ref{thm:diagramconstruction}).

We now fix a partition $\lambda\in \Pn_{k,n}$ and focus on the Postnikov diagram
$D_{\lambda}$. This diagram contains a face $F(\lambda)$ labelled $J_{\lambda}$.
Proposition~\ref{Prop:VectorFieldAdditivity}
applies to every face $F$ of $D_{\lambda}$
except $F(\lambda)$, but we need a version
of this proposition for the case $F=F(\lambda)$ also. 
We first introduce more notation.

\begin{defn} \label{d:partitionlabels}
Fix a partition $\lambda\in \Pn_{k,n}$.
Set
$$K_{\lambda}=\{i\in J_{\lambda}\,:\,i+1\not\in J_{\lambda}\}.$$
For $i\in K_{\lambda}$, let
$\addbox{\lambda}{i}$ be the partition
with the property that
\begin{equation}\label{e:Jlambdataui} 
J_{\addbox{\lambda}{i}}=(J_{\lambda}\setminus \{i\})\cup \{i+1\}.
\end{equation}
This is the partition which corresponds to adding a box to $\lambda$ adjacent to the edge numbered $i$ in the associated path (see Section~\ref{s:Plucker}).

If $F(\lambda)$ is internal in $D_\lambda$ then in terms of the notation in Definition~\ref{d:adjacentfaces} we have
$$K_{\lambda}=\{a_1<\cdots <a_r\},$$
and
$\addbox{\lambda}{a_j}=\lambda'(j)$ for $j=1,\ldots ,r$, by the construction of $D_{\lambda}$ in
Section~\ref{s:specialdiagram}. In particular \eqref{e:Jlambdataui} reflects the fact that for $a_j=i$ the adjacent clockwise strand $c_{j-1}$ around $F(\lambda)$  is $i+1$.  For the other adjacent faces $\lambda(j)$ we note that 
\begin{equation}\label{e:Jlambdaj}
J_{\lambda(j)}=(J_{\lambda}\setminus \{a_j\})\cup \{a_{j+1}+1\}.
\end{equation}
since $c_j=a_{j+1}+1$ for $F(\lambda)$.
If $F(\lambda)$ is on the boundary then $J_{\lambda}=L_i$ and 
$K_{\lambda}=\{i\}$ and the unique alternating face adjacent to $F(\lambda)$ is labelled with $(J_{\lambda}\setminus \{i\})\cup \{i+1\}$, which equals $\widehat{L}_i$. The corresponding partition is $\addbox{\lambda}{i}$.
\end{defn}

\begin{defn} \label{d:kappa}
In the case $k\not=1,n-1$, we have defined (see Definition~\ref{d:Fi}) the faces $\Fi$ and $\Fip$ to be the faces
in $D_{\lambda}$ adjacent to the unique crossing point $\PPi$ of strands $i$ and $i+1$.
If $k=1$ or $n-1$, there may be more than
one such crossing point, but we may define
$\Fi$ and $\Fip$ in the same way: the
definition does not depend on $\PPi$.

The faces $\Fi$ and $\Fip$ have labels $I_i$ and $I'_i$ respectively.
Let $\kappa_{i,\lambda}$ and $\kappa'_{i,\lambda}$
be the corresponding partitions.
\end{defn}

Our aim is to prove equation~
\eqref{e:vfequation}; i.e.\ to compute
$X_{\lambda}^{(\m)}W_q$. Thus we need to compute
$X_{\lambda}\frac{p_{\widehat{L}_i}}{p_{L_i}}$
for each $i\in [1,n]$.
We saw in Theorem~\ref{t:pmproperties} that
$$\frac{p_{\widehat{L}_i}}{p_{L_i}}=
\frac{\p_{I'_i}}{\p_{I_i}}+\cdots $$
where $\frac{\p_{I'_i}}{\p_{I_i}}$ is
the summand corresponding to $\Mi$ in the formula~\eqref{e:Lihatterm}.
In order to compute the action of $X_{\lambda}^{(\m)}$ on
$\frac{\p_{I'_i}}{\p_{I_i}}$,
we need to know the value of
$c^{(\m)}_{\lambda}(\kappa'_{i,\lambda})-c^{(\m)}
_{\lambda}(\kappa_{i,\lambda})$
for each $i$.
We first consider the case where $i\in K_{\lambda}$,
observing the following (which follows from
the construction of $D_{\lambda}$).
 
\begin{rem} \label{r:coincides}
If $i\in K_{\lambda}$, the faces $F(\lambda)$ and $\Fi$ coincide. We have $\kappa_{i,\lambda}=\lambda$
and $\kappa'_{i,\lambda}=\addbox{\lambda}{i}$,
since $I'_i=(I_i\setminus \{i\})\cup \{i+1\}$. See the construction in Section~\ref{s:specialdiagram}.
\end{rem}


\begin{lem} \label{l:caddabox}
Fix a partition $\lambda\in \Pn_{k,n}$, and let $D=D_{\lambda}$.
Then, for $i\in K_{\lambda}$ and $\m\in [1,n]$, we have
\begin{equation*}
c^{(\m)}_{\lambda}(\addbox{\lambda}{i})-c^{(\m)}_{\lambda}(\lambda)=
\begin{cases}
0, & \text{if $i=\m$}; \\
1, & \text{if $i\not=m$}.
\end{cases}
\end{equation*}
\end{lem}
\begin{proof}
Recall that, for any partition $\mu$, we have
$c_{\mu}^{(n)}(\mu)=c_{\mu}(\mu)=0$.
We first show that:
\begin{equation} \label{e:partbeqn}
c_{\lambda}(\addbox{\lambda}{i})=
\begin{cases}
0, & \text{if $i=n$}; \\
1, & \text{otherwise}.
\end{cases}
\end{equation}
For $i\in K_{\lambda}$, we have
$J_{\lambda}\setminus J_{\addbox{\lambda}{i}}=\{i\}$ and $J_{\addbox{\lambda}{i}}\setminus J_{\lambda}=\{i+1\}$.
If $i\not=n$, then $i+1>i$, so
$c_{\lambda}(\addbox{\lambda}{i})=1$.
If $i=n$, then $i+1=1<n=i$, so
$c_{\lambda}(\addbox{\lambda}{i})=0$, and~\eqref{e:partbeqn} is shown.

Using the definition (see equation~\eqref{e:clambdam}) of $c_{\lambda}^{(\m)}$, we have that
\begin{equation*}
\begin{split}
c_{\lambda}^{(\m)}\left(\addbox{\lambda}{i}\right)-
c_{\lambda}^{(\m)}(\lambda) &=
c_{\lambda^{(\m)}}\left(\addbox{\lambda}{i}^{(\m)}\right)-
c_{\lambda^{(\m)}}\left(\lambda^{(\m)}\right) \\
&=c_{\lambda^{(\m)}}\left(\addbox{\lambda^{(\m)}}{i-\m}\right),
\end{split}
\end{equation*}
where the last equality follows from~\eqref{e:Jlambdataui}.
The result then follows from~\eqref{e:partbeqn}
(replacing $\lambda$ with $\lambda^{(\m)}$).
\end{proof}

Using the definition of $X_{\lambda}^{(\m)}$ (see equation~\eqref{e:twistedclusterfield} and Lemma~\ref{l:vectorfield})
Lemma~\ref{l:caddabox} implies that,
for $i\in K_{\lambda}$,
\begin{equation} \label{e:iinKlambda}
X^{(\m)}_{\lambda} \frac{\p_{I'_i}}{\p_{I_i}}
=
\begin{dcases}
0, & i=\m; \\
\frac{\p_{I'_i}}{\p_{I_i}}p_{\lambda}, & i\not=\m.
\end{dcases}
\end{equation}

In order to compute the action of $X_{\lambda}^{(\m)}$ on the summands in the formula~\eqref{e:vfequation}
corresponding to the matchings other than
$\Mi$ (in the case $i\in K_{\lambda}$), we will need the following.

\begin{prop} \label{p:nearlambda}
Assume that $F(\lambda)$ is an internal face of $D_{\lambda}$.
Then, for any $\m\in [1,n]$, we have:
$$\sum_{j=1}^r c_{\lambda}^{(\m)}(\lambda'(j)) - \sum_{j=1}^r c_{\lambda}^{(\m)}(\lambda(j))=1.$$
\end{prop}

\begin{proof}
We first prove that:
\begin{equation} \label{e:nearlambda}
\sum_{j=1}^r c_{\lambda}(\lambda'(j)) - \sum_{j=1}^r c_{\lambda}(\lambda(j))=1.
\end{equation}

Recall that
$$K_{\lambda}=\{a_1<a_2<\cdots <a_r\}.$$
We will interpret the subscripts $j$ of the $a_j$ modulo $r$. By our assumption on $J_{\lambda}$, we have $r>1$.

By equation~\eqref{e:partbeqn} in
Lemma~\ref{l:caddabox}, we have
\begin{equation}
\label{e:lambdacase}
\sum_{j=1}^r c_{\lambda}\left(\lambda'(j)\right)=
\sum_{j=1}^r c_{\lambda}\left(\addbox{\lambda}{a_j}\right)=
\begin{cases}
r, & \text{if } a_r\not=n; \\
r-1, & \text{if } a_r=n.
\end{cases}
\end{equation}
We will show that
\begin{equation}
\label{e:lambdaprimecase}
\sum_{j=1}^r c_{\lambda}\left(\lambda(j)\right)=
\begin{cases}
r-1, & \text{if } a_r\not=n; \\
r-2, & \text{if } a_r=n;
\end{cases}
\end{equation}
from which~\eqref{e:nearlambda} follows.

By~\eqref{e:Jlambdaj}, we have $J_{\lambda}\setminus J_{\lambda(j)}=\{a_j\}$ and
$J_{\lambda(j)}\setminus J_{\lambda}=\{a_{j+1}+1\}$.

Suppose first that $a_r\not=n$.
If $1\leq j\leq r-1$
then $a_{j+1}+1>a_j$, so $c_{\lambda}(\lambda(j))=1$.
In the case $j=r$, we have
$$a_{r+1}+1=a_1+1<a_2\leq a_r,$$
as $a_1+1\not\in K_{\lambda}$ and $r\geq 2$.
Hence, $c_{\lambda}(\lambda(r))=0$.
We obtain $\sum_{j=1}^r c_{\lambda}(\lambda(j))=r-1$,
as required.

We are left with the case $a_r=n$.
If $1\leq j\leq r-2$, then $a_{j+1}+1>a_j$,
so $c_{\lambda}(\lambda(j))=1$. We have, for the case
$j=r-1$ (reducing modulo $n$ as usual)
$$a_{(r-1)+1}+1=a_r+1=n+1=1<a_{r-1},$$
since $r-1\geq 1$ and $1\not\in K_{\lambda}$.
Hence $c_{\lambda}(\lambda(r-1))=0$.
For the case $j=r$, we have
$$a_{r+1}+1=a_1+1<a_2\leq a_r=n,$$
so $a_{r+1}+1<a_r$, and $c_{\lambda}(\lambda(r))=0$.
We obtain $\sum_{i=1}^r c_{\lambda}(\lambda(j))=r-2$.

We have now shown that~\eqref{e:nearlambda} holds,
which is the case $\m=n$ of the proposition.
For the general case, we have, using the definition of $c_{\lambda}^{(\m)}$,
\begin{equation}
\sum_{j=1}^r c_{\lambda}^{(\m)}\left(\lambda'(j)\right) - \sum_{j=1}^r c_{\lambda}^{(\m)}(\lambda(j))
=
\sum_{j=1}^r c_{\lambda^{(\m)}}\left(\lambda'(j)^{(\m)}\right)-
\sum_{j=1}^r c_{\lambda^{(\m)}}\left(\lambda(j)^{(\m)}\right).
\end{equation}
By the construction of $D_{\lambda}$ (see Remark~\ref{r:rotateddiagrams}), the sets
$\{\lambda(j)^{(\m)}\,:\,j=1,\ldots ,r\}$ and
$\{\lambda^{(\m)}(j)\,:\,j=1,\ldots ,r\}$ coincide, and the sets
$\{\lambda'(j)^{(\m)}\,:\,j=1,\ldots ,r\}$ and
$\{(\lambda')^{(\m)}(j)\,:\,j=1,\ldots ,r\}$ coincide.
The result follows.
\end{proof}

To obtain the analogue of equation~\eqref{e:iinKlambda}
in the case $i\not\in K_{\lambda}$, we need to compute
$c^{(\m)}_{\lambda}\left(\kappa'_{i,\lambda}\right)-c^{(\m)}_{\lambda}(\kappa_{i,\lambda})$
for $i\not\in K_{\lambda}$; this is covered
by the following proposition.

\begin{prop} \label{p:inotK}
Fix a partition $\lambda\in \Pn_{k,n}$.
Let $D=D_{\lambda}$ be the Postnikov diagram constructed in Theorem~\ref{thm:diagramconstruction}.
Suppose that $i\not\in K_{\lambda}$.
Then we have:
$$c_{\lambda}^{(\m)}\left(\kappa'_{i,\lambda}\right)-c_{\lambda}^{(\m)}(\kappa_{i,\lambda})=
\begin{cases}
-1, & i=\m; \\
0, & i\not=\m.
\end{cases}
$$
\end{prop}
\begin{proof}
We first consider the case $\m=n$, i.e.\ we
show that:
\begin{equation} \label{e:casemn}
c_{\lambda}\left(\kappa'_{i,\lambda}\right)-c_{\lambda}(\kappa_{i,\lambda})=
\begin{cases}
-1, & i=n; \\
0, & i\not=n.
\end{cases}
\end{equation}

We divide the proof into three cases, depending on whether $i,i+1$
lie in $J_{\lambda}$ or not (note that, since
we assume $i\not\in K_{\lambda}$, we cannot have
the case $i\in J_{\lambda},i+1\not\in J_{\lambda}$).

\noindent
\textbf{Case (a)} $i,i+1\not \in J_{\lambda}$. \\
Then we have:
\begin{align*}
I'_i \setminus J_{\lambda} &=((I_i\setminus J_{\lambda}) \setminus \{i\})\cup \{i+1\}; \\
J_{\lambda} \setminus I'_i &= J_{\lambda}\setminus I_i.
\end{align*}
If $i\not=n$, then, since $i,i+1\not\in J_{\lambda}\setminus I_i$
and $i,i+1\not\in J_{\lambda}\setminus I'_i$,
we have that $c_{\lambda}(\kappa_{i,\lambda})=c_{\lambda}\left(\kappa'_{i,\lambda}\right)$, as required.
We now consider the case $i=n$.
Since $I_n$ and $J_{\lambda}$ are noncrossing and $n\in I_n\setminus J_{\lambda}$, we may write $I_n\setminus J_{\lambda}=K_1\sqcup K_2$, where $K_1<J_{\lambda}\setminus I_n<K_2$ and $n\in K_2$. Note that $K_1$ may be empty.
We have $c_{\lambda}(\kappa_{n,\lambda})=|K_2|$.
Let $K'_1=K_1\cup \{1\}$
and $K'_2=K_2\setminus \{n\}$.
Then we have
$I'_n\setminus J_{\lambda}=K'_1\sqcup K'_2$,
with $K'_1<J_{\lambda}\setminus I_n<K'_2$.
Note that $K'_1$ is nonempty, but $K'_2$ may
be empty.
We then have $c_{\lambda}\left(\kappa'_{n,\lambda}\right)=|K'_2|=|K_2|-1=
c_{\lambda}(\kappa_{n,\lambda})-1$ as required.

\noindent \textbf{Case (b)} $i\not \in J_{\lambda}, i+1\in J_{\lambda}.$ \\
Then we have $i\in I_i\setminus J_{\lambda}$,
$i+1\in J_{\lambda}\setminus I_i$, and:
\begin{align*}
I'_i \setminus J_{\lambda} &=(I_i\setminus J_{\lambda}) \setminus \{i\}; \\
J_{\lambda} \setminus I'_i &= (J_{\lambda}\setminus I_i)\setminus \{i+1\}.
\end{align*}
Since $I_i$ and $J_{\lambda}$ are noncrossing, we must have one of the
following three cases.


\noindent \emph{Case (b)(i)}: $J_{\lambda}\setminus I_i=K_1\sqcup K_2$, $K,_1$ and $K_2$ are nonempty
and $K_1<I_i\setminus J_{\lambda}<K_2$.
We have $c_{\lambda}(\kappa_{i,\lambda})=|K_1|$.
Since $i\in I_i\setminus J_{\lambda}$ and
$i+1\in J_{\lambda}\setminus I_i$, we must have
$i+1\in K_2$.
Let $K'_2=K_2\setminus \{i+1\}$.
Then we have $J_{\lambda}\setminus I'_i=K_1\sqcup K'_2$, where $K_1<I'_i\setminus J_{\lambda} <K'_2$. Note that $K'_2$ may be empty. 
We see that
$c_{\lambda}\left(\kappa'_{i,\lambda}\right)=|K_1|=c_{\lambda}(\kappa_{i,\lambda})$,
as required
(note that $i<i+1$ in this case, so $i\not=n$).

\noindent \emph{Case (b)(ii)}: $J_{\lambda} \setminus 
I_i<I_i\setminus J_{\lambda}$.
Then, since $i\in I_i\setminus J_{\lambda}$
and $i+1\in J_{\lambda}\setminus I_i$, we have
$i=n$. We have $J_{\lambda} \setminus I'_i<I'_i\setminus 
J_{\lambda}$,
so $c_{\lambda}\left(\kappa'_{i,\lambda}\right)=c_{\lambda}
(\kappa_{i,\lambda})-1$,
as required.

\noindent \emph{Case (b)(iii)}:
$I_i\setminus J_{\lambda}=K_1\sqcup K_2$, $K_1$ is
nonempty and $K_1<J_{\lambda}\setminus I_i<K_2$.
We have $c_{\lambda}(\kappa_{i,\lambda})=|K_2|$.
Since $i\in I_i\setminus J_{\lambda}$ and
$i+1\in J_{\lambda}\setminus I_i$ we have $i\in K_1$.
Let $K'_1=K_1\setminus \{i\}$.
Then we have
$I'_i\setminus J_{\lambda}=K'_1\sqcup K_2$,
where $K'_1<J_{\lambda}\setminus I'_i<K_2$.
Note that $K'_1$ may be empty. We see that
$c_{\lambda}\left(\kappa'_{i,\lambda}\right)=|K_2|=c_{\lambda}(\kappa_{i,\lambda})$
as required (note that $i<i+1$ in this case so $i\not=n$).

This completes case (b).

\noindent
\textbf{Case (c)} $i,i+1\in J_{\lambda}$.
Then we have:
\begin{align*}
I'_i \setminus J_{\lambda} &=I_i\setminus J_{\lambda}; \\
J_{\lambda} \setminus I'_i &= ((J_{\lambda}\setminus I_i)\setminus \{i+1\})\cup \{i\}.
\end{align*}
If $i\not=n$, then, since $i,i+1\not\in J_{\lambda}\setminus I_i$
and $i,i+1\not\in J_{\lambda}\setminus I'_i$,
we have that $c_{\lambda}(\kappa_{i,\lambda})=c_{\lambda}\left(\kappa'_{i,\lambda}\right)$, as required.
We now consider the case $i=n$.
Since $I_n$ and $J_{\lambda}$ are noncrossing and $1\in J_{\lambda}\setminus I_n$, we may write $J_{\lambda}\setminus I_n=K_1\sqcup K_2$, where $K_1<I_n\setminus J_{\lambda}<K_2$ and $1\in K_1$. Note that $K_2$ may be empty.
We have $c_{\lambda}(\kappa_{n,\lambda})=|K_1|$.
Let $K'_1=K_1\setminus \{1\}$ and $K'_2=K_2\cup \{n\}$.
We have
$J_{\lambda}\setminus I'_i=K'_1\sqcup K'_2$,
where $K'_1<I'_i\setminus J_{\lambda}<K'_2$.
Note that $K'_2$ is nonempty, but $K'_1$ may be
empty.
We then have $c_{\lambda}\left(\kappa'_{n,\lambda}\right)=|K'_1|=|K_1|-1=
c_{\lambda}(\kappa_{n,\lambda})-1$ as required.
This completes case (c).

We have completed the proof of~\eqref{e:casemn}, and thus the proof of the statement required for $\m=n$.

The result for arbitrary $\m$ follows
from Remark~\ref{r:rotateddiagrams} and~\eqref{e:casemn} as follows, using the
definition~\eqref{e:clambdam} of
$c_{\lambda}^{(\m)}$:
\begin{equation*}
\begin{split}
c^{(\m)}_{\lambda}\left(\kappa'_{i,\lambda}\right)-c^{(\m)}_{\lambda}(\kappa_{i,\lambda}) &=
c_{\lambda^{(\m)}}\left((\kappa'_{i,\lambda})^{(\m)}\right)-c_{\lambda^{(\m)}}\left(\kappa_{i,\lambda}^{(\m)}\right) \\
&= c_{\lambda^{(\m)}}\left(\kappa'_{i-\m,\lambda^{(\m)}}\right)-c_{\lambda^{(\m)}}\left(\kappa_{{i-\m},\lambda^{(\m)}}\right) \\
&=
\begin{cases}
-1, & i-\m=n \mod n; \\
0, & i-\m\not=n \mod n;
\end{cases}
\\
&=
\begin{cases}
-1, & i=\m; \\
0, & i\not=\m; \\
\end{cases}
\end{split}
\end{equation*}
as required.
\end{proof}

Proposition~\ref{p:inotK} implies that, for $i\not\in K_{\lambda}$, we have
\begin{equation} \label{e:inotinKlambda}
X^{(\m)}_{\lambda} \frac{\p_{I'_i}}{\p_{I_i}}
=
\begin{dcases}
-\frac{\p_{I'_i}}{\p_{I_i}}p_{\lambda}, & i=\m; \\
0, & i\not=\m.
\end{dcases}
\end{equation}

\section{Action of the vector field $X_{\lambda}$ on $W_q$}
\label{s:actionvectorfield2}
Fix a partition $\lambda\in \Pn_{k,n}$.
Our main aim in this section is to compute the action of $X^{(m)}_{\lambda}$
on $W$ for each $m\in [1,n]$ (Theorem~\ref{t:nonequivariantaction}). We include
the cases $k=1,n-1$.
Let $D=D_{\lambda}$ be the Postnikov diagram associated to $\lambda$ constructed in
Theorem~\ref{thm:diagramconstruction}.
Recall that $D$ has a face $F(\lambda)$
labelled $J_{\lambda}$.
Let $\CC$ denote the Postnikov extended cluster associated to $D$. 

We collect together the information we
will need. 
Recall first that, on the cluster torus $\Xcheck_{\CC}$, the vector field
$X_{\lambda}^{(\m)}$ is given by the
following formula (see~\eqref{e:clusterfield}).
\begin{equation} \label{e:recallvectorfield}
X_{\lambda}^{(\m)}=\p_{\lambda} \sum_{\p_{\mu}\in \CC}
c_{\lambda}^{(\m)}(\mu)\p_{\mu}\frac{\partial}{\partial \p_{\mu}}.
\end{equation}
Recall also that the superpotential $W_q$
on $\Xcheck$ is given by:
\begin{equation} \label{e:recallsuperpotential}
W_q=\sum_{i=1}^n q^{\delta_{i,n}} \frac{\p_{\widehat{L}_i}}{\p_{L_i}},
\end{equation}
where, for $i\in [1,n]$,
$L_i=[i-k+1,i]$ and
$\widehat{L}_i=[i-k+1,i-1]\cup \{i+1\}$.

Let $D_1,\ldots ,D_n$ be the Postnikov diagrams associated to $D$ in Definition~\ref{d:Di}, with corresponding weighted dual bipartite graphs $G_1,G_2,\ldots ,G_n$.

Let $$\coeff=\frac{\p_{L_{i-1}}\p_{L_{i+1}}\cdots \p_{L_{i+k}}}{\cluster}.$$
Then, by Theorem~\ref{thm:MarSco}, we have
the following formula in~$\C[\Xcheck_{\CC}]$:
\begin{equation} \label{e:recallMS}
\frac{\p_{\widehat{L}_i}}{\p_{L_i}}=\sum_{M} \coeff w_M,
\end{equation}
where the sum is over all perfect matchings $M$ of $G_i$ and
$w_M$ denotes the matching monomial associated to a perfect matching $M$.

Suppose first that $k\not=1,n-1$.
Recall from Definition~\ref{d:Fi}
that $\Fi$ is the alternating face
which is to the left of strand $i$ and to the right of strand $i+1$ and adjacent to the crossing point
$\PPi$ of strands $i$ and $\reduce{i+1}$ .
It has label $I_i$. Furthermore, $\Fip$ is the alternating face adjacent to $\PPi$ on the other side of $\gamma_i$, i.e.\ to the right of strand $i$ and to the left of strand $i+1$, with label $I'_i$.
Recall also that:
$$K_{\lambda}=\{i\in J_{\lambda}\,:\,i+1\not\in J_{\lambda}\}.$$

\begin{rem} \label{r:Flambda}
If $i\in K_{\lambda}$ we
have seen (Remark~\ref{r:coincides}) that
the faces $F(\lambda)$ and $F_i$ coincide.
If $i\not\in K_{\lambda}$ then, since the label of every face of $\Giin$ contains $i$ but not $i+1$, the face $F(\lambda)$ cannot be a face of $\Giin$.
\end{rem}

By Corollary~\ref{c:elementarycomponents} that
the elementary components of $G_i$
consist of $\Giin$ (see Definition~\ref{d:Giin}) together with certain
single edges on the remaining vertices.
By Theorem~\ref{t:pmproperties}, if $\Fi$ is an 
internal face, then the perfect matchings on $G_i$ 
are obtained from an initial matching, $\Mi$, in 
which $\Fi$ is flippable, namely by flipping $\Fi$ and then performing a sequence of flips (possibly empty) not involving
$\Fi$. The flips all take place in faces
of $\Giin$.
If $\Fi$ is a boundary face, there is a unique perfect matching $\Mi$
on $G_i$. 

If $k=1$ or $n-1$, recall that $\Fi$ is defined in the  same way as above (and doesn't depend on a 
choice of crossing point $\PPi$ of strands 
$i$ and $i+1$). If $i\in K_{\lambda}$ then
the faces $F(\lambda)$ and $F_i$ coincide.
For all $i$, we have that $\Fi$ is a boundary face and  there is a unique matching $\Mi$ on $G_i$
(see Remark~\ref{r:Ximatchingk1}).
Each elementary component of $G_i$ consists of
a single edge $e$ where $e$ is an edge in $\Mi$
(see Corollary~\ref{c:elementarycomponents}).

For any $k$ we have, by Theorem~\ref{t:pmproperties}:
\begin{equation} \label{e:initialshort}
\coeff w_{\Mi}=\frac{\p_{I'_i}}{\p_{I_i}}
\end{equation}
for all $i\in [1,n]$.

We also have the following (see~\eqref{e:iinKlambda} and~\eqref{e:inotinKlambda}):
\begin{equation} \label{e:firstterm}
X^{(\m)}_{\lambda} \frac{\p_{I'_i}}{\p_{I_i}}
=
\begin{dcases}
0, & i=\m, i\in K_{\lambda}; \\
\frac{\p_{I'_i}}{\p_{I_i}}p_{\lambda}, & i\not=\m, i\in K_{\lambda}; \\
-\frac{\p_{I'_i}}{\p_{I_i}}p_{\lambda}, & i=\m, i\not\in K_{\lambda}; \\
0, & i\not=\m, i\not\in K_{\lambda}.
\end{dcases}
\end{equation}


\begin{lem} \label{lem:flipchangeinw}
Let $M,M'$ be perfect matchings on the
bipartite graph $G$ dual to a Postnikov
diagram $D$, with $M'$ obtained from $M$ by flipping around an $M$-flippable face $F_0$ of $G$. Then
$$\frac{w_{M'}}{w_{M}}=\frac{\prod_{F} \p_{F}}{\prod_{F'} \p_{F'}},$$
where the product in the numerator is over the faces
$F$ of $G$ sharing an edge in $M$ with $F_0$ and the product in the denominator is over the
faces $F'$ of $G$ sharing an edge in $M'$ with
$F_0$.
\end{lem}
\begin{proof}
This is easily seen to hold by considering the weights on the edges in $G$ around $F_0$.
\end{proof}

We now have all the ingredients we need to prove the main result of this section.

\begin{thm} \label{t:nonequivariantaction}
Let $\lambda$ be an arbitrary Young diagram in $\Pn_{k,n}$ and $m\in [1,n]$. Then we have:
\begin{equation} \label{e:noneq}
X^{(\m)}_{\lambda}W_q=\left(
\sum_{\mu}\p_{\mu}+q\sum_{\nu}\p_{\nu}\right) - q^{\delta_{mn}}\frac{\p_{\widehat{L}_m}}{\p_{L_m}}\p_{\lambda},
\end{equation}
where $\mu$, $\nu$ are exactly as in the quantum version of Monk's rule for $\sigma^{\Box}*_q \sigma^{\lambda}$.
\end{thm}

\begin{proof}
It suffices to check this on the cluster
torus $\Xcheck_{\CC}$, since it is open
dense in $\Xcheck$.
Note that~\eqref{e:noneq} can be rewritten as:
\begin{align} \label{e:noneqrewrite}
X_\lambda^{(\m)}W_q
=\left( \sum_{i\in K_{\lambda}}q^{\delta_{in}}\p_{\addbox{\lambda}{i}} \right)
-q^{\delta_{mn}}\frac{\p_{\widehat{L}_m}}{\p_{L_m}}\p_{\lambda}.
\end{align}
As in Definition~\ref{d:partitionlabels}, we write
$$K_{\lambda}=\{a_1<a_2<\cdots <a_r\}.$$
We will show that
\begin{equation} \label{e:wewillshowthat}
X_{\lambda}^{(\m)}\frac{\p_{\widehat{L}_i}}{\p_{L_i}}=
\begin{dcases}
\p_{\addbox{\lambda}{i}}-\frac{\p_{\widehat{L}_i}}{\p_{L_i}}\p_{\lambda}, & i=m, i\in K_{\lambda}; \\
\p_{\addbox{\lambda}{i}}, & i\not=m, i\in K_{\lambda}; \\
-\frac{\p_{\widehat{L}_i}}{\p_{L_i}}\p_{\lambda}, & i=m, i\not\in K_{\lambda}; \\
0, & i\not=m, i\not\in K_{\lambda}.
\end{dcases}
\end{equation}
The result follows from this, since then we have:
\begin{align*}
X_\lambda^{(\m)}W_q &= X_{\lambda}\sum_{i=1}^n q^{\delta_{in}}
\frac{\p_{\widehat{L}_i}}{\p_{L_i}} \\
&=\left( \sum_{i\in K_{\lambda}}q^{\delta_{in}}\p_{\addbox{\lambda}{i}} \right)
-q^{\delta_{mn}}\frac{\p_{\widehat{L}_m}}{\p_{L_m}}\p_{\lambda}.
\end{align*}

We divide the proof into three cases:
the first with $F_i$ an internal face of $D$ and
$i\in K_{\lambda}$, the second with $F_i$ an
internal face of $D$ and $i\not\in K_{\lambda}$, and the third with $F_i$ a boundary face of $D$.

\noindent\textbf{Case I}: Suppose that $F_i$ is an
internal face of $D$ and $i\in K_{\lambda}$.
By Remark~\ref{r:Flambda}, $F(\lambda)=\Fi$, with label $I_i=J_{\lambda}$;
by assumption this is not a boundary face of $D$.
The face $\Fip$ is labelled with $I'_i=J_{\tau_i(\lambda)}$.
(see Definition~\ref{d:partitionlabels}).
By~\eqref{e:initialshort} and~\eqref{e:firstterm}
and, we have:
\begin{equation}
\label{e:Ifirstpart}
X^{(\m)}_{\lambda} \left(\coeff w_{\Mi}\right)=
\begin{cases}
0, & i=\m; \\
\coeff w_{\Mi}\p_{\lambda}, & i\not=\m.
\end{cases}
\end{equation}
Let $\Mi'$ be the perfect matching obtained by flipping $\Mi$ at the face $F(\lambda)=\Fi$. Then we have, by Lemma~\ref{lem:flipchangeinw} and
Proposition~\ref{p:nearlambda}:
\begin{align}
\label{e:Isecondpart}
\begin{split}
X^{(\m)}_{\lambda} \left( \frac{w_{\Mi'}}{w_{\Mi}} \right) &=
X^{(\m)}_{\lambda} \left( \frac{\prod_{j=1}^r p_{\lambda(j)}}{\prod_{j=1}^r p_{\lambda'(j)}} \right) \\
&= \left( \sum_{j=1}^r c_{\lambda}^{(\m)}(
\lambda(j))-\sum_{j=1}^r c_{\lambda}^{(\m)}(\lambda'(j)) \right) \frac{\prod_{j=1}^r p_{\lambda(j)}}{\prod_{j=1}^r p_{\lambda'(j)}}p_{\lambda}=-\frac{w_{\Mi'}}{w_{\Mi}}\p_{\lambda},
\end{split}
\end{align}
using the definition of $X_{\lambda}^{(\m)}$.
Hence, by~\eqref{e:Ifirstpart},~\eqref{e:Isecondpart} and the Leibniz rule,
\begin{equation}
\label{e:Ithirdpart}
X^{(\m)}_{\lambda} \left(\coeff w_{\Mi'}\right)=
X^{(\m)}_{\lambda} \left( \coeff w_{\Mi} \frac{w_{\Mi'}}{w_{\Mi}} \right) =
\begin{cases}
-\coeff w_{\Mi'} p_{\lambda}, & i=\m; \\
0, & i\not=\m.
\end{cases}
\end{equation}
Let $\Mi''$ be the perfect matching obtained from $\Mi'$ by flipping at some $\Mi'$-flippable face not equal to $F(\lambda)$.
By Lemma~\ref{lem:flipchangeinw} and Proposition~\ref{Prop:VectorFieldAdditivity},
\begin{equation}
\label{e:Ifourthpart}
X^{(\m)}_{\lambda} \left( \frac{w_{\Mi''}}{w_{\Mi'}} \right)=0.
\end{equation}
Hence, by~\eqref{e:Ithirdpart},~\eqref{e:Ifourthpart} and the Leibniz rule,
$$X^{(\m)}_{\lambda}( \coeff w_{\Mi''})=
\begin{cases}
-\coeff w_{\Mi''} p_{\lambda}, & i=\m; \\
0, & i\not=\m.
\end{cases}
$$
By Theorem~\ref{t:pmproperties}, any perfect matching on $G_i$ can
be reached from $M_i$ by a sequence of face flips
involving faces of $\Giin$ other than $F(\lambda)$. So, repeating the argument above
and using~\eqref{e:recallMS}, we obtain the
following, where the sums are over all perfect
matchings $M$ of $G_i$.
\begin{equation*}
\begin{split}
X^{(\m)}_{\lambda} \left( \frac{\p_{\widehat{L}_i}}{\p_{L_i}} \right) &=
X_{\lambda}^{(\m)} \left( \sum_{M} \coeff w_M \right)
=\begin{dcases}
\coeff w_{\Mi} p_{\lambda}-\sum_{M} \coeff w_M p_{\lambda}, & i=m; \\
\coeff w_{\Mi} p_{\lambda}, & i\not=\m;
\end{dcases} \\
&=\begin{dcases}
\frac{p_{I'_i}}{p_{I_i}}p_{\lambda}-\frac{p_{\widehat{L}_i}}{p_{L_i}}p_{\lambda}, & i=\m; \\
\frac{p_{I'_i}}{p_{I_i}}p_{\lambda}, & i\not=\m.
\end{dcases}
\end{split}
\end{equation*}
Note that $p_{I_i}=p_{\lambda}$ and
$p_{I'_i}=p_{\tau_i(\lambda)}$.
Hence, we have:
\begin{equation*}
X^{(\m)}_{\lambda} \left( \frac{\p_{\widehat{L}_i}}{\p_{L_i}} \right) 
=
\begin{dcases}
p_{\addbox{\lambda}{i}}-
\frac{p_{\widehat{L}_i}}{p_{L_i}}p_{\lambda}, & i=\m; \\
p_{\addbox{\lambda}{i}}, & i\not=\m;
\end{dcases}
\end{equation*}
as required in this case.

\noindent \textbf{Case II}: Suppose that $F_i$ is
an internal face and $i\not \in K_{\lambda}$. \\
Then, using~\eqref{e:initialshort} and~\eqref{e:firstterm}, we have that:
\begin{equation}
\label{e:IIfirstpart}
X^{(\m)}_{\lambda} \left(\coeff w_{\Mi}\right)=
\begin{cases}
-\coeff w_{\Mi} p_{\lambda}, & i=\m; \\
0, & i\not=\m.
\end{cases}
\end{equation}
Since $i\not\in K_{\lambda}$, $F(\lambda)$ is
not a face of $\Giin$ by Remark~\ref{r:Flambda}.
Hence all perfect matchings can be obtained from $\Mi$ by flips not
involving $F(\lambda)$ by Theorem~\ref{t:pmproperties}.
If $\Mi'$ is obtained from $\Mi$ by the
flip of a face in $\Giin$, then,
by the definition of $X_{\lambda}^{(\m)}$
and applying Lemma~\ref{lem:flipchangeinw} and
Proposition~\ref{Prop:VectorFieldAdditivity}:
\begin{equation}
\label{e:IIsecondpart}
X^{(\m)}_{\lambda} \left( \frac{w_{\Mi'}}{w_{\Mi}} \right)=0.
\end{equation}
Hence, by~\eqref{e:IIfirstpart},~\eqref{e:IIsecondpart} and the Leibniz rule,
$$X^{(\m)}_{\lambda} \left(\coeff w_{\Mi'}\right)=
\begin{cases}
-\coeff w_{\Mi'} p_{\lambda}, & i=\m; \\
0, & i\not=\m.
\end{cases}
$$
Repeating this argument, we see that for all perfect matchings $M$ of $G_i$,
$$X^{(\m)}_{\lambda} \left(\coeff w_M\right)=
\begin{cases}
-\coeff w_M p_{\lambda}, & i=\m; \\
0, & i\not=\m;
\end{cases}$$
and therefore, using~\eqref{e:recallMS},
we obtain the following, where the sum is over
all perfect matchings $M$ of $G_i$:
\begin{equation*}
\begin{split}
X^{(\m)}_{\lambda} \left( \frac{\p_{\widehat{L}_i}}{\p_{L_i}} \right)&=
\begin{cases}
-\sum_M \coeff w_M p_{\lambda}, & i=\m; \\
0, & i\not=\m;
\end{cases} \\
&=
\begin{dcases}
-\frac{p_{\widehat{L}_m}}{p_{L_m}}p_{\lambda}, & i=\m; \\
0, & i\not=\m;
\end{dcases}
\end{split}
\end{equation*}
as required in this case.

\noindent\textbf{Case III}: Suppose that $F_i$ lies on the boundary of $D$.
Then its label, $I_i=L_j=[j-k+1,j]$ for some $j\in [1,n]$.
Since $i\in I_i$ and $i+1\not\in I_i$, we must have $j=i$ and $L_i=I_i$, so $\widehat{L}_i=I'_i$.

We consider the cases of~\eqref{e:wewillshowthat}.
If $i\in K_{\lambda}$ then, by Remark~\ref{r:Flambda}, $p_{I_i}=
p_{\lambda}$ and $p_{I'_i}=p_{\addbox{\lambda}{i}}$. If, in addition, $i=\m$, then
\begin{equation}
\label{e:IIIfirstpart}
p_{\addbox{\lambda}{i}}-\frac{p_{\widehat{L}_i}}{p_{L_i}}p_{\lambda}=p_{I'_i}-\frac{p_{I'_i}}{p_{I_i}}p_{I_i}=0.
\end{equation}
If $i\not=\m$, then
\begin{equation}
\label{e:IIIsecondpart}
p_{\addbox{\lambda}{i}}=p_{I'_i}=\frac{p_{I'_i}}{p_{I_i}}p_{\lambda}.
\end{equation}
Also, if $i\not\in K_{\lambda}$ and $i=\m$,
then
\begin{equation}
\label{e:IIIthirdpart}
-\frac{p_{\widehat{L}_i}}{p_{L_i}}p_{\lambda}=
-\frac{p_{I'_i}}{p_{I_i}}p_{\lambda}.
\end{equation}
Equation~\eqref{e:wewillshowthat} in this
case now follows from~\eqref{e:IIIfirstpart},~\eqref{e:IIIsecondpart},~\eqref{e:IIIthirdpart} and~\eqref{e:firstterm}.
The proposition is proved.
\end{proof}

\section{Completion of the proof of Theorem~\ref{t:main1}}
\label{s:completionofproof}
In this section we complete the proof of Theorem~\ref{t:main1}. Namely, we will show that:
\begin{align} \label{e:dwdtaction3}
[q\dWdq \p_{\lambda}\omega] &=
\sum_{\mu}[\p_{\mu}\omega]+q\sum_{\nu}[\p_{\nu}\omega]; \\
\label{e:dwdtaction4}
\frac{1}{z}[W_q\p_{\lambda}\omega] &=
\frac{n}{z}\left(\sum_{\mu}[\p_{\mu}\omega]+q\sum_{\nu}[\p_{\nu}\omega] \right) -|\lambda|[\p_{\lambda}\omega],
\end{align}
where $\mu$, $\nu$ are exactly as in the quantum Monk's rule for
$\sigma^{\Box}*_q\sigma^{\lambda}$.

We first of all note a corollary to Theorem~\ref{t:nonequivariantaction}.

\begin{cor} \label{c:nonequivariantaction}
Let $\lambda$ be an arbitrary Young diagram in $\Pn_{k,n}$. Then we have:
\begin{enumerate}[(a)]
\item
$$X_{\lambda}W_q=
\left(
 \sum_{\mu} \p_{\mu}+q\sum_{\nu}\p_{\nu}
\right)-q \dWdq p_{\lambda};$$
\item
$$\sum_{m=1}^n X^{(m)}_{\lambda}W_q=n\left(
 \sum_{\mu} \p_{\mu}+q\sum_{\nu}\p_{\nu}
\right)
- W_qp_{\lambda},$$
where $\mu$, $\nu$ are exactly as in the quantum Monk's rule for
$\sigma^{\Box}*_q\sigma^{\lambda}$.
\end{enumerate}
\end{cor}

\begin{proof}
Part (a) is the case $m=n$ in Theorem~\ref{t:nonequivariantaction}. To see part (b) we add up the cases $m=1,2,\ldots ,n$.
\end{proof}

Let $\xi$ be a regular vector field on $\Xcheck$ and denote by $i_{\xi}\omega$ the insertion of $\xi$ into
$\omega$. Then we obtain a relation 
$$\left[d(i_{\xi}\omega)+\frac{1}{z} dW_q \wedge i_{\xi}\omega\right]=0$$
in $G^W$ by applying $d+\frac{1}{z} dW_q \wedge \ -\ $ to the $(n-1)$-form $i_{\xi}\omega$.
Since $dW_q\wedge \omega=0$ we have
$ dW_q\wedge i_{\xi}\omega=(i_{\xi}(dW_q))\omega$.
Therefore the relation in $G^W$ reads
\begin{equation} \label{e:insertion}
\left[d(i_{\xi}\omega)\right]+\frac{1}{z}\left[(\xi\cdot W_q)\omega\right]=0,
\end{equation}

%

We compute the first term in~\eqref{e:insertion}
in the case $\xi=X_{\lambda}^{(\m)}$. Recall that the $m$-th twist of the empty partition $\emptyset$, denoted by $\emptyset^{(\m)}$, is equal to
$\mu_{n-m}$ and corresponds to the $k$-subset $J_{n-m}=\{n-m+1,\dotsc, n-m+k\}$ (interpreted cyclically modulo $n$).

\begin{lem} \label{l:differentialofinsertion}
Let $\lambda\in \Pn_{k,n}$. Then we have:
\begin{equation}
\label{e:diffinsertion}
\left[d\left(i_{X_{\lambda}^{(\m)}}\omega\right)\right]=-c_{\lambda^{(\m)}}(\emptyset^{(\m)})[\p_{\lambda}\omega].
\end{equation}
\end{lem}

\begin{proof}
It suffices to check this on the cluster
torus $\Xcheck_{\CC}$, where $\CC$ is the
Postnikov extended cluster corresponding to $D_{\lambda}$.
We have:
\begin{equation*}
d\left(i_{X_{\lambda}^{(\m)}}\omega\right) =
\sum_{\mu\in \CC} d\left( i_{c_{\lambda}^{(\m)}(\mu)\p_{\lambda}\p_{\mu}\frac{\partial}{\partial \p_\mu}}\omega\right).
\end{equation*}
If $\mu\ne \lambda$ then 
\begin{equation*}
\begin{split}
d\left( i_{c_{\lambda}^{(\m)}(\mu)\p_{\lambda}\p_{\mu}\frac{\partial}{\partial \p_\mu}}\omega\right)
&=
\pm d\left( c_{\lambda}^{(\m)}(\mu) p_{\lambda}\bigwedge_{\varepsilon\in \CC, \varepsilon\not=\mu} \frac{d\p_{\varepsilon}}{\p_{\varepsilon}}\right) \\
&=\pm c_{\lambda}^{(\m)}(\mu)dp_{\lambda}\wedge\bigwedge_{\varepsilon\in \CC,\varepsilon\not=\mu} \frac{d\p_{\varepsilon}}{\p_{\varepsilon}}=0.
\end{split}
\end{equation*}
Therefore the only non-zero summand is the one where $\mu=\lambda$. We may write $\omega$ in the extended cluster $\CC$ as $\omega =\pm\frac{d\p_\lambda}{\p_\lambda}\wedge\bigwedge_{\varepsilon\in \CC,\varepsilon\not=\lambda} \frac{d\p_{\varepsilon}}{\p_{\varepsilon}}$. Then the $\mu=\lambda$ summand is 
\begin{equation*}
\begin{split}
d\left(c_{\lambda}^{(\m)}(\lambda)\p_{\lambda} i_{\p_{\lambda}\frac{\partial}{\partial \p_\lambda}}\omega\right)&=
\pm c_{\lambda}^{(\m)}(\lambda)dp_{\lambda}\wedge\bigwedge_{\varepsilon\in \CC,\varepsilon\not=\lambda} \frac{d\p_{\varepsilon}}{\p_{\varepsilon}} \\
&=c_{\lambda}^{(\m)}(\lambda)p_{\lambda}\omega \\
&=\left(c_{\lambda^{(\m)}}(\lambda^{(\m)})-c_{\lambda^{(\m)}}\left(\emptyset^{(\m)}\right)\right)p_{\lambda}\omega \\
&=-c_{\lambda^{(\m)}}\left(\emptyset^{(\m)}\right)p_{\lambda}\omega,
\end{split}
\end{equation*}
and we are done.
\end{proof}

Since $c_{\lambda}(\emptyset)=0$, it follows from Lemma~\ref{l:differentialofinsertion} that
\begin{align}
\left[d\left(i_{X_{\lambda}}\omega\right)\right] &=0; \label{e:firstinsertion} \\
\left[d\left(i_{\sum_{\m=1}^n X_{\lambda}^{(\m)}}\right) \omega\right] &=
\left(-\sum_{\m=1}^n c_{\lambda^{(\m)}}\left(\emptyset^{(\m)}\right) \right) [p_{\lambda}\omega].
\label{e:secondinsertion}
\end{align}
We need a simpler form for the coefficient in the second equation.
This will be given Lemma~\ref{l:boundarysum} below.
Recall that
$J_i$ is the $k$-subset corresponding to $\mu_i$. The following statement follows
immediately from the definitions.

\begin{lem} \label{l:boundarycoefficients}
Let $\lambda\in\Pn_{k,n}$.
Then for $1\leq i\leq n$, we have
$$c_{\lambda}({\mu_i})=
\begin{cases}
|[1,i]\cap J_{\lambda}|, & 1\leq i\leq n-k; \\
|[i+1,n]\setminus J_{\lambda}|, & n-k+1\leq i\leq n.
\end{cases}
$$
\hfill \qed
\end{lem}

\begin{lem} \label{l:boundarysum}
Let $\lambda \in \Pn_{k,n}$. Then
$$\sum_{\m=1}^{n} c_{\lambda^{(\m)}}\left(\emptyset^{(\m)}\right)=|\lambda|.$$
\end{lem}

\begin{proof}
Firstly, we note that $c_{\lambda}^{(n)}(\emptyset)=0$. Recall also that
$J_{\emptyset^{(\m)}}=J_{n-\m}$.
If $1\leq \m\leq k-1$, then $n-k+1\leq n-\m\leq n-1$, so, by
Lemma~\ref{l:boundarycoefficients},
\begin{equation*}
\begin{split}
c_{\lambda^{(\m)}}\left(\emptyset^{(\m)}\right) &=|[n-m+1,n]\setminus J_{\lambda^{(\m)}}| \\
&= |[n-\m+1,n]\setminus (J_{\lambda}-\m)| \\
&= |[1,\m]\setminus J_{\lambda}|.
\end{split}
\end{equation*}
An element $j\in [1,k-1]\setminus J_{\lambda}$ contributes $1$ to
the term $[1,\m]\setminus J_{\lambda}$ in the sum
$$\sum_{\m=1}^{k-1} |[1,\m]\setminus J_{\lambda}|$$
for all $m\geq j$, and zero otherwise. It follows that
\begin{equation}
\label{e:kminusj}
\sum_{\m=1}^{k-1} c_{\lambda^{(\m)}}\left(\emptyset^{(\m)}\right) =
\sum_{j\in[1,k-1]\setminus J_{\lambda}} k-j.
\end{equation}
We have
$$[1,n]\setminus J_{\lambda}=
\{k-\lambda_1+1,\ldots ,k-\lambda_{n-k}+(n-k)\},$$
so
$$[1,k-1]\setminus J_{\lambda}
=
\{k-\lambda_1+1,\ldots ,k-\lambda_s+s\},
$$
where $s$ is maximal such that
$\lambda_s>s$.
Hence the sum in~\eqref{e:kminusj} can
be rewritten as
$$\sum_{r=1}^s k-(k-\lambda_r+r)=
\sum_{r=1}^s \lambda_r-r,$$
which is the number of boxes in $\lambda$
strictly to the right of the leading
diagonal.

If $k\leq \m \leq n-1$, then $1\leq n-m\leq n-k$ and, by Lemma~\ref{l:boundarycoefficients},
\begin{equation*}
\begin{split}
c_{\lambda^{(\m)}}\left(\emptyset^{(\m)}\right) &= |[1,n-m] \cap J_{\lambda^{(\m)}}| \\
&= |[1,n-m]\cap (J_{\lambda}-\m)| \\
&= |[m+1,n]\cap J_{\lambda}|.
\end{split}
\end{equation*}
An element $j$ in $[k+1,n]\cap J_{\lambda}$ contributes $1$ to the term
$|[m+1,n]\cap J_{\lambda}|$ in the sum
$$\sum_{m=k}^{n-1} |[m+1,n]\cap J_{\lambda}|$$
if $m\leq j-1$, and zero otherwise. It follows that
\begin{equation}
\label{e:jminusk}
\sum_{\m=k}^{n-1} c_{\lambda^{(\m)}}\left(\emptyset^{(\m)}\right) =
\sum_{j\in J_{\lambda}\cap [k+1,n]} j-k,
\end{equation}
Let $\lambda'$ be the transpose of $\lambda$.
Then we have:
$$J_{\lambda}=\{k+\lambda'_1,k+\lambda'_2-1,\ldots, k+\lambda'_k-k+1\}.$$
Hence,
$$[k+1,n]\cap J_{\lambda}=\{k+\lambda'_1,k+\lambda'_2-1,\ldots ,k+\lambda'_{u}-u+1\},$$
where $u$ is maximal such that $\lambda'_u\geq u$.
Therefore, the sum in~\eqref{e:jminusk} can
be rewritten as
$$\sum_{r=1}^u (k+\lambda'_r-r+1)-k=
\sum_{\lambda'_r\geq r} \lambda'_r-r+1,$$
which is the number of boxes in
$\lambda$ on or below the
leading diagonal. Combining this with the above gives the claimed result.
\end{proof}

By~\eqref{e:secondinsertion} and Lemma~\ref{l:boundarysum}, we have:
\begin{equation}
\label{e:secondinsertion2}
\left[d\left(i_{\sum_{\m=1}^n X_{\lambda}^{(\m)}}\right) \omega\right] =
-|\lambda|[p_{\lambda}\omega].
\end{equation}

\begin{thm} \label{t:dwdtaction}
Let $\lambda\in \Pn_{k,n}$. Then
\begin{enumerate}[(a)]
\item
$$
q\dWdq [\p_{\lambda}\omega]=
\sum_{\mu}[\p_{\mu}\omega]+q\sum_{\nu}[\p_{\nu}\omega],\text{ and}
$$
\item
$$\frac{1}{z}[W_q\p_{\lambda}\omega]=
\frac{n}{z}\left(\sum_{\mu}[\p_{\mu}\omega]+q\sum_{\nu}[\p_{\nu}\omega] \right) -|\lambda|[\p_{\lambda}\omega],
$$
where, in each case, $\mu$, $\nu$ are exactly as in the quantum Monk's rule
for $\sigma^{\Box}*_q\sigma^{\lambda}$.
\end{enumerate}
\end{thm}

\begin{proof}
Part (a) follows from~\eqref{e:insertion}
in the case $\xi=X_{\lambda}$, using
Corollary~\ref{c:nonequivariantaction}(a) and~\eqref{e:firstinsertion}.
Part (b) follows from~\eqref{e:insertion} in
the case $\xi=\sum_{\m=1}^n X_{\lambda}^{(\m)}$,
using Corollary~\ref{c:nonequivariantaction}(b) and~\eqref{e:secondinsertion2}.
\end{proof}

\begin{proof}[Proof of Theorem~\ref{t:main1}]
By Theorem~\ref{t:dwdtaction}, equations~\eqref{e:dwdtaction1} and~\eqref{e:dwdtaction2} hold. This, together with Lemma~\ref{l:freebasis2}, completes the proof of Theorem~\ref{t:main1}.
\end{proof}

\section{Background for mirror symmetry in the torus-equivariant setting}
\label{s:Tequivariant}
We now turn to the torus-equivariant mirror theorem, Theorem~\ref{t:equivmain}. We begin in  Section~\ref{s:eqqH} by reviewing the structure of the small equivariant quantum cohomology ring of a Grassmannian. We refer to \cite{Anderson:EquivariantCohomology} for background on equivariant cohomology, and \cite{Mihalcea} for relevant background on equivariant quantum cohomology. In Section~\ref{s:Tequivariantsuperpotential} we recall the equivariant version of the superpotential (introduced for general $G/P$ in \cite{Rie:MSgen}) and describe it in the case of the Grassmannian in terms of Pl\"ucker coordinates. The main ingredient to the proof of the mirror theorem, Theorem~\ref{t:equivmain}, is Theorem~\ref{t:dwdtactionequivariant} proved in Section~\ref{s:actionvectorfieldequivariant}.  Namely, in this section we work out the action on the equivariant superpotential of the vector fields $X_\lambda$ constructed in Section~\ref{s:vectorfield}.   

Let us first fix our basic set-up regarding the torus and its action on $X$. Recall that $T^\vee$ denotes the maximal torus of diagonal matrices in $GL_n^\vee(\C)$, the general linear group in the $A$-model. It naturally acts on $\C^n$ and on $X=Gr_{n-k}(\C^n)$. We denote the standard basis of $\C^n$ by $v_1,\dotsc, v_n$. The weight of the action of $T^\vee$ on the span of $v_i$ is denoted by $\varepsilon^\vee_i\in \Hom(T^\vee,\C^*)$, and we have $\varepsilon^\vee=(\varepsilon_i^\vee):T\overset\sim\to(\C^*)^n$. As usual we think of the lattice $X_*(T^\vee)=\Hom(T^\vee,\C^*)$ as embedded in the dual $(\mathfrak h^\vee)^*$ of the Lie algebra $\mathfrak h^\vee$ of $T^\vee$, and use additive notation for characters.

\subsection{The equivariant cohomology ring of $X$}\label{s:ehomology}
Our conventions regarding the equivariant cohomology ring $H^*_{T^\vee}(X,\C)$ of the Grassmannian $X$ are as follows. First recall that the equivariant cohomology of $X$ is a free module over the equivariant cohomology of a point. Moreover, the equivariant cohomology of a point 
is a polynomial ring 
\[
H^*_{T^\vee}(pt)=\C[x_1,\dotsc, x_n].
\]    
To be completely explicit, using the  
Borel construction and the isomorphism $\varepsilon^\vee:T^\vee\cong (\C^*)^n$ we have $H^*_{T^\vee}(pt)=H^*(BT^\vee)=H^*(\prod_{i=1}^n\C P^\infty)$, 
and our conventions are that $x_i$ is the first Chern class of the line bundle coming from the $\mathcal O(1)$ of the $i$-th factor in $\prod_{i=1}^n\C P^\infty$. 
This class $x_i$ is also the equivariant first Chern class of the one-dimensional representation $-\varepsilon_i^\vee$ of $T^\vee$, interpreted as an equivariant line bundle on the point. Therefore we have natural identifications
$$
H_{T^\vee}^{*}(pt,\C)=\C[x_1,\dotsc, x_n]=S^\bullet\left(\left(\mathfrak h^\vee\right)^*\right)=\C[\mathfrak h^\vee],
$$
with $x_i=-\varepsilon_i^\vee$. 

The Schubert basis in the equivariant setting
is made up of equivariant fundamental classes of certain $T^\vee$-invariant Schubert varieties in $X$, which we need to choose explicitly as follows.
Our Schubert varieties $X^\lambda$ are 
 obtained as as the closures of $B^\vee_+$-orbits in $X$. Recall that $J_\lambda$ records the $k$ horizontal steps in the Young diagram $\lambda\in\Pn_{k,n}$, compare Section~\ref{s:GrassmannianBmodel}. Let $\Vertical(\lambda)=[1,n]\setminus J_\lambda$ and define $[v_\lambda]\in Gr_{n-k}(\C^n)$ by
\[
[v_\lambda]=\left<v_{j}\ | \ j\in\Vertical(\lambda)\right>_\C.
\]
We define
\[
X^\lambda:=\overline{B^\vee_+\cdot [v_\lambda]}.
\]
With this definition, the Schubert variety denoted by $X^\lambda$ has complex codimension $|\lambda |$, the number of boxes in $\lambda$. We denote by
$
\sigma_{T^\vee}^{\lambda}
$
the associated equivariant fundamental class $[X^\lambda]_{T^\vee}\in H_{T^\vee}^{2|\lambda |}(X)$.

The equivariant version of Monk's rule involves the following linear combinations of equivariant parameters,
\begin{equation}\label{e:xlambda}
x_\lambda:=\sum_{j\in \Vertical(\lambda)}x_{j}.
\end{equation} 
Note that under the identification $x_i=-\varepsilon_i^\vee$ the $x_\lambda$ are the {\it negatives} of the weights of the $(n-k)$-th fundamental representation of $GL_n^\vee(\C)$. In particular $x_{\lambda_{max}}=x_1+\dotsc + x_{n-k}$ is the negative of the highest weight,
and $x_{\emptyset}=x_{k+1}+x_{k+2}+\cdots+x_{n}$ the negative of the lowest weight. Here $\lambda_{max}$ refers to the maximal Young diagram, the $(n-k)\x k$ rectangle.

We now consider three analogues of the `hyperplane class' $\sigma^{\ydiagram {1}}$ in the equivariant setting.

 
\begin{enumerate}
\item
We have the $T^\vee$-equivariant first Chern class of the equivariant line bundle $\mathcal O(1)$ coming from the Pl\"ucker embedding,
\[
\zeta:=c_1^{T^\vee}(\mathcal O(1)).
\]  
\item
In $H_{T^\vee}^2(X)$ we also have the Schubert class $\sigma_{T^\vee}^{\ydiagram {1}}$ defined above,
corresponding to the $B^\vee_+$-invariant Schubert divisor $X^{\ydiagram{1}}$.
One can check that
\[
\sigma_{T^\vee}^{\ydiagram{1}}= c^{T^\vee}_1(\mathcal O(1)\otimes L_{\varepsilon_{k+1}+\dotsc +\varepsilon_{n} })=c^{T^\vee}_1(\mathcal O(1))-(x_{k+1}+\dotsc+x_n)=\zeta-x_\emptyset.
\]
\item
Finally, we have the alternative equivariant Schubert class
$\tilde{\sigma}_{T^\vee}^{\ydiagram {1}}$,
which corresponds to the $B^{\vee}_-$-invariant Schubert divisor 
\[
\widetilde{X}^{\ydiagram{1}}=\overline{B^\vee_-\cdot v_{\lambda_{submax}}},
\]
where $\lambda_{submax}$ denotes the Young diagram in
$\Pn_{k,n}$ obtained by removing one box from the maximal Young diagram. This class is related to the other two choices by 
\[
\tilde{\sigma}_{T^\vee}^{\ydiagram{1}}=
\zeta-(x_{1}+\dotsc+x_{n-k})=\zeta-x_{\lambda_{max}}=\sigma_{T^\vee}^{\ydiagram{1}}+x_\emptyset-x_{\lambda_{max}}.
\]

\end{enumerate}

\subsection{The small equivariant quantum cohomology of $X$}\label{s:eqqH}

The equivariant quantum cohomology ring of $X$ is denoted by $qH^*_{T^\vee}(X)$. It is defined by  using $T^\vee$-equivariant versions of Gromov-Witten invariants \cite{Lu:EquivQCoh} to specify a $q$-deformed cup product structure on $H_{T^\vee}^*(X,\C)\otimes\C[q]$; see also \cite{GK:FlagQCoh}.  In the case of the Grassmannian $X$ the structure of the ring $qH^*_{T^\vee}(X)$ was worked out by Mihalcea \cite{Mihalcea}.



We now recall the equivariant version of the quantum Monk's rule 
\cite[Section~1.1]{Mihalcea}. 
Expressed in our conventions, this states that quantum multiplication with the equivariant Chern class of $\mathcal O(1)$, is given by the formula
\begin{equation}\label{e:EquivariantQMonks}
 \zeta\star_{q,x} \sigma_{T^\vee}^{\lambda}=
\sum_\mu \sigma_{T^\vee}^{\mu}+ q\sum_{\nu}\sigma^{\nu}_{T^\vee}+ x_\lambda \sigma_{T^\vee}^\lambda,
\end{equation}
where the first two summands on the right hand side are as in the non-equivariant quantum Monk's rule and $x_\lambda$ is given in Equation~\eqref{e:xlambda}.

We will also need a special $T^\vee$-invariant anti-canonical divisor, $X_{ac}$. Note that we have $\Z/n\Z$-action on the Grassmannian $X$ analogously to the one defined in Section~\ref{s:GrassmannianBmodel} for the $B$-model. (Indeed $X$ and $\Xcheckbar$ are isomorphic varieties).  The divisor $X_{ac}$ is the $\Z/n\Z$-orbit of the divisor $X^{\ydiagram{1}}$. If we denote by $X^{\ydiagram{1}}(i)$ the $i$-th translate of $X^{\ydiagram{1}}$ under the $\Z/n\Z$-action then 
\begin{equation}\label{e:Xac}
X_{ac}=\bigcup_{i=1}^n X^{\ydiagram{1}}(i).
\end{equation}
Note that $X_{ac}$ is the union of $n$ distinct hyperplanes in $X$,
including $X^{\ydiagram{1}}$ (where $i=n$) and $\widetilde{X}^{\ydiagram{1}}$ (where $i=n-k$). This is the Langlands dual version of the divisor $D$ from~\eqref{e:acdivisorDef}.

The equivariant fundamental class $[X_{ac}]_{T^\vee}$ of $X_{ac}$ is given by
\[
[X_{ac}]_{T^\vee}= n \zeta - \sum_{j=1}^n (x_{j+1} +\dotsc x_{j+n-k})=n \zeta-(n-k)\left(\sum_{i=1}^n x_i\right)
\] 
where the indices are interpreted modulo $n$. Hence we have
\begin{equation}
[X_{ac}]_{T^\vee} \star_{q,x} \sigma_{T^\vee}^{\lambda}=
n\sum_\mu \sigma_{T^\vee}^{\mu}+ qn\sum_{\nu}\sigma^{\nu}_{T^\vee}+ nx_\lambda \sigma_{T^\vee}^\lambda-(n-k)\left(\sum_{i=1}^n x_i\right) \sigma_{T^\vee}^{\lambda},
\end{equation}
where the terms in the first two summands are as in the
non-equivariant quantum Monk's rule.

\section{The $T^\vee$-equivariant version of the superpotential}
\label{s:Tequivariantsuperpotential}

By the equivariant superpotential of the target space $X$ we mean a deformation of the usual superpotential to a (multi-valued) map involving the equivariant parameters, which encodes structures from the equivariant quantum cohomology of  $X$. 
A torus-equivariant version of the superpotential for general type partial flag varieties was introduced in \cite[Section~4]{Rie:MSgen}, where it was denoted $\mathcal F_P+\ln(\phi)$, and was shown to recover the equivariant quantum cohomology rings in their presentation due to Dale Peterson \cite{Pet:QCoh}. In this section we express this equivariant superpotential, in the special case of the Grassmannian $X=Gr_{n-k}(\C^n)$, in terms of the Pl\"ucker coordinates on the mirror Grassmannian $\Xcheck$. 

We may think of the superpotential $W:\Xcheck\x \C^*_q\to \C$ as a section of a trivial line bundle $\C$ on $\Xcheck\x \C^*_q$. The $T^\vee$-equivariant version of the superpotential will be a multi-valued section of the trivial vector bundle $\C\oplus \mathfrak h$ on $\Xcheck\x \C^*_q$, so a multi-valued  map
 \[
\Weq:\Xcheck\x\C^*_q\xdashrightarrow{(W,\Phi)} 
\C\oplus \mathfrak h.
\]

We give our new definition of $\Weq$ in
the Grassmannian setting first, followed by the original,
more general definition of the equivariant superpotential for homogeneous spaces from 
\cite{Rie:MSgen}. Then we will demonstrate that the two are 
equivalent when the homogeneous space is a Grassmannian. 

\begin{defn}\label{d:Weq}
Recall that we have a natural identification of  $\mathfrak h$ with $H_{T^\vee}^2(\{pt\})$ which sends $-\varepsilon_i^\vee$ to the equivariant parameter $x_i$, see Section~\ref{s:ehomology}.
We define $\Weq:\Xcheck\x\C^*_q\xdashrightarrow{\quad \ }
\C\oplus \mathfrak h$ by
\begin{equation}
\label{e:Weq}
\Weq=W+ \ln(q)(x_1+\dotsc+ x_{n-k})+ \ln(\p_{\mu_{1}})(x_2-x_1)+\ln(\p_{\mu_{2}})(x_3-x_2)+\dotsc + \ln(\p_{\mu_{n-1}})(x_n-x_{n-1}),
\end{equation}
where $W$ is as in Definition~\ref{d:W} and we keep in mind that $\p_{\mu_n}=\p_{\emptyset}=1$. Therefore $\ln(\p_{\mu_n})=0$. 

Note that we have $x_{i+1}-x_{i}=\alpha_i^\vee$ and $x_1+\dotsc+x_{n-k}=-\omega_{n-k}^\vee$, since the equivariant parameters
are related to
the usual basis of $\mathfrak h$ by $x_i=-\varepsilon_i^\vee$. 
\end{defn}

Recall the notations from Section~\ref{s:RichardsonBmodel}. We have isomorphisms
\[
\begin{array}{ccccc}
\Xcheck\x \C^*_{q}&\overset{\psi_L}\longleftarrow &B_-\cap U_+\tilde T^{W_P}\dot w_P\dot w_0\inv U_+ &\overset{\psi_R}\longrightarrow & \RR\x \C^*_{q}. \\
 (Pb,\alpha_{n-k}(t))&\leftarrow & b=u_1 t \dot w_P \dot w_0\inv u_2 &\mapsto & (b\dot w_0 B_-,\alpha_{n-k}(t)).
\end{array}
\]

Note that $B_-\cap U_+\tilde T^{W_P}\dot w_P\dot w_0\inv U_+$ is a subset of the Borel subgroup $B_-$ of $GL_n(\C)$. Let
 $\pi$ be the projection $B_-\to T $ sending an element of the Borel subgroup $B_-=T U_-$ onto its torus factor. We add this map to the diagram above, giving:
\[
\begin{array}{ccccc}
\Xcheck\x \C^*_{q}&\overset{\psi_L}\longleftarrow &B_-\cap U_+\tilde T^{W_P}\dot w_P\dot w_0\inv U_+ &\overset{\psi_R}\longrightarrow & \RR\x \C^*_{q}. \\
&& \downarrow\pi & &  \\
&& T\ &&
\end{array}
\]
The inverse to $\exp:\mathfrak {h}\to T$  defines a multivalued map
\[
\ln_{T}:T\xdashrightarrow{\quad} \mathfrak h  
\]   
 \begin{defn}\cite[Section~4]{Rie:MSgen} The equivariant Lie-theoretic superpotential 
 \[
 \Feq:\RR\x \C_{q}^*\xdashrightarrow{\qquad} \C\oplus\mathfrak h
 \]
  is defined by adding an $\mathfrak h$ component to the superpotential $\mathcal F:\RR\x \C_{q}^*\longrightarrow \C$ from Definition~\ref{d:RichardsonBmodel} as follows.
 Consider the composition 
 \[
 \Phi=\ln_{T}\circ\pi\circ\psi_R\inv: 
 \RR\x \C_{q}^*\To \mathfrak h.
 \] 
 Then  
\[
\mathcal F^{\operatorname{eq}} = \mathcal F +  \Phi.
\]
 \end{defn}

This definition is a slight variation of the definition of the equivariant superpotential from \cite{Rie:MSgen}, the difference stemming from the fact that~\cite{Rie:MSgen} considered the maximal torus of $PSL_n(\C)$ whereas here $T$ is the maximal torus of $GL_n(\C)$. 
Composing with the map $\C\oplus \mathfrak h\to \C\oplus \mathfrak h_{PSL_n}$ defined by quotienting out the center of $\mathfrak{gl}_n$ recovers the original equivariant superpotential from  \cite[Section~4]{Rie:MSgen} associated to the action of the maximal torus of $PSL_n(\C)$. (This equivariant superpotential is denoted  $\mathcal F_P+\ln(\phi)$ in \cite{Rie:MSgen}.)

 The main goal of this section is to prove the following comparison result.
 
 \begin{prop} \label{p:EquivariantComparison}
With the definitions as above, the following diagram commutes, 
\[
\begin{array}{ccccc}
\Xcheck\x \C^*_{q}&\overset{\psi_L}\longleftarrow &B_-\cap U_+\tilde T^{W_P}\dot w_P\dot w_0\inv U_+ &\overset{\psi_R}\longrightarrow & \RR\x \C^*_{q}. \\
\downarrow \Weq& &  & & \downarrow \mathcal F^{\operatorname{eq}} \\
\C \oplus \mathfrak h&& = && \C \oplus \mathfrak h.
\end{array}
\]
 Therefore the equivariant superpotentials $(\Xcheck,\Weq_q)$ and $(\RR,\Feq_q)$ are equivalent.
 \end{prop}

This proposition is an extension of Proposition~\ref{p:comparisonMRR}. We begin with some remarks. Let $b=u_1t\dot w_P\dot w_0^{-1}\in B^-\cap U_+\tilde{T}^{W_P}\dot w_P \dot w_0^{-1}U^+$.
Then, by Proposition~\ref{p:comparisonMRR}, we have
\begin{equation}
\begin{split}
\mathcal{F}^{eq}(\psi_R(b)) &=
\mathcal{F}(\psi_R(b))+\ln_T(\pi(b)) \\
&=W(\psi_L(b))+\sum_{i=1}^n x_i \ln(b_{ii}),
\end{split}
\end{equation}
where $b_{ii}$ denotes the $i$th diagonal entry of the matrix $b$.
Looking at~\eqref{e:Weq}, we see that to prove Proposition~\ref{p:EquivariantComparison},
it is sufficient to show that the following holds, where
$q=\alpha_{n-k}(t)$:
\begin{equation}
\label{e:equivariantaim1}
\begin{split}
\sum_{i=1}^n x_i \ln(b_{ii})&=
\ln(q)(x_1+\dotsc+ x_{n-k})+ (x_2-x_1)\ln(\p_{\mu_{12}}(Pb))+(x_3-x_2)\ln(\p_{\mu_{23}}(Pb))+\dotsc \\
&\quad\quad+ (x_{n}-x_{n-1})\ln(\p_{\mu_{n-1,n}}(Pb)).
\end{split}
\end{equation}
Note that, since $u_1,u_2\in U_+$ and $t\in \tilde T^{W_P}$, we have
\begin{equation*}
\begin{split}
\Delta_{J_n}^{[n-k+1,n]}(b) &=
\Delta_{[1,k]}^{[n-k+1,n]}(u_1t\dot w_P\dot w_0^{-1}u_2) \\
&= \Delta_{[1,k]}^{[n-k+1,n]}
(u_1t\dot w_P\dot w_0^{-1}) \\
&= \Delta_{[1,k]}^{[n-k+1,n]}
(t\dot w_P\dot w_0^{-1})=1,
\end{split}
\end{equation*}
where the last step is an easy calculation.
Since our convention is that $p_{\emptyset}=1$, it follows
that, for $i=1,\ldots ,n-1$,
\begin{equation*}
p_{\mu_i}(Pb)=
\Delta_{J_i}^{[n-k+1,n]}(b).
\end{equation*}

Thus,~\eqref{e:equivariantaim1} is equivalent to
\begin{equation}
\label{e:equivariantaim2}
\begin{split}
\sum_{i=1}^n x_i \ln(b_{ii})&=
\ln(q)(x_1+\dotsc+ x_{n-k})+ (x_2-x_1)\ln\left(\Delta_{J_1}^{[n-k+1,n]}(b)\right)+ \\
&\quad\quad+(x_3-x_2)\ln(\Delta_{J_2}^{[n-k+1,n]}(b))+\dotsc+
(x_{n}-x_{n-1})\ln\left(\Delta_{J_{n-1}}^{[n-k+1,n]}(b)\right).
\end{split}
\end{equation}
To prove~\eqref{e:equivariantaim2}, we will
write the diagonal entries of $b$ in terms of the
minors $\Delta_{J_i}^{[n-k+1,n]}$ of $b$.

\begin{lem} \label{l:diagonalb}
Let $b=u_+t\dot w_P\dot w_0^{-1}u_2\in B^-\cap U_+\tilde{T}^{W_P}\dot w_P\dot w_0^{-1}U_+$, with $q=\alpha_{n-k}(t)$.
Then, for $1\leq i\leq n$, we have the following:
$$b_{ii}=
\begin{dcases}
q\frac{1}{\Delta_{J_1}^{[n-k+1,n]}(b)}, & \text{if }i=1; \\
q\frac{\Delta_{J_{i-1}}^{[n-k+1,n]}(b)}{\Delta_{J_i}^{[n-k+1,n]}(b)}, & \text{if }2\leq i\leq n-k; \\
\frac{\Delta_{J_{i-1}}^{[n-k+1,n]}(b)}{\Delta_{J_i}^{[n-k+1,n]}(b)}, & \text{if }n-k+1 \leq i\leq n-1; \\
\Delta_{J_{n-1}}^{[n-k+1,n]}(b), & \text{if }i=n.
\end{dcases}
$$
\end{lem}
\begin{proof}
Write
$b=u_1 \telt \dot w_Pw_0^{-1} u_2$, with $u_1,u_2\in U_+$ and $\telt\in \tilde{T}^{W_P}$.
Then $b^{-1}=u_2^{-1}\dot w_0\dot w_P^{-1} \telt^{-1}u_1^{-1}$.
Hence,
$$(b^{-1})_{ii}=\frac{\Delta_{[1,i]}^{[1,i]}(u_2^{-1}\dot w_0\dot w_P^{-1} \telt^{-1})}{\Delta_{[1,i-1]}^{[1,i-1]}(u_2^{-1}\dot w_0\dot w_P^{-1} \telt^{-1})},$$
so
\begin{equation}
\label{e:biiequation}
b_{ii}=\frac{\Delta_{[1,i-1]}^{[1,i-1]}(u_2^{-1}\dot w_0\dot w_P^{-1} \telt^{-1})}{\Delta_{[1,i]}^{[1,i]}(u_2^{-1}\dot w_0\dot w_P^{-1} \telt^{-1})}.
\end{equation}

Recall that $\telt$ is a diagonal matrix
with
$$t_{ii}=\begin{cases}
q, & \text{if }1\leq i\leq n-k; \\
1, & \text{if }i+1\leq i\leq n.
\end{cases}
$$

We claim that, for $1\leq i\leq n$,
\begin{equation}
\label{e:partialeq}
\Delta_{[1,i]}^{[1,i]}(u_2^{-1}\dot w_0\dot w_P^{-1} \telt^{-1})=
\begin{dcases}
q^{-i}\Delta_{J_i}^{[n-k+1,n]}(b), & \text{if } 1\leq i\leq n-k; \\
q^{-(n-k)}\Delta_{J_i}^{[n-k+1,n]}(b), & \text{if } n-k\leq i\leq n-1; \\
q^{-(n-k)}, & \text{if }i=n.
\end{dcases}
\end{equation}
The result then follows 
from~\eqref{e:biiequation} and~\eqref{e:partialeq}.

To prove the claim, we consider each of the
three cases. We suppose first that
$1\leq i\leq n-k$. Then
\begin{align*}
\Delta_{[1,i]}^{[1,i]}\left(u_2^{-1}\dot w_0\dot w_P^{-1} \telt^{-1}\right)
&=q^{-i}\Delta_{[1,i]}^{[1,i]}\left(u_2^{-1}\dot w_0\dot w_P^{-1}\right) \\
&=q^{-i}(-1)^{ik}\Delta_{[k+1,k+i]}^{[1,i]}(u_2^{-1}) \\
&=q^{-i}(-1)^{ik}(-1)^s\Delta_{[i+1,n]}^{[1,k]\cup [i+k+1,n]}(u_2),
\end{align*}
where $s=(1+2+\cdots +i)+((k+1)+(k+2)+\cdots +(k+i))$,
using Jacobi's Theorem for the minors of an inverse matrix in the last step.
Noting that $s$ is congruent to $ik$ mod $2$ and
that $u_2$ is upper unitriangular, we obtain:
\begin{align*}
\begin{split}
\Delta_{[1,i]}^{[1,i]}\left(u_2^{-1}\dot w_0\dot w_P^{-1} \telt^{-1}\right) &=
q^{-i}\Delta_{[i+1,i+k]}^{[1,k]}(u_2). \\
&=q^{-i}\Delta_{[i+1,i+k]}^{[n-k+1,n]}\left(t\dot w_P\dot w_0^{-1}u_2\right) \\
&=q^{-i}\Delta_{J_i}^{[n-k+1,n]}(b),
\end{split}
\end{align*}
as required in this case.

Next, suppose that $n-k+1\leq i\leq n-1$.
Then we have
\begin{align*}
\Delta_{[1,i]}^{[1,i]}\left(u_2^{-1}\dot w_0\dot w_P^{-1} \telt^{-1}\right)
&=q^{-(n-k)}\Delta_{[1,i]}^{[1,i]}\left(u_2^{-1}\dot w_0\dot w_P^{-1}\right) \\
&=q^{-(n-k)}(-1)^{k(n-k)}\Delta_{[k+1,n]\cup [1,i+k-n]}^{[1,i]}\left(u_2^{-1}\right) \\
&=q^{-(n-k)}(-1)^{k(n-k)+(n-k)(i+k-n)}\Delta_{[1,i+k-n,]\cup [k+1,n]}^{[1,i]}\left(u_2^{-1}\right) \\
&=q^{-(n-k)}(-1)^{k(n-k)+(n-k)(i+k-n)}\Delta_{[k+1,n]}^{[i+k-n+1,i]}\left(u_2^{-1}\right),
\end{align*}
using in the last step the fact that $u_2^{-1}$ is upper unitriangular.
Applying Jacobi's theorem, we have:
$$\Delta_{[1,i]}^{[1,i]}\left(u_2^{-1}\dot w_0\dot w_P^{-1} \telt^{-1}\right)
=q^{-(n-k)}(-1)^t\Delta_{[1,i+k-n]\cup [i+1,n]}^{[1,k]}(u_2),$$
where
\begin{align*}
t&=k(n-k)+(n-k)(i+k-n)+\sum_{j=1}^k j +\sum_{j=1}^{i+k-n} j +\sum_{j=1}^{n-i} (i+j) \\
&=k(n-k)+(n-k)(i+k-n)+\sum_{j=1}^k j +\sum_{j=n-i+1}^k (j-n+i) +\sum_{j=1}^{n-i} (i+j) \\
&=k(n-k)+(n-k)(i+k-n)-n(i+k-n) +\sum_{j=1}^k j +\sum_{j=1}^{k} (i+j) \\
&=k(n-k)+(n-k)(i+k-n)-n(i+k-n) +k(i+k+1) \\
&=2kn-k(k-1),
\end{align*}
which is even. Hence
\begin{align*}
\Delta_{[1,i]}^{[1,i]}\left(u_2^{-1}\dot w_0\dot w_P^{-1} \telt^{-1}\right)
&=q^{-(n-k)}\Delta_{[1,i+k-n]\cup [i+1,n]}^{[1,k]}(u_2) \\
&=q^{-(n-k)}\Delta_{[1,i+k-n]\cup [i+1,n]}^{[n-k+1,n]}\left(\dot w_P\dot w_0^{-1}u_2\right) \\
&=q^{-(n-k)}\Delta_{[1,i+k-n]\cup [i+1,n]}^{[n-k+1,n]}\left(u_1\telt \dot w_P\dot w_0^{-1}u_2\right) \\
&=q^{-(n-k)}\Delta_{J_i}^{[n-k+1,n]}(b),
\end{align*}
as required in this case.

Finally, we consider the case $i=n$. We have:
\begin{align*}
\Delta_{[1,n]}^{[1,n]}\left(u_2^{-1} \dot w_0\dot w_P^{-1} \telt^{-1}\right)
&=q^{-(n-k)}\Delta_{[1,n]}^{[1,n]}\left(u_2^{-1}\dot w_0\dot w_P^{-1}\right) \\
&=q^{-(n-k)}(-1)^{k(n-k)}\Delta_{[k+1,n]\cup [1,k]}^{[1,n]}\left(u_2^{-1}\right) \\
&=q^{-(n-k)}(-1)^{k(n-k)}(-1)^{k(n-k)}\Delta_{[1,n]}^{[1,n]}\left(u_2^{-1}\right) \\
&=q^{-(n-k)},
\end{align*}
since $u_2^{-1}$ is upper unitriangular.
The result is shown.
\end{proof}

\begin{proof}[Proof of Proposition~\ref{p:EquivariantComparison}]
By Lemma~\ref{l:diagonalb}, we have:
\begin{equation*}
\begin{split}
\sum_{i=1}^n x_i \ln(b_{ii})
&=\ln(q)(x_1+x_2+\cdots +x_{n-k}) 
-x_1\ln\left(\Delta_{J_1}^{[n-k+1,n]}(b)\right)
\\
&\quad\quad
+x_n\ln\left(\Delta_{J_{n-1}}^{[n-k+1,n]}(b)\right)+\sum_{i=2}^{n-1}
\left(
x_i\ln\left(\Delta_{J_{i-1}}^{[n-k+1,n]}(b)\right)-
x_i\ln\left(\Delta_{J_{i}}^{[n-k+1,n]}(b)\right)
\right)
\\
&=\ln(q)(x_1+x_2+\cdots+x_{n-k})
+\sum_{i=1}^{n-1}\left( (x_{i+1}-x_i)\ln\left(\Delta_{J_i}^{[n-k+1,n]}(b)\right) \right).
\end{split}
\end{equation*}
Hence~\eqref{e:equivariantaim2} holds, and we are done.
\end{proof}

\section{Action of the vector field: $T^\vee$-equivariant case}
\label{s:actionvectorfieldequivariant}
In this section we prove the formulas needed to complete
the proof of Theorem~\ref{t:equivmain}.

\begin{thm} \label{t:dwdtactionequivariant}
Let $\lambda\in \Pn_{k,n}$. Then we
have:
\begin{enumerate}
\item[(a)]
\begin{equation*}
[q\dWeqdq \p_{\lambda}\omega] =
\left( \sum_{\mu}[\p_{\mu}\omega]+q\sum_{\nu}[\p_{\nu}\omega] \right)
+x_{\lambda}{[\p_{\lambda}\omega]},
\end{equation*}
and
\item[(b)]
\begin{equation*}
-\frac{1}{z}[W p_{\lambda}\omega]=
|\lambda|[p_{\lambda}\omega]
-\frac{n}{z}\left( \sum_{\mu}[p_{\mu}\omega]+q[p_{\nu}\omega]\right)-
\frac{n}{z}x_{\lambda}[p_{\lambda}\omega]
+\frac{(n-k)}{z}\left(\sum_{j=1}^n x_j\right)[p_{\lambda}\omega],
\end{equation*}
\end{enumerate}
where, in each case, $\mu$, $\nu$ are exactly as in the quantum Monk's rule for $\sigma^{\Box}*_q\sigma^{\lambda}$.\end{thm}


To prove Theorem~\ref{t:dwdtactionequivariant},
we compute the action of the vector field $X_{\lambda}$ on
$\Weqq$.
Note that $\Weqq=W_q+\Wtildeq$, where
$$
\Wtildeq=
\ln(q)(x_1+x_2+\cdots +x_{n-k})+\sum_{i=1}^{n-1} (x_{i+1}-x_i) \ln\left(\p_{\mu_i}\right).
$$
Recall that
$$x_{\lambda}=\sum_{i\in \Vertical(\lambda)}x_i.$$
In particular, we have 
$x_{\lambda_{\text{max}}}=x_1+\cdots +x_{n-k}$.

\begin{lem} \label{l:xdifference}
Let $\m\in [1,n]$. Then
$$
\sum_{i=1}^{n-1} c_{\lambda}^{(\m)}(\mu_i)(x_{i+1}-x_i)=x_{\lambda}-(x_{m+1}+\cdots +x_{m+n-k}).
$$
\end{lem}
\begin{proof}
We interpret the subscripts of the $\mu_i$ modulo $n$,
with representatives in $[1,n]$.
We have:
$$
\sum_{i=1}^{n-1} c_{\lambda}^{(\m)}(\mu_i)(x_{i+1}-x_i)=
\sum_{i=1}^n \left(c_{\lambda}^{(\m)}(\mu_{i-1})-c_{\lambda}^{(\m)}(\mu_i)\right)x_i,
$$
noting that $c_{\lambda}^{(\m)}(\mu_n)=c_{\lambda}^{(\m)}(\emptyset)=0$.
Thus the statement in the lemma is equivalent to the statement:
$$
c_{\lambda}^{(\m)}(\mu_{i-1})-c_{\lambda}^{(\m)}(\mu_i)=
\begin{cases}
0, & i\in \Vertical(\lambda),\ i\in [m+1,m+n-k]; \\
-1, & i\not\in \Vertical(\lambda),\ i\in [m+1,m+n-k]; \\
1, & i\in \Vertical(\lambda),\ i\not\in [m+1,m+n-k]; \\
0, & i\not\in \Vertical(\lambda),\ i\not\in [m+1,m+n-k],
\end{cases}
$$
for all $i\in [1,n]$. We first show that this holds for $m=n$, i.e.\ that:
\begin{equation}
\label{e:cdifferencen}
c_{\lambda}(\mu_{i-1})-c_{\lambda}(\mu_i)=
\begin{cases}
0, & i\in \Vertical(\lambda),\ i\in [1,n-k]; \\
-1, & i\not\in \Vertical(\lambda),\ i\in [1,n-k]; \\
1, & i\in \Vertical(\lambda),\ i\not\in [1,n-k]; \\
0, & i\not\in \Vertical(\lambda),\ i\not\in [1,n-k],
\end{cases}
\end{equation}
for all $i\in [1,n]$. Note that Lemma~\ref{l:boundarycoefficients} can be
restated as:
\begin{equation}
\label{e:boundarycoefficients}
c_{\lambda}(\mu_i)=
\begin{cases}
[1,i]\setminus \Vertical(\lambda), & 1\leq i\leq n-k; \\
[i+1,n]\cap \Vertical(\lambda), & n-k+1\leq i\leq n.
\end{cases}
\end{equation}
For $i=1$, we have $c_{\lambda}(\mu_n)-c_{\lambda}(\mu_1)=
-c_{\lambda}(\mu_1)$ and the result follows
from~\eqref{e:boundarycoefficients}. For $2\leq i\leq n-k$,
it follows from~\eqref{e:boundarycoefficients} that
$c_{\lambda}(\mu_{i-1})-c_{\lambda}(\mu_i)=0$ if $i\in \Vertical(\lambda)$
and is equal to $-1$ if $i\not\in\Vertical(\lambda)$, giving the result in this case.
The case $n-k+2\leq i\leq n$ is similar, leaving the case $i=n-k+1$.

We have, using the fact that $|\Vertical(\lambda)|=n-k$,
\begin{equation*}
\begin{split}
c_{\lambda}(\mu_{n-k})-c_{\lambda}(\mu_{n-k+1}) &=
|[1,n-k]\setminus \Vertical(\lambda)|-|[n-k+2,n]\cap \Vertical(\lambda)| \\
&=n-k-|[1,n-k]\cap \Vertical(\lambda)|-|[n-k+2,n]\cap \Vertical(\lambda)| \\
&=
\begin{cases}
n-k-|[1,n-k]\cap \Vertical(\lambda)|-|[n-k+1,n]\cap \Vertical(\lambda)|+1,
& n-k+1\in \Vertical(\lambda); \\
n-k-|[1,n-k]\cap \Vertical(\lambda)|-|[n-k+1,n]\cap \Vertical(\lambda)|,
& n-k+1\not\in \Vertical(\lambda);
\end{cases} \\
&=
\begin{cases} 1, & n-k+1\in \Vertical(\lambda); \\
0, & n-k+1\not\in \Vertical(\lambda);
\end{cases}
\end{split}
\end{equation*}
as required.
For arbitrary $\m\in [1,n]$, we have, recalling the definition of
$c_{\lambda}^{(\m)}$ (equation~\eqref{e:clambdam}) and
using~\eqref{e:cdifferencen},
\begin{equation}
\label{e:cdifferencepart}
\begin{split}
c_{\lambda}^{(\m)}(\mu_{i-1})-c_{\lambda}^{(\m)}(\mu_i) &=
c_{\lambda^{(\m)}}\left(\mu_{i-1}^{(\m)}\right)-c_{\lambda^{(\m)}}\left(\mu_i^{(\m)}\right) \\
&=
\begin{cases}
0, & i-\m\in \Vertical(\lambda^{(\m)}),\ i-\m\in [1,n-k]; \\
-1, & i-\m\not\in \Vertical(\lambda^{(\m)}),\ i-\m\in [1,n-k]; \\
1, & i-\m\in \Vertical(\lambda^{(\m)}),\  i-\m \not\in [1,n-k]; \\
0, & i-\m\not\in \Vertical(\lambda^{(\m)}),\ i-\m \not\in [1,n-k].
\end{cases}
\end{split}
\end{equation}
The result then follows from~\eqref{e:cdifferencepart}, noting that
$\Vertical(\lambda^{(\m)})=\Vertical(\lambda)-m$ (working mod $n$).
\end{proof}

\begin{prop} \label{p:equivariantaction}
Let $\lambda\in \Pn_{k,n}$. Then:
$$X_{\lambda}^{(\m)} \Wtildeq=(x_{\lambda}-(x_{m+1}+\cdots +x_{m+n-k}))\p_{\lambda}.$$
\end{prop}
\begin{proof}
For any $i\in [1,n]$,
$$\p_{\mu_i}\frac{\partial}{\partial \p_{\mu_i}}\ln \p_{\mu_i}=1,$$
Hence, by Lemma~\ref{l:xdifference},
$$X_{\lambda}^{(\m)}\Wtildeq=\sum_{i=1}^{n-1}
c_{\lambda}^{(\m)}(\mu_i)\alpha^{\vee}_ip_{\lambda}=
(x_{\lambda}-(x_{m+1}+\cdots +x_{m+n-k}))p_{\lambda},
$$
\end{proof}

\begin{prop}
\label{p:termsofW}
Let $\lambda\in \Pn_{k,n}$ and $m\in [1,n]$. Then
we have:
\begin{equation*}
\begin{split}
\frac{1}{z}q^{\delta_{mn}}[\frac{\p_{\widehat{L}_m}}{\p_{L_m}}\p_{\lambda}\omega]
&=
\frac{1}{z}
\left(
\sum_{\mu}[\p_{\mu}\omega]+q\sum_{\nu}[\p_{\nu}\omega] \right)
+\frac{1}{z}(x_{\lambda}-(x_{m+1}+\cdots +x_{m+n-k}))[p_{\lambda}\omega] \\
&\quad\quad\quad -c_{\lambda^{(m)}}\left(\emptyset^{(\m)}\right)[p_{\lambda}\omega],
\end{split}
\end{equation*}
\end{prop}
\begin{proof}
Arguing as for equation~\eqref{e:insertion}, we have the following:
\begin{equation} \label{e:equivariantinsertion}
\left[d\left(i_{\xi}\omega\right)\right]+\frac{1}{z}\left[\left(\xi\cdot \Weqq\right)\omega\right]=0,
\end{equation}
for a regular vector field $\xi$ on
$\Xcheck$.
We will apply this in the case $\xi=X_{\lambda}^{(\m)}$, for each
$m\in [1,n]$.

By Lemma~\ref{l:differentialofinsertion}, we have:
\begin{equation}
\label{e:mcasefirstterm}
\left[d\left(i_{X_{\lambda}^{(\m)}}\right)\omega\right]=
-c_{\lambda^{(m)}}\left(\emptyset^{(\m)}\right)[p_{\lambda}\omega],
\end{equation}
giving the first term in~\eqref{e:equivariantinsertion}.

For the second term, we first note that, by
Theorem~\ref{t:nonequivariantaction}, we have:
\begin{equation} \label{e:mcasesecondtermA}
X^{(\m)}_{\lambda}W_q=\left(
\sum_{\mu}\p_{\mu}+q\sum_{\nu}\p_{\nu}\right) - q^{\delta_{mn}}\frac{\p_{\widehat{L}_m}}{\p_{L_m}}\p_{\lambda}.
\end{equation}
By Proposition~\ref{p:equivariantaction}, we have:
\begin{equation}
\label{e:mcasesecondtermB}
X_{\lambda}^{(\m)} \Wtildeq=(x_{\lambda}-(x_{m+1}+\cdots +x_{m+n-k}))\p_{\lambda}.
\end{equation}
Combining~\eqref{e:mcasesecondtermA} and~\eqref{e:mcasesecondtermB}, we obtain:
\begin{equation}
\label{e:mcasesecondterm}
X_{\lambda}^{(\m)} \Weq=
\left(
\sum_{\mu}\p_{\mu}+q\sum_{\nu}\p_{\nu}\right) - q^{\delta_{mn}}\frac{\p_{\widehat{L}_m}}{\p_{L_m}}\p_{\lambda}
+(x_{\lambda}-(x_{m+1}+\cdots +x_{m+n-k}))\p_{\lambda}.
\end{equation}

Substituting~\eqref{e:mcasefirstterm} and~\eqref{e:mcasesecondterm} into~\eqref{e:equivariantinsertion}, we obtain:
$$
\frac{1}{z}q^{\delta_{mn}}\left[\frac{\p_{\widehat{L}_m}}{\p_{L_m}}\p_{\lambda}\omega\right]
=
\frac{1}{z}
\left(
\sum_{\mu}[\p_{\mu}\omega]+q\sum_{\nu}[\p_{\nu}\omega] \right)
+\frac{1}{z}(x_{\lambda}-(x_{m+1}+\cdots +x_{m+n-k}))[p_{\lambda}\omega]
-c_{\lambda^{(m)}}\left(\emptyset^{(\m)}\right)[p_{\lambda}\omega],
$$
as required.
\end{proof}

We can now prove the following enhanced version of Proposition~\ref{p:eqLiecomp}.

\begin{prop}\label{p:eqLiecompEnhanced}
The Jacobi ring of $(\Xcheck,\Weq)$ is isomorphic to the equivariant quantum cohomology ring $qH^*_{T^\vee}(X,\C)[q\inv]$ via an isomorphism which satisfies
\[
[\p_\lambda]\mapsto\sigma^\lambda_{T^\vee}.
\]
Moreover this isomorphism sends the summands of $W$ to equivariant fundamental classes of $T^\vee$-invariant divisors $X^{\ydiagram{1}}(i)$ from Section~\ref{s:eqqH}.  Namely for $i\ne n-k$, 
\[
\left[\frac{\p_{\widehat\mu_i}}{\p_{\mu_i}}\right]\mapsto [X^{\ydiagram{1}}(i)]_{T^\vee},
\]
and for $i=n-k$
\[
q\left[\frac{\p_{\widehat\mu_{n-k}}}{\p_{\mu_{n-k}}}\right]\mapsto [ X^{\ydiagram{1}}(n-k)]_{T^\vee}
=\widetilde{\sigma}^{\, \ydiagram{1}}_{T^\vee}
\]
noting that $\widetilde X^{\ydiagram{1}}$ is the $(n-k)$-th shift of $X^{\ydiagram{1}}$.
\end{prop}

\begin{proof}[Proof of Proposition~\ref{p:eqLiecomp}]
By the comparison result, Proposition~\ref{p:EquivariantComparison}, together with \cite[Theorem~4.1]{Rie:MSgen} and Peterson's theory, see  \cite[Corollary~4.2]{Rie:MSgen}, we know that the Jacobi ring of $(\Xcheck,\Weq_q)$ is isomorphic to the quantum cohomology $qH^*_{T^\vee}(X,\C)[q\inv]$ via an isomorphism of graded (compare \eqref{e:gradingastorusaction}) rings,  which fixes the $x_i$.
Moreover the image of $\p_\lambda$ is $\sigma^\lambda_{T^\vee}$ up to possible summands in the ideal generated by the equivariant parameters, by Proposition~\ref{p:SchubertPlucker}.  Therefore the $\p_\lambda$ form an additive basis of the Jacobi ring as module over $\C[x_1,\dotsc, x_n,q,q\inv]$. 

 Consider \eqref{e:mcasesecondterm} from above. Recall that $\p_{L_k}=\p_\emptyset=1$, see \eqref{e:Lihatterm}. Setting $m=k$ we obtain the following relation in the Jacobi ring of $\Weq_q$,
\begin{equation}\label{e:JacobiMonk}
 {\p_{\ydiagram{1}}}\ \p_\lambda=\sum_{\mu}\p_{\mu}+q\sum_{\nu}\p_{\nu}
+(x_{\lambda}-(x_{k+1}+\cdots +x_{n}))\p_{\lambda}.
\end{equation}
For example (assuming $n-k>1$), we have
\[
(\p_{\ydiagram{1}}+(x_{k+1}+\cdots +x_{n})) p_{\ydiagram{1}}= \p_{\ydiagram{2}}+\p_{\ydiagram{1,1}}+x_\lambda\p_{\ydiagram{1}},
\]
in the Jacobi ring, which we can compare with the relation from the equivariant quantum Monk's rule,
\[
\zeta\star_{q,x} \sigma^{\ydiagram{1}}_{T^\vee}= \sigma^{\ydiagram{2}}_{T^\vee}+\sigma^{\ydiagram{1,1}}_{T^\vee}+x_\lambda\sigma^{\ydiagram{1}}_{T^\vee}.
\]
It now follows from \eqref{e:JacobiMonk} by \cite[Corollary~7.1]{Mihalcea} that the isomorphism from the Jacobi ring to quantum cohomology must take $p_\lambda$ to $\sigma^{\lambda}_{T^\vee}$.

In particular $\p_{\ydiagram{1}}$ maps to $[X^{\ydiagram{1}}]_{T^\vee}$. The equation~\eqref{e:mcasesecondterm} for $\lambda=\emptyset$ implies the identity in the Jacobi ring, if $m\ne n$,
\[
\frac{\p_{\widehat L_m}}{\p_{L_m}}=p_{\ydiagram{1}}+x_{\emptyset}-(x_{m+1}+\dotsc + x_{m+n-k}),
\]
and for $m=n$ the identity,
\[
q\frac{\p_{\widehat L_n}}{\p_{L_n}}=p_{\ydiagram{1}}+(x_{k+1}+\dotsc+x_n)-(x_{1}+\dotsc + x_{n-k}).
\]
Therefore under the isomorphism with quantum cohomology we have
\[
\frac{\p_{\widehat L_m}}{\p_{L_m}}\mapsto\zeta - (x_{m+1}+\dotsc + x_{m+n-k})
\]
and
\[
q\frac{\p_{\widehat L_n}}{\p_{L_n}}\ \mapsto \ \zeta - (x_{1}+\dotsc + x_{n-k}).
\]
Since the $\Z/n\Z$-action on $X=Gr_{n-k}(\C^*)$ comes from the cyclic permutation of the basis $v_1,\dotsc, v_n$ of $\C^n$, and fixes  $\zeta$, the equivariant Chern class of $\mathcal O(1)$,  we have that the fundamental class $[X^{\ydiagram{1}}(m)]_{T^\vee}$ is related to 
\[
[X^{\ydiagram{1}}]_{T^\vee}=\zeta - (x_{k+1}+\dotsc + x_{n})
\] 
by cyclic permutation of the equivariant parameters. Therefore
\[
[X^{\ydiagram{1}}(i)]_{T^\vee}=\zeta - (x_{k+i+1}+\dotsc + x_{n+i})
\] 
with indices taken modulo $n$, and this agrees with the image of
\[
\frac{\p_{\widehat{\mu_i}}}{\p_{\mu_i}}=\frac{\p_{\widehat L_{k+i}}}{\p_{L_{k+i}}}
\]
respectively of 
\[q\frac{\p_{\widehat{\mu_{n-k}}}}{\p_{\mu_{n-k}}}=q\frac{\p_{\widehat L_{n}}}{\p_{L_{n}}},
\]
if $i=n-k$, which was to be proved. 
\end{proof}

Finally, we complete the proof of Theorem~\ref{t:dwdtactionequivariant}.
\begin{proof}[Proof of Theorem~\ref{t:dwdtactionequivariant}]
Putting $m=n$ in the statement in Proposition~\ref{p:termsofW}, we obtain:
$$
q[\frac{\p_{\widehat{L}_n}}{\p_{L_n}}\p_{\lambda}\omega]
=
\left(
\sum_{\mu}[\p_{\mu}\omega]+q\sum_{\nu}[\p_{\nu}\omega] \right)
+(x_{\lambda}-x_{\lambda_{\text{max}}})[p_{\lambda}\omega].
$$
We also have:
\begin{equation*}
\begin{split}
q\dWeqdq &= q\dWdq +x_{\lambda_{\text{max}}} \\
&= 
q\frac{\p_{\widehat{L}_n}}{\p_{L_n}}+x_{\lambda_{\text{max}}},
\end{split}
\end{equation*}
and part (a) of Theorem~\ref{t:dwdtactionequivariant} follows.

For part (b), we use Lemma~\ref{l:boundarysum} and the sum
of the cases $m=1,\ldots ,n$ in Proposition~\ref{p:termsofW}.
\end{proof}

\begin{proof}[Proof of Theorem~\ref{t:equivmain}]
The free basis lemma, Lemma~\ref{l:freebasis2},  also has an equivariant version.
This is just obtained by replacing $\C[ z^{\pm 1}, q^{\pm 1}]$ by
$H^*_{T^\vee}(pt)[ z^{\pm 1},q^{\pm 1}]$, which does not affect the proof. Theorem~\ref{t:equivmain} now follows from Theorem~\ref{t:dwdtactionequivariant} and the equivariant free basis lemma.
\end{proof}

\end{document}